\documentclass[12pt,reqno]{amsart}

\usepackage[utf8]{inputenc}
\usepackage[T1]{fontenc}
\usepackage{amsmath,amssymb,amsthm}
\usepackage[a4paper,margin=28mm]{geometry}
\usepackage{setspace}
\usepackage{enumitem}
\usepackage{hyperref}
\hypersetup{
  colorlinks,
  citecolor=blue,
  linkcolor=blue,
  urlcolor=blue, 
}
\usepackage{marginnote}
\usepackage{xcolor}
\usepackage{upgreek}
\usepackage{tikz}
\usetikzlibrary{calc,intersections,patterns}
\usepgflibrary{shadings}
\usepackage{mathtools}
\usepackage{esint}
\usepackage{verbatim}
\usepackage{mathrsfs}
\usepackage{stmaryrd}
\usepackage[
  colorinlistoftodos,
  backgroundcolor=green!20!white,
  bordercolor=green!75!black,
  linecolor=green!75!black
]{todonotes}
\usepackage{cleveref}

\numberwithin{equation}{section}
\theoremstyle{plain}
\newtheorem{theorem}{Theorem}
\newtheorem{lemma}{Lemma}[section]
\newtheorem{proposition}[lemma]{Proposition}
\newtheorem{corollary}[lemma]{Corollary}

\theoremstyle{definition}
\newtheorem{definition}{Definition}

\newtheorem{remark}{Remark}
\newtheorem{example}{Example}

\newcommand{\Z}{\mathbf{Z}}
\newcommand{\C}{\mathbf{C}}
\newcommand{\R}{\mathbf{R}}
\newcommand{\Sp}{\mathbf{S}}
\newcommand{\vol}{\mathrm{vol}}
\newcommand{\dist}{\operatorname{dist}}
\newcommand{\Tub}{\mathrm{Tub}}
\newcommand{\eps}{\varepsilon}
\newcommand{\inn}{\mathrm{in}}
\newcommand{\out}{\mathrm{out}}

\newcommand{\SU}{\operatorname{SU}}
\newcommand{\su}{\mathfrak{su}}
\newcommand{\so}{\mathfrak{so}}
\newcommand{\U}{\operatorname{U}}
\DeclareMathOperator{\Ad}{Ad}
\DeclareMathOperator{\ad}{ad}
\DeclareMathOperator{\tr}{tr}

\DeclareMathOperator{\supp}{supp}

\newcommand{\restr}{\mathbin{\vrule height 1.4ex depth -0.3ex width .07ex
  \vrule height .07ex depth 0ex width .7ex}}

\begin{document}

\title{Nonabelian Higgs fields from minimal submanifolds}
\author{Marco Badran}
\address{Bocconi University, Via Roentgen 1, 20136 Milan, Italy}
 	\email{marco.badran@unibocconi.it}

\begin{abstract}
We construct $\mathrm{SU}(2)$ Yang--Mills--Higgs fields on a broad class of noncompact Riemannian manifolds, concentrating in the large-mass limit along prescribed nondegenerate minimal submanifolds of codimension three. The solutions have uniformly bounded rescaled energy. On closed manifolds with $b_1(M)=0$, adding a suitably scaled Higgs potential gives an analogous concentration result. Our construction completes one direction of the correspondence proposed by Parise, Pigati and Stern \cite[Question~1.1]{Parise-Pigati-Stern2026}, realising prescribed minimal submanifolds of codimension three as concentration sets of $\SU(2)$ Yang--Mills--Higgs fields.
\end{abstract}

\maketitle

\setcounter{tocdepth}{1}
\tableofcontents

\section{Introduction}
The emergence of minimal hypersurfaces from nonlinear elliptic equations has been a central theme in geometric analysis since the 1970s. The work of Modica and Mortola \cite{Modica-Mortola1977} revealed how hypersurface area arises as the limit of diffuse phase-transition energies. This insight initiated a programme that has since established a rich correspondence between minimal hypersurfaces and the transition layers of solutions to the Allen–Cahn equation \cite{Hutchinson-Tonegawa2000,Guaraco2018,Gaspar-Guaraco2018,Chodosh-Mantoulidis2023,Chodosh-Mantoulidis2020}.

\smallskip

Gauge theory provides a natural setting for extending this correspondence to higher codimension. In codimension two, the self-dual abelian Higgs model, with structure group $\U(1)$, plays a role analogous to that of the Allen–Cahn equation, relating critical points of its energy to minimal submanifolds through a singular concentration limit \cite{Pigati-Stern2021,Parise-Pigati-Stern2024a,Parise-Pigati-Stern2024b,BadranDelPino23,BadranDelPino24}.

\smallskip

In codimension three, the corresponding theory is the nonabelian Yang–Mills–Higgs model with structure group $\SU(2)$. Parise, Pigati and Stern \cite{Parise-Pigati-Stern2026} recently proved that suitably rescaled Yang–Mills–Higgs energies converge variationally to the codimension-three volume functional, laying the foundations for a correspondence between nonabelian Higgs fields and minimal submanifolds. The present paper contributes to this programme by proving the following theorem, stated here informally:

\medskip

\noindent{\bf Theorem A.} \emph{On a broad class of closed and noncompact Riemannian manifolds, the $\SU(2)$ Yang--Mills--Higgs equations---with a suitably scaled Higgs potential in the closed case---admit families of solutions concentrating, in the large-mass limit, along any prescribed closed, topologically admissible, nondegenerate minimal submanifold of codimension three, with uniformly bounded rescaled energy.}

\medskip

We first introduce the model and the geometric hypotheses, deferring the precise statements to \Cref{thm: main noncompact,thm: main} below.

\medskip

Let $(M,g)$ be an oriented Riemannian manifold of dimension $n\geq 4$. The $\SU(2)$-Higgs  model couples a connection $A$ on the trivial bundle $P\coloneqq M\times\SU(2)\to M$ to a Higgs field $\Psi$, a section of $E=\ad P\coloneqq P\times_{\Ad}\su(2)\simeq M\times\su(2)$. In a fixed trivialisation, $A\in\Omega^1(M,\su(2))$ and $\Psi\in\Omega^0(M,\su(2))$, with covariant derivative and curvature\footnote{We use the notation $[A\wedge B]$ as the wedge product of forms combined with the Lie bracket in $\su(2)$, that is, $[A\wedge B](X,Y)\coloneqq [A(X),B(Y)]-[A(Y),B(X)]$. In particular $[A\wedge A](X,Y)=2[A(X),A(Y)]$.}
\begin{equation*}
	d_A\Psi\coloneqq d\Psi+[A,\Psi],\qquad F_A\coloneqq dA+\frac12[A\wedge A].
\end{equation*}
The Yang--Mills--Higgs energy is
\begin{equation}\label{eq: energy no_pot}
	E(\Psi,A)\coloneqq \frac12\int_M\left(|d_A\Psi|^2+|F_A|^2\right)d\mathrm{vol}_g.
\end{equation}

On $\R^3$, the model admits special finite-energy solutions, called \emph{monopoles}, satisfying the first-order Bogomolnyi equation \cite{Bogomolnyi1976,Jaffe-Taubes1980}
\begin{equation}\label{eq: monopole}
	*F_A=-d_A\Psi.
\end{equation}
This equation characterises equality in the lower bound given by the Bogomolnyi decomposition
\begin{equation*}
	E(\Psi,A)=\frac12\int_{\R^3}|*F_A+d_A\Psi|^2\,dx+2\pi mk,
\end{equation*}
where $m$ and $k$ are respectively the \emph{mass} and the \emph{charge} of $(\Psi,A)$, given by the well-defined quantities
\begin{equation*}
	m\coloneqq\lim_{|x|\to\infty}|\Psi(x)|>0,\qquad k\coloneqq\lim_{R\to\infty}\deg\left(\left.\frac{\Psi}{|\Psi|}\right|_{\partial B_R}\right)\in\Z.
\end{equation*}
Monopoles therefore minimise the energy at fixed mass and charge. The explicit charge-one solution of Prasad and Sommerfield \cite{Prasad-Sommerfield1975} will be used below. Taubes constructed higher-charge monopoles by gluing \cite{Taubes1981Monopoles}.

\medskip

On a noncompact manifold, one similarly prescribes $|\Psi|\to m$ at infinity. The large mass limit $m\to +\infty$ is the regime in which the functional $\Gamma$-converges to the $(n-3)$-volume by \cite{Parise-Pigati-Stern2026}. A more convenient formulation follows by setting $\eps\coloneqq m^{-1}$ and $\Phi\coloneqq\eps\Psi$, so that $|\Phi|\to1$ at infinity and the normalised energy is given by
\begin{equation}\label{eq: energy nc_eps}
	E_\eps(\Phi,A)\coloneqq \eps E(\eps^{-1}\Phi,A)=\frac12\int_M\left(\frac1\eps|d_A\Phi|^2+\eps|F_A|^2\right)d\mathrm{vol}_g.
\end{equation}

On a closed manifold, varying $\Phi$ in its own direction at a critical point of \eqref{eq: energy no_pot} gives $\int_M|d_A\Phi|^2\,d\mathrm{vol}_g=0$. The Higgs field is therefore parallel and cannot form monopole cores. To allow concentration, we add a potential as in \cite{Parise-Pigati-Stern2026} favouring unit-length Higgs fields and consider
\begin{equation}\label{eq: energy}
	E_\eps(\Phi,A)\coloneqq \frac12\int_M\left(\frac{1}{\eps}|d_A\Phi|^2+\eps|F_A|^2+\frac{1}{\delta_\eps}(1-|\Phi|^2)^2\right)d\mathrm{vol}_g,
\end{equation}
where the rate $\delta_\eps>0$ is chosen, for fixed exponents $0<\tau<\tau'<\frac12$ and a sufficiently large constant $C_*$, so that
\begin{equation}\label{eq: deltaeps}
 C_*\eps^3\leq\delta_\eps\leq\eps^{1+2\tau'}.
\end{equation}
The potential replaces the condition at infinity by penalising deviations from $|\Phi|=1$.

\medskip

We treat two geometric settings, referred to throughout as the ``compact'' and ``noncompact'' cases. On noncompact manifolds, we study the energy \eqref{eq: energy nc_eps} and its associated second-order Yang--Mills--Higgs equations. On closed manifolds, the obstruction to monopole concentration leads us to the perturbed energy \eqref{eq: energy}. As in classical Ginzburg--Landau theory, the potential favours unit-length Higgs fields away from the cores, where topological defects are located. We now specify the geometric assumptions on the ambient manifold and the prescribed concentration set.

\begin{definition}[Compact admissible ambient space]\label{def: admissible amb} We say that a closed Riemannian manifold $(M,g)$ is an admissible ambient space if it is a smooth, complete, connected and oriented manifold of dimension $n\geq 4$, satisfying $b_1(M)=0$.
\end{definition}

For the noncompact version of admissibility, we need the following definition. On a noncompact manifold $(M,g)$ choose a point $o\in M$ and define $R\colon M\to (0,+\infty)$ by
\begin{equation*}
	R(x)\coloneqq (1+\dist_g(x,o)^2)^\frac{1}{2}.
\end{equation*}
For $\tau\in(0,1)$, $\nu>0$ and integers $k,p\geq0$, consider the spaces $C^{k,\tau}_{\nu}(M;\Lambda^pT^*M)$ of $C^{k,\tau}$ $p$-forms satisfying
\begin{equation*}
	\|\omega\|_{C^{k,\tau}_{\nu}}\coloneqq\sup_{x\in M}R(x)^\nu\|\omega\|_{C^{k,\tau}(B_{r_0}(x))}<\infty,
\end{equation*}
where $r_0>0$ is the fixed radius in the bounded geometry condition below.

\begin{definition}[Noncompact admissible ambient space]\label{def: admissible amb noncomp} We say that a noncompact Riemannian manifold $(M,g)$ is an admissible ambient space if it is a smooth, complete, connected and oriented manifold of dimension $n\geq 4$, without boundary and satisfying
\begin{itemize}
	\item bounded geometry: there are constants $r_0,c_1,c_2>0$ such that, in suitable coordinates on each $B_{r_0}(x_0)$,
	\begin{equation*}
		c_1^{-1}I\leq (g_{ij}(x))_{ij}\leq c_1I\quad(x\in B_{r_0}(x_0)),\qquad \|g_{ij}\|_{C^4(B_{r_0}(x_0))}\leq c_2
	\end{equation*}
	uniformly on $x_0\in M$;
	\item there exists a compact set $K\subset M$ such that $H^2_{\mathrm{dR}}(M\setminus K)=0$;
	\item there exists a positive minimal scalar Green function $G$ satisfying $\Delta_{g,x}G(\cdot,y)=\delta_y$ and $G(x,y)\to0$ as $x\to\infty$, for every $y\in M$;
	\item there exists $\nu>0$ such that, for $p=0,1$ and every $f\in C^{0,\tau}_{\nu+2}(M;\Lambda^pT^*M)$, the equation
	\begin{equation*}
		\Delta_pu=f
	\end{equation*}
	has a unique solution $u\in C^{2,\tau}_{\nu}(M;\Lambda^pT^*M)$, satisfying
	\begin{equation*}
		\|u\|_{C^{2,\tau}_{\nu}}\leq C\|f\|_{C^{0,\tau}_{\nu+2}},
	\end{equation*}
	where $C$ is independent of $f$;
	\item the exponent $\nu$ satisfies
	\begin{equation}\label{eq: noncompact weighted volume}
		\int_M R(x)^{-2\nu-1}\,d\mathrm{vol}_g(x)<\infty.
	\end{equation}
\end{itemize}
\end{definition}

\begin{example}
The following are examples of spaces satisfying \Cref{def: admissible amb noncomp}.
\begin{itemize}
 \item Euclidean space $(\R^n,g_{\mathrm{Eucl}})$, for $n\geq4$, with any $\nu\in(\frac{n-1}{2},n-2)$.

 \item The Eguchi--Hanson space $(T^*\Sp^2,g_{\mathrm{EH}})$ \cite{Eguchi-Hanson1979}, with $\frac32<\nu<2$. This complete, nonflat Ricci-flat four-manifold is asymptotic to $\R^4/\{\pm1\}$.
 \item The products $\R^d\times\Sp^k$ with the Euclidean and round metrics, for $d\geq4$ and $k\geq3$, with $\nu\in(\frac{d-1}{2},d-2)$.
 \item The punctured manifolds $N\setminus\{p\}$, where $N$ is a closed, connected, oriented $n$-manifold, $n\geq4$, with $b_1(N)=0$, equipped with a complete metric whose end is isometric to the complement of a Euclidean ball. Here, we can choose again $\nu\in(\frac{n-1}{2},n-2)$.
\end{itemize}
\end{example}

\begin{remark}
For a manifold with one Euclidean end and $b_1(M)=0$, every smooth change of metric supported in a compact set preserves ambient admissibility. Thus the geometry of the compact interior can be varied freely. This leaves considerable freedom in choosing the geometry near the concentration set.
\end{remark}

\begin{definition}[Admissible submanifold]\label{def: admissible} Let $(M,g)$ be either a compact or noncompact admissible ambient space.
	We say that $\Gamma\subset M$ is admissible if 
	\begin{itemize}
		\item it is a smooth, oriented, closed, embedded minimal submanifold of codimension $3$;
		\item there exists a $C^\infty$ map $F\colon M\to\R^3$ such that $\Gamma=F^{-1}(0)$, and $0$ is a regular value for $F$;
		\item it is nondegenerate, in the sense that the Jacobi operator
		\begin{equation*}
			J_{\Gamma}\colon C^{2,\tau}(N\Gamma)\to C^{0,\tau}(N\Gamma)
		\end{equation*}
		has trivial kernel.
		\item If $M$ is noncompact, $\Gamma\subset K$.
	\end{itemize}
\end{definition}

Our building blocks are the standard radial, mass-one, charge-one critical points $(\Phi_\lambda,A_\lambda)$ of the model energy 
\begin{equation}\label{eq: Elambda}
	E_\lambda(\Phi,A)\coloneqq \frac12\int_{\R^3}\left(|d_A\Phi|^2+|F_A|^2+\frac{\lambda}{2}(1-|\Phi|^2)^2\right)\,dy,
\end{equation}
that is, the Bogomolnyi--Prasad--Sommerfield (BPS) monopole for $\lambda=0$ and the 't Hooft--Polyakov monopoles for $\lambda>0$, see \Cref{sec: BPStHP}. Their energies will be denoted by $c_\lambda\coloneqq E_\lambda(\Phi_\lambda,A_\lambda)$, with $c_0=2\pi$. We use the normalised fields and energies \eqref{eq: energy nc_eps} and \eqref{eq: energy} throughout, and set
\begin{equation*}
\lambda_\eps\coloneqq\begin{cases}
0,&M\text{ noncompact},\\
2\eps^3/\delta_\eps,&M\text{ compact},
\end{cases}
\end{equation*}
where the compact rate satisfies \eqref{eq: deltaeps}.

To describe the profiles near $\Gamma$, fix an orthonormal frame $(\nu_1,\nu_2,\nu_3)$ of $N\Gamma$, supplied by admissibility. For normal fields $h_\eps\to0$ in $C^{2,\tau}(N\Gamma)$, define the shifted normal rescaling
\begin{equation}\label{eq: introductory normal rescaling}
f_{\eps,q}(y)\coloneqq\exp_q\left(h_\eps(q)+\eps y^a\nu_a(q)\right),\qquad q\in\Gamma,\quad y\in B_R\subset\R^3,
\end{equation}
where $\exp_q$ is the Riemannian exponential map and $R>0$ is fixed. Convergence to the normal profile means that such fields $h_\eps$ and smooth gauge choices exist for which, for every $R>0$,
\begin{equation}\label{eq: normal profile convergence}
\sup_{q\in\Gamma}\left(\|f_{\eps,q}^*\Phi_\eps-\Phi_{\lambda_\eps}\|_{C^{2,\tau}(B_R)}+\|f_{\eps,q}^*A_\eps-A_{\lambda_\eps}\|_{C^{2,\tau}(B_R)}\right)\longrightarrow0.
\end{equation}
The derivatives here are in the normal variable $y$; the pullback already includes the scaling of the connection one-form. Writing $e_\eps(\Phi,A)$ for the energy density in \eqref{eq: energy nc_eps} or \eqref{eq: energy}, respectively, energy concentration is expressed by
\begin{equation}\label{eq: energy concentration}
e_\eps(\Phi_\eps,A_\eps)\,d\mathrm{vol}_g-c_{\lambda_\eps}\mathcal H^{n-3}\restr_\Gamma\stackrel{*}{\rightharpoonup}0
\end{equation}
as Radon measures. The coefficient depends on the potential rate: if $\lambda_\eps\to\bar\lambda$, then $c_{\lambda_\eps}\to c_{\bar\lambda}$. In particular, $\delta_\eps/\eps^3\to\infty$ gives $c_{\lambda_\eps}\to2\pi$, recovering the BPS energy density.

\begin{theorem}\label{thm: main noncompact}
Let $\Gamma$ be an admissible submanifold in a noncompact admissible ambient space $(M,g)$. For every sufficiently small $\eps>0$, there are finite-energy critical points $(\Phi_\eps,A_\eps)$ of \eqref{eq: energy nc_eps}, with $|\Phi_\eps|\to1$ at infinity, satisfying \eqref{eq: normal profile convergence} and \eqref{eq: energy concentration}. Moreover,
\begin{equation*}
\lim_{\eps\to0}E_\eps(\Phi_\eps,A_\eps)=2\pi\,\mathscr{H}^{n-3}(\Gamma).
\end{equation*}
\end{theorem}
\begin{theorem}\label{thm: main}
Let $\Gamma$ be an admissible submanifold in a compact admissible ambient space $(M,g)$, and fix a rate $\delta_\eps$ as in \eqref{eq: deltaeps}. For every sufficiently small $\eps>0$, there are critical points $(\Phi_\eps,A_\eps)$ of \eqref{eq: energy} satisfying \eqref{eq: normal profile convergence} and \eqref{eq: energy concentration}. Moreover,
\begin{equation*}
E_\eps(\Phi_\eps,A_\eps)-c_{\lambda_\eps}\,\mathscr{H}^{n-3}(\Gamma)\longrightarrow0.
\end{equation*}
\end{theorem}
We develop the two proofs in parallel: the main difference lies in the approximate solutions, while the linear theory and nonlinear scheme have a common structure.

\begin{remark}
We make three comments on the admissibility assumptions in \Cref{def: admissible amb,def: admissible amb noncomp,def: admissible}.
\begin{itemize}
	\item In the compact case, the hypothesis $b_1(M)=0$ appears to be technical, but we do not know how to remove it. It would be interesting to determine whether a further reduction, possibly combined with a suitable geometric modification of the initial approximation, could extend the construction to the general case.
	\item The existence of a global defining map $F$ reflects the topology expected of a multiplicity-one concentration set. In this sense, the assumption is consistent with the class of limiting submanifolds that the construction is intended to describe.
	\item The assumption $H^2_{\mathrm{dR}}(M\setminus K)=0$ in \Cref{def: admissible amb noncomp} can be replaced by the vanishing of a single cohomology class. More precisely, it suffices that the defining map $F$ can be chosen so that, for some compact neighbourhood $K$ of $\Gamma$,
	\begin{equation*}
		\left[\left.\left(\frac{F}{|F|}\right)^*\omega_{\Sp^2}\right|_{M\setminus K}\right]=0\qquad\text{in }H^2_{\mathrm{dR}}(M\setminus K),
	\end{equation*}
	where $\omega_{\Sp^2}$ is the standard area form on $\Sp^2$.
\end{itemize}
\end{remark}

\subsection{Related works}
In dimension three, the first-order Bogomolnyi equation \eqref{eq: monopole} provides a distinguished class of critical points of \eqref{eq: energy no_pot}. On asymptotically conical three-manifolds with $b_2(M)=0$, Oliveira \cite{Oliveira2016AC} constructed large-mass monopoles concentrating at prescribed points by gluing methods. Fadel and Oliveira \cite{Fadel-Oliveira2019} subsequently analysed sequences of fixed-charge monopoles in the large-mass limit, identifying the finite concentration set with the limiting Higgs zeros and describing the Euclidean monopole bubbles that account for the concentrated energy.

\smallskip

For the model with a Higgs potential, Cheng, Fadel and Lara \cite{Cheng-Fadel-Lara2025} constructed nontrivial critical points on closed three-manifolds by min--max, with uniformly bounded rescaled energy. The assumption $b_1(M)=0$ ensures that these solutions are irreducible. 

\smallskip

Related higher-dimensional results concern first-order monopole equations on manifolds with special holonomy. Oliveira \cite{Oliveira2016Stenzel} constructed invariant Calabi--Yau monopoles on the Stenzel manifold $T^*S^3$. Li \cite{Li2025Monopoles} developed a large-mass limit theory on asymptotically conical $G_2$-manifolds and Calabi--Yau three-folds, extracting singular abelian limits along calibrated cycles and proving energy identities for normal monopole bubbles. Fadel and Oliveira \cite{Fadel-Oliveira2026} further related calibrated energy concentration, limiting Higgs zeros and abelianisation. These first-order systems imply the potential-free Yang--Mills--Higgs equations but impose additional geometric constraints: the limiting cycles are coassociative or special Lagrangian, respectively. 

\medskip

On general Riemannian ambients, Parise, Pigati and Stern \cite{Parise-Pigati-Stern2026} established the variational convergence to codimension-three area recalled above. In the abelian $\U(1)$ model, the author and M.~del Pino \cite{BadranDelPino23,BadranDelPino24} constructed solutions concentrating along prescribed nondegenerate minimal submanifolds of codimension two by gluing, with precise control of their local profiles. A complementary approach was developed by De Philippis and Pigati \cite{DePhilippis-Pigati2024}, whose purely variational construction establishes concentration at the level of energy measures and applies to several related phase-transition energies. These two approaches provide different forms of control: gluing resolves the local structure of the solutions, while the variational method flexibly treats a broader class of functionals within a common framework.

\subsection{Proof strategy and comparison with models in lower codimension}

The proofs of \Cref{thm: main noncompact,thm: main} follow the geometric gluing strategy developed for the Allen--Cahn equation \cite{Pacard-Ritore2003,delPino-Kowalczyk-Wei2011,delPino-Kowalczyk-Wei2013,Badran-delPino-Guaraco2026}, the abelian Higgs model \cite{BadranDelPino23,BadranDelPino24} and many other problems. We outline the steps common to these constructions.

\medskip

We first construct an approximate solution by placing $\eps$-rescaled model solutions in the normal fibres of $\Gamma$ and extending them to the ambient manifold using cutoffs and the geometric and topological assumptions. The compact and noncompact constructions are carried out in \Cref{sec: approx compact,sec: noncompact approximation}, and their errors are estimated in \Cref{sec: approximation errors}.

\medskip

We then seek an exact solution as a small perturbation of this approximation, measured in the norms of \Cref{sec: norms}. Taylor expansion separates the linearised operator from a quadratic remainder, leading to the nonlinear scheme in \Cref{sec: nonlinear scheme}. The quadratic estimate needed for the contraction is proved in \Cref{sec: quadratic estimate proof}.

\medskip

The symmetries of the model, in particular translation invariance, produce approximate kernel elements that obstruct a uniform inverse on the full space. An infinite-dimensional Lyapunov--Schmidt reduction separates these modes from their complement and allows the centre of the normal profile to vary along $\Gamma$. The uniform inverse on the complement is established in \Cref{sec: invertibility}; the contraction in \Cref{sec: complementary nonlinear equation} then produces a small correction for fixed modulation parameters.

\medskip

Projection onto the translation modes yields an equation for the normal displacement, whose leading operator is the Jacobi operator $J_\Gamma$. The estimates in \Cref{sec: reduced equation proof} control the remainder, while nondegeneracy of $\Gamma$ provides the inverse needed for the contraction in \Cref{sec: reduced equation}. This determines the position of the concentrating profile.

\medskip

The nonabelian model introduces several additional difficulties with respect to the Allen--Cahn and abelian Higgs case.

\begin{itemize}[leftmargin=*,itemsep=0.5\baselineskip]
\item Away from Higgs zeros, we split the Lie algebra components into a \emph{longitudinal} part, parallel to the Higgs field, and a \emph{transverse} part, orthogonal to it. The initial approximation leaves an $O(1)$ error in the exterior longitudinal connection equation, as shown in \eqref{eq: initial error}. We cancel this error by solving a Hodge equation for a longitudinal correction and matching it to a smooth replacement near the core; see \Cref{sec: approx compact,sec: noncompact approximation}.
\item After normal rescaling, the transverse equations have a positive zeroth-order term and are schematically of the form $-\Delta+\eps^{-2}$, as in Allen--Cahn or abelian Higgs, whereas the longitudinal connection equation is of Poisson type, $-\Delta$. The absence of this positive term allows contributions from different scales to accumulate, making uniform estimates more delicate. The weighted dyadic norms of \Cref{sec: norms} distinguish these components and sum the longitudinal source contributions over scales. They provide the framework for the inverse estimates in \Cref{sec: invertibility}.
\item The Poisson character of the longitudinal equations described in the preceding item makes the feedback from the outer problem into the inner equation harder to control. We therefore seek a single global correction and construct a projected right inverse on the ambient manifold, rather than solving a coupled inner--outer system for separate corrections; see \Cref{sec: invertibility}. This streamlines the linear theory, but makes the reduction somewhat more delicate: the reduced coefficients must account for the global correction and its longitudinal tail, as analysed in \Cref{sec: reduced equation proof}.
\item  Besides translations, the kernel of the normal linearised operator contains a phase mode, described in \Cref{sec: second variation}. Its amplitude $\theta$ varies along $\Gamma$ and satisfies a reduced equation with leading operator $\Delta_\Gamma$; see \Cref{sec: reduced equation}. For connected $\Gamma$, we invert this operator on functions of mean zero, leaving a constant phase residual. A generalisation of the gauge-compatibility argument used in the abelian Higgs model \cite{BadranDelPino23,BadranDelPino24} eliminates this residual and recovers the original equations; see \Cref{sec: technical gauge}. The extension to several components is discussed in Appendix~\ref{app: disconnected components}.
\end{itemize}

\section{Setting}\label{sec:setting}
In what follows we fix the following constants, which should be interpreted as absolute parameters in the construction.
\begin{itemize}
	\item $0<\tau<\frac12$: the H\"older exponent of the Banach spaces, see \Cref{sec: norms}.
	\item $\lambda_*>0$ is the constant predicted by \Cref{lem: normal weighted coercivity}.
	\item $C_*>2/\lambda_*$ determines the lower bound of the rate for $\delta_\eps$, see \eqref{eq: deltaeps}.
	\item $K>0$ a constant determining the size of the parameters $h,\theta$, see \eqref{eq: assumptions htheta}.
\end{itemize}
We will often use the notation $A\lesssim B$ to say that $A\leq CB$ for a constant $C>0$ silently depending on the parameters above.

\smallskip

The Lie algebra $\su(2)$ denotes the space of traceless skew-Hermitian matrices 
\begin{equation*}
	\su(2)=\{H\in\mathrm{Mat}_2(\C): \tr H=0,\quad H^\dagger =-H\}
\end{equation*}
with basis
\begin{equation*}
	T_1\coloneqq \begin{pmatrix} i & 0 \\ 0 & -i \end{pmatrix},\qquad
	T_2\coloneqq \begin{pmatrix} 0 & 1 \\ -1 & 0 \end{pmatrix},\qquad
	T_3\coloneqq \begin{pmatrix} 0 & i \\ i & 0 \end{pmatrix}.
\end{equation*}
which is orthonormal with respect to the inner product $\langle X,Y\rangle\coloneqq-\frac12\tr(XY)$, and satisfies $[T_\alpha,T_\beta]=2\eps_{\alpha\beta\gamma}T_\gamma$.
For $u\in\SU(2)$, the adjoint action on $\su(2)$ is
\begin{equation}\label{eq: adjoint action}
 \Ad_u\colon\su(2)\to\su(2),\qquad \Ad_u X\coloneqq uXu^{-1}.
\end{equation}
We use the same notation for its pointwise action on Lie-algebra-valued forms and on pairs, acting on each component.

\medskip

\subsection{The nonabelian Higgs model} We retain the rate $\delta_\eps$ and the exponents fixed in \eqref{eq: deltaeps}.
Observe that this is larger than the Parise--Pigati--Stern \cite{Parise-Pigati-Stern2026} bound $\eps^2\ll\delta_\eps\ll\eps$ on the lower end, while it is slightly more restrictive on the upper end. The lower bound comes from the available range of model coercivity, see \Cref{lem: normal weighted coercivity}. The upper end is a bound that ensures enough smallness for the error of approximation. 

\medskip

To account for different scaling properties of functions and forms, we define the $\eps$-dependent $L^2$ pairing on pairs
\begin{equation}\label{eq: L^2 inner product}
	\langle(\Phi,A),(\Psi,B)\rangle_{L^2_\eps(M)}\coloneqq \int_{M}\eps^{-1}\langle \Phi,\Psi\rangle +\eps A\cdot B.
\end{equation}
Then, the first variation of \eqref{eq: energy} is 
\begin{equation*}
	\frac{d}{dt}\Big\vert_{t=0}E_\eps(\Phi+t\Psi,A+tB)=\langle {\bf S}_\eps(\Phi,A),(\Psi,B)\rangle_{L^2_\eps(M)}
\end{equation*}
where
\begin{equation}\label{eq:euler-lagrange}
	{\bf S}_\eps(\Phi,A)\coloneqq \begin{pmatrix}
		d_A^*d_A\Phi+\frac{2\eps}{\delta_\eps}(|\Phi|^2-1)\Phi\\[6pt]
		d_A^*F_A+\frac{1}{\eps^{2}}[\Phi,d_A\Phi]
	\end{pmatrix}.
\end{equation}
Here $d_A^*$ denotes the formal $L^2$-adjoint of $d_A$, with respect to $g$. The corresponding expression for the first variation of \eqref{eq: energy nc_eps} is the same without the potential term in the field component. 

\medskip

The energy is invariant under $\SU(2)$-gauge transformations. For $u\in \Omega^0(M,\SU(2))$, the transformed fields are
\begin{equation*}
	\Phi^u=u^{-1}\Phi u,\quad A^u=u^{-1}Au+u^{-1}du
\end{equation*}
and satisfy $E_\eps(\Phi^u,A^u)=E_\eps(\Phi,A)$.

\subsection{The BPS and the 't Hooft--Polyakov monopoles}\label{sec: BPStHP} Fix $\lambda\geq 0$. The entire model we use as building block of the construction is a smooth, finite energy critical point of \eqref{eq: Elambda}
satisfying $|\Phi|\to 1$. This is called the 't Hooft--Polyakov monopole \cite{tHooft1974,Polyakov1974}, which we will denote with $(\Phi_\lambda,A_\lambda)$; the subscript $\lambda$ highlights its dependence on it. Note that choosing 
\begin{equation}\label{eq: lambdaeps}
	\lambda=\lambda_\eps\coloneqq \frac{2\eps^3}{\delta_\eps},\qquad 0<\lambda_\eps\leq\lambda_*,\qquad \lambda_\eps/\eps^2\to\infty,
\end{equation}
the energy \eqref{eq: Elambda} corresponds to an $\eps$-rescaling of \eqref{eq: energy}. When dealing with the compact problem, we fix this choice of $\lambda$ hereon, leaving implicit its dependence on $\eps$ where it causes no confusion.

Explicitly, we can represent $(\Phi_\lambda,A_\lambda)$ as 
\begin{equation}\label{eq:monopole-ansatz}
	\Phi_\lambda(y)\coloneqq \varphi_\lambda(|y|)\frac{y^\alpha T_\alpha}{|y|},\quad A_\lambda(y)\coloneqq \alpha_\lambda(|y|)\eps_{\alpha\beta\gamma}\frac{y^\alpha T_\beta }{2|y|^2}dy^\gamma
\end{equation}
The profiles $\varphi_\lambda$ and $\alpha_\lambda$ solve
\begin{equation}\label{eq:monopole-profiles}
	\begin{cases}
		\varphi_\lambda''+\dfrac2r\varphi_\lambda'=\dfrac{2(1-\alpha_\lambda)^2\varphi_\lambda}{r^2}+\lambda(\varphi_\lambda^2-1)\varphi_\lambda,\\[6pt]
		\alpha_\lambda''=(1-\alpha_\lambda)\left(\dfrac{\alpha_\lambda(2-\alpha_\lambda)}{r^2}-4\varphi_\lambda^2\right),
	\end{cases}
\end{equation}
with boundary conditions
\begin{equation*}
	\varphi_\lambda(0)=\alpha_\lambda(0)=0,\qquad
	\lim_{r\to\infty}\varphi_\lambda(r)=\lim_{r\to\infty}\alpha_\lambda(r)=1.
\end{equation*}
The standard radial minimising solution has expansions $\varphi_\lambda(r)=ar+O(r^3)$ and $\alpha_\lambda(r)=br^2+O(r^4)$ at the origin, ensuring that \eqref{eq:monopole-ansatz} extends smoothly across $y=0$. 

Lastly, we remark that the 't Hooft--Polyakov monopole corresponding to $\lambda=0$ is the classical Bogomolnyi--Prasad--Sommerfield monopole \cite{Bogomolnyi1976,Prasad-Sommerfield1975}
\begin{equation}\label{eq: BPS}
	\varphi_{0}(r)=\coth(2r)-\frac1{2r},\qquad
	\alpha_{0}(r)=1-\frac{2r}{\sinh(2r)}.
\end{equation}

\subsection{Second variation, invariances and linear kernel}\label{sec: second variation}
At a smooth pair $(\Phi,A)$, the second variation of energy \eqref{eq: energy} in the direction $(\Psi,B)$ is
\begin{equation*}
	\frac{d^2}{dt^2}\Big\vert_{t=0}E_\eps(\Phi+t\Psi,A+tB)=	\widetilde{\bf Q}_{(\Phi,A)}(\Psi,B),
\end{equation*}
where
\begin{equation*}
\begin{split}
	\widetilde {\bf Q}_{(\Phi,A)}(\Psi,B)&\coloneqq \frac1\eps\int_M\left(|d_A\Psi+[B,\Phi]|^2+2\langle d_A\Phi,[B,\Psi]\rangle\right)\,d\vol_g\\
	&\quad+\eps\int_M\left(|d_AB|^2+\langle F_A,[B\wedge B]\rangle\right)\,d\vol_g\\
	&\quad+\frac{2}{\delta_\eps}\int_M\left((|\Phi|^2-1)|\Psi|^2+2\langle\Phi,\Psi\rangle^2\right)\,d\vol_g.
\end{split}
\end{equation*}
The second variation of \eqref{eq: energy nc_eps} is the same but without the last line.
At a critical pair $(\Phi,A)$, gauge invariance generates the infinite-dimensional subspace
\begin{equation*}
	\mathcal{Z}_{\mathrm{gauge}}\coloneqq \left\lbrace {\bf G}_{(\Phi,A)}[\gamma]\coloneqq ([\Phi,\gamma],d_A\gamma): \gamma\in \Omega^0(M,\su(2)) \right\rbrace
\end{equation*}
contained in the kernel of the linearised operator, which therefore cannot be elliptic. For any smooth pair $(\Phi,A)$, define the gauge-corrected quadratic form
\begin{equation}\label{eq: gauge decom Q}
 	{\bf Q}_{(\Phi,A)}(\Psi,B)\coloneqq \widetilde {\bf Q}_{(\Phi,A)}(\Psi,B)+\frac{1}{\eps}\int_M\left|{\bf G}_{(\Phi,A)}^*[(\Psi,B)]\right|^2\,d\vol_g.
\end{equation}
where ${\bf G}_{(\Phi,A)}^*[(\Psi,B)]=\eps d^*_AB-\frac{1}{\eps}[\Phi,\Psi]$ is the adjoint of ${\bf G}_{(\Phi,A)}$ with respect to \eqref{eq: L^2 inner product}. 
The corresponding gauge-corrected linearised operator is
\begin{equation}\label{eq: model linearised operator}
	{\bf L}_{(\Phi,A)}(\Psi,B)\coloneqq \begin{pmatrix}
		\Delta_A\Psi-\eps^{-2}[\Phi,[\Phi,\Psi]]+2\ad_B^*d_A\Phi+\frac{2\eps}{\delta_\eps}\bigl((|\Phi|^2-1)\Psi+2\langle\Phi,\Psi\rangle\Phi\bigr)\\[6pt]
		\Delta_AB+\ad_B^*F_A-\eps^{-2}[\Phi,[\Phi,B]]-2\eps^{-2}[d_A\Phi,\Psi]
	\end{pmatrix},
\end{equation}
where $\Delta_A\coloneqq d_A^*d_A+d_Ad_A^*$ is the covariant Hodge Laplacian and $\ad_B^*$ is the pointwise adjoint of $\ad_B\omega\coloneqq[B\wedge\omega]$, given explicitly by 
\begin{equation*}
	\ad_B^*\eta=-\sum_{j=1}^n[\iota_{e_j}B,\iota_{e_j}\eta].
\end{equation*}
Here $(e_j)$ is any local orthonormal frame and $\iota_{e_j}$ denotes the contraction. This expression is independent of the frame and satisfies $\langle\ad_B^*\eta,\omega\rangle=\langle\eta,[B\wedge\omega]\rangle$ pointwise for every $k$-form $\eta$ and $(k-1)$-form $\omega$, with values in $\su(2)$ and $k\geq1$.
The linearised operator satisfies
\begin{equation*}
	{\bf Q}_{(\Phi,A)}(\Psi,B)=\langle {\bf L}_{(\Phi,A)}(\Psi,B),(\Psi,B)\rangle_{L^2_\eps(M)}.
\end{equation*}
For the Euclidean monopole $(\Phi_\lambda,A_\lambda)$, let ${\bf L}_\lambda$ denote the gauge-corrected Hessian of $E_\lambda$ with the ordinary $L^2$ pairing. It has four explicit kernel elements,
\begin{equation}\label{eq: linearised kernel}
	\mathrm{span}\{{\bf Z}_1,{\bf Z}_2,{\bf Z}_3,{\bf W}\}\subseteq \ker{\bf L}_\lambda\cap L^2(\R^3).
\end{equation}
The elements
\begin{equation}\label{eq: monopole translations}
	{\bf Z}_a=(\iota_{e_a}d_{A_\lambda}\Phi_\lambda,\iota_{e_a}F_{A_\lambda}),\qquad a=1,2,3,
\end{equation}
are the gauge-corrected infinitesimal translations.
The element ${\bf W}$, which we now describe, arises from a gauge transformation itself: this means that 
\begin{equation*}
	\widetilde {\bf L}_{\lambda}[{\bf W}]={\bf G}_{\lambda}{\bf G}_{\lambda}^*[{\bf W}]=0,
\end{equation*}
where ${\bf G}_\lambda\coloneqq {\bf G}_{(\Phi_\lambda,A_\lambda)}$.
The pair ${\bf W}$ represents an infinitesimal internal rotation about the Higgs direction . On $\R^3\setminus\{0\}$, set
\begin{equation*}
	\sigma(y)=\frac{y^a T_a}{|y|},\qquad A_\sigma=-\frac14[\sigma,d\sigma].
\end{equation*}
Then,
\begin{equation*}
\Phi_\lambda=\varphi\sigma,\quad A_\lambda=\alpha A_\sigma\quad \text{and}\quad d_{A_\sigma}\sigma=0	
\end{equation*}
Let $q_\lambda\colon [0,+\infty)\to \R$ be the solution to  
\begin{equation}\label{eq: qlambda}
	 q_\lambda''+\frac2r q_\lambda'-\frac{2(1-\alpha(r))^2}{r^2}q_\lambda=0,\qquad q_\lambda(0)=0,\qquad\lim_{r\to\infty}q_\lambda(r)=1.
\end{equation}
It can be checked that $q_\lambda(r)=cr+O(r^3)$ at the origin and $q_\lambda(r)\to1$ at infinity. For every $t$, we set 
\begin{equation*}
	u_t\coloneqq \exp(t\, q_\lambda\sigma)\in \SU(2)
\end{equation*}
and define 
\begin{equation}\label{eq: monopole phase mode}
	{\bf W}\coloneqq \frac{d}{dt}\Big|_{t=0}(\Phi_\lambda,A_\lambda)^{u_t}=(0,d_{A_\lambda}(q_\lambda\sigma))=(0,q_\lambda'\sigma\,dr+q_\lambda(1-\alpha)d\sigma).
\end{equation}
The choice of $q_\lambda$ in \eqref{eq: qlambda} ensures that $d_{A_\lambda}^*d_{A_\lambda}(q_\lambda\sigma)=0$, thus fixing ${\bf G}_\lambda^*[{\bf W}]=0$.

\begin{lemma}\label{lem: monopole ker}
	There exists $\lambda_0>0$ such that for all $\lambda\in [0,\lambda_0)$ the inclusion \eqref{eq: linearised kernel} is an equality. 
\end{lemma}

\begin{proof}
The BPS nondegeneracy theorem at $\lambda=0$ \cite[Theorem~3.1, Corollary~3.4]{Stuart1994Monopoles}, and \cite[Theorem~5.6]{Androulakis-Dostoglou1998} for small $\lambda>0$, give nonnegativity of the ungauged Hessian and precisely three translation classes modulo gauge. The gauge quotient here includes transformations with constant limiting phase; the framed kernel also retains the phase mode, which we recover below. Hence, for ${\bf v}\in\ker{\bf L}_\lambda\cap L^2$,
\begin{equation*}
0=\langle{\bf L}_\lambda{\bf v},{\bf v}\rangle_{L^2}
=\widetilde{\bf Q}_\lambda({\bf v})+\|{\bf G}_\lambda^*{\bf v}\|_{L^2}^2
\end{equation*}
implies ${\bf G}_\lambda^*{\bf v}=0$ and ${\bf v}=\sum_{\alpha=1}^3a_\alpha{\bf Z}_\alpha+{\bf G}_\lambda\xi$. Since ${\bf G}_\lambda^*{\bf Z}_\alpha=0$, we have ${\bf G}_\lambda^*{\bf G}_\lambda\xi=0$ and ${\bf G}_\lambda\xi\in L^2$.

Outside a fixed ball, write $\xi=f\sigma+\zeta$, with $\zeta\perp\sigma$. Finite gauge energy gives $\zeta\in L^2$ and
\begin{equation*}
 df=\langle d_{A_\lambda}\xi,\sigma\rangle+\langle\zeta,d_{A_\lambda}\sigma\rangle\in L^2.
\end{equation*}
The homogeneous Sobolev inequality therefore gives $f-c\in L^6$ for a constant $c$. Since $d_{A_\lambda}\sigma$ decays exponentially,
\begin{equation*}
 d_{A_\lambda}\zeta=d_{A_\lambda}\xi-df\,\sigma-f\,d_{A_\lambda}\sigma\in L^2,
\end{equation*}
and the same inequality, applied after a cutoff, gives $\zeta\in L^6$. Finally, $1-q_\lambda=O(r^{-1})\in L^6$ outside the ball, so $\eta\coloneqq\xi-cq_\lambda\sigma\in L^6(\R^3)$. By \eqref{eq: qlambda}, ${\bf G}_\lambda^*{\bf G}_\lambda\eta=0$. Testing with $\chi_R^2\eta$, where $\chi_R=1$ on $B_R$, vanishes outside $B_{2R}$ and $|d\chi_R|\le C/R$, yields
\begin{equation*}
\int_{B_R}|{\bf G}_\lambda\eta|^2
\le 4\int_{B_{2R}\setminus B_R}|d\chi_R|^2|\eta|^2
\le C\|\eta\|_{L^6(B_{2R}\setminus B_R)}^2\longrightarrow0.
\end{equation*}
Consequently ${\bf G}_\lambda\eta=0$, so ${\bf v}=\sum_{\alpha=1}^3a_\alpha{\bf Z}_\alpha+c{\bf W}$. The reverse inclusion was established above.
\end{proof}

\section{The approximate solution: compact case}\label{sec: approx compact}
For simplicity, we assume throughout the construction and proof that $\Gamma$ is connected. The modifications required for finitely many connected components are described in Appendix~\ref{app: disconnected components}.

\medskip

In this section we construct the approximate solution that will be used in the perturbative scheme. The steps are the following: first, define a pair locally near $\Gamma$ by placing 't Hooft--Polyakov monopoles in its normal bundle fibres, and extend it globally using a cut-off and exponential decay. The second assumption in \Cref{def: admissible} is a topological condition ensuring that such extension is possible. This approximation turns out to be far from being a solution because of geometric terms appearing in the neutral connection component of the error of approximation in regions away from $\Gamma$. To repair this, we modify the approximate solution by adding a neutral connection term, which is carefully obtained by gluing different profiles near and away from the core. Lastly, we include the translation and phase fields defined in \eqref{eq: shifted phase approximation} to the approximation, to allow the flexibility needed to solve the reduced equation. 

\medskip

In what follows, we use Einstein summation convention. Moreover, we use $1\leq i,j,k\dots\leq n-3$ to denote indices of tangential variables, $1\leq a,b,c,\dots\leq 3$ to denote indices of normal variables.

\subsection{The initial approximation and large errors from the neutral connection}
We begin the construction on an admissible compact ambient space $(M,g)$. Remark that the restriction of $dF$ to $N\Gamma$ gives a natural trivialisation, which provides an orthonormal frame $(\nu_1,\nu_2,\nu_3)$. For $\delta>0$ sufficiently small, the Fermi map is a diffeomorphism
\begin{equation*}
\mathscr F_{\Gamma}\colon\Gamma\times B_\delta^3\longrightarrow\Tub_\delta(\Gamma),\qquad \mathscr F_{\Gamma}(q,y)=\exp_q\left(y^a\nu_a(q)\right).
\end{equation*}
Here $\Tub_\delta(\Gamma)=\{p\in M:\dist(p,\Gamma)<\delta\}$ and we can always assume
\begin{equation*}
F(\mathscr F_{\Gamma}(q,y))=y.
\end{equation*}
We define 
\begin{equation*}
	\sigma\coloneqq\frac{F^a T_a}{|F|}\colon M\setminus\Gamma\to \su(2),\qquad A_\sigma\coloneqq -\frac14[\sigma,d\sigma]\in\Omega^1(M\setminus \Gamma,\su(2))
\end{equation*}
We use as interior approximation the pullback of the $\eps$-rescaled 't Hooft--Polyakov monopole via Fermi coordinates. To extend it to the whole $M$, consider a monotone, smooth cut-off function $\eta\colon [0,+\infty)\to [0,1]$ such that $\eta\equiv 1$ in $[0,1]$, $\eta\equiv 0$ in $(2,+\infty)$ and $|\nabla\eta|\leq 2$. From $\eta$, we construct a new cut-off $\chi_{\Gamma}\colon M\to [0,1]$ by setting 
\begin{equation*}
	\chi_{\Gamma}(x)\coloneqq 
	\begin{cases}
		\eta \left(\frac{2|y|}{\delta}\right)&\text{if } x=\mathscr{F}_{\Gamma}(q,y)\in \Tub_\delta(\Gamma)\\
		0&\text{if } x\in M\setminus \Tub_\delta(\Gamma).
	\end{cases}
\end{equation*}
We define the global scalar functions
\begin{equation*}
\begin{split}
	\varphi_\eps(x)&\coloneqq \chi_{\Gamma}(x)\varphi_{\lambda_\eps}\left(\frac{|y|}{\eps}\right)+(1-\chi_{\Gamma}(x)),\\ 
	\alpha_\eps(x)&\coloneqq \chi_{\Gamma}(x)\alpha_{\lambda_\eps}\left(\frac{|y|}{\eps}\right)+(1-\chi_{\Gamma}(x)),
\end{split}
\end{equation*}
where on the support of $\chi_{\Gamma}$ we understand $x=\mathscr{F}_{\Gamma}(q,y)$.
Then, the global approximate solution is the pair $(\Phi^0_{\Gamma},A^0_{\Gamma})$ defined by 
\begin{equation*}
	\Phi^0_{\Gamma}\coloneqq \varphi_\eps\sigma,\qquad 	A^0_{\Gamma}\coloneqq \alpha_\eps A_\sigma
\end{equation*}

\medskip

We use the error of approximation \eqref{eq:euler-lagrange} to measure how ``good'' the approximation is. The field component of the error vanishes on $\{\chi_\Gamma=0\}$ since $|\sigma|=1$ and $d_{A_\sigma}\sigma=0$. On the other hand, the connection component of the error need not be small. Indeed, setting
\begin{equation*}
	f_0\coloneqq \langle \sigma,F_{A_\sigma}\rangle\in\Omega^2(M\setminus\Gamma),
\end{equation*}
we have
\begin{equation}\label{eq: curvature identity}
	F_{A_\sigma}(X,Y)=-\frac14[d\sigma(X),d\sigma(Y)]=\sigma f_0(X,Y).
\end{equation}
We claim that $f_0=-\frac12\sigma^*\omega_{\Sp^2}$. Indeed, identifying $\su(2)$ with $\R^3$ through the oriented orthonormal basis $(T_1,T_2,T_3)$, we have $[u,v]=2u\times v$ and $(\omega_{\Sp^2})_p(u,v)=\langle p,u\times v\rangle$ for $u,v\in T_p\Sp^2$. Thus, for $X,Y\in T_xM$, \eqref{eq: curvature identity} gives
\begin{equation*}
\begin{split}
 f_0(X,Y)&=-\frac14\langle\sigma,[d\sigma(X),d\sigma(Y)]\rangle\\
 &=-\frac12\langle\sigma,d\sigma(X)\times d\sigma(Y)\rangle
 =-\frac12(\sigma^*\omega_{\Sp^2})(X,Y).
\end{split}
\end{equation*}
Hence, we have 
\begin{equation}\label{eq: initial error}
	{\bf S}_\eps(\Phi_\Gamma^0,A_\Gamma^0)\vert_{\{\chi_\Gamma=0\}}=(0,(d_g^*f_0)\sigma)
\end{equation}
which is $O(1)$. In other words, the canonical initial approximation obtained by normal slices of the monopole generates large errors in the connection component parallel to the field, in the Lie algebra sense. 

\subsection{The improved approximation and its error}
We first cancel the large connection error \eqref{eq: initial error} in the outer region $\{\chi_\Gamma=0\}$ by adding a connection term $b\sigma$, where $b$ is a real one-form. The equation for $b$ is
\begin{equation}\label{eq: eq for b}
	\Delta_1b=-d_g^*f_0,\quad d_g^*b=0
\end{equation}
where $\Delta_1$ is the Hodge Laplacian acting on 1-forms. 
Indeed, the corrected pair on $\{\chi_\Gamma=0\}$ is $(\sigma,A_\sigma+b\sigma)$. Since $d_{A_\sigma}\sigma=0$ and $[\sigma,\sigma]=0$, we have
\begin{equation*}
	d_{A_\sigma+b\sigma}\sigma=0,\qquad
	F_{A_\sigma+b\sigma}=F_{A_\sigma}+d_{A_\sigma}(b\sigma)
	=\sigma(f_0+db).
\end{equation*}
Thus the Higgs equation remains satisfied, while the connection equation reduces to
\begin{equation*}
	{\bf S}_\eps(\sigma,A_\sigma+b\sigma)
	=\bigl(0,\sigma d_g^*(f_0+db)\bigr).
\end{equation*}
and imposing the Coulomb condition $d_g^*b=0$, the connection error becomes $d_g^*f_0+\Delta_1b$.

\medskip
The existence of a solution to \eqref{eq: eq for b} is the content of \Cref{lem: exterior correction} below. Since the solution is in general not smooth on $\Gamma$ we need a precise asymptotic description to match it with a smooth replacement. Recall the smooth orthonormal normal frame $(\nu_1,\nu_2,\nu_3)$ fixed above using $F$. In local coordinates $(q^j)$ on $\Gamma$, set $e_j=\partial_{q^j}$ and define the smooth normal connection coefficients by
\begin{equation}\label{eq: omega}
	\nabla^\perp_{e_j}\nu_b=\omega_{jb}^a\nu_a,\qquad \omega_{jb}^a=-\omega_{ja}^b,
\end{equation}
The second relation follows by differentiating $\langle\nu_a,\nu_b\rangle=\delta_{ab}$. Thus $\omega_j=(\omega_{jb}^a)_{a,b}\in\so(3)$, and the local expression $\omega\coloneqq\omega_jdq^j$ defines a smooth form in $\Omega^1(\Gamma;\so(3))$. Its lift $\rho\coloneqq\rho_jdq^j\in\Omega^1(\Gamma;\su(2))$ under the adjoint isomorphism is determined by
\begin{equation}\label{eq: rho}
	[\rho_j,T_b]=\omega_{jb}^{a}T_a\qquad\Longleftrightarrow\qquad \rho_j=\frac14\eps_{abc}\omega_{ja}^{b}T_c.
\end{equation}
\begin{lemma}\label{lem: exterior correction}
There is a unique distributional solution $b$ of \eqref{eq: eq for b} with $b\perp\ker\Delta_1$. It is smooth on $M\setminus\Gamma$ and belongs to $L^\infty(M)\cap W^{2,p}(M;T^*M)$ for every $1<p<\frac32$. Moreover, $b$ can be decomposed as $b=b_s+b_r$, where $b_r\in C^{0,\gamma}(M;T^*M)$ for every $0<\gamma<1$ and, on $\Tub_\delta(\Gamma)\setminus\Gamma$,
\begin{equation*}
 b_s=\chi_\Gamma\langle\pi^*\rho,\sigma\rangle=\chi_\Gamma\langle \rho_j,\sigma\rangle dq^j,
\end{equation*}
where $\pi\colon\Tub_\delta(\Gamma)\to\Gamma$ is the nearest-point projection and $b_s$ is understood to be zero outside the tube.
\end{lemma}

\begin{proof}
	We first argue that $d^*_gf_0\in L^{p}(M,T^*M)$ for $1<p<\frac{3}{2}$.
Put $r=\dist(\cdot,\Gamma)=|y|$ in Fermi coordinates. By \eqref{eq: curvature identity}, we have 
\begin{equation*}
	f_0(X,Y)=-\frac{1}{4}\langle \sigma,[d\sigma(X),d\sigma(Y)]\rangle
\end{equation*}
and since $\sigma=\frac{y^aT_a}r$, we readily check that 
\begin{equation*}
 (f_0)_{ab}=-\frac{\eps_{cab}y^c}{2r^3},\qquad
 (f_0)_{ia}=(f_0)_{ij}=0,\qquad
 \partial_{y^a}(f_0)_{ab}=0,
\end{equation*}
the last identity following by antisymmetry. Thus $d_{\bar g}^*f_0=0$ for the product metric $\bar g=g_\Gamma+\delta_{ab}\,dy^a\,dy^b$. Since $g-\bar g=O(r)$ with bounded first derivatives, comparison of the first-order codifferentials yields
\begin{equation*}
 |d_g^*f_0|=|(d_g^*-d_{\bar g}^*)f_0|
 \le Cr|\partial f_0|+C|f_0|\le Cr^{-2}.
\end{equation*}
The normal volume element is comparable to $r^2\,dr\,d\mathrm{vol}_{\Sp^2}$, from which the claim follows. On $M_t=M\setminus\Tub_t(\Gamma)$ the outward normal is $-\partial_r$ and $\iota_{\partial_r}f_0=0$. Hence, for every harmonic one-form $\zeta$,
\begin{equation*}
 \int_{M_t}\langle d_g^*f_0,\zeta\rangle
 =\int_{M_t}\langle f_0,d\zeta\rangle
 -\int_{\partial M_t}\langle\iota_{-\partial_r}f_0,\zeta\rangle=0.
\end{equation*}
Letting $t\to 0$ we obtain orthogonality with harmonic forms. The Hodge Green operator $G_1$ for $\Delta_1$ therefore gives the unique normalised solution $b=-G_1d_g^*f_0$, with the asserted $W^{2,p}$ and off-core smoothness. Furthermore, $\Delta_0d_g^*b=-(d_g^*)^2f_0=0$ and its average is 0 on each connected component of $M$, thus $d_g^*b=0$.

\smallskip

Next, we prove the decomposition. Since $\langle \rho_j,\sigma\rangle$ is a restriction of a linear function on the sphere $\Sp^2$ (for $q$ fixed), it is a spherical harmonic of degree 1. Using the comparison with the product metric and the bounds $|b_s|+r|\nabla b_s|+r^2|\nabla^2b_s|\leq C$, we see that 
\begin{equation*}
	\Delta_1b_s=\frac{2}{r^2}\langle\rho_j,\sigma\rangle\,dq^j+O(r^{-1})
\end{equation*} 
On the other hand, a careful expansion of $d^*_gf_0$ also shows
\begin{equation*}
	d^*_gf_0=-\frac{2}{r^2}\langle\rho_j,\sigma\rangle\,dq^j+O(r^{-1}).
\end{equation*}
Thus, we get the remainder bound for $b_r\coloneqq b-b_s$
\begin{equation}\label{eq: exterior regular source}
	\Delta_1b_r=R\coloneqq -d^*_gf_0-\Delta_1b_s,\qquad r|R|\leq C.
\end{equation}
The Green kernel estimates $|K(x,z)|\leq Cd_M(x,z)^{2-n}$, $|\nabla_xK(x,z)|\leq Cd_M(x,z)^{1-n}$ imply the stated regularity for $b_r$. We also need regularity along $\Gamma$. Differentiating the smooth Fermi coefficients in the cancellation above gives $|\partial_q^kR|\leq C_k r^{-1}$ for every $k$. In local Fermi coordinates, integration by parts transfers tangential derivatives of the Green kernel to the source; simultaneous tangential derivatives in its two variables preserve its singular order. Thus, up to a smooth term, $\partial_q^k b_r$ is a finite sum of integrals with source $\partial_{q'}^\ell R$, $\ell\leq k$, and kernels satisfying the same two bounds. The preceding Green estimate consequently gives
\begin{equation}\label{eq: exterior tangential regularity}
 \partial_q^k b_r\in C^{0,\gamma}_{\mathrm{loc}}\quad(k\geq0,\ 0<\gamma<1),
 \qquad \beta\coloneqq b_r|_\Gamma\in C^\infty(\Gamma;T^*M|_\Gamma).
\end{equation} 
\end{proof}

Unfortunately, the outer correction $b\sigma$ need not extend even continuously across $\Gamma$ by the singularity of $\sigma$. To repair this, write $\beta\coloneqq b_r\vert_\Gamma=\beta_jdq^j+\beta_a dy^a\in C^\infty(\Gamma;T^*M\vert_\Gamma)$. For any fixed $0<\gamma<1$,
\begin{equation*}
	b\sigma=\langle \rho_j,\sigma\rangle \sigma dq^j+\beta_j\sigma dq^j+\beta_a\sigma dy^a+O(r^\gamma).
\end{equation*}
We match the three leading terms with smooth inner fields.

\begin{itemize}
	\item \emph{The angular part $\langle \rho_j,\sigma\rangle \sigma dq^j$}. Define the connection correction
	\begin{equation*}
		M\coloneqq M_jdq^j,\qquad M_j\coloneqq \rho_j-\alpha_\lambda\left(\tfrac{r}{\eps}\right)\left(\rho_j-\langle \rho_j,\sigma\rangle \sigma\right).
	\end{equation*}
	The smooth radial expansion of $\alpha_\lambda$ gives $M\vert_\Gamma=\rho$, while $\alpha_\lambda\to1$ gives the required angular limit as $\tfrac{r}{\eps}\to\infty$.
	\item \emph{The tangential trace $\beta_j\sigma dq^j$}. Use the lift $\xi_\eps\beta_jdq^j$, where
	\begin{equation}\label{eq: q_eps}
		\xi_\eps\coloneqq q_\eps\sigma,\qquad q_\eps\coloneqq\chi_\Gamma q_\lambda\left(\tfrac{r}{\eps}\right)+(1-\chi_\Gamma).
	\end{equation}
	The radial expansion of $q_\lambda$ makes $\xi_\eps$ smooth at the core, and $q_\lambda\to1$ supplies the matching limit.
	\item \emph{The normal trace $\beta_a\sigma dy^a$}. Use the smooth gauge
	\begin{equation}\label{eq: normal trace gauge}
		g_\eps\coloneqq\exp((\beta_a y^a)\xi_\eps)\in\SU(2).
	\end{equation}
	In the abelian limit its action adds $\sigma\,d(\beta_a y^a)=\beta_a\sigma\,dy^a+y^a(\partial_{q^j}\beta_a)\sigma\,dq^j$, supplying the normal trace.
\end{itemize}
Define
\begin{equation*}
	A^{\inn}\coloneqq g_\eps^{-1}\bigl(A_\Gamma^0+M+\xi_\eps\beta_jdq^j\bigr)g_\eps+g_\eps^{-1}dg_\eps.
\end{equation*}
Since $g_\eps$ commutes with $\Phi_\Gamma^0$, the Higgs field in this gauge transformation is unchanged.

We match the inner and outer connections at radius $\sqrt\eps$, using
\begin{equation}\label{eq: s_eps}
 \kappa_\eps(\mathscr F_\Gamma(q,y))\coloneqq\eta\left(\frac{|y|}{\sqrt\eps}\right),
\end{equation}
extending $\kappa_\eps$ by zero outside the tube. Define
\begin{equation*}
 \Phi_\Gamma^{\mathrm c}\coloneqq\Phi_\Gamma^0,\qquad A_\Gamma^{\mathrm c}\coloneqq\kappa_\eps A^{\inn}+(1-\kappa_\eps)(A_\Gamma^0+b\sigma).
\end{equation*}

\subsection{Vertical shift and phase adjustment}
Lastly, we introduce the parameters 
\begin{equation*}
	h\in C^{2,\tau}(\Gamma,\R^3),\qquad \theta\in C^{2,\tau}(\Gamma,\R)
\end{equation*}
satisfying
\begin{equation}\label{eq: assumptions htheta}
 	\|h\|_{C^{2,\tau}}\leq K\eps|\log\eps|^2,\qquad \|\theta\|_{C^{2,\tau}}\leq K,\qquad \int_\Gamma\theta=0. 
\end{equation}
Using the fixed normal trivialisation, for small $h$ define the diffeomorphism $\mathscr{T}_h\colon M\to M$
\begin{equation*}
	\mathscr T_h(\mathscr F_\Gamma(q,y))\coloneqq
	\mathscr F_\Gamma\left(q,y+\eta\left(\frac{2|y|}{\delta}\right)h(q)\right),
\end{equation*}
extended by the identity outside $\Tub_\delta(\Gamma)$, so that $\mathscr T_h(\Gamma)=\Gamma_h\coloneqq\exp_\Gamma h$. Set
\begin{equation}\label{eq: phase gauge}
	\bar\theta\coloneqq\chi_\Gamma(\theta\circ\pi),\qquad
	u_\theta\coloneqq\exp(\bar\theta\xi_\eps),
\end{equation}
where $\bar\theta=0$ outside the tube.
Define the shifted compact approximation by
\begin{equation}\label{eq: shifted phase approximation}
\begin{split}
 \Phi_{h,\theta}^{\mathrm c}&\coloneqq(\mathscr T_h^{-1})^*\Phi_\Gamma^{\mathrm c},\\
 A_{h,\theta}^{\mathrm c}&\coloneqq(\mathscr T_h^{-1})^*\bigl(u_\theta^{-1}A_\Gamma^{\mathrm c}u_\theta+u_\theta^{-1}du_\theta-\xi_\eps d\bar\theta\bigr).
\end{split}
\end{equation}
The phase does not change the Higgs field because $[\xi_\eps,\Phi_\Gamma^{\mathrm c}]=0$. Differentiating the connection at $\theta=0$ gives
\begin{equation*}
 \left.\partial_t A_{0,t\theta}^{\mathrm c}\right|_{t=0}=d_{A_\Gamma^{\mathrm c}}(\bar\theta\xi_\eps)-\xi_\eps d\bar\theta=\bar\theta\,d_{A_\Gamma^{\mathrm c}}\xi_\eps.
\end{equation*}
Thus the subtraction in \eqref{eq: shifted phase approximation} removes the term $\xi_\eps d\bar\theta$ from the infinitesimal gauge variation.

\section{The approximate solution: noncompact case}\label{sec: noncompact approximation}
Let $M$ be a noncompact admissible ambient space and $\Gamma\subset M$ an admissible submanifold. We retain the notation from the preceding section, and set $\lambda=0$, so that the normal profiles are $\varphi_0$ and $\alpha_0$.

\subsection{The initial approximation}
Let $G$ be the positive minimal Green function in \Cref{def: admissible amb noncomp}, normalised by $\Delta_{g,x}G(x,q)=\delta_q$. Define the scalar potential
\begin{equation}\label{eq: noncompact scalar potential}
 H(x)\coloneqq2\pi\int_\Gamma G(x,q)\,d\mathrm{vol}_\Gamma(q),
 \qquad x\in M\setminus\Gamma.
\end{equation}
It satisfies $\Delta_gH=2\pi\delta_\Gamma$ distributionally and tends to zero at infinity. Minimality of $\Gamma$ and this normalisation give
\begin{equation*}
 H(x)=\frac1{2r(x)}+O(1),\qquad r(x)\coloneqq\dist_g(x,\Gamma),
\end{equation*}
uniformly near $\Gamma$, matching the Coulomb term in the BPS Higgs' field profile \eqref{eq: BPS}.
Using the cutoffs in \eqref{eq: s_eps} and \Cref{sec: approx compact}, set
\begin{equation*}
\begin{split}
 \varphi_\eps&\coloneqq\kappa_\eps\varphi_0\left(\frac r\eps\right)+(1-\kappa_\eps)(1-\eps H),\\
 \alpha_\eps&\coloneqq\chi_\Gamma\alpha_0\left(\frac r\eps\right)+(1-\chi_\Gamma).
\end{split}
\end{equation*}
The profile terms are defined in the Fermi tube and extended by zero after multiplication by their cutoffs. 
The reason behind the different cut-off scale in the field part is the slow Higgs' tail \eqref{eq: BPS}. Confining the matching error to a shrinking tube, together with the improvement below, ensures enough global smallness of the error of approximation.

The initial approximation is
\begin{equation*}
 \Phi_\Gamma^0\coloneqq\varphi_\eps\sigma,
 \qquad A_\Gamma^0\coloneqq\alpha_\eps A_\sigma.
\end{equation*}
This pair extends smoothly across $\Gamma$, agrees with the rescaled BPS monopole near $\Gamma$, and equals $((1-\eps H)\sigma,A_\sigma)$ outside the tube. In particular, $|\Phi_\Gamma^0|\to1$ at infinity.

\medskip

As in the compact case, this choice produces an $O(1)$ error in the connection side and thus needs improvement. 

\subsection{The improved approximation}
\begin{lemma}\label{lem: noncompact exterior correction}
There is a real one-form $b$, smooth on $M\setminus\Gamma$ and coclosed near $\Gamma$, such that
\begin{equation*}
 f_1\coloneqq f_0+db,\qquad d_g^*f_1=0,\qquad \|f_1\|_{C^{2,\tau}_\nu(M\setminus\Tub_\delta(\Gamma))}<\infty.
\end{equation*}
Near $\Gamma$ it has the decomposition $b=b_s+b_r$ of \Cref{lem: exterior correction}, with $b_r\in C^{0,\gamma}$ for every $0<\gamma<1$ and smooth tangential trace $\beta\coloneqq b_r|_\Gamma$.
\end{lemma}
\begin{proof}
First, we make the curvature compactly supported. Since $\Gamma\subset K$ and $H^2_{\mathrm{dR}}(M\setminus K)=0$, write $f_0=d\alpha$ outside $K$. Cutting off $-\alpha$ gives a smooth one-form $a$, zero near $\Gamma$, for which
\begin{equation*}
 \widetilde f_0\coloneqq f_0+da=0\qquad\text{outside a compact set}.
\end{equation*}
Let $\widetilde M$ be an auxiliary closed manifold agreeing with $M$ on $\mathrm{spt}\widetilde f_0$. There, we can solve $\Delta_1\widetilde b=-d^*\widetilde f_0$ orthogonally to harmonic one-forms, as in \Cref{lem: exterior correction}. The source is orthogonal to harmonic one-forms because their exterior derivative vanishes. Cutting off $\widetilde b$ away from $\Gamma$ to obtain a compactly supported $b_0$, coclosed near $\Gamma$, with
\begin{equation*}
 v\coloneqq-d^*\widetilde f_0-\Delta_1b_0\in C_c^\infty(M;T^*M).
\end{equation*}
Next, apply the inverse in \Cref{def: admissible amb noncomp}, $u\coloneqq \Delta_1^{-1}v$ and set 
\begin{equation*}
b\coloneqq a+b_0+u.
\end{equation*}
The function $d^*(b_0+u)$ is distributionally harmonic, hence smooth, and tends to zero at infinity, hence $d^*(b_0+u)\equiv 0$ by maximum principle. Therefore
\begin{equation*}
 d^*(f_0+db)=d^*\widetilde f_0+d^*d(b_0+u)=d^*\widetilde f_0+\Delta_1(b_0+u)=0.
\end{equation*}
Outside a compact set $f_1=du$ is harmonic. The weighted bound for $u$ and interior estimates in bounded-geometry charts give the claimed bound for $f_1$.

Finally, near $\Gamma$ the equation is $\Delta_1b=-d^*f_0$. The local decomposition in \Cref{lem: exterior correction} applies. 
\end{proof}

The scalar potential has a smooth regular trace
\begin{equation*}
 H_\Gamma\coloneqq\left.\left(H-\frac1{2r}\right)\right|_\Gamma\in C^\infty(\Gamma).
\end{equation*}
Indeed, minimality gives $\Delta_g(H-\chi_\Gamma/(2r))=O(r^{-1})$, also after tangential differentiation, so the preceding local argument applies. In the following formulas, fields on $\Gamma$ are pulled back by $\pi$.

To match the regular term $-\eps H_\Gamma$, use the infinitesimal mass variation of the BPS monopole:
\begin{equation}\label{eq: noncompact mass variation}
\begin{split}
 \left.\partial_m\left[m\varphi_0\left(\frac{mr}{\eps}\right)\right]\right|_{m=1}&=\varphi_0\left(\frac r\eps\right)+\frac r\eps\varphi_0'\left(\frac r\eps\right),\\
 \left.\partial_m\alpha_0\left(\frac{mr}{\eps}\right)\right|_{m=1}&=\frac r\eps\alpha_0'\left(\frac r\eps\right).
\end{split}
\end{equation}
Set ${\bf C}_\eps\coloneqq(C_\eps^\Phi,C_\eps^A)$, where
\begin{equation}\label{eq: noncompact mass correction}
\begin{split}
 C_\eps^\Phi&\coloneqq-\eps H_\Gamma\left[\varphi_0\left(\frac r\eps\right)+\frac r\eps\varphi_0'\left(\frac r\eps\right)\right]\sigma,\\
 C_\eps^A&\coloneqq-\eps H_\Gamma\frac r\eps\alpha_0'\left(\frac r\eps\right)\left[A_\sigma-(\rho_j-\langle\rho_j,\sigma\rangle\sigma)dq^j\right].
\end{split}
\end{equation}
Retain $\xi_\eps$ and $g_\eps$ from the compact construction, with $\lambda=0$ and the trace $\beta$ just obtained. Define
\begin{equation*}
\begin{split}
 M&\coloneqq\left[\rho_j-\alpha_0\left(\frac r\eps\right)(\rho_j-\langle\rho_j,\sigma\rangle\sigma)\right]dq^j+C_\eps^A,\\
 A^{\inn}&\coloneqq g_\eps^{-1}\bigl(A_\Gamma^0+M+\xi_\eps\beta_jdq^j\bigr)g_\eps+g_\eps^{-1}dg_\eps.
\end{split}
\end{equation*}
The improved approximation is
\begin{equation*}
\begin{split}
 \Phi_\Gamma^{\mathrm{nc}}&\coloneqq\Phi_\Gamma^0+\kappa_\eps C_\eps^\Phi,\\
 A_\Gamma^{\mathrm{nc}}&\coloneqq\kappa_\eps A^{\inn}+(1-\kappa_\eps)(A_\Gamma^0+b\sigma).
\end{split}
\end{equation*}
All terms supported in the tube are extended by zero. Introducing the translation and phase as in \eqref{eq: shifted phase approximation}, set
\begin{equation}\label{eq: noncompact shifted approximation}
\begin{split}
 \Phi_{h,\theta}^{\mathrm{nc}}&\coloneqq(\mathscr T_h^{-1})^*\Phi_\Gamma^{\mathrm{nc}},\\
 A_{h,\theta}^{\mathrm{nc}}&\coloneqq(\mathscr T_h^{-1})^*\bigl(u_\theta^{-1}A_\Gamma^{\mathrm{nc}}u_\theta+u_\theta^{-1}du_\theta-\xi_\eps d\bar\theta\bigr).
\end{split}
\end{equation}

In statements common to both cases, we denote with $(\Phi_{h,\theta},A_{h,\theta})$ the corresponding compact or noncompact approximation.

\section{Norms}\label{sec: norms} Here we define the topology in which we measure the error of approximation and the corresponding perturbation. The norms will depend on $\eps>0$, the parameters $h$ and $\theta$ and the background approximation $(\Phi_{h,\theta},A_{h,\theta})$. To define the norm, we need a few preliminary definitions.

\begin{itemize}
	\item Put $r\coloneqq\dist(\cdot,\Gamma)$, so that $r=|y|$ on $\Tub_\delta(\Gamma)$. Fix a constant $M_0>0$ and consider the dyadic covering 
\begin{equation*}
	\Omega_j\coloneqq \{r_j<r<4r_j\},\qquad r_j\coloneqq 2^{j-1}M_0\eps ,\qquad j=1,\dots,J_\eps
\end{equation*}
where $J_\eps\coloneqq 1+\lfloor\log_2\frac{\delta}{4M_0\eps} \rfloor=O(|\log\eps|)$. Also, set
\begin{equation*}
	\begin{aligned}
	\Omega_0&\coloneqq \{r<2M_0\eps\},& r_0&\coloneqq\eps,\\
	\Omega_{J_\eps+1}&\coloneqq\Omega_\out\coloneqq \{r>\delta/4\},& r_{J_\eps+1}&\coloneqq\delta.
	\end{aligned}
\end{equation*}
Define the transported regions as $\Omega_{j,h}\coloneqq \mathscr{T}_h(\Omega_j)$, $j=0,\dots,J_\eps+1$.
\item For an $\su(2)$-valued tensor $T$ defined on any precompact region of $M\setminus \Gamma_h$, we have the following smooth splitting into longitudinal and transverse components:
\begin{equation*}
	T=T^\perp+T^\parallel,\qquad T^\parallel\coloneqq \frac{\langle T,\Phi_{h,\theta}\rangle\Phi_{h,\theta}}{|\Phi_{h,\theta}|^2},\qquad T^\perp\coloneqq T-T^\parallel.
\end{equation*}
\item For a tensor $T$ on a domain $\Omega$, a H\"older exponent $\tau\in (0,1)$ and a scale $s>0$, define the H\"older norm 
\begin{equation*}
	\|T\|_{\tau,s,\Omega}\coloneqq \|T\|_{L^\infty(\Omega)}+s^\tau[T]_{\tau,\Omega}
\end{equation*}
where the H\"older seminorm is computed on pairs of distinct points whose distance is at most $s$, and tensors are compared by parallel transport for the Levi--Civita connection and $A_{h,\theta}$ along the unique minimising geodesic (taking $\delta$ below the injectivity radius).
\end{itemize}
\subsubsection*{The compact case} Let $M$ be a compact admissible ambient space. 
For a $C^{0,\tau}$ pair $(\Psi,B)$ on $M$, define the source norm
\begin{equation}\label{eq: source norm}
	\begin{split}
		\|(\Psi,B)\|_{Y_{\eps,h,\theta}}&\coloneqq
			r_0^2\bigl(\|\Psi\|_{\tau,r_0,\Omega_{0,h}}+\eps\|B\|_{\tau,r_0,\Omega_{0,h}}\bigr)\\
			&+\sum_{j=1}^{J_\eps+1}r_j^2\bigl(\|\Psi^\parallel\|_{\tau,r_j,\Omega_{j,h}}+\eps\|B^\parallel\|_{\tau,r_j,\Omega_{j,h}}\bigr)\\
			&+\max_{1\leq j\leq J_\eps+1}r_j^2\bigl(\|\Psi^\perp\|_{\tau,r_j,\Omega_{j,h}}+\eps\|B^\perp\|_{\tau,r_j,\Omega_{j,h}}\bigr).
	\end{split}
\end{equation}
Define the solution norm of a $C^{2,\tau}$ pair by
\begin{equation}\label{eq: sol norm}
	\begin{split}
	\|(\Psi,B)\|_{X_{\eps,h,\theta}}
	&\coloneqq\max_{0\leq j\leq J_\eps+1}\sum_{k=0}^2r_j^k
	\bigl(\|\nabla_{A_{h,\theta}}^k\Psi\|_{\tau,r_j,\Omega_{j,h}}+\eps\|\nabla_{A_{h,\theta}}^kB\|_{\tau,r_j,\Omega_{j,h}}\bigr)\\
	&+\max_{1\leq j\leq J_\eps+1}\frac{r_j^2}{\eps^2}
	\bigl(\|\Psi^\perp\|_{\tau,r_j,\Omega_{j,h}}+\eps\|B^\perp\|_{\tau,r_j,\Omega_{j,h}}\bigr)\\
	&+\frac{2\lambda}{\eps^2}\max_{1\leq j\leq J_\eps+1}r_j^2
	\|\Psi^\parallel\|_{\tau,r_j,\Omega_{j,h}}.
	\end{split}
\end{equation}
The terms ``solution norm'' and ``source norm'' come from the linearised equation, where we use the $Y_{\eps,h,\theta}$ norm of the right-hand side to control the $X_{\eps,h,\theta}$ norm of the solution, see \Cref{prop: uniform inverse} below.

\subsubsection*{The noncompact case}
In the noncompact case, the definition of the norms \eqref{eq: source norm} and \eqref{eq: sol norm} is almost identical, with only two differences. The first is setting $\lambda=0$ in \eqref{eq: sol norm}, the second is including the decay at infinity:
let $\nu>0$ be the exponent from \Cref{def: admissible amb noncomp}. We replace the terms in \eqref{eq: source norm} and \eqref{eq: sol norm}, corresponding to the index $j=J_\eps+1$ with the following weighted exterior norms:
\begin{equation}\label{eq: noncompact exterior source norm}
 \|(\Psi,B)\|_{Y^\nu_{\eps,h,\theta}(\Omega_{J_\eps+1,h})}
 \coloneqq\|\Psi\|_{C^{0,\tau}_{\nu+2}}+\eps\|B\|_{C^{0,\tau}_{\nu+2}},
\end{equation}
and
\begin{equation}\label{eq: noncompact exterior solution norm}
\begin{split}
 \|(\Psi,B)\|_{X^\nu_{\eps,h,\theta}(\Omega_{J_\eps+1,h})}
 &\coloneqq \|\Psi^\parallel\|_{C^{2,\tau}_\nu}+\eps\|B^\parallel\|_{C^{2,\tau}_\nu}\\
 &+\|\Delta_{A_{h,\theta}}\Psi^\parallel\|_{C^{0,\tau}_{\nu+2}}+\eps\|\Delta_{A_{h,\theta}}B^\parallel\|_{C^{0,\tau}_{\nu+2}}\\
 &+\|\Psi^\perp\|_{C^{2,\tau}_{\nu+2}}+\eps\|B^\perp\|_{C^{2,\tau}_{\nu+2}}\\
 &+\eps^{-2}\left(\|\Psi^\perp\|_{C^{0,\tau}_{\nu+2}}
 +\eps\|B^\perp\|_{C^{0,\tau}_{\nu+2}}\right).
\end{split}
\end{equation}
Here the weighted norms are covariant and taken on $\Omega_{J_\eps+1,h}$, and $\Delta_{A_{h,\theta}}$ denotes the covariant Hodge Laplacian in degree zero or one, respectively. We denote by $X_{\eps,h,\theta}$ and $Y_{\eps,h,\theta}$ the Banach spaces of pairs with finite solution and source norm, respectively, using the full H\"older spaces.

\begin{remark}
The norms just defined deserve some explanation.
The splitting into longitudinal and transverse components reflects the different action of the Higgs field on these directions. Away from the core, the transverse equations contain a positive zeroth-order term of size $\eps^{-2}$, whereas the longitudinal connection equation is of Poisson type. The longitudinal Higgs equation has the additional term $2\lambda\eps^{-2}$, which disappears in the noncompact case, when $\lambda=0$. The corresponding weights in the solution norm \eqref{eq: sol norm} record this zeroth-order control. The core is treated separately, without using a projection that becomes undefined at zeros of the Higgs field.

The dyadic decomposition connects the core scale $\eps$ to the fixed tubular scale. The factor $\eps$ multiplying the connection components in \eqref{eq: source norm} and \eqref{eq: sol norm} accounts for their scaling as one-forms.

The distinction between sums and maxima is already visible in the three-dimensional Poisson model. Write $f=\sum_j f_j$, where $f_j$ is supported in a normal annulus of radius comparable to $r_j$. The Newtonian potential $u$ of $f$ satisfies
$$
|u(0)|\leq \frac{1}{4\pi}\sum_j\int_{\R^3}\frac{|f_j(y)|}{|y|}\,dy\lesssim \sum_j r_j^2\|f_j\|_{L^\infty}.
$$
For example, the source $f(y)=|y|^{-2}\mathbf{1}_{\{\eps<|y|<\delta\}}$ gives $u(0)=\log(\delta/\eps)$, although each weighted annular contribution is bounded independently of $\eps$. Thus a maximum over scales misses the possible accumulation, whereas the sum in the longitudinal source norm \eqref{eq: source norm} captures it and yields a bound whose constant is independent of the number of annuli.

The positive transverse zeroth-order term prevents the same accumulation, so a maximum suffices there. Finally, the solution norm \eqref{eq: sol norm} takes maxima because it measures uniform control across scales, without introducing an additional factor proportional to the number of annuli.
\end{remark}

Lastly, for any integer $k\geq0$ and $f\in C^{k,\tau}(\Gamma;\R^m)$, define
\begin{equation*}
	\|f\|_{C^{k,\tau}_\eps(\Gamma)}\coloneqq
	\sum_{j=0}^k\eps^j\|\nabla_\Gamma^j f\|_{L^\infty(\Gamma)}
	+\eps^{k+\tau}[\nabla_\Gamma^k f]_{\tau,\Gamma}.
\end{equation*}
Here derivatives and the ordinary H\"older seminorm use the fixed metric on $\Gamma$, componentwise in $\R^m$. We write $C^{k,\tau}_\eps(\Gamma;\R^m)$ for $C^{k,\tau}(\Gamma;\R^m)$ equipped with this norm.

\medskip

We can now measure the error of approximation. The proof of the following result is deferred to \Cref{sec: approximation errors}.

\begin{lemma}\label{lem: error approx}
For all sufficiently small $\eps$ and all parameters $h,\theta$ satisfying \eqref{eq: assumptions htheta},
\begin{equation*}
\begin{split}
 \|{\bf S}_\eps(\Phi_{h,\theta}^{\mathrm c},A_{h,\theta}^{\mathrm c})\|_{Y_{\eps,h,\theta}}&\lesssim \sqrt{\eps\delta_\eps}+\eps^{3/2}(1+|\log\eps|),\\
 \|{\bf S}_\eps(\Phi_{h,\theta}^{\mathrm{nc}},A_{h,\theta}^{\mathrm{nc}})\|_{Y_{\eps,h,\theta}}&\lesssim \eps^{3/2}(1+|\log\eps|)
\end{split}
\end{equation*}
where the expression of ${\bf S}_\eps$ does not contain the potential in the second line.
\end{lemma}

\section{The nonlinear scheme}\label{sec: nonlinear scheme}
In this section we prove the existence of a perturbation reducing the equation to a problem for $h$ and $\theta$. We will state several results and postpone their proof to following sections for clarity of exposition. Our goal is to find a pair ${\bf\Psi}\coloneqq(\Psi,B)$ such that 
\begin{equation}\label{eq: final eq}
	{\bf S}_\eps(\Phi_\eps,A_\eps)=0,\qquad (\Phi_\eps,A_\eps)\coloneqq (\Phi_{h,\theta}+\Psi,A_{h,\theta}+B)
\end{equation}
Set
\begin{equation*}
	{\bf G}_{h,\theta}\coloneqq{\bf G}_{(\Phi_{h,\theta},A_{h,\theta})},
	\qquad
	{\bf G}_{h,\theta}^*{\bf\Psi}
	=\eps d_{A_{h,\theta}}^*B-\eps^{-1}[\Phi_{h,\theta},\Psi].
\end{equation*}
Transport the smooth phase generator from \eqref{eq: q_eps} by defining $\xi_{\eps,h}\coloneqq(\mathscr T_h^{-1})^*\xi_\eps$.
Since $u_\theta$ commutes with $\xi_\eps$, this generator is independent of $\theta$ and satisfies $[\Phi_{h,\theta},\xi_{\eps,h}]=0$.

\medskip

Fix a nonnegative $\chi\in C_c^\infty(M)$ equal to one on a neighbourhood of all the shifted gluing regions; when $M$ is compact, take $\chi\equiv1$. Define the projection onto the span of $\xi_{\eps,h}$ by
\begin{equation}\label{eq: phase projection}
	\Pi_h\zeta\coloneqq
	\frac{\int_M\chi\langle\zeta,\xi_{\eps,h}\rangle\,d\vol_g}
	{\int_M\chi|\xi_{\eps,h}|^2\,d\vol_g}\,\xi_{\eps,h}.
\end{equation}
For each correction, put
\begin{equation*}
	\gamma_{\bf\Psi}\coloneqq(I-\Pi_h){\bf G}_{h,\theta}^*{\bf\Psi}.
\end{equation*}
Rather than \eqref{eq: final eq}, we work with the modified equation
\begin{equation}\label{eq: modified perturbation}
	{\bf S}_\eps(\Phi_{h,\theta}+\Psi,A_{h,\theta}+B)
	+\eps^{-1}{\bf G}_{(\Phi_{h,\theta}+\Psi,A_{h,\theta}+B)}
	\gamma_{\bf\Psi}=0.
\end{equation}
This modification allows us to use the gauge-corrected linearised operator at leading order and to cancel the phase residual; see \Cref{lem: constant phase residual} below. 
Using the notation of \Cref{sec: second variation}, we write
\begin{equation*}
{\bf L}_{h,\theta} \coloneqq{\bf L}_{(\Phi_{h,\theta},A_{h,\theta})}=D{\bf S}_\eps(\Phi_{h,\theta},A_{h,\theta}) +\eps^{-1}{\bf G}_{h,\theta}{\bf G}_{h,\theta}^*.
\end{equation*}
The linearisation of the left-hand side of \eqref{eq: modified perturbation} at ${\bf\Psi}=0$ is therefore
\begin{equation}\label{eq: restricted linearised operator}
	\widehat{\bf L}_{h,\theta}
	={\bf L}_{h,\theta}
	-\eps^{-1}{\bf G}_{h,\theta}\Pi_h{\bf G}_{h,\theta}^*,
\end{equation}
where the subtracted term is explicitly given by
\begin{equation*}
	\eps^{-1}{\bf G}_{h,\theta}\Pi_h{\bf G}_{h,\theta}^*{\bf\Psi}
	=
	\frac{\langle{\bf\Psi},{\bf G}_{h,\theta}(\chi\xi_{\eps,h})\rangle_{L^2_\eps(M)}}
	{\eps\int_M\chi|\xi_{\eps,h}|^2\,d\vol_g}
	\,{\bf G}_{h,\theta}\xi_{\eps,h}.
\end{equation*}
We also define the nonlinear quadratic remainder by
\begin{equation}\label{eq: nonlinear remainder definition}
\begin{split}
	\mathcal N_{\eps,h,\theta}({\bf\Psi})
	&\coloneqq{\bf S}_\eps(\Phi_{h,\theta}+\Psi,A_{h,\theta}+B)
	-{\bf S}_\eps(\Phi_{h,\theta},A_{h,\theta})\\
	&-D{\bf S}_\eps(\Phi_{h,\theta},A_{h,\theta}){\bf\Psi}+\eps^{-1}\bigl([\Psi,\gamma_{\bf\Psi}],[B,\gamma_{\bf\Psi}]\bigr).
\end{split}
\end{equation}
This vanishes to second order at zero. Equation \eqref{eq: modified perturbation} consequently becomes
\begin{equation*}
	\widehat{\bf L}_{h,\theta}{\bf\Psi}
	=-{\bf S}_\eps(\Phi_{h,\theta},A_{h,\theta})
	-\mathcal N_{\eps,h,\theta}({\bf\Psi}).
\end{equation*}
\begin{lemma}\label{lem: quadratic estimate}
Supposing that \eqref{eq: assumptions htheta} holds, there is a constant $C>0$ such that for any pairs ${\bf\Psi_1},{\bf\Psi_2}$ satisfying
\begin{equation*}
 \|{\bf\Psi}_j\|_{X_{\eps,h,\theta}}\leq R\leq1,\quad j=1,2,
\end{equation*}
one has
\begin{equation*}
 \|\mathcal N_{\eps,h,\theta}({\bf\Psi}_1)
       -\mathcal N_{\eps,h,\theta}({\bf\Psi}_2)\|_{Y_{\eps,h,\theta}}
 \le CR\|{\bf\Psi}_1-{\bf\Psi}_2\|_{X_{\eps,h,\theta}}.
\end{equation*}
\end{lemma}
\Cref{lem: quadratic estimate} is proved in \Cref{sec: quadratic estimate proof}.

\subsection{Projections and the linear theory}\label{sec: projections linear theory}
Define the rescaled Fermi map $\mathscr F_{\eps,h}\colon\Gamma\times B_{\delta/\eps}^3\to\mathscr T_h(\Tub_\delta(\Gamma))$
as
\begin{equation}\label{eq: rescaled Fermi map}
	\mathscr F_{\eps,h}(q,y)\coloneqq\mathscr T_h\bigl(\mathscr F_\Gamma(q,\eps y)\bigr).
\end{equation}
Let $\nu_{b,h}(q,z)$ be $\nu_b(q)$ parallel transported first along the normal radial geodesic to $\mathscr F_\Gamma(q,z)$, then along the short geodesic to $\mathscr T_h(\mathscr F_\Gamma(q,z))$. In the fibre formulas below, $F_h(q,z)\coloneqq\mathscr T_h(\mathscr F_\Gamma(q,z))$ denotes the Fermi map in physical normal coordinates. With $\lambda=\lambda_\eps$ in the compact case and $\lambda=0$ in the noncompact case, define the lifted modes by
\begin{equation}\label{eq: lifted monopole modes}
\begin{aligned}
 {\bf Z}_{a,\eps,h,\theta}\bigl(\mathscr F_{\eps,h}(q,y)\bigr)
 &\coloneqq\Ad_{(g_\eps u_\theta)^{-1}}
 \left(D_a\Phi_\lambda,\,\eps^{-1}(F_{A_\lambda})_{ab}\nu_{b,h}^{\flat}\right),
 &&a=1,2,3,\\[3pt]
 {\bf Z}_{4,\eps,h,\theta}\bigl(\mathscr F_{\eps,h}(q,y)\bigr)
 &\coloneqq\Ad_{(g_\eps u_\theta)^{-1}}
 \left(0,\,\eps^{-1}D_b(q_\lambda\sigma)\nu_{b,h}^{\flat}\right).
\end{aligned}
\end{equation}
Here
\begin{equation}\label{eq: model normal covariant derivative}
 D_a\coloneqq\partial_{y^a}+[A_{\lambda,a},\,\cdot\,],\qquad a=1,2,3.
\end{equation}
Model quantities are evaluated at $y$, the transported frame at $(q,\eps y)$, and the gauge at $\mathscr F_{\eps,0}(q,y)$.

\medskip

Next, consider the weighted cut-off 
\begin{equation*}
	\rho_\eps(y)\coloneqq \frac{\eta(4\delta^{-1}|y|)}{1+\frac{|y|^2}{\eps^2}}
\end{equation*}
Define
\begin{equation*}
 \mathfrak B_{\eps,h,\theta}\colon X_{\eps,h,\theta}\to C^{0,\tau}_\eps(\Gamma;\R^4),\qquad {\bf V}\longmapsto\bigl(\mathfrak b_j[{\bf V}]\bigr)_{j=1}^4,
\end{equation*}
where
\begin{equation}\label{eq: fibre constraints}
 \mathfrak b_j[{\bf V}](q)\coloneqq\frac1{\eps^{2}}\int_{B_{\delta/2}^3}\rho_\eps(z)\left\langle{\bf V}(F_h(q,z)),{\bf Z}_{j,\eps,h,\theta}(F_h(q,z))\right\rangle_\eps\,dz.
\end{equation}
The reconstruction map is
\begin{equation*}
 \mathfrak B_{\eps,h,\theta}^*\colon C^{0,\tau}_\eps(\Gamma;\R^4)\longrightarrow Y_{\eps,h,\theta}
\end{equation*}
as the adjoint for $L^2(\Gamma;\R^4)$ and $L^2_\eps(M)$:
\begin{equation}\label{eq: dual residual injection}
 \langle\mathfrak B_{\eps,h,\theta}^*\mathfrak b,{\bf V}\rangle_{L^2_\eps(M)}=\int_\Gamma\mathfrak b\cdot\mathfrak B_{\eps,h,\theta}[{\bf V}]d\mathrm{vol} _\Gamma.
\end{equation}
More explicitly, if $F_h^*d\vol_g=J_h(q,z)\,d\vol_\Gamma\,dz$, then
\begin{equation}\label{eq: explicit residual injection}
 (\mathfrak B_{\eps,h,\theta}^*\mathfrak b)(F_h(q,z))=\frac{\rho_\eps(z)}{\eps^2J_h(q,z)}\sum_{j=1}^4\mathfrak b_j(q){\bf Z}_{j,\eps,h,\theta}(F_h(q,z)),
\end{equation}
extended by zero outside the tube. For fixed $\eps,h,\theta$, suppress these subscripts and set
\begin{equation*}
 \mathcal E_{\eps,h,\theta}\coloneqq\{{\bf\Psi}\in X_{\eps,h,\theta}:\mathfrak B_{\eps,h,\theta}{\bf\Psi}=0\}.
\end{equation*}
The four orthogonality conditions hold at every point of $\Gamma$; the correction itself is defined on all of $M$.

\begin{proposition}\label[proposition]{prop: uniform inverse}
Under \eqref{eq: deltaeps} and \eqref{eq: assumptions htheta}, there exists $\eps_0>0$ such that, for $0<\eps<\eps_0$ and every ${\bf F}\in Y_{\eps,h,\theta}$, there is a unique pair
\begin{equation*}
 ({\bf\Psi},\mathfrak b)\in\mathcal E_{\eps,h,\theta}\times C^{0,\tau}_\eps(\Gamma;\R^4)
\end{equation*}
satisfying
\begin{equation}\label{eq: constrained linear equation}
 {\bf L}_{\eps,h,\theta}{\bf\Psi}={\bf F}+\mathfrak B_{\eps,h,\theta}^*\mathfrak b.
\end{equation}
Moreover, with $X=X_{\eps,h,\theta}$ and $Y=Y_{\eps,h,\theta}$,
\begin{equation}\label{eq: uniform inverse estimate}
 \|{\bf\Psi}\|_{X_{\eps,h,\theta}}+\|\mathfrak b\|_{C^{0,\tau}_\eps}+\|\mathfrak B^*\mathfrak b\|_{Y_{\eps,h,\theta}}\lesssim \|{\bf F}\|_{Y_{\eps,h,\theta}}.
\end{equation}
The same assertions hold for $\widehat{\bf L}_{h,\theta}$ in place of ${\bf L}_{h,\theta}$.
\end{proposition}

The proof of \Cref{prop: uniform inverse} is deferred to \Cref{sec: invertibility}.

\subsection{The nonlinear contraction}\label{sec: complementary nonlinear equation}
Let $\mathcal T_{\eps,h,\theta}$ denote the correction component of the inverse for $\widehat{\bf L}_{\eps,h,\theta}$ in \Cref{prop: uniform inverse}.
By \Cref{lem: error approx}, we have 
\begin{equation*}
	\sup_{h,\theta}\|{\bf S}_\eps(\Phi_{h,\theta},A_{h,\theta})\|_{Y_{\eps,h,\theta}}\lesssim a_\eps
\end{equation*}
where the supremum ranges over the space defined by \eqref{eq: assumptions htheta} and 
\begin{equation}\label{eq: error rate for contraction}
a_\eps\coloneqq \begin{cases}
 \sqrt{\eps\delta_\eps}+\eps^{3/2}(1+|\log\eps|),&M\text{ compact},\\
 \eps^{3/2}(1+|\log\eps|),&M\text{ noncompact}.
 \end{cases}
\end{equation}
For fixed $h,\theta$, consider
\begin{equation*}
 {\bf\Psi}=-\mathcal T_{\eps,h,\theta}({\bf S}_\eps(\Phi_{h,\theta},A_{h,\theta})+\mathcal N_{\eps,h,\theta}({\bf\Psi})).
\end{equation*}
By \Cref{lem: quadratic estimate}, the right-hand side preserves the closed ball of radius $C a_\eps$ in $\mathcal E_{\eps,h,\theta}$ and has Lipschitz constant at most $C' a_\eps$, where $C,C'>0$ are large constants. Therefore, it is a contraction in this ball for sufficiently small $\eps$ and it has a unique fixed point by Banach fixed point theorem. Recovering $\mathfrak b=\mathfrak b_{\eps,h,\theta}$ from the same mixed inverse gives
\begin{equation}\label{eq: complementary nonlinear system}
\begin{split}
 \widehat{\bf L}_{\eps,h,\theta}{\bf\Psi}&=-{\bf S}_\eps(\Phi_{h,\theta},A_{h,\theta})-\mathcal N_{\eps,h,\theta}({\bf\Psi})+\mathfrak B_{\eps,h,\theta}^*\mathfrak b_{\eps,h,\theta},\\
 \mathfrak B_{\eps,h,\theta}{\bf\Psi}&=0,
\end{split}
\end{equation}
and
\begin{equation}\label{eq: complementary solution bound}
 \|{\bf\Psi}\|_{X_{\eps,h,\theta}}+\|\mathfrak b_{\eps,h,\theta}\|_{C^{0,\tau}_\eps}+\|\mathfrak B_{\eps,h,\theta}^*\mathfrak b_{\eps,h,\theta}\|_{Y_{\eps,h,\theta}}\lesssim  a_\eps.
\end{equation}
Consequently, the perturbed pair in \eqref{eq: final eq}, still depending on $h,\theta$, satisfies the exact equation
\begin{equation}\label{eq: nonlinear residual equation}
 {\bf S}_\eps(\Phi_\eps,A_\eps)=\mathfrak B_{\eps,h,\theta}^*\mathfrak b_{\eps,h,\theta}-\eps^{-1}{\bf G}_{(\Phi_\eps,A_\eps)}\gamma_{\bf\Psi}.
\end{equation}
Our goal is to choose the parameters so that 
\begin{equation}\label{eq: b=0}
	\mathfrak b_{\eps,h,\theta}=(0,0,0,c_\eps)\quad \text{on }\Gamma,
\end{equation}
for some $c_\eps\in\R$. Then \Cref{lem: constant phase residual} gives ${\bf S}_\eps(\Phi_\eps,A_\eps)=0$. We do so in the next section.

\subsection{The reduced equation}\label{sec: reduced equation}

Consider the operator $\mathcal J\colon C^{2,\tau}(\Gamma;\R^4)\to C^{0,\tau}(\Gamma;\R^4)$ defined by
\begin{equation*}
	\mathcal J(h,\theta)\coloneqq (J_\Gamma h,\Delta_\Gamma\theta),
\end{equation*}
where $h=(h_1,h_2,h_3)$ is identified with the normal field $h^\alpha\nu_\alpha$.
Define the projection
\begin{equation}\label{eq: mean zero projection}
	\Pi_0(g,\varphi)\coloneqq
	\left(g,\varphi-
	\fint_{\Gamma}\varphi\,d\vol_\Gamma\right).
\end{equation}
Let $\mathcal J_0$ be the restriction of $\mathcal J$ to pairs whose phase has zero integral on $\Gamma$, with the same restriction on its codomain. Nondegeneracy of $\Gamma$ makes $\mathcal J_0$ invertible.

Define
\begin{equation*}
	d_\eps\bigl((h,\theta),(\widetilde h,\widetilde\theta)\bigr)
	\coloneqq\frac{\|h-\widetilde h\|_{C^{2,\tau}}}{\eps|\log\eps|^2}
	+\|\theta-\widetilde\theta\|_{C^{2,\tau}}.
\end{equation*}
We work on the complete closed ball
\begin{equation*}
	\mathcal B_{\eps,K}\coloneqq
	\left\{(h,\theta):d_\eps\bigl((h,\theta),(0,0)\bigr)\leq K,
	\quad\int_{\Gamma}\theta=0\right\}.
\end{equation*}
The distance condition implies the bounds in \eqref{eq: assumptions htheta}.

\begin{proposition}\label[proposition]{prop: reduced equation}
There is a map
$\mathcal R_\eps=(\mathcal R_\eps^{\mathrm{tr}},\mathcal R_\eps^{\mathrm{ph}})\colon\mathcal B_{\eps,K}\to C^{0,\tau}(\Gamma;\R^3)\times C^{0,\tau}(\Gamma;\R)=C^{0,\tau}(\Gamma;\R^4)$ such that \eqref{eq: b=0} is equivalent to
\begin{equation}\label{eq: reduced Jacobi Laplacian system}
	\mathcal J(h,\theta)=\Pi_0\mathcal R_\eps(h,\theta).
\end{equation}
Moreover, there are rates $\omega_\eps,\upsilon_\eps\to0$ such that
\begin{equation}\label{eq: reduced size estimates}
\begin{split}
	\|\mathcal R_\eps^{\mathrm{tr}}(h,\theta)\|_{C^{0,\tau}}
	&\lesssim \eps|\log\eps|^2+\omega_\eps\|h\|_{C^{2,\tau}},\\
	\|\mathcal R_\eps^{\mathrm{ph}}(h,\theta)\|_{C^{0,\tau}}
	&\lesssim 1,
\end{split}
\end{equation}
For all $(h_j,\theta_j)\in\mathcal B_{\eps,K}$, $j=1,2$,
\begin{equation}\label{eq: reduced Lipschitz estimates}
\begin{split}
 \|\mathcal R_\eps^{\mathrm{tr}}(h_1,\theta_1)-\mathcal R_\eps^{\mathrm{tr}}(h_2,\theta_2)\|_{C^{0,\tau}}
 &\lesssim \upsilon_\eps\|h_1-h_2\|_{C^{2,\tau}}
 +\upsilon_\eps\eps|\log\eps|^2\|\theta_1-\theta_2\|_{C^{2,\tau}},\\
 \|\mathcal R_\eps^{\mathrm{ph}}(h_1,\theta_1)-\mathcal R_\eps^{\mathrm{ph}}(h_2,\theta_2)\|_{C^{0,\tau}}
 &\lesssim \upsilon_\eps\left(
 \frac{\|h_1-h_2\|_{C^{2,\tau}}}{\eps|\log\eps|^2}
 +\|\theta_1-\theta_2\|_{C^{2,\tau}}\right).
\end{split}
\end{equation}
\end{proposition}
The proof of \Cref{prop: reduced equation}, including the identification and estimates of $\mathcal R_\eps$, is given in \Cref{sec: reduced equation proof}. Set
\begin{equation*}
	\mathcal T_\eps\coloneqq\mathcal J_0^{-1}\Pi_0\mathcal R_\eps.
\end{equation*}
For $K$ sufficiently large and $0<\eps<\eps_0(K)$, the estimates in \Cref{prop: reduced equation}, Schauder bounds for $\mathcal J_0^{-1}$ and 
Banach's fixed point theorem give a unique $(h_\eps,\theta_\eps)\in\mathcal B_{\eps,K}$ satisfying
\begin{equation*}
	\mathcal J(h_\eps,\theta_\eps)=\Pi_0\mathcal R_\eps(h_\eps,\theta_\eps).
\end{equation*}
hence, we proved \eqref{eq: b=0}.
By \eqref{eq: nonlinear residual equation}, the corresponding perturbed pair satisfies
\begin{equation}\label{eq: remaining phase residual}
	{\bf S}_\eps(\Phi_\eps,A_\eps)
	=\mathfrak B_{\eps,h_\eps,\theta_\eps}^*(0,0,0,c_\eps)
	-\eps^{-1}{\bf G}_{(\Phi_\eps,A_\eps)}\gamma_{\bf\Psi},
\end{equation}
for some constant $c_\eps\in\R$. The following lemma removes this remaining phase residual. 
\begin{lemma}\label{lem: constant phase residual}
There exist $\eps_0,m_0>0$ such that, if $0<\eps<\eps_0$, the parameters satisfy \eqref{eq: assumptions htheta},
\begin{equation*}
 {\bf\Psi}\in X_{\eps,h,\theta},\qquad
 \|{\bf\Psi}\|_{X_{\eps,h,\theta}}\leq m_0,
\end{equation*}
and a constant $c_\eps\in\R$ satisfies
\begin{equation*}
 {\bf S}_\eps(\Phi_{h,\theta}+\Psi,A_{h,\theta}+B)
 +\eps^{-1}{\bf G}_{(\Phi_{h,\theta}+\Psi,A_{h,\theta}+B)}\gamma_{\bf\Psi}
 =\mathfrak B_{\eps,h,\theta}^*(0,0,0,c_\eps),
\end{equation*}
then
\begin{equation*}
 c_\eps=0,\qquad
 {\bf G}_{(\Phi_{h,\theta}+\Psi,A_{h,\theta}+B)}\gamma_{\bf\Psi}=0.
\end{equation*}
This holds in both the compact and noncompact cases.
\end{lemma}
The proof is deferred to \Cref{sec: technical gauge}.

\begin{proof}[Proof of Theorems \ref{thm: main noncompact} and \ref{thm: main}]
The fixed points give parameters $(h_\eps,\theta_\eps)$ and a correction ${\bf\Psi}$ satisfying \eqref{eq: remaining phase residual}. Moreover, $\|{\bf\Psi}\|_{X_{\eps,h_\eps,\theta_\eps}}\lesssim a_\eps\to0$. Hence \Cref{lem: constant phase residual} gives $c_\eps=0$ and \eqref{eq: final eq}.

Pull back to each normal fibre by \eqref{eq: introductory normal rescaling} and undo the inner and phase gauges. On every fixed $B_R$, the normal components of the approximation differ from $(\Phi_{\lambda_\eps},A_{\lambda_\eps})$ by $O_R(\eps)$ in $C^{2,\tau}$, and the pulled-back correction is $O_R(a_\eps)$, uniformly along $\Gamma$. Since $h_\eps\to0$ in $C^{2,\tau}$, this proves \eqref{eq: normal profile convergence}.

For the noncompact exterior, choose fixed smooth compact domains $K\Subset K_1$ containing all gluing regions. If a real scalar or one-form $v$ satisfies
\begin{equation*}
 \|v\|_{C^{2,\tau}_\nu(M\setminus K)}
 +\|\Delta_pv\|_{C^{0,\tau}_{\nu+2}(M\setminus K)}\leq C_v,
 \qquad p=0,1,
\end{equation*}
testing against a cutoff times $v$ gives
\begin{equation}\label{eq: exterior graph energy}
\begin{split}
 \int_{M\setminus K_1}(|dv|^2+|d^*v|^2)\,d\vol_g
 &\lesssim C_v^2+\int_{M\setminus K}|v|\,|\Delta_pv|\,d\vol_g\\
 &\quad+\limsup_{L\to\infty}L^{-2}\int_{\{L<R<2L\}}|v|^2\,d\vol_g
 \lesssim C_v^2.
\end{split}
\end{equation}
Here $d^*v=0$ for scalars; the last term vanishes because $R^{-2\nu-2}$ is integrable by \eqref{eq: noncompact weighted volume}. A cutoff of $H$ has compactly supported Laplacian, so the scalar inverse in \Cref{def: admissible amb noncomp} and the maximum principle give $H\in C^{2,\tau}_\nu$ off $K$. The form $u$ in the proof of \Cref{lem: noncompact exterior correction} also belongs to this space, has compactly supported Laplacian, and satisfies $f_1=du$ off $K$. Thus \eqref{eq: exterior graph energy} yields
\begin{equation*}
 \int_{M\setminus K_1}(|dH|^2+|f_1|^2)\,d\vol_g<\infty.
\end{equation*}
The exterior approximation $((1-\eps H)\sigma,A_D)$ consequently has energy $O(\eps)$.

Normal rescaling gives the energy density $c_{\lambda_\eps}$ along the shifted submanifold. The radial profile estimates give a uniform $O(R^{-1})$ bound for the model energy outside $B_R$; the geometric and matching corrections contribute $o(1)$. First letting $\eps\to0$ and then $R\to\infty$ proves \eqref{eq: energy concentration} and the asserted total-energy asymptotics for the approximation.

These conclusions persist under the correction. Indeed, the solution norm bounds the squared weighted $L^2$ differences of the covariant Higgs derivatives and curvatures by $Ca_\eps^2/\eps=o(1)$. On the noncompact exterior, write $\Psi^\parallel=\psi\sigma$ and $B^\parallel=b\sigma$. Apply \eqref{eq: exterior graph energy} to $\psi$ and $\eps b$, with $C_v\lesssim a_\eps$, using \eqref{eq: noncompact exterior solution norm}. This gives precisely the stated bound for their derivative terms. The transverse terms and commutators satisfy the same bound directly from the weights $\nu+2$ in that norm and $[\R\sigma,\R\sigma]=0$. In the compact case, its longitudinal weight also gives
\begin{equation*}
\frac1{\delta_\eps}\int_M\left(|\Phi_{h_\eps,\theta_\eps}+\Psi|^2-|\Phi_{h_\eps,\theta_\eps}|^2\right)^2\,d\vol_g\lesssim\frac{a_\eps^2}{\eps}=o(1).
\end{equation*}
Cauchy--Schwarz and the bounded energy of the approximation therefore yield
\begin{equation*}
\left\|e_\eps(\Phi_\eps,A_\eps)-e_\eps(\Phi_{h_\eps,\theta_\eps},A_{h_\eps,\theta_\eps})\right\|_{L^1(M)}=o(1).
\end{equation*}
This proves the measure convergence and energy asymptotics for the exact solutions. Finally, continuity of the model profiles and the uniform tail bound give $c_{\lambda_\eps}\to c_{\bar\lambda}$ whenever $\lambda_\eps\to\bar\lambda$, with $c_0=2\pi$.
\end{proof}

\section{Errors of approximation}\label{sec: approximation errors}
In this section we prove \Cref{lem: error approx}, starting with the compact case.
\subsection{The error in the compact case}

First, we record the profile and matching estimates. The scalar decay length is
\begin{equation}\label{eq: approximation scalar length}
\ell_\eps\coloneqq\frac{\eps}{\sqrt{2\lambda_\eps}} =\frac12\sqrt{\frac{\delta_\eps}{\eps}},\qquad \frac{\sqrt C_*}{2}\eps\leq \ell_\eps\leq \eps^{\tau'},
\end{equation}
where $\tau'\in (\tau,\frac12)$ and we used \eqref{eq: deltaeps}. We need two preliminary results.
\begin{lemma}\label{lem: approximation profile bounds}
There exist constants $M_0>0$ and $c\in (0,1)$, independent of $\eps$, such that, for all sufficiently small $\eps>0$, $r\geq M_0\eps$ and $0\leq k\leq3$,
\begin{equation}\label{eq: approximation profile bounds}
\begin{split}
r^k\left|\partial_r^k\left(1-\varphi_{\lambda_\eps}(r/\eps)\right)\right| &\lesssim \frac{\eps}{r}e^{-cr/\ell_\eps},\\
\eps^k\left|\partial_r^k\left(1-\alpha_{\lambda_\eps}(r/\eps)\right)\right| &\lesssim e^{-cr/\eps},\\
r^k\left|\partial_r^k\left(1-q_{\lambda_\eps}(r/\eps)\right)\right| &\lesssim \frac{\eps}{r}.
\end{split}
\end{equation}
\end{lemma}
\begin{proof}
Apply comparison to the radial profile equations and integrate \eqref{eq: qlambda}, then differentiate these equations to obtain the derivative bounds.
\end{proof}

The following matching estimate refines \Cref{lem: exterior correction}.
\begin{lemma}\label{lem: approximation regular matching}
Let $b_r$ and $\beta\coloneqq b_r|_\Gamma$ be as in \Cref{lem: exterior correction}, and extend $\beta$ smoothly using Fermi coordinates. For all sufficiently small $r>0$, let 
\begin{equation*}
	\Omega_r\coloneqq \{x\in M:r<\dist_g(x,\Gamma)<2r\}.
\end{equation*}
Then,
\begin{equation}\label{eq: approximation regular matching}
\sum_{k=0}^2r^k\|\nabla^k(b_r-\beta)\|_{\tau,r}\lesssim r(1+|\log r|).
\end{equation}
\end{lemma}
\begin{proof}
The source $R$ defined in \eqref{eq: exterior regular source} satisfies
\begin{equation}\label{eq: bound R}
 \Delta_1b_r=R,\qquad
 \rho\|R\|_{\tau,\rho,B_{\rho/2}(x)}\lesssim 1,\qquad
 \int_{B_t(z)}|R|\,d\vol_g\lesssim t^{n-1},
\end{equation}
where $\rho=r(x)>0$ and $0<t\leq2t_0$, for a fixed sufficiently small $t_0>0$. Since $b_1(M)=0$, the Hodge Green kernel $K$ from the proof of \Cref{lem: exterior correction} gives
\begin{equation*}
 b_r(x)=\int_M K(x,z)R(z)\,d\vol_g(z).
\end{equation*}
For $0<d_M(x,x')<t_0/4$, split the integral over $B_{2d_M(x,x')}(x)\cup B_{2d_M(x,x')}(x')$ and its complement. The kernel bounds and the preceding mass estimate give
\begin{equation*}
\begin{split}
 |b_r(x)-b_r(x')|&\lesssim d_M(x,x')+\sum_{p\in\{x,x'\}}\int_{B_{3d_M(x,x')}(p)}\frac{|R(z)|}{d_M(p,z)^{n-2}}\,d\vol_g(z)\\
 &+d_M(x,x')\sum_{\substack{j\, :\, 2^jd_M(x,x')<t_0}}
 \frac{\displaystyle\int_{B_{2^{j+1}d_M(x,x')}(x)}|R|\,d\vol_g}{(2^jd_M(x,x'))^{n-1}}\\
 &\lesssim d_M(x,x')\Bigl(1+\sum_{k=0}^{\infty}2^{-k}+\sum_{\substack{j\, :\,2^jd_M(x,x')<t_0}}1\Bigr)\\
 &\lesssim d_M(x,x')(1+|\log d_M(x,x')|).
\end{split}
\end{equation*}
Sending $x'$ to the nearest point $p$ of $x$ in $\Gamma$, we get $|b_r(x)-\beta(p)|\lesssim r(x)(1+|\log r(x)|).$
Extending $\beta$ smoothly in the fixed Fermi coordinates and setting $w=b_r-\beta$ we find $|w(x)|\lesssim r(x)(1+|\log r(x)|)$. Scaled Schauder estimates with $\rho=r(x)$ now yield, using the second bound in \eqref{eq: bound R} and smoothness of $\Delta_1\beta$,
\begin{equation*}
\begin{split}
 \sum_{k=0}^2\rho^k\|\nabla^kw\|_{\tau,\rho,B_{\rho/4}(x)}&\lesssim \|w\|_{L^\infty(B_{\rho/2}(x))}
 +\rho^2\|R-\Delta_1\beta\|_{\tau,\rho,B_{\rho/2}(x)}\\
 &\lesssim \rho(1+|\log\rho|).
\end{split}
\end{equation*}
Covering the tubular annuli by these balls proves \eqref{eq: approximation regular matching}.
\end{proof}

\begin{proof}[Proof of \Cref{lem: error approx}, compact case]
We first estimate the error for $h=\theta=0$ in the following three regions:
\begin{equation*}
 \{r<\sqrt\eps\},\quad \{\sqrt\eps\leq r\leq2\sqrt\eps\},\quad
 \{r>2\sqrt\eps\}.
\end{equation*}

\subsubsection*{The core and inner annuli: $r<\sqrt\eps$}
On this region define the complete inner pair of \Cref{sec: approx compact}, before applying $g_\eps$, by
\begin{equation*}
 \Phi\coloneqq\Phi_\Gamma^0=\varphi_{\lambda_\eps}(r/\eps)\sigma,
 \qquad A\coloneqq A_\Gamma^0+M+\xi_\eps\beta_i\,dq^i.
\end{equation*}
Thus $\Phi_\Gamma^{\mathrm c}=\Phi$ and $A_\Gamma^{\mathrm c}=g_\eps^{-1}Ag_\eps+g_\eps^{-1}dg_\eps$.
Since $g_\eps$ commutes with $\Phi$, gauge covariance allows us to compute the estimate with $(\Phi,A)$; the H\"older norms from \Cref{sec: norms}, defined using covariant parallel transport, are preserved by this change of gauge.

\smallskip

In Fermi coordinates set, for a vector field $V$,
\begin{equation*}
 X_i\coloneqq\partial_{q^i}-\omega_{ib}^{a}y^b\partial_{y^a},
 \qquad D_Vu\coloneqq V(u)+[A(V),u],\qquad D_a\coloneqq D_{\partial_{y^a}}.
\end{equation*}
where $\omega_{ib}^a$ are as in \eqref{eq: omega}.
The definition of $M$ in \Cref{sec: approx compact}, together with \eqref{eq: q_eps} and \eqref{eq: rho}, gives
\begin{equation}\label{eq: approximation horizontal cancellation}
\begin{aligned}
 A(X_i)&=\rho_i+\beta_i\xi_\eps,
 &D_{X_i}\Phi&=D_{X_i}\xi_\eps=0,\\
 F_A(X_i,\partial_{y^a})&=-\beta_iD_a\xi_\eps,
 &F_A(\partial_{y^a},\partial_{y^b})&=F_{A_\Gamma^0}(\partial_{y^a},\partial_{y^b}).
\end{aligned}
\end{equation}
Indeed, $X_i\sigma=-[\rho_i,\sigma]$ and $A_\sigma(X_i)=\rho_i-\langle\rho_i,\sigma\rangle\sigma$; rotation equivariance of the normal connection gives the curvature identities. Since $A(\partial_{y^a})=A_\Gamma^0(\partial_{y^a})$, \eqref{eq: approximation horizontal cancellation} and \eqref{eq: qlambda} yield
\begin{equation}\label{eq: approximation normal cancellation}
 D_aD_a\xi_\eps=0,
 \qquad D_a\bigl(F_A(X_i,\partial_{y^a})\bigr)
 =-\beta_iD_aD_a\xi_\eps=0.
\end{equation}
Also, define the tangential curvature two-form
\begin{equation*}
 F_A^{\mathrm{tan}}\coloneqq\frac12F_A(X_i,X_k)\,dq^i\wedge dq^k.
\end{equation*}
We have
\begin{equation}\label{eq: FAtan}
 |F_A^{\mathrm{tan}}|+r|\nabla_A F_A^{\mathrm{tan}}|\lesssim 1,
 \qquad |d_A^*F_A^{\mathrm{tan}}|\lesssim r^{-1}.
\end{equation}
Write the residual of the actual approximation as
\begin{equation*}
 (E_\Phi,E_A)\coloneqq{\bf S}_\eps(\Phi_\Gamma^{\mathrm c},A_\Gamma^{\mathrm c}).
\end{equation*}
We use the regions and scales of \Cref{sec: norms}:
\begin{equation*}
 \Omega_0=\{r<2M_0\eps\},\quad r_0=\eps,\qquad
 \Omega_j=\{r_j<r<4r_j\},\quad r_j=2^{j-1}M_0\eps.
\end{equation*}
For now, take $1\leq j\leq J_\eps$ with $4r_j\leq \sqrt\eps$, so that $\Omega_j$ is contained in the region $\{r<\sqrt\eps\}$. In \eqref{eq: source norm}, the Higgs and connection components on $\Omega_j$ carry the weights $r_j^2$ and $\eps r_j^2$, respectively. The projections of \Cref{sec: norms} along $\sigma$ and onto $\sigma^\perp$ are uniformly bounded in the scaled H\"older norms, so the full-component expression
\begin{equation*}
 r_j^2\bigl(\|E_\Phi\|_{\tau,r_j,\Omega_j}
 +\eps\|E_A\|_{\tau,r_j,\Omega_j}\bigr)
\end{equation*}
controls each of the corresponding projected contributions, up to a uniform constant. Recall that the longitudinal contributions are summed and the transverse contributions maximised, by \eqref{eq: source norm}.

In the frame $(X_i,\partial_{y^a})$, the normal and mixed metric coefficients differ from their product values by $O(r^2)$, and the tangential coefficients differ by $O(r)$. Minimality gives the radial volume factor $r^2(1+O(r^2))$. Using \eqref{eq: approximation horizontal cancellation} and \eqref{eq:monopole-profiles}, the Euclidean normal terms cancel, leaving a radial contribution to $E_\Phi$ bounded by $Cr|\partial_r\varphi_{\lambda_\eps}(r/\eps)|$. By \eqref{eq: approximation profile bounds}, its contribution with the source-norm weight is
\begin{equation}\label{eq: approximation inner Higgs bound}
 r_j^2\bigl\|r\partial_r\varphi_{\lambda_\eps}(r/\eps)\bigr\|_{\tau,r_j,\Omega_j}
 \lesssim \eps r_j e^{-cr_j/\ell_\eps}.
\end{equation}
For $E_A$, use \eqref{eq: approximation horizontal cancellation} and \eqref{eq: approximation normal cancellation}. The metric errors in the normal curvature are bounded by $C(r^2|\nabla_A F_A|+r|F_A|)$, \eqref{eq: FAtan} controls $d_A^*F_A^{\mathrm{tan}}$, and \eqref{eq: approximation profile bounds} bounds the remaining trace terms by $C\eps/r^2$. Equation \eqref{eq: approximation profile bounds} also controls the scaled H\"older seminorms. Since $r_j<r<4r_j$ on $\Omega_j$, multiplication by the connection weight $\eps r_j^2$ gives
\begin{equation}\label{eq: approximation inner connection bound}
\eps r_j^2\left(r_j^2\|\nabla_A F_A\|_{\tau,r_j,\Omega_j} +r_j\|F_A\|_{\tau,r_j,\Omega_j}+r_j^{-1}+\eps r_j^{-2}\right)\lesssim (\eps r_j+\eps^2).
\end{equation}
The norms involving $A$ in this calculation use its covariant parallel transport; gauge covariance transfers \eqref{eq: approximation inner connection bound} to the residual $(E_\Phi,E_A)$. Together with \eqref{eq: approximation inner Higgs bound}, this gives
\begin{equation}\label{eq: approximation inner shells}
 r_j^2\bigl(\|E_\Phi\|_{\tau,r_j,\Omega_j}+\eps\|E_A\|_{\tau,r_j,\Omega_j}\bigr)
 \lesssim \eps r_j e^{-cr_j/\ell_\eps}+\eps r_j+\eps^2.
\end{equation}
On the core, rescaling \eqref{eq:monopole-profiles}, \eqref{eq: approximation horizontal cancellation} and \eqref{eq: approximation normal cancellation} by $r_0=\eps$ gives the first term of \eqref{eq: source norm}:
\begin{equation}\label{eq: approximation core bound}
 r_0^2(\|E_\Phi\|_{\tau,r_0,\Omega_0}
 +\eps\|E_A\|_{\tau,r_0,\Omega_0})\lesssim \eps^2.
\end{equation}
The finitely many shells meeting $r=\sqrt\eps$ have $r_j\asymp \sqrt\eps$; their inner portions satisfy \eqref{eq: approximation inner shells}, and their full norms follow by combining it with \eqref{eq: approximation gluing shells} below. 

\subsubsection*{The gluing annulus: $\sqrt\eps\leq r\leq2\sqrt\eps$}
For each $1\leq j\leq J_\eps$ whose shell meets this region, put
\begin{equation*}
 \Omega_{j,\mathrm g}\coloneqq\Omega_j\cap\{\sqrt\eps\leq r\leq2\sqrt\eps\}.
\end{equation*}
There are uniformly finitely many such indices, and $r_j\asymp \sqrt\eps$. Define the connection difference
\begin{equation*}
 \mathcal D_\eps\coloneqq A^{\inn}-(A_\Gamma^0+b\sigma).
\end{equation*}
Here $\nabla$ denotes covariant differentiation using $A_\Gamma^{\mathrm c}$ and the Levi--Civita connection. By \eqref{eq: approximation regular matching} and \eqref{eq: approximation profile bounds}, the two connections being interpolated satisfy
\begin{equation}\label{eq: approximation gluing difference}
 \sum_{k=0}^2r_j^k\|\nabla^k\mathcal D_\eps\|_{\tau,r_j,\Omega_{j,\mathrm g}}
 \lesssim r_j(1+|\log r_j|)+\frac{\eps}{r_j}
 +\eps^{-N}e^{-cr_j/\eps}
\end{equation}
for a fixed $N>0$, independent of $j$ and $\eps$. The angular and tangential trace differences are 
\begin{equation*}
\begin{split}
 M-b_s\sigma&=(1-\alpha_{\lambda_\eps}(r/\eps))
 (\rho_i-\langle\rho_i,\sigma\rangle\sigma)dq^i,\\
 \xi_\eps\beta_idq^i-\beta_i\sigma dq^i&=(q_\eps-1)\beta_i\sigma dq^i.
\end{split}
\end{equation*}
In the residual of $A_\Gamma^{\mathrm c}=A_\Gamma^0+b\sigma+\kappa_\eps\mathcal D_\eps$, let $(E_{\Phi,\kappa},E_{A,\kappa})$ denote the terms containing a derivative of $\kappa_\eps$. By \eqref{eq: s_eps}, $|\nabla^k\kappa_\eps|\leq C\eps^{-k/2}$ for $k=1,2$. Using \eqref{eq: approximation gluing difference}, their contribution with the weights in \eqref{eq: source norm} satisfies
\begin{equation}\label{eq: approximation cutoff error}
\begin{split}
 r_j^2\bigl(\|E_{\Phi,\kappa}\|_{\tau,r_j,\Omega_{j,\mathrm g}}
 +\eps\|E_{A,\kappa}\|_{\tau,r_j,\Omega_{j,\mathrm g}}\bigr)
 &\lesssim \eps\sum_{k=0}^2r_j^k\|\nabla^k\mathcal D_\eps\|_{\tau,r_j,\Omega_{j,\mathrm g}}
 +\eps^{-N}e^{-cr_j/\eps}\\
 &\lesssim \eps r_j(1+|\log r_j|)+\frac{\eps^2}{r_j}.
\end{split}
\end{equation}
Here $\mathcal D_\eps^\parallel$ commutes with $\Phi_\Gamma^0$, while \eqref{eq: approximation profile bounds} makes $\mathcal D_\eps^\perp$ and its scaled derivatives exponentially small. Thus the commutator terms with coefficient $\eps^{-2}$ are absorbed by the exponential remainder. By \eqref{eq: approximation gluing difference}, the terms quadratic in $\mathcal D_\eps$ contribute at most $C(\eps r_j(1+|\log r_j|)+\eps^2/r_j)$ as well.

Combining \eqref{eq: approximation cutoff error} and \eqref{eq: approximation inner shells} with the outer-pair calculation leading to \eqref{eq: approximation outer shells} below gives
\begin{equation}\label{eq: approximation gluing shells}
r_j^2\bigl(\|E_\Phi\|_{\tau,r_j,\Omega_{j,\mathrm g}}+\eps\|E_A\|_{\tau,r_j,\Omega_{j,\mathrm g}}\bigr) \lesssim \eps r_j e^{-cr_j/\ell_\eps} +\eps r_j(1+|\log r_j|)+\frac{\eps^2}{r_j}.
\end{equation}

By \eqref{eq: source norm}, the left-hand side of \eqref{eq: approximation gluing shells} controls the longitudinal and transverse contributions on $\Omega_{j,\mathrm g}$. The local bounds in \eqref{eq: approximation inner shells}, \eqref{eq: approximation gluing shells} and \eqref{eq: approximation outer shells} also control H\"older pairs across the common boundaries, giving the full norm on each crossing shell at the same scale $r_j$.

\subsubsection*{The outer region: $r>2\sqrt\eps$}
By \eqref{eq: s_eps}, $\kappa_\eps=0$ here. With the compact profiles $\varphi_\eps,\alpha_\eps$ of \Cref{sec: approx compact}, the complete pair is
\begin{equation*}
 \Phi_\Gamma^{\mathrm c}=\varphi_\eps\sigma,\qquad
 A_\Gamma^{\mathrm c}=\alpha_\eps A_\sigma+b\sigma.
\end{equation*}
Writing $A_D$ for the exterior connection \eqref{eq: exterior connection}, \eqref{eq: eq for b} and \eqref{eq: curvature identity} give
\begin{equation}\label{eq: approximation exterior cancellation}
 d_{A_D}\sigma=0,\qquad F_{A_D}=\sigma(f_0+db),\qquad
 d_{A_D}^*F_{A_D}=\sigma d_g^*(f_0+db)=0.
\end{equation}
For $1\leq j\leq J_\eps$, put
\begin{equation*}
 \Omega_{j,\mathrm o}\coloneqq\Omega_j\cap\{r>2\sqrt\eps\}.
\end{equation*}
On $2\sqrt\eps<r<\delta/2$, the profiles are $\varphi_{\lambda_\eps}(r/\eps)$ and $\alpha_{\lambda_\eps}(r/\eps)$. By \eqref{eq: approximation profile bounds}, the difference $A_\Gamma^{\mathrm c}-A_D=(\alpha_{\lambda_\eps}(r/\eps)-1)A_\sigma$ and its derivatives are exponentially small. On $\delta/2\leq r<\delta$, derivatives of $\chi_\Gamma$ multiply the same profile tails. Thus \eqref{eq: approximation exterior cancellation}, \eqref{eq: approximation profile bounds} and \eqref{eq:monopole-profiles} give, for every nonempty $\Omega_{j,\mathrm o}$,
\begin{equation}\label{eq: approximation outer shells}
 r_j^2\bigl(\|E_\Phi\|_{\tau,r_j,\Omega_{j,\mathrm o}}
 +\eps\|E_A\|_{\tau,r_j,\Omega_{j,\mathrm o}}\bigr)
 \lesssim \eps r_j e^{-cr_j/\ell_\eps}+\eps^{-N}e^{-cr_j/\eps},
\end{equation}
where $N>0$ is fixed. The left-hand side uses exactly the Higgs and connection weights in \eqref{eq: source norm}; it controls each projected contribution on this part of $\Omega_j$. The first term on the right bounds the radial Higgs error, and the second bounds the transverse connection and Higgs errors. Fixed cutoff terms obey the same estimate after decreasing $c>0$.

The remaining norm region is $\Omega_{J_\eps+1}=\{r>\delta/4\}$, with its prescribed scale $r_{J_\eps+1}=\delta$. By \eqref{eq: deltaeps}, $\ell_\eps\leq\tfrac12\eps^{\tau'}$, so \eqref{eq: approximation profile bounds} makes the profile tails and their derivatives smaller than every power of $\eps$. Hence \eqref{eq: approximation outer shells} gives, for every $P>0$,
\begin{equation}\label{eq: approximation fixed cutoff}
 r_{J_\eps+1}^2\bigl(\|E_\Phi\|_{\tau,r_{J_\eps+1},\Omega_{J_\eps+1}}
 +\eps\|E_A\|_{\tau,r_{J_\eps+1},\Omega_{J_\eps+1}}\bigr)
 \lesssim_P\eps^P.
\end{equation}
This controls the $j=J_\eps+1$ contributions in \eqref{eq: source norm}. On $r\geq\delta$, the construction in \Cref{sec: approx compact} gives $(\Phi_\Gamma^{\mathrm c},A_\Gamma^{\mathrm c})=(\sigma,A_D)$, so \eqref{eq: approximation exterior cancellation} gives $(E_\Phi,E_A)=(0,0)$. By \eqref{eq: shifted phase approximation}, the shifted residual also vanishes there.

\smallskip\noindent\emph{Dependence on $h$ and $\theta$.}
We now include the parameters in all three regions.
Pulling \eqref{eq: shifted phase approximation} back by $\mathscr T_h$ and undoing the phase gauge gives the pair $(\Phi_\Gamma^{\mathrm c},A_\Gamma^{\mathrm c}-\xi_\eps d\bar\theta)$ with metric $\mathscr T_h^*g$. Let $(E_{\Phi,h,\theta},E_{A,h,\theta})$ denote its residual in these fixed coordinates. Since $[\xi_\eps,\Phi_\Gamma^{\mathrm c}]=0$, the phase changes the curvature by
\begin{equation}\label{eq: approximation phase curvature}
 F_{A_\Gamma^{\mathrm c}-\xi_\eps d\bar\theta}=F_{A_\Gamma^{\mathrm c}}-d_{A_\Gamma^{\mathrm c}}\xi_\eps\wedge d\bar\theta.
\end{equation}
Using \eqref{eq: approximation normal cancellation} in \eqref{eq: approximation phase curvature}, the phase terms left in the residual contain tangential derivatives of $\theta$ or metric errors. Equations \eqref{eq: approximation profile bounds} and \eqref{eq: assumptions htheta} bound their weighted norm on each annulus by $C\eps^2$. On the gluing annulus, \eqref{eq: qlambda} and \eqref{eq: approximation profile bounds} make the defect in the Euclidean normal equation for $\xi_\eps$ exponentially small.

Replacing $g$ by $\mathscr T_h^*g$ changes the normal principal coefficients by $O(r\|h\|_{C^{2,\tau}}+\|h\|_{C^{2,\tau}}^2)$, the mixed coefficients by $O(\|h\|_{C^{2,\tau}})$, and their first derivatives by $O(\|h\|_{C^{2,\tau}})$. Applying these coefficient bounds with \eqref{eq: approximation profile bounds}, \eqref{eq: approximation horizontal cancellation} and \eqref{eq: FAtan} gives the following estimate for the change in the residual on each $\Omega_j$, $1\leq j\leq J_\eps$. The left-hand side uses the same weights as the corresponding source-norm contributions:
\begin{equation}\label{eq: approximation parameter shells}
\begin{split}
 &r_j^2\bigl(\|E_{\Phi,h,\theta}-E_\Phi\|_{\tau,r_j,\Omega_j}
 +\eps\|E_{A,h,\theta}-E_A\|_{\tau,r_j,\Omega_j}\bigr)\lesssim \eps^2+\eps\|h\|_{C^{2,\tau}}
 +\frac{\eps\|h\|_{C^{2,\tau}}^2}{\eps+r_j}.
\end{split}
\end{equation}
On the core, \eqref{eq: approximation parameter shells} holds with $j=0$ and $r_0=\eps$. On $\Omega_{J_\eps+1}$, using $r_{J_\eps+1}=\delta$, the comparison gives
\begin{equation}\label{eq: approximation exterior parameters}
\begin{split}
 &r_{J_\eps+1}^2\bigl(\|E_{\Phi,h,\theta}-E_\Phi\|_{\tau,r_{J_\eps+1},\Omega_{J_\eps+1}}
 +\eps\|E_{A,h,\theta}-E_A\|_{\tau,r_{J_\eps+1},\Omega_{J_\eps+1}}\bigr)\lesssim \eps^2+\eps\|h\|_{C^{2,\tau}}.
\end{split}
\end{equation}
On $\{r\geq\delta\}$, both residuals vanish by \eqref{eq: approximation exterior cancellation} and \eqref{eq: shifted phase approximation}. The source norms before and after pullback and gauge transformation are uniformly equivalent under \eqref{eq: assumptions htheta}.

\smallskip\noindent\emph{Summation of the error estimates.}
To obtain the full source norm \eqref{eq: source norm}, sum the longitudinal contributions over the annuli and bound the maximum of the transverse contributions by their sum. The dyadic radii from \Cref{sec: norms} satisfy
\begin{equation}\label{eq: approximation dyadic sums}
 \sum_{j=1}^{J_\eps}r_j e^{-cr_j/\ell_\eps}\lesssim \ell_\eps,\qquad
 \sum_{\substack{1\leq j\leq J_\eps\\r_j\leq2\sqrt\eps}}r_j\lesssim \sqrt\eps,\qquad
 \sum_{j=1}^{J_\eps}\frac{\eps}{\eps+r_j}\lesssim 1.
\end{equation}
Applying \eqref{eq: approximation dyadic sums} to \eqref{eq: approximation inner shells} and \eqref{eq: approximation outer shells}, the radial Higgs error contributes $\eps\ell_\eps$, the inner connection error contributes $\eps^{3/2}$, and the $\eps^2$ terms contribute $\eps^2(1+|\log\eps|)$. Add \eqref{eq: approximation core bound}, \eqref{eq: approximation gluing shells} and \eqref{eq: approximation fixed cutoff}, taking $P\geq2$ in the latter. Summing \eqref{eq: approximation parameter shells} by \eqref{eq: approximation dyadic sums} and including \eqref{eq: approximation exterior parameters} yields
\begin{equation*}
\begin{split}
\|{\bf S}_\eps(\Phi_{h,\theta}^{\mathrm c},A_{h,\theta}^{\mathrm c})\|_{Y_{\eps,h,\theta}}&\lesssim \biggl(\eps\ell_\eps+\eps^{3/2}(1+|\log\eps|)+\eps^2(1+|\log\eps|)\\
&\hspace{2em}+\eps(1+|\log\eps|)\|h\|_{C^{2,\tau}}+\|h\|_{C^{2,\tau}}^2\biggr)\\
&\lesssim \sqrt{\eps\delta_\eps}+\eps^{3/2}(1+|\log\eps|)+\eps^2(1+|\log\eps|)^4\\
&\lesssim \sqrt{\eps\delta_\eps}+\eps^{3/2}(1+|\log\eps|).
\end{split}
\end{equation*}
The second inequality uses \eqref{eq: approximation scalar length} and \eqref{eq: assumptions htheta}; the last uses $\eps^{1/2}(1+|\log\eps|)^3\to0$. All constants are uniform for $0<\lambda_\eps\leq\lambda_*$. Finally, the fixed-cutoff contribution is controlled by \eqref{eq: approximation fixed cutoff}.

\end{proof}

\subsection{The error in the noncompact case}
\begin{proof}[Proof of \Cref{lem: error approx}, noncompact case]
Use the regions and scales of \Cref{sec: norms}, as in the compact case, and set
\begin{equation*}
 (E_\Phi,E_A)\coloneqq{\bf S}_\eps(\Phi_\Gamma^{\mathrm{nc}},A_\Gamma^{\mathrm{nc}}).
\end{equation*}
Only \eqref{eq: noncompact mass correction} and the scalar matching require additional estimates. As in \Cref{sec: noncompact approximation}, minimality gives $\Delta_g(H-(2r)^{-1})=O(r^{-1})$, with the scaled bounds of \eqref{eq: bound R}. Following the same steps as in the proof of \Cref{lem: approximation regular matching} it can be shown that, for $1\leq j\leq J_\eps$,
\begin{equation}\label{eq: noncompact approximation matching}
\sum_{k=0}^2r_j^k\Bigl(\|\nabla^k(H-(2r)^{-1}-H_\Gamma)\|_{\tau,r_j,\Omega_j}+\|\nabla^k(b_r-\beta)\|_{\tau,r_j,\Omega_j}\Bigr) \lesssim r_j(1+|\log r_j|).
\end{equation}

\subsubsection*{The core and inner annuli: $r<\sqrt\eps$}
Before adding ${\bf C}_\eps$, the compact calculation \eqref{eq: approximation horizontal cancellation}--\eqref{eq: approximation inner shells} applies with $\lambda=0$ and with the scalar exponential factor omitted. 

Fix $q\in\Gamma$ and set $t=r/\eps$ and $s=\eps H_\Gamma(q)$. The mass variation \eqref{eq: noncompact mass variation} is
\begin{equation*}
 U_m(t)\coloneqq\binom{m\varphi_0(mt)}{\alpha_0(mt)},\qquad
 V(t)\coloneqq\left.\partial_mU_m(t)\right|_{m=1}
 =\binom{\varphi_0(t)+t\varphi_0'(t)}{t\alpha_0'(t)}.
\end{equation*}
Thus \eqref{eq: noncompact mass correction} replaces $U_1$ by $U_1-sV$, including the corresponding angular correction. By \eqref{eq: BPS}, $V(t)\to(1,0)$ as $t\to\infty$, so its Higgs component supplies the required regular term $-\eps H_\Gamma(q)$.
Writing the residual of \eqref{eq:monopole-profiles} at $\lambda=0$ explicitly, set
\begin{equation*}
 \mathcal R_0\binom{\varphi}{\alpha}
 \coloneqq
 \begin{pmatrix}
  \varphi''+\dfrac2t\varphi'-\dfrac{2(1-\alpha)^2\varphi}{t^2}\\[6pt]
  \alpha''-(1-\alpha)\left(\dfrac{\alpha(2-\alpha)}{t^2}-4\varphi^2\right)
 \end{pmatrix},
\end{equation*}
where primes denote derivatives in $t$. Since $U_m$ is an exact solution for every $m>0$,
\begin{equation*}
 \mathcal R_0(U_m)=0,\qquad
 D\mathcal R_0(U_1)V
 =\left.\partial_m\mathcal R_0(U_m)\right|_{m=1}=0.
\end{equation*}
Since $s$ is constant in the normal variables, the corrected normal residual is therefore
\begin{equation}\label{eq: noncompact mass normal cancellation}
 \mathcal R_0(U_1-sV)
 =\mathcal R_0(U_1)
 -sD\mathcal R_0(U_1)V+\mathcal Q_s
 =\mathcal Q_s,
\end{equation}
where $\mathcal Q_s$ contains only terms of degree at least two in $s$. This is the cancellation obtained by using the coupled mass derivatives. To control the remainder, compare with the exact solution of mass $1-s$:
\begin{equation*}
 U_1-sV-U_{1-s}
 =-s^2\int_0^1(1-\theta)\left.\partial_m^2U_m\right|_{m=1-\theta s}\,d\theta.
\end{equation*}
By \eqref{eq: BPS}, $m\varphi_0(mt)=m\coth(2mt)-(2t)^{-1}$, so the Coulomb term disappears under mass differentiation and $\partial_m^2U_m$ decays exponentially. The quadratic Taylor remainder contributes at most $C\eps^2$ with the weights in \eqref{eq: source norm}, on the core and after summation over the inner shells.

The trace lift \eqref{eq: q_eps} still uses $q_0(r/\eps)\sigma$. For the comparison mass $m=1-\eps H_\Gamma(q)$, \eqref{eq: qlambda} gives the normal-equation defect
\begin{equation}\label{eq: noncompact phase profile defect}
 \frac{2}{r^2}\left((1-\alpha_0(mr/\eps))^2-(1-\alpha_0(r/\eps))^2\right)
 q_0(r/\eps)\sigma.
\end{equation}
By \eqref{eq: BPS} and \eqref{eq: qlambda}, \eqref{eq: noncompact phase profile defect} is bounded, together with scaled derivatives, by $C\eps^{-1}(1+r/\eps)^N e^{-cr/\eps}$ for a fixed $N>0$. Multiplication by the connection weight in \eqref{eq: source norm} and dyadic summation give $C\eps^2$; the rescaling used in \eqref{eq: approximation core bound} gives the same bound on $\Omega_0$.

By \eqref{eq: noncompact mass correction}, $[\Phi_\Gamma^0,C_\eps^\Phi]=0$, so the scalar mass commutator vanishes. Tangential differentiation of \eqref{eq: noncompact mass correction}, using \eqref{eq: BPS}, adds scalar terms, and connection terms multiplied by $\eps$, bounded by
\begin{equation}\label{eq: noncompact mass derivative bound}
 C\left(\frac{\eps|\nabla_\Gamma H_\Gamma|}{\eps+r}
 +\eps|\nabla_\Gamma^2H_\Gamma|+\eps^2|\nabla_\Gamma H_\Gamma|^2\right).
\end{equation}
Here $\nabla_\Gamma$ is the connection of the induced metric on $\Gamma$. The contribution of \eqref{eq: noncompact mass derivative bound} with the weight in \eqref{eq: source norm} is at most $C\eps r_j$. Combining it with \eqref{eq: approximation inner shells} at $\lambda=0$ and \eqref{eq: noncompact phase profile defect} gives, for $4r_j\leq \sqrt\eps$,
\begin{equation}\label{eq: noncompact inner shells}
 r_j^2\bigl(\|E_\Phi\|_{\tau,r_j,\Omega_j}+\eps\|E_A\|_{\tau,r_j,\Omega_j}\bigr)
 \lesssim \eps r_j+\eps^2.
\end{equation}
The core rescaling in \eqref{eq: approximation core bound}, with \eqref{eq: noncompact mass correction}, gives at most $C\eps^2$. By \eqref{eq: source norm}, these bounds control the projected contributions; crossing shells are handled by combining \eqref{eq: noncompact inner shells} with \eqref{eq: noncompact gluing shells} and \eqref{eq: noncompact outer shells} below.

\subsubsection*{The gluing annulus: $\sqrt\eps\leq r\leq2\sqrt\eps$}
By \eqref{eq: BPS} and \eqref{eq: noncompact mass correction}, the difference between the corrected inner Higgs field and $(1-\eps H)\sigma$ is
\begin{equation}\label{eq: noncompact Higgs matching difference}
 \eps\left(H-\frac1{2r}-H_\Gamma\right)\sigma
\end{equation}
up to exponentially small terms. Equation \eqref{eq: noncompact approximation matching} bounds the first two scaled derivatives of \eqref{eq: noncompact Higgs matching difference} on $\Omega_{j,\mathrm g}$ by $C\eps r_j(1+|\log r_j|)$. By \eqref{eq: BPS}, the connection correction in \eqref{eq: noncompact mass correction} is exponentially small there, so \eqref{eq: approximation gluing difference} and \eqref{eq: approximation cutoff error} apply. Since both scalar and longitudinal connection differences commute, they produce no additional commutator term. Consequently,
\begin{equation}\label{eq: noncompact gluing shells}
 r_j^2\bigl(\|E_\Phi\|_{\tau,r_j,\Omega_{j,\mathrm g}}
 +\eps\|E_A\|_{\tau,r_j,\Omega_{j,\mathrm g}}\bigr)
 \lesssim \eps r_j(1+|\log r_j|)+\frac{\eps^2}{r_j}.
\end{equation}
The weights are those of \eqref{eq: source norm}, and the dyadic construction in \Cref{sec: norms} gives uniformly finitely many such shells, with $r_j\asymp \sqrt\eps$.

\subsubsection*{The outer region: $r>2\sqrt\eps$}
By \eqref{eq: s_eps} and the construction in \Cref{sec: noncompact approximation}, the Higgs field here is exactly $(1-\eps H)\sigma$. With $A_D$ as in \eqref{eq: exterior connection}, \Cref{lem: noncompact exterior correction} and \eqref{eq: noncompact scalar potential} give
\begin{equation}\label{eq: noncompact exterior cancellation}
 \Delta_gH=0,\qquad d_{A_D}\sigma=0,\qquad d_{A_D}^*F_{A_D}=\sigma d_g^*f_1=0.
\end{equation}
Using \eqref{eq: noncompact exterior cancellation} in the calculation of \eqref{eq: approximation outer shells} removes the radial Higgs error; \eqref{eq: BPS} controls the remaining connection tails. For $1\leq j\leq J_\eps$ with $\Omega_{j,\mathrm o}\ne\varnothing$,
\begin{equation}\label{eq: noncompact outer shells}
 r_j^2\bigl(\|E_\Phi\|_{\tau,r_j,\Omega_{j,\mathrm o}}
 +\eps\|E_A\|_{\tau,r_j,\Omega_{j,\mathrm o}}\bigr)
 \lesssim \eps^{-N}e^{-cr_j/\eps}.
\end{equation}
On $\Omega_{J_\eps+1}=\{r>\delta/4\}$, use \eqref{eq: noncompact exterior source norm}. Equation \eqref{eq: noncompact outer shells} makes the residual exponentially small on the fixed tube, while \eqref{eq: noncompact exterior cancellation} makes it zero for $r\geq\delta$. Hence, for every $P>0$,
\begin{equation}\label{eq: noncompact fixed exterior bound}
 \|E_\Phi\|_{C^{0,\tau}_{\nu+2}(\Omega_{J_\eps+1})}
 +\eps\|E_A\|_{C^{0,\tau}_{\nu+2}(\Omega_{J_\eps+1})}
 \lesssim_P\eps^P.
\end{equation}

\smallskip\noindent\emph{Parameters and summation.}
Using \eqref{eq: noncompact shifted approximation}, the parameter comparison \eqref{eq: approximation phase curvature}, \eqref{eq: approximation parameter shells} and \eqref{eq: approximation exterior parameters} applies in the same reference regions. The additional terms are controlled by \eqref{eq: noncompact phase profile defect} and \eqref{eq: noncompact mass derivative bound}, while \eqref{eq: noncompact approximation matching} gives $|\nabla^kH|\leq Cr^{-1-k}$ for $k\leq2$. The weights in \eqref{eq: noncompact exterior source norm} are bounded wherever the residual is nonzero, and \eqref{eq: noncompact shifted approximation} is unchanged outside the tube. Summing \eqref{eq: noncompact inner shells}, \eqref{eq: noncompact gluing shells} and \eqref{eq: noncompact outer shells} by \eqref{eq: approximation dyadic sums}, and adding the core contribution, \eqref{eq: noncompact fixed exterior bound} and the parameter bounds \eqref{eq: approximation parameter shells} and \eqref{eq: approximation exterior parameters}, yields
\begin{equation*}
\begin{split}
 \|{\bf S}_\eps(\Phi_{h,\theta}^{\mathrm{nc}},A_{h,\theta}^{\mathrm{nc}})\|_{Y_{\eps,h,\theta}}&\lesssim\biggl(\eps^{3/2}(1+|\log\eps|)
 +\eps^2(1+|\log\eps|)\\
 &\hspace{2em}+\eps(1+|\log\eps|)\|h\|_{C^{2,\tau}}+\|h\|_{C^{2,\tau}}^2\biggr)\\
 &\lesssim \eps^{3/2}(1+|\log\eps|),
\end{split}
\end{equation*}
where the last inequality uses \eqref{eq: assumptions htheta}.
\end{proof}

\section{Invertibility theory for the linearised operator}\label{sec: invertibility}

We prove \Cref{prop: uniform inverse}, constructing the constrained inverses and establishing their uniform bounds. The argument has the following steps.
\begin{itemize}
 \item Lemmas~\ref{lem: normal weighted coercivity} and~\ref{lem: product Liouville} establish normal coercivity under the fibre constraints and a product Liouville theorem for fields with normal decay.
 \item \Cref{prop: global constrained coercivity} gives global coercivity for the compact variational problem, which provides weak solutions.
 \item \Cref{lem: local estimates supremum remainder} and \Cref{prop: compact supremum estimate} give local and supremum bounds; \Cref{cor: local estimates supremum remainder} combines them into the full compact a priori estimate.
 \item \Cref{lem: noncompact local remainder} and \Cref{prop: noncompact a priori estimate} give the corresponding noncompact estimates, using the weighted exterior inverses.
 \item \Cref{lem: noncompact product constructions} gives the product inverse and longitudinal extensions used to construct the noncompact inverse in \Cref{prop: noncompact ordinary inverse}.
 \item Corollaries~\ref{cor: compact restricted inverse} and~\ref{cor: noncompact restricted inverse} control the gauge projection as a small perturbation, giving the inverses for $\widehat{\bf L}_{h,\theta}$.
\end{itemize}

\medskip

\begin{definition}[Rescaling convention]\label{def: rescaling convention}
For a physical pair ${\bf V}=(V_\Phi,V_A)$, use the Fermi map \eqref{eq: rescaled Fermi map} and the transported normal frame $\nu_{a,h}$ from \Cref{sec: projections linear theory}. Let $e_i$, $1\leq i\leq n-3$, be a local orthonormal frame of $T\Gamma$, with dual frame $e^i$, and transport it along the same geodesics as $\nu_{a,h}$ to obtain $e_{i,h}$. With the gauges \eqref{eq: normal trace gauge}, \eqref{eq: phase gauge} and the adjoint action \eqref{eq: adjoint action}, define
\begin{equation*}
\begin{aligned}
 U_\eps(q,y)&\coloneqq(g_\eps u_\theta)(\mathscr F_{\eps,0}(q,y)),\\
 \widehat V_\Phi(q,y)&\coloneqq\Ad_{U_\eps(q,y)}(\mathscr F_{\eps,h}^*V_\Phi)(q,y),\\
 \widehat V_a(q,y)&\coloneqq\eps\Ad_{U_\eps(q,y)}
 \left[V_A\big|_{\mathscr F_{\eps,h}(q,y)}(\nu_{a,h}(q,\eps y))\right],\quad a=1,2,3,\\
 \widehat b_i(q,y)&\coloneqq\eps\Ad_{U_\eps(q,y)}
 \left[V_A\big|_{\mathscr F_{\eps,h}(q,y)}(e_{i,h}(q,\eps y))\right].
\end{aligned}
\end{equation*}
Write $\widehat V_N=\widehat V_a\,dy^a$, $\widehat b=\widehat b_i e^i$, and $\widehat{\bf V}=(\widehat V_\Phi,\widehat V_N)$. For a ${\bf F}=(F_\Phi,F_A)$, the rescaled normal forcing is
\begin{equation*}
 f\coloneqq\eps^2\widehat{\bf F}
   =\bigl(\eps^2\widehat F_\Phi,\eps^2\widehat F_a\,dy^a\bigr).
\end{equation*}
\end{definition}
Thus $\widehat{\bf V}$ records the scalar and normal components, while $\widehat b$ records the tangential component. The adjoint action undoes the construction gauge, evaluated at the unshifted point $\mathscr F_{\eps,0}(q,y)$. The hat acts on transported-frame coefficients: each one-form coefficient is multiplied by $\eps$, the scalar is unscaled, and $q$ is unchanged. The same convention applies to sources and operator outputs. The additional source factor $\eps^2$ normalises the second-order operator, so the total factors on the physical scalar and one-form source coefficients are $\eps^2$ and $\eps^3$, respectively.

\medskip

For a pair ${\bf\Psi}=(\Psi,B)$ on $\R^3$ and with $(\Phi_\lambda,A_\lambda)$ as in \eqref{eq:monopole-ansatz}, define
\begin{equation}\label{eq: normal energy norm}
\|{\bf\Psi}\|_{\mathcal H_\lambda}^2\coloneqq \int_{\R^3}\bigl(|\nabla_{A_\lambda}\Psi|^2+|\nabla_{A_\lambda}B|^2 +|[\Phi_\lambda,\Psi]|^2+|[\Phi_\lambda,B]|^2+\lambda|\langle\Phi_\lambda,\Psi\rangle|^2\bigr)\,dy.
\end{equation}
Let $\mathcal H_\lambda$ be the completion of smooth compactly supported pairs in this norm. Recall the kernel pairs ${\bf Z}_1,{\bf Z}_2,{\bf Z}_3,{\bf W}$ from \Cref{sec: second variation}. For simplicity, we put ${\bf Z}_4\coloneqq{\bf W}$ and, for $R>0$, set
\begin{equation*}
 \eta_R(y)\coloneqq\eta\left(\frac{|y|}{R}\right).
\end{equation*}
For $i=1,\ldots,4$, define the pair-valued tests
\begin{equation}\label{eq: normal weighted tests}
 {\bf E}_{i,R}(y)\coloneqq\frac{\eta_R(y)}{1+|y|^2}{\bf Z}_i(y),
 \qquad 0<R\leq\infty,\qquad \eta_\infty\equiv1.
\end{equation}
The modes ${\bf Z}_i$ and hence ${\bf E}_{i,R}$ are evaluated at the same parameter $\lambda$; this dependence is suppressed in the notation.

We start by proving the coercivity of the monopole.

\begin{lemma}\label{lem: normal weighted coercivity}
There exist $\lambda_*,R_*,c,C>0$ such that, for every $0\leq\lambda\leq\lambda_*$, $R\geq R_*$ and ${\bf\Psi}\in\mathcal H_\lambda$,
\begin{equation*}
 {\bf Q}_\lambda({\bf\Psi}) \geq c\|{\bf\Psi}\|_{\mathcal H_\lambda}^2 -C\sum_{j=1}^4\left| \int_{\R^3}\langle{\bf\Psi},{\bf E}_{j,R}\rangle\,dy\right|^2.
\end{equation*}
In particular, ${\bf Q}_\lambda$ is uniformly coercive on the subspace where these four integrals vanish.
\end{lemma}
\begin{proof}
First, record that the small-coupling stability of \cite[Theorem~5.6]{Androulakis-Dostoglou1998} and \eqref{eq: gauge decom Q} gives ${\bf Q}_\lambda\geq0$. Also, note that Hardy inequality and the bounds 
\begin{equation}\label{eq: Kato}
	|d|\Psi||^2\leq |\nabla_{A_\lambda}\Psi|^2,\quad |d|B||^2\leq |\nabla_{A_\lambda}B|^2
\end{equation}
give
\begin{equation}\label{eq: normal Hardy bound}
 \int_{\R^3}\frac{|\Psi|^2+|B|^2}{|y|^2}\,dy \leq4\int_{\R^3}(|\nabla_{A_\lambda}\Psi|^2+|\nabla_{A_\lambda}B|^2)\,dy.
\end{equation}
For a fixed sufficiently large $R_0$, we claim
\begin{equation}\label{eq: normal compact remainder}
 \|{\bf\Psi}\|_{\mathcal H_\lambda}^2 \lesssim_{R_0} {\bf Q}_\lambda({\bf\Psi})+ \|{\bf\Psi}\|_{L^2(B_{R_0})}^2.
\end{equation}
To prove \eqref{eq: normal compact remainder}, write, for $r=|y|>0$,
\begin{equation}\label{eq: normal component splitting}
 \Psi^\parallel=\langle\Psi,\sigma\rangle\sigma,\quad
 \Psi^\perp=\Psi-\Psi^\parallel,
\end{equation}
with the same splitting for each coefficient of $B$. Integrating the covariant Hodge Laplacian in \eqref{eq: model linearised operator} by parts gives
\begin{equation}\label{eq: normal quadratic form expansion}
\begin{split}
 {\bf Q}_\lambda({\bf\Psi})=\int_{\R^3}\Bigl(&|\nabla_{A_\lambda}\Psi|^2+|\nabla_{A_\lambda}B|^2
 +\lambda(3\varphi_\lambda^2-1)|\Psi^\parallel|^2\\
 &+\bigl(4\varphi_\lambda^2+\lambda(\varphi_\lambda^2-1)\bigr)|\Psi^\perp|^2
 +4\varphi_\lambda^2|B^\perp|^2\Bigr)\,dy
 +\mathcal R_\lambda({\bf\Psi}).
\end{split}
\end{equation}
Here $\mathcal R_\lambda$ is the integral of the remaining zeroth-order terms, which contain $F_{A_\lambda}$ or $d_{A_\lambda}\Phi_\lambda$. Their components along $\sigma$ have size $O(r^{-2})$ and pair only with two perpendicular components of $(\Psi,B)$; their perpendicular components decay exponentially. Consequently, for $R_0$ sufficiently large,
\begin{equation}\label{eq: normal curvature remainder}
 |\mathcal R_\lambda({\bf\Psi})|
 \lesssim \int_{r>R_0}\left[r^{-2}(|\Psi^\perp|^2+|B^\perp|^2)
 +e^{-cr}(|\Psi|^2+|B|^2)\right]dy
 +\|{\bf\Psi}\|_{L^2(B_{R_0})}^2.
\end{equation}
By \eqref{eq: approximation profile bounds} and \eqref{eq: BPS}, the profile bounds are uniform for small $\lambda$:
\begin{equation}\label{eq: normal profile bounds}
 \varphi_\lambda(r)\to 1,\qquad
 |\varphi_\lambda'(r)|\lesssim r^{-2},\qquad
 |1-\alpha_\lambda(r)|+|\alpha_\lambda'(r)|\lesssim e^{-cr}.
\end{equation}
By \eqref{eq: normal profile bounds}, after choosing $R_0$ large and $\lambda_*$ small, the coefficient of $|\Psi^\parallel|^2$ in \eqref{eq: normal quadratic form expansion} is bounded below by $c\lambda$, while the coefficients of $|\Psi^\perp|^2$ and $|B^\perp|^2$ are bounded below by a fixed $c>0$. The $Cr^{-2}$ term in \eqref{eq: normal curvature remainder} is absorbed by the coefficients of $|\Psi^\perp|^2+|B^\perp|^2$. Moreover, \eqref{eq: normal Hardy bound} gives
\begin{equation}\label{eq: normal exponential absorption}
 \int_{r>R_0}e^{-cr}(|\Psi|^2+|B|^2)\,dy
 \lesssim \sup_{r>R_0}(r^2e^{-cr})
 \int_{\R^3}(|\nabla_{A_\lambda}\Psi|^2+|\nabla_{A_\lambda}B|^2)\,dy.
\end{equation}
The coefficient in \eqref{eq: normal exponential absorption} tends to zero as $R_0\to\infty$. Since
$|[\Phi_\lambda,\Psi]|^2=4\varphi_\lambda^2|\Psi^\perp|^2$ and likewise for $B$, the exterior zeroth-order terms control the corresponding terms of \eqref{eq: normal energy norm}. On $B_{R_0}$, all zeroth-order coefficients are uniformly bounded, so their contribution is bounded below by $-C\|{\bf\Psi}\|_{L^2(B_{R_0})}^2$. Combining \eqref{eq: normal curvature remainder}--\eqref{eq: normal exponential absorption} with the derivative integral in \eqref{eq: normal quadratic form expansion} gives
\begin{equation*}
 {\bf Q}_\lambda({\bf\Psi})
 \geq c\|{\bf\Psi}\|_{\mathcal H_\lambda}^2-C\|{\bf\Psi}\|_{L^2(B_{R_0})}^2,
\end{equation*}
which is \eqref{eq: normal compact remainder}, first for compactly supported pairs and then by density.

\smallskip

Next, we show that the nullspace in $\mathcal H_0$ is the span of the four modes in \eqref{eq: linearised kernel}. Let ${\bf\Psi}=(\Psi,B)\in\mathcal H_0$ satisfy ${\bf Q}_0({\bf\Psi})=0$. The Bogomolny factorisation of ${\bf Q}_0$ gives
\begin{equation}\label{eq: normal BPS first order}
 d_{A_0}\Psi+[B,\Phi_0]+*d_{A_0}B=0,
 \qquad d_{A_0}^*B-[\Phi_0,\Psi]=0.
\end{equation}
We claim that ${\bf\Psi}\in L^2(\R^3)$, so \Cref{lem: monopole ker} identifies it with a combination of the four modes. Sobolev inequality and \eqref{eq: Kato} give
\begin{equation}\label{eq: normal Sobolev bound}
 \|{\bf\Psi}\|_{L^6(\R^3)}
 \lesssim \|d|{\bf\Psi}|\|_{L^2(\R^3)}
 \lesssim \|\nabla_{A_0}{\bf\Psi}\|_{L^2(\R^3)}
 \lesssim \|{\bf\Psi}\|_{\mathcal H_0}.
\end{equation}
Equation \eqref{eq: normal Sobolev bound} extends to $\mathcal H_0$ by its definition as a completion. Since ${\bf L}_0{\bf\Psi}=0$, elliptic regularity makes the pair smooth; interior estimates on unit balls and \eqref{eq: normal Sobolev bound} show that ${\bf\Psi}$ and $\nabla_{A_0}{\bf\Psi}$ are bounded and tend to zero at infinity.

To estimate ${\bf\Psi}^\perp=(\Psi^\perp,B^\perp)$, observe that
\begin{equation}\label{eq: normal BPS direction derivative}
 d_{A_0}\sigma=(1-\alpha_0)d\sigma,
 \qquad 1-\alpha_0(r)=\frac{2r}{\sinh(2r)}.
\end{equation}
By \eqref{eq: normal BPS direction derivative}, the covariant derivatives of the projection onto $\sigma^\perp$ decay exponentially. Projecting ${\bf L}_0{\bf\Psi}=0$ and using \eqref{eq: normal quadratic form expansion}, \eqref{eq: normal curvature remainder} and \eqref{eq: normal profile bounds} gives, for large $r$,
\begin{equation}\label{eq: normal transverse equation}
 \left|\left(\nabla_{A_0}^*\nabla_{A_0}+4\varphi_0^2\right){\bf\Psi}^\perp\right|
 \leq Cr^{-2}|{\bf\Psi}^\perp|
 +Ce^{-cr}\bigl(|{\bf\Psi}|+|\nabla_{A_0}{\bf\Psi}|\bigr).
\end{equation}
In particular, the coupling to ${\bf\Psi}^\parallel$ is exponentially small. Choose $\mu>0$ and a sufficiently large radius so that $4\varphi_0^2-Cr^{-2}\geq\mu^2$. Equations \eqref{eq: Kato} and \eqref{eq: normal transverse equation} then imply, weakly on this exterior region,
\begin{equation}\label{eq: normal transverse comparison}
 \left(-\sum_{i=1}^3\partial_{y^i}^2+\mu^2\right)|{\bf\Psi}^\perp|
 \lesssim e^{-cr}.
\end{equation}
Since ${\bf\Psi}^\perp\to0$ at infinity, comparison in \eqref{eq: normal transverse comparison} with $Ce^{-\gamma r}$, for $0<\gamma<\min\{c,\mu\}$, proves $|{\bf\Psi}^\perp|\leq Ce^{-\gamma r}$. Interior elliptic estimates give the same decay for its derivatives, after decreasing $\gamma$ if necessary.

Write $B=B_i\,dy^i$ and let $\langle B,\sigma\rangle=\langle B_i,\sigma\rangle\,dy^i$ denote its real coefficient one-form along $\sigma$. Taking the inner product of \eqref{eq: normal BPS first order} with $\sigma$ now gives
\begin{equation}\label{eq: normal longitudinal first order}
 d\langle\Psi,\sigma\rangle+*d\langle B,\sigma\rangle=O(e^{-\gamma r}),
 \qquad d^*\langle B,\sigma\rangle=O(e^{-\gamma r}),
\end{equation}
also after differentiation. Consequently, \eqref{eq: normal longitudinal first order} gives
\begin{equation}\label{eq: normal longitudinal Laplacians}
 d^*d\langle\Psi,\sigma\rangle=O(e^{-\gamma r}),\qquad
 d^*d\langle B_i,\sigma\rangle=O(e^{-\gamma r}),\quad i=1,2,3.
\end{equation}
Multiply these coefficients by a smooth cutoff that is zero near the origin and one for large $r$. By \eqref{eq: normal longitudinal Laplacians}, their Euclidean Laplacians then decay exponentially outside a compact set. The Newton-potential representation gives their asymptotic expansions; the difference from this potential is an entire harmonic $L^6$ function and hence vanishes. Thus there exist real constants $a,c_1,c_2,c_3$, the coefficients of the leading $1/r$ terms, such that
\begin{equation}\label{eq: normal longitudinal tails}
 \langle\Psi,\sigma\rangle=\frac{a}{r}+O(r^{-2}),\qquad \langle B_i,\sigma\rangle=\frac{c_i}{r}+O(r^{-2}),\quad i=1,2,3.
\end{equation}
The first derivatives of the remainders in \eqref{eq: normal longitudinal tails} are $O(r^{-3})$. Evaluating the second equation in \eqref{eq: normal longitudinal first order} at $y=r\omega$, with $\omega\in\Sp^2$, yields
\begin{equation*}
 \frac{\sum_{i=1}^3c_i\omega_i}{r^2}+O(r^{-3})=O(e^{-\gamma r}).
\end{equation*}
Multiplying by $r^2$ and letting $r\to\infty$ gives $\sum_i c_i\omega_i=0$ for every $\omega$, so $c_1=c_2=c_3=0$. The first equation in \eqref{eq: normal longitudinal first order} then becomes
\begin{equation*}
 -\frac{a}{r^2}\,dr+O(r^{-3})=O(e^{-\gamma r}),
\end{equation*}
and therefore $a=0$ as well. Hence $|{\bf\Psi}|=O(r^{-2})$, which is square-integrable at infinity in dimension three. Smoothness on the ball gives ${\bf\Psi}\in L^2(\R^3)$, proving the claim.

\smallskip

Next, define the moments and the $4\times4$ Gram matrix
\begin{equation}\label{eq: normal moment Gram matrix}
\begin{split}
 \mathfrak{m}_i({\bf\Psi})&\coloneqq\int_{\R^3}\frac{\eta_R}{1+|y|^2}
 \langle{\bf\Psi},{\bf Z}_i\rangle\,dy,\\
 G_{\lambda,R,ij}&\coloneqq\int_{\R^3}\frac{\eta_R}{1+|y|^2}
 \langle{\bf Z}_j,{\bf Z}_i\rangle\,dy,\qquad 1\leq i,j\leq4.
\end{split}
\end{equation}
The modes are evaluated at $\lambda$. Choose $R_*$ so that the four modes at $\lambda=0$ are linearly independent on the ball where $\eta_{R_*}=1$. Equation \eqref{eq: normal moment Gram matrix} then gives a positive lower bound for $G_{0,R}$, uniformly for $R\geq R_*$. Local convergence of the modes as $\lambda\to0$ preserves this lower bound for small $\lambda$.

The uniform bounds $|{\bf Z}_j(y)|\lesssim (1+|y|)^{-2}$ from \eqref{eq: monopole translations}, \eqref{eq: monopole phase mode} and \eqref{eq: normal profile bounds}, together with \eqref{eq: normal Hardy bound}, show that
\begin{equation}\label{eq: normal moment tail}
 \left|\int_{|y|>T}\frac{\eta_R\langle{\bf\Psi},{\bf Z}_j\rangle}{1+|y|^2}\,dy\right| \lesssim T^{-3/2}\|{\bf\Psi}\|_{\mathcal H_\lambda},\qquad T\geq1.
\end{equation}
Suppose that coercivity on $\mathfrak{m}_j=0$ fails on every sufficiently small interval of $\lambda$. Choose $\lambda_k\to0$, $R_k\geq R_*$ and pairs satisfying
\begin{equation*}
 \|{\bf\Psi}_k\|_{\mathcal H_{\lambda_k}}=1,\qquad \mathfrak{m}_j({\bf\Psi}_k)=0,\qquad {\bf Q}_{\lambda_k}({\bf\Psi}_k)\longrightarrow0.
\end{equation*} Local compactness gives a limit in $\mathcal H_0$, strongly in $L^2_{\rm loc}$. By \eqref{eq: normal compact remainder} this limit is nonzero. Cauchy--Schwarz for the nonnegative forms shows that it lies in the BPS nullspace. After taking $R_k\to R_\infty\in[R_*,\infty]$, \eqref{eq: normal moment tail} passes all four zero moments to the limit. Writing the limit as $\sum_{j=1}^4a_j{\bf Z}_j$ at $\lambda=0$, its zero moments give $G_{0,R_\infty}a=0$. The positive lower bound for \eqref{eq: normal moment Gram matrix} forces $a=0$, contradicting the nonzero limit.

For an arbitrary pair, let $\mathfrak{m}=(\mathfrak{m}_1({\bf\Psi}),\ldots,\mathfrak{m}_4({\bf\Psi}))^T$, set $a=G_{\lambda,R}^{-1}\mathfrak{m}$, and put
\begin{equation*}
 {\bf\Psi}^{\circ}\coloneqq{\bf\Psi}-\sum_{j=1}^4a_j{\bf Z}_j.
\end{equation*}
Then $\mathfrak{m}_i({\bf\Psi}^{\circ})=\mathfrak{m}_i({\bf\Psi})-(G_{\lambda,R}a)_i=0$. The uniform inverse bound for \eqref{eq: normal moment Gram matrix} and the bounded norms \eqref{eq: normal energy norm} of the modes give
$\|\sum_ja_j{\bf Z}_j\|_{\mathcal H_\lambda}\leq C|\mathfrak{m}|$.
Since ${\bf L}_\lambda{\bf Z}_j=0$ by \eqref{eq: linearised kernel}, subtracting this combination leaves ${\bf Q}_\lambda$ unchanged. Therefore
\begin{equation*}
 \|{\bf\Psi}\|_{\mathcal H_\lambda}^2
 \leq2\|{\bf\Psi}^{\circ}\|_{\mathcal H_\lambda}^2+C|\mathfrak{m}|^2
 \leq C{\bf Q}_\lambda({\bf\Psi})+C\sum_{j=1}^4|\mathfrak{m}_j({\bf\Psi})|^2,
\end{equation*}
which is the assertion. The same argument permits $R=\infty$, with $\eta_\infty=1$.
\end{proof}

\begin{lemma}\label{lem: product Liouville}
Let $(x,y)\in\R^n=\R^{n-3}\times\R^3$. For $\lambda\in[0,\lambda_*]$, put ${\bf E}_i={\bf E}_{i,\infty}$, with the tests from \eqref{eq: normal weighted tests} evaluated at $\lambda$. Suppose that ${\bf\Psi}=(\Psi,B)$ is a bounded $C^2$ pair on $\R^{n-3}\times\R^3$, with $B$ a one-form in the normal variables, and that $g_1,\ldots,g_4$ are bounded functions on $\R^{n-3}$ satisfying
\begin{equation}\label{eq: product Liouville system}
\begin{split}
 \left(-\Delta_x+{\bf L}_\lambda\right){\bf\Psi}(x,y) &=\sum_{j=1}^4g_j(x){\bf E}_j(y),\\
 \int_{\R^3}\langle{\bf\Psi}(x,y),{\bf E}_j(y)\rangle\,dy&=0
\end{split}
\end{equation}
for every $x$ and $j=1,\ldots,4$. If, for some $\nu>\frac12$,
\begin{equation}\label{eq: product Liouville decay}
 \sup_{(x,y)\in\R^{n-3}\times\R^3} (1+|y|)^\nu\bigl(|\Psi(x,y)|+|B(x,y)|\bigr)<\infty,
\end{equation}
then ${\bf\Psi}=0$ and $g_j=0$ for every $j$. Likewise, a bounded $\su(2)$-valued function $b$ satisfying
\begin{equation}\label{eq: product Liouville tangential equation}
 \left(-\Delta_x+\nabla_{A_\lambda}^*
 \nabla_{A_\lambda}\right)b -[\Phi_\lambda,[\Phi_\lambda,b]]=0
\end{equation}
and the same normal decay vanishes identically.
\end{lemma}
\begin{proof}
For $s>0$ the matrix
\begin{equation}\label{eq: product resolvent matrix}
 A(s)_{ij}=\left\langle {\bf E}_i,({\bf L}_\lambda+s)^{-1}{\bf E}_j\right\rangle_{L^2(\R^3)}
\end{equation}
is positive definite by the nonnegativity established in the proof of \Cref{lem: normal weighted coercivity} and the independence of \eqref{eq: normal weighted tests}. Write $g=(g_1,\ldots,g_4)^T$, and use a hat for the Fourier transform in $x$, with frequency $\zeta\in\R^{n-3}$. Localise to a compact set $K\subset\R^{n-3}\setminus\{0\}$, and test the Fourier transform of \eqref{eq: product Liouville system} with $({\bf L}_\lambda+|\zeta|^2)^{-1}{\bf E}_i$. The moment conditions in \eqref{eq: product Liouville system} give
\begin{equation*}
 A(|\zeta|^2)\widehat g=0,\qquad\text{hence}\qquad\widehat g=0
\end{equation*}
on this set by the positivity of \eqref{eq: product resolvent matrix}. Testing next with $({\bf L}_\lambda+|\zeta|^2)^{-1}f$, for arbitrary compactly supported smooth normal pairs $f$, gives $\widehat{\bf\Psi}=0$ there. These Fourier transforms are understood as tempered distributions. To justify the tests, set
\begin{equation*}
 u_\zeta\coloneqq({\bf L}_\lambda+|\zeta|^2)^{-1}f,
 \qquad s_0\coloneqq\min_{\zeta\in K}|\zeta|^2>0,
\end{equation*}
where $f$ is either ${\bf E}_i$ or a smooth compactly supported normal pair. Use the componentwise splitting $u_\zeta=u_\zeta^\parallel+u_\zeta^\perp$ along $\R\sigma\oplus\sigma^\perp$, defined in \eqref{eq: normal component splitting}. The quadratic-form expansion \eqref{eq: normal quadratic form expansion} shows that, outside a sufficiently large ball, the zeroth-order quadratic form after adding $|\zeta|^2$, including the curvature terms, is bounded below by $(s_0/2)|u_\zeta|^2$. Equation \eqref{eq: Kato} therefore gives
\begin{equation}\label{eq: product resolvent comparison}
 \left(-\sum_{a=1}^3\partial_{y^a}^2+s_0/2\right)|u_\zeta|\leq |f|.
\end{equation}
The $L^2$ resolvent solution tends to zero at infinity by interior estimates. For $f={\bf E}_i=O((1+|y|)^{-4})$, a sufficiently large multiple of $(1+|y|)^{-4}$ is a supersolution of \eqref{eq: product resolvent comparison} and bounds $|u_\zeta|$. For compactly supported $f$, comparison outside its support with $Ce^{-\gamma|y|}$, where $0<\gamma<\sqrt{s_0/2}$, gives exponential decay. Interior elliptic estimates give the corresponding bounds for spatial derivatives. Differentiating the resolvent equation gives
\begin{equation}\label{eq: product resolvent frequency derivative}
 ({\bf L}_\lambda+|\zeta|^2)\partial_{\zeta_\alpha}u_\zeta=-2\zeta_\alpha u_\zeta,
 \qquad \alpha=1,\ldots,n-3.
\end{equation}
Repeated differentiation of \eqref{eq: product resolvent frequency derivative} and comparison in \eqref{eq: product resolvent comparison} give these decay bounds for every frequency derivative, uniformly on $K$. They make the tests integrable in $y$ and make the integration-by-parts errors from a normal cutoff tend to zero, justifying their pairing with the bounded solution.

Next, both Fourier transforms are supported at zero. A distribution with this support is a finite sum of derivatives of the delta distribution; boundedness in $x$ leaves only the constant term. Equation \eqref{eq: product Liouville system}, tested against ${\bf Z}_i$ using \eqref{eq: linearised kernel}, gives
\begin{equation*}
 \sum_j g_j\int_{\R^3}\frac{\langle{\bf Z}_j,{\bf Z}_i\rangle}{1+|y|^2}\,dy=0.
\end{equation*}
The matrix in \eqref{eq: normal moment Gram matrix} with $R=\infty$ is invertible, so $g=0$. The integrations by parts are valid by \eqref{eq: monopole translations}, \eqref{eq: monopole phase mode}, \eqref{eq: normal profile bounds} and local elliptic estimates. Testing the homogeneous equation with a normal cutoff times ${\bf\Psi}$ now gives finite homogeneous energy: by \eqref{eq: product Liouville decay}, the cutoff loss on $R<|y|<2R$ is $O(R^{1-2\nu})=o(1)$, and \eqref{eq: normal compact remainder} bounds the remaining interior term. Consequently ${\bf Q}_\lambda({\bf\Psi})=0$, and \Cref{lem: normal weighted coercivity}, with $R=\infty$, gives ${\bf\Psi}=0$.

Finally, apply the Fourier argument to \eqref{eq: product Liouville tangential equation}, whose normal operator is nonnegative, without a coefficient matrix. It makes $b$ independent of $x$; cutoff testing then gives
\begin{equation*}
 \int_{\R^3}\left(|\nabla_{A_\lambda}b|^2+|[\Phi_\lambda,b]|^2\right)=0.
\end{equation*}
Its decay implies $b=0$.
\end{proof}

\subsection{The compact case} We first assume that $M$ is compact and admissible. Let  $r(x)\coloneqq\dist_g(x,\Gamma)$ and put $r_h\coloneqq r\circ\mathscr T_h^{-1}$. Define $\mathcal H_{\eps,h,\theta}$ as the completion of smooth pairs ${\bf\Psi}=(\Psi,B)$ on $M$ in the norm
\begin{equation}\label{eq: compact energy norm}
\begin{split}
 \|{\bf\Psi}\|_{\mathcal H_{\eps,h,\theta}}^2&\coloneqq \frac1\eps\int_M\left(|\nabla_{A_{h,\theta}}\Psi|^2 +\frac{|\Psi|^2}{(\eps+r_h)^2}\right)\,d\vol_g\\
 &+\eps\int_M\left(|\nabla_{A_{h,\theta}}B|^2 +\frac{|B|^2}{(\eps+r_h)^2}\right)\,d\vol_g\\
 &+\frac1{\eps^3}\int_M\left(|[\Phi_{h,\theta},\Psi]|^2 +\lambda_\eps|\langle\Phi_{h,\theta},\Psi\rangle|^2\right)\,d\vol_g\\
 &+\frac1\eps\int_M|[\Phi_{h,\theta},B]|^2\,d\vol_g.
\end{split}
\end{equation}
For fixed $\eps,h,\theta$, this norm is equivalent to the usual $H^1$ norm. The map $\mathfrak B_{\eps,h,\theta}$ extends continuously to this space.

\begin{proposition}\label[proposition]{prop: global constrained coercivity}
There exist $c,\eps_0>0$ such that, for $\eps\in (0,\eps_0)$ and parameters $(h,\theta)$ satisfying \eqref{eq: assumptions htheta},
\begin{equation*}
 {\bf Q}_{(\Phi_{h,\theta},A_{h,\theta})}({\bf\Psi}) \geq c\|{\bf\Psi}\|_{\mathcal H_{\eps,h,\theta}}^2
\end{equation*}
for every ${\bf\Psi}\in\mathcal H_{\eps,h,\theta}$ satisfying $\mathfrak B_{\eps,h,\theta}{\bf\Psi}=0$. 
\end{proposition}

\begin{proof}
Write for brevity ${\bf Q}={\bf Q}_{(\Phi_{h,\theta},A_{h,\theta})}$ and $\mathcal H=\mathcal H_{\eps,h,\theta}$. The codimension-three Hardy inequality and \eqref{eq: Kato} give, for scalar or one-form fields $v$,
\begin{equation}\label{eq: compact Hardy bound}
 \|(\eps+r_h)^{-1}v\|_2^2\lesssim \|\nabla_{A_{h,\theta}}v\|_2^2+\|v\|_2^2.
\end{equation}
Choose smooth real cutoffs $\eta_\eps,\zeta_\eps$ with $\eta_\eps^2+\zeta_\eps^2=1$, $\eta_\eps=1$ for $r_h\leq\tfrac12\eps^{7/8}$ and $\eta_\eps=0$ for $r_h\geq\tfrac14\eps^{3/4}$. A logarithmic interpolation gives
\begin{equation}\label{eq: compact localization cutoffs}
 |d\eta_\eps|^2+|d\zeta_\eps|^2\leq\frac{C}{r_h^2|\log\eps|^2}.
\end{equation}
Using \eqref{eq: compact localization cutoffs} and \eqref{eq: compact energy norm}, expansion of the derivative terms in \eqref{eq: gauge decom Q} yields
\begin{equation}\label{eq: compact quadratic localization}
\begin{split}
 {\bf Q}(\eta_\eps{\bf\Psi})+{\bf Q}(\zeta_\eps{\bf\Psi})-{\bf Q}({\bf\Psi})&=\int_M(|d\eta_\eps|^2+|d\zeta_\eps|^2)
 (\eps^{-1}|\Psi|^2+\eps|B|^2)\,d\vol_g\\
 &\lesssim |\log\eps|^{-2}\|{\bf\Psi}\|_{\mathcal H}^2.
\end{split}
\end{equation}
We prove the following two claims, with $c>0$ independent of $\eps,h,\theta$.
\begin{itemize}
	\item Inner almost-coercivity:
	\begin{equation}\label{eq: global inner coercivity}
 {\bf Q}(\eta_\eps{\bf\Psi})\geq c\|\eta_\eps{\bf\Psi}\|_{\mathcal H}^2-o(1)\|{\bf\Psi}\|_{\mathcal H}^2.
\end{equation}
	\item Outer coercivity:
	\begin{equation}\label{eq: global outer coercivity}
 {\bf Q}(\zeta_\eps{\bf\Psi})\geq c\|\zeta_\eps{\bf\Psi}\|_{\mathcal H}^2.
\end{equation}
\end{itemize}
These claims close the proof: since $\eta_\eps^2+\zeta_\eps^2=1$, expansion of the derivatives in \eqref{eq: compact energy norm} gives
\begin{equation}\label{eq: compact energy localization}
 \|{\bf\Psi}\|_{\mathcal H}^2\leq\|\eta_\eps{\bf\Psi}\|_{\mathcal H}^2+\|\zeta_\eps{\bf\Psi}\|_{\mathcal H}^2.
\end{equation}
Adding \eqref{eq: global inner coercivity} and \eqref{eq: global outer coercivity} and using \eqref{eq: compact quadratic localization} and \eqref{eq: compact energy localization} yields
\begin{equation*}
 {\bf Q}({\bf\Psi})\geq\bigl(c-o(1)-C|\log\eps|^{-2}\bigr)\|{\bf\Psi}\|_{\mathcal H}^2
 \geq\frac c2\|{\bf\Psi}\|_{\mathcal H}^2
\end{equation*}
for sufficiently small $\eps$.

\subsubsection*{Inner almost-coercivity}
Put ${\bf V}=(V_\Phi,V_A)\coloneqq \eta_\eps{\bf\Psi}$ and use the rescaled fields $\widehat{\bf V},\widehat b$ of \Cref{def: rescaling convention}. All inner rescaled fields vanish for $|y|\geq\tfrac14\eps^{-1/4}$.

In the product metric, the change of fibre volume is $dz=\eps^3dy$. Denote the resulting quadratic form by
\begin{equation}\label{eq: compact product quadratic form}
\begin{split}
 {\bf Q}^{\mathrm{prod}}_\eps(\widehat{\bf V},\widehat b)
 &\coloneqq\int_\Gamma {\bf Q}_{\lambda_\eps}(\widehat{\bf V}(q,\cdot))\,d\vol_\Gamma\\
 &\quad+\int_{\Gamma\times\R^3}\sum_i
 \left(|\nabla_{A_{\lambda_\eps}}\widehat b_i|^2+|[\Phi_{\lambda_\eps},\widehat b_i]|^2\right)\,d\vol_\Gamma\,dy\\
 &\quad+\eps^2\int_{\Gamma\times\R^3}
 \left(|\nabla_\Gamma\widehat V_\Phi|^2+|\nabla_\Gamma\widehat V_N|^2+|\nabla_\Gamma\widehat b|^2\right)\,d\vol_\Gamma\,dy\\
 &\quad+\mathcal R_{\Gamma,\eps}(\widehat b),
\end{split}
\end{equation}
where the curvature contribution from the Hodge identity on $\Gamma$ is
\begin{equation*}
 \mathcal R_{\Gamma,\eps}(\widehat b)
 \coloneqq\eps^2\int_{\Gamma\times\R^3}
 \operatorname{Ric}_\Gamma(e_i,e_j)\langle\widehat b_i,\widehat b_j\rangle\,d\vol_\Gamma\,dy.
\end{equation*}
Here $\nabla_\Gamma$ uses the fixed normal trivialisation and the Levi--Civita connection on tangential covectors. The mixed derivative terms cancel by integration by parts: with $D_a=\partial_{y^a}+[A_{\lambda_\eps,a},\cdot]$, their two contributions have opposite signs and the same integral,
\begin{equation*}
 \sum_i\int_{\Gamma\times\R^3}\langle D_a\widehat b_i,\nabla_i\widehat V_a\rangle
 =\sum_i\int_{\Gamma\times\R^3}\langle D_a\widehat V_a,\nabla_i\widehat b_i\rangle.
\end{equation*}
Since $\operatorname{Ric}_\Gamma$ is bounded and the inner support has $|y|\leq\tfrac14\eps^{-1/4}$, \eqref{eq: normal Hardy bound} gives
\begin{equation}\label{eq: compact Ricci error bound}
 |\mathcal R_{\Gamma,\eps}(\widehat b)|\lesssim \eps^2\int|\widehat b|^2
 \lesssim \eps^{3/2}\int\frac{|\widehat b|^2}{|y|^2}
 \lesssim \eps^{3/2}\sum_i\int|\nabla_{A_{\lambda_\eps}}\widehat b_i|^2.
\end{equation}
The Fermi metric and frame errors are bounded by $C(\eps^{3/4}+\|h\|_{C^2})$. These include the graph-induced tangential connection coefficients: in flat coordinates they are $-\eps^{-1}\partial_i h^a(A_\lambda)_a(y)$, so their rescaled size is $O(\|Dh\|)$ even though their physical size need not be bounded. By \eqref{eq: approximation profile bounds}, \eqref{eq: approximation regular matching} and \eqref{eq: shifted phase approximation}, after removal of $g_\eps u_\theta$ and separation of these graph terms, the remaining additional connection coefficients have bounded size, normal derivatives bounded by $C(\eps+r_h)^{-1}$, and exponentially decaying components perpendicular to $\sigma$. The mixed Higgs terms use
\begin{equation*}
 \sup_{r_h\geq0}(\eps+r_h)^2\eps^{-1}e^{-cr_h/\eps}\lesssim \eps.
\end{equation*}
Cauchy--Schwarz and \eqref{eq: normal Hardy bound} therefore give
\begin{equation}\label{eq: compact product comparison}
 \bigl|{\bf Q}({\bf V})-{\bf Q}^{\mathrm{prod}}_\eps(\widehat{\bf V},\widehat b)\bigr|
 \leq C\bigl(\eps^{3/4}+\eps+\|h\|_{C^2}\bigr)\|{\bf V}\|_{\mathcal H}^2=o(1)\|{\bf V}\|_{\mathcal H}^2.
\end{equation}

The constraints \eqref{eq: fibre constraints} become exactly the weighted moments of $\widehat{\bf\Psi}$ in this transported frame. The fixed physical cutoff equals one on the inner support. Thus the moments of $\widehat{\bf V}$ are supported in $\{|y|\geq\tfrac12\eps^{-1/8}\}$. Equation \eqref{eq: normal moment tail} gives
\begin{equation}\label{eq: compact inner moment bound}
 \int_\Gamma\sum_{j=1}^4\left|\int_{\R^3}
 \frac{\langle\widehat{\bf V},{\bf Z}_j\rangle}{1+|y|^2}\,dy\right|^2\,d\vol_\Gamma
 \leq C\eps^{3/8}\|{\bf\Psi}\|_{\mathcal H}^2.
\end{equation}
Apply \Cref{lem: normal weighted coercivity} to $\widehat{\bf V}$ and \eqref{eq: normal Hardy bound} to each $\widehat b_i$. Combining \eqref{eq: compact product quadratic form}, \eqref{eq: compact Ricci error bound}, \eqref{eq: compact product comparison} and \eqref{eq: compact inner moment bound} proves \eqref{eq: global inner coercivity}.

\subsubsection*{Outer coercivity}
Now, put ${\bf V}=(V_\Phi,V_A)\coloneqq \zeta_\eps{\bf\Psi}$. On the support of ${\bf V}$, where $r_h\geq\tfrac12\eps^{7/8}$, define
\begin{equation*}
 \sigma_h\coloneqq\frac{\Phi_{h,\theta}}{|\Phi_{h,\theta}|},\qquad
 \psi\coloneqq\langle V_\Phi,\sigma_h\rangle,\qquad
 b\coloneqq\langle V_A,\sigma_h\rangle,
 \qquad {\bf V}^\perp\coloneqq{\bf V}-(\psi\sigma_h,b\sigma_h).
\end{equation*}
Extend the real scalar $\psi$ and real one-form $b$ by zero through the inner tube. Equations \eqref{eq: approximation profile bounds}, \eqref{eq: approximation horizontal cancellation}, \eqref{eq: FAtan} and \eqref{eq: shifted phase approximation} give, for a fixed $N>0$,
\begin{equation}\label{eq: compact outer background bounds}
 \sum_{k=0}^1|\nabla_{A_{h,\theta}}^k d_{A_{h,\theta}}\sigma_h|
 +|F_{A_{h,\theta}}^\perp|\lesssim \eps^{-N}e^{-cr_h/\eps},
 \qquad |F_{A_{h,\theta}}|\lesssim (1+r_h^{-2}).
\end{equation}
Here the curvature is split along $\R\sigma_h\oplus\sigma_h^\perp$, and covariant derivatives use $A_{h,\theta}$ and the Levi--Civita connection. The Higgs field equals $|\Phi_{h,\theta}|\sigma_h$, with $|\Phi_{h,\theta}|\geq c>0$ here. For ${\bf Q}({\bf V}^\perp)$, the positive terms include
\begin{equation}\label{eq: compact transverse mass terms}
 4\eps^{-3}|\Phi_{h,\theta}|^2|V_\Phi^\perp|^2
 +4\eps^{-1}|\Phi_{h,\theta}|^2|V_A^\perp|^2.
\end{equation}
The curvature and mixed Higgs terms are bounded in absolute value by
\begin{equation}\label{eq: compact transverse curvature bound}
 C(1+r_h^{-2})(\eps^{-1}|V_\Phi^\perp|^2+\eps|V_A^\perp|^2).
\end{equation}
Since $\eps^2r_h^{-2}\leq4\eps^{1/4}$, \eqref{eq: compact transverse curvature bound} is absorbed by \eqref{eq: compact transverse mass terms} for small $\eps$. The remaining perpendicular potential coefficient satisfies $|\lambda_\eps(|\Phi_{h,\theta}|^2-1)|\lesssim \eps/r_h\lesssim \eps^{1/8}$ and is absorbed as well.

Define the scalar and one-form expression
\begin{equation}\label{eq: compact longitudinal quadratic form}
\begin{split}
 {\bf Q}^{\parallel}_\eps(\psi,b)\coloneqq{}&
 \eps^{-1}\int_M\left(|d\psi|^2+
 \frac{\lambda_\eps}{\eps^2}(3|\Phi_{h,\theta}|^2-1)|\psi|^2\right)\,d\vol_g\\
 &+\eps\int_M(|db|^2+|d^*b|^2)\,d\vol_g.
\end{split}
\end{equation}
The errors in using \eqref{eq: compact longitudinal quadratic form} and separating ${\bf V}^\perp$ can be read from
\begin{equation}\label{eq: compact longitudinal derivative identities}
\begin{aligned}
 d_{A_{h,\theta}}(\psi\sigma_h)&=(d\psi)\sigma_h+\psi\,d_{A_{h,\theta}}\sigma_h,\\
 \nabla_{A_{h,\theta}}(b\sigma_h)&=(\nabla b)\sigma_h+b\otimes d_{A_{h,\theta}}\sigma_h,\\
 \langle\nabla_{A_{h,\theta}}{\bf V}^\perp,\sigma_h\rangle
 &=-\langle{\bf V}^\perp,d_{A_{h,\theta}}\sigma_h\rangle.
\end{aligned}
\end{equation}
The last identity is interpreted componentwise. All mixed zeroth-order terms contain either $F_{A_{h,\theta}}^\perp$ or $(d_{A_{h,\theta}}\Phi_{h,\theta})^\perp=|\Phi_{h,\theta}|d_{A_{h,\theta}}\sigma_h$. Equations \eqref{eq: compact outer background bounds} and \eqref{eq: compact longitudinal derivative identities}, with Cauchy--Schwarz, therefore give, after increasing $N$ to include the powers of $\eps$ in ${\bf Q}$,
\begin{equation}\label{eq: compact outer splitting error}
\begin{split}
 \left|{\bf Q}({\bf V})-{\bf Q}({\bf V}^\perp)-{\bf Q}^{\parallel}_\eps(\psi,b)\right|\lesssim \eps^{-N}e^{-c\eps^{-1/8}}\|{\bf V}\|_{\mathcal H}^2
 =o(1)\|{\bf V}\|_{\mathcal H}^2.
\end{split}
\end{equation}
Here $3|\Phi_{h,\theta}|^2-1\geq c>0$. By \eqref{eq: lambdaeps}, $\lambda_\eps/\eps^2\geq1$ for small $\eps$, so \eqref{eq: compact Hardy bound} gives
\begin{equation}\label{eq: compact scalar Hardy bound}
\begin{split}
 &\eps^{-1}\left(\|d\psi\|_2^2+\|(\eps+r_h)^{-1}\psi\|_2^2
 +\frac{\lambda_\eps}{\eps^2}\||\Phi_{h,\theta}|\psi\|_2^2\right)\lesssim \eps^{-1}\left(\|d\psi\|_2^2+\frac{\lambda_\eps}{\eps^2}\|\psi\|_2^2\right).
\end{split}
\end{equation}
For the one-form, $b_1(M)=0$ by \Cref{def: admissible amb}; the Hodge estimate and \eqref{eq: compact Hardy bound} give
\begin{equation}\label{eq: compact one form coercivity}
 \eps\left(\|\nabla b\|_2^2+\|(\eps+r_h)^{-1}b\|_2^2\right)
 \lesssim \eps\|b\|_{H^1}^2
 \lesssim \eps(\|db\|_2^2+\|d^*b\|_2^2).
\end{equation}
Equations \eqref{eq: compact longitudinal derivative identities} also compare the longitudinal terms in \eqref{eq: compact energy norm} with the left-hand sides of \eqref{eq: compact scalar Hardy bound} and \eqref{eq: compact one form coercivity}, with an error bounded by $C\eps^{-N}e^{-c\eps^{-1/8}}\|{\bf V}\|_{\mathcal H}^2$. Together with the absorption following \eqref{eq: compact transverse mass terms}--\eqref{eq: compact transverse curvature bound}, they imply
\begin{equation}\label{eq: compact outer norm comparison}
 \|{\bf V}\|_{\mathcal H}^2
 \leq C\left({\bf Q}({\bf V}^\perp)+{\bf Q}^{\parallel}_\eps(\psi,b)\right)
 +o(1)\|{\bf V}\|_{\mathcal H}^2.
\end{equation}
Combining \eqref{eq: compact outer norm comparison} with \eqref{eq: compact outer splitting error} and absorbing the $o(1)$ terms proves \eqref{eq: global outer coercivity}. The proof is concluded.
\end{proof}

We now establish \eqref{eq: uniform inverse estimate} with an extra supremum remainder. Set 
\begin{equation*}
	\|(\Psi,B)\|_{L^\infty(M)}\coloneqq \|\Psi\|_{L^\infty(M)}+\eps\|B\|_{L^\infty(M)}.
\end{equation*}

\begin{lemma}\label{lem: local estimates supremum remainder}
There exist $C,\eps_1>0$ such that the following holds for $\eps\in(0,\eps_1)$. Let ${\bf F}\in Y_{\eps,h,\theta}$ for some $h,\theta$ satisfying \eqref{eq: assumptions htheta}, and suppose that
\begin{equation*}
 ({\bf\Psi},\mathfrak b)\in\mathcal E_{\eps,h,\theta} \times C^{0,\tau}_\eps(\Gamma;\R^4),\qquad {\bf\Psi}=(\Psi,B),
\end{equation*}
satisfy
\begin{equation*}
 {\bf L}_{\eps,h,\theta}{\bf\Psi} ={\bf F}+\mathfrak B_{\eps,h,\theta}^*\mathfrak b.
\end{equation*}
Then
\begin{equation}\label{eq: local estimates supremum remainder}
\|{\bf\Psi}\|_{X_{\eps,h,\theta}} +\|\mathfrak b\|_{C^{0,\tau}_\eps(\Gamma;\R^4)} +\|\mathfrak B_{\eps,h,\theta}^*\mathfrak b\|_{Y_{\eps,h,\theta}}\lesssim \|{\bf F}\|_{Y_{\eps,h,\theta}} +\|{\bf \Psi}\|_{L^\infty(M)}.
\end{equation}
\end{lemma}

\begin{proof}
We prove the following three estimates:
\begin{align}
	\|\mathfrak B_{\eps,h,\theta}^*\mathfrak b\|_{Y_{\eps,h,\theta}}&\lesssim \|\mathfrak b\|_{C^{0,\tau}_\eps},\label{eq: residual injection estimate proof}\\
	\|\mathfrak b\|_{C^{0,\tau}_\eps} &\lesssim \|{\bf F}\|_{Y_{\eps,h,\theta}}+\|\Psi\|_{\tau,\eps,M}+\eps\|B\|_{\tau,\eps,M}+o(1)\|{\bf\Psi}\|_{X_{\eps,h,\theta}},\label{eq: coefficient local estimate proof}\\
	\|{\bf\Psi}\|_{X_{\eps,h,\theta}} &\lesssim \|{\bf F}\|_{Y_{\eps,h,\theta}} +\|\mathfrak b\|_{C^{0,\tau}_\eps} +\|{\bf\Psi}\|_{L^\infty(M)}.\label{eq: local Schauder estimate proof}
\end{align}
To see that they imply the lemma, interpolate, for every $a>0$,
\begin{equation}\label{eq: compact interpolation}
 \|\Psi\|_{\tau,\eps,M}+\eps\|B\|_{\tau,\eps,M}
 \leq a\|{\bf\Psi}\|_{X_{\eps,h,\theta}}
 +C_a\|{\bf\Psi}\|_{L^\infty(M)}.
\end{equation}
Insert \eqref{eq: compact interpolation} into \eqref{eq: coefficient local estimate proof}, then into
\eqref{eq: local Schauder estimate proof}. Choosing $a$ and $\eps$
small absorbs the solution norm; \eqref{eq: residual injection estimate proof}
then gives \eqref{eq: local estimates supremum remainder}.

\subsubsection*{Proof of \eqref{eq: residual injection estimate proof}}
Use the regions $\Omega_{j,h}$ and scales $r_j$ of \Cref{sec: norms}, and put
$t_j=r_j/\eps$. On $\Omega_{j,h}$, $1\leq j\leq J_\eps$, the weighted modes
in \eqref{eq: explicit residual injection} decay as $O(t_j^{-4})$. Multiplication
by the weights $r_j^2$ and $\eps r_j^2$ in \eqref{eq: source norm} gives
\begin{equation*}
r_j^2\left(\|(\mathfrak B_{\eps,h,\theta}^*\mathfrak b)_\Phi^\parallel\|_{\tau,r_j,\Omega_{j,h}} +\eps\|(\mathfrak B_{\eps,h,\theta}^*\mathfrak b)_A^\parallel\|_{\tau,r_j,\Omega_{j,h}}\right)\lesssim t_j^{-2}(1+t_j^\tau)\|\mathfrak b\|_{C^{0,\tau}_\eps}.
\end{equation*}
The factor $1+t_j^\tau$ accounts for estimating the coefficient $\mathfrak b$, controlled at H\"older scale $\eps$, in the source norm at scale $r_j=\eps t_j$. By \eqref{eq: lifted monopole modes} and \eqref{eq: approximation profile bounds}, the transverse components of ${\bf Z}_{i,\eps,h,\theta}$, $i=1,\ldots,4$, decay exponentially in $r_h/\eps$ and the Jacobian $J_h^{-1}$ in \eqref{eq: explicit residual injection} has bounded (scaled) norm. The core term is bounded by $C\|\mathfrak b\|_{C^{0,\tau}_\eps}$. On $\Omega_{J_\eps+1,h}$, the same calculation uses $r_{J_\eps+1}=\delta$ and $t_{J_\eps+1}=\delta/\eps$; derivatives of the fixed cutoff are bounded at scale $\delta$. Since $\tau<1$, summing $t_j^{-2}(1+t_j^\tau)$ proves \eqref{eq: residual injection estimate proof}.

\subsubsection*{Proof of \eqref{eq: coefficient local estimate proof}}
Use ${\bf E}_i={\bf E}_{i,\delta/(4\eps)}$ from \eqref{eq: normal weighted tests}, evaluated at $\lambda=\lambda_\eps$.
Apply \Cref{def: rescaling convention} to ${\bf\Psi}$, ${\bf F}$ and ${\bf L}_{\eps,h,\theta}{\bf\Psi}$. The constraint \eqref{eq: fibre constraints} becomes
\begin{equation}\label{eq: compact rescaled constraint}
 \int_{\R^3}\langle {\bf E}_i,\widehat{\bf\Psi}\rangle\,dy=0,\qquad q\in\Gamma.
\end{equation}
Testing \eqref{eq: constrained linear equation} in these rescaled coordinates against ${\bf E}_i$ gives
\begin{equation}\label{eq: actual multiplier identity}
\begin{split}
 \sum_{j=1}^4G_{\eps,h,ij}(q)\mathfrak b_j(q)
 &=\int_{\R^3}\langle{\bf L}_{\lambda_\eps}{\bf E}_i,\widehat{\bf\Psi}\rangle\,dy
 -\eps^2\int_{\R^3}\langle {\bf E}_i,\widehat{\bf F}\rangle\,dy
 +\mathcal E_i(q),\\
 G_{\eps,h,ij}(q)&\coloneqq\int_{\R^3}
 \frac{\langle {\bf E}_i(y),{\bf E}_j(y)\rangle}{J_h(q,\eps y)}\,dy,
\end{split}
\end{equation}
where the tested operator error is
\begin{equation*}
 \mathcal E_i(q)\coloneqq\eps^2\int_{\R^3}
 \langle {\bf E}_i,\widehat{{\bf L}_{\eps,h,\theta}{\bf\Psi}}\rangle\,dy
 -\int_{\R^3}\langle{\bf L}_{\lambda_\eps}{\bf E}_i,\widehat{\bf\Psi}\rangle\,dy.
\end{equation*}
The Jacobian factor $J_h^{-1}$ comes directly from \eqref{eq: explicit residual injection}. Since $J_h$ is uniformly bounded above and below and the four weighted modes are independent, their Gram matrix satisfies
\begin{equation}\label{eq: compact multiplier Gram positivity}
 c^TG_{\eps,h}(q)c=\int_{\R^3}
 \frac{|\sum_{i=1}^4c_i{\bf E}_i|^2}{J_h(q,\eps y)}\,dy\geq c_*|c|^2,
 \qquad c\in\R^4,
\end{equation}
with $c_*>0$ uniform in $q,\eps,h,\theta$. The scaled H\"older norms of this matrix and its inverse are uniformly bounded as well.

We estimate the terms in \eqref{eq: actual multiplier identity}. By \eqref{eq: linearised kernel}, ${\bf L}_{\lambda_\eps}{\bf Z}_i=0$; differentiating the weight in \eqref{eq: normal weighted tests} gives
\begin{equation}\label{eq: compact weighted test operator bound}
 |{\bf L}_{\lambda_\eps}{\bf E}_i(y)|\leq C(1+|y|)^{-6},
\end{equation}
including on the cutoff annulus. By \eqref{eq: compact weighted test operator bound}, the first integral is bounded in $C^{0,\tau}_\eps(\Gamma)$ by $C(\|\Psi\|_{\tau,\eps,M}+\eps\|B\|_{\tau,\eps,M})$. The decay $|{\bf E}_i(y)|\leq C(1+|y|)^{-4}$ and \eqref{eq: source norm} bound the second integral, including its factor $\eps^2$, by $C\|{\bf F}\|_{Y_{\eps,h,\theta}}$.

To estimate $\mathcal E$, use the formal self-adjointness of ${\bf L}_{\lambda_\eps}$ to write
\begin{equation*}
 \mathcal E_i(q)=\int_{\R^3}\left\langle {\bf E}_i,
 \eps^2\widehat{{\bf L}_{\eps,h,\theta}{\bf\Psi}}
 -{\bf L}_{\lambda_\eps}\widehat{\bf\Psi}\right\rangle\,dy.
\end{equation*}
In physical Fermi coordinates $(q,z)$, let $g_h^{ab},g_h^{ka},g_h^{k\ell}$ be the inverse metric coefficients of $F_h^*g$, evaluated below at $(q,\eps y)$. Summing normal indices $a,b=1,2,3$ and tangential indices $k,\ell=1,\ldots,n-3$, the operator difference is
\begin{equation}\label{eq: multiplier operator difference}
\begin{split}
 \eps^2\widehat{{\bf L}_{\eps,h,\theta}{\bf\Psi}}
 -{\bf L}_{\lambda_\eps}\widehat{\bf\Psi}
 ={}&-(g_h^{ab}-\delta^{ab})\partial_{y^a}\partial_{y^b}\widehat{\bf\Psi}
 -2\eps g_h^{ka}\partial_{q^k}\partial_{y^a}\widehat{\bf\Psi}\\
 &-\eps^2g_h^{k\ell}\partial_{q^k}\partial_{q^\ell}\widehat{\bf\Psi}
 +\mathcal R^{(1)}_\eps{\bf\Psi}.
\end{split}
\end{equation}
Here $\mathcal R^{(1)}_\eps$ is the part of this operator difference of order at most one: it contains the connection, frame, curvature and Higgs coefficient differences, including couplings to tangential one-form components.

In the tangential second-order term, subtract the inverse metric $g_\Gamma^{k\ell}(q)$ of $\Gamma$. Its pairing vanishes by differentiating \eqref{eq: compact rescaled constraint} twice in $q$, since ${\bf E}_i$ is independent of $q$. It remains to estimate the three metric differences $g_h^{ab}-\delta^{ab}$, $g_h^{ka}$ and $g_h^{k\ell}-g_\Gamma^{k\ell}$, and $\mathcal R^{(1)}_\eps$. Integrate normal derivatives onto the metric coefficients and ${\bf E}_i$; compact support removes boundary terms. The Fermi metric estimates in \Cref{sec: approximation errors}, together with \eqref{eq: normal weighted tests} and \eqref{eq: compact weighted test operator bound}, control these integrals by \eqref{eq: sol norm}. The scaled H\"older estimates introduce a factor $t^\tau$; by \eqref{eq: approximation regular matching} and \eqref{eq: approximation profile bounds}, the bounded connection corrections gain $\eps$ near the core, and their transverse components decay exponentially. Together these give
\begin{equation}\label{eq: compact multiplier error bound}
\begin{split}
 \|\mathcal E\|_{C^{0,\tau}_\eps}
 &\lesssim \left(\eps+\eps\int_1^{\delta/\eps}t^{-3+\tau}\,dt
 +\eps^{-\tau}\|h\|_{C^{2,\tau}}(1+|\log\eps|)\right)
 \|{\bf\Psi}\|_{X_{\eps,h,\theta}}\\
 &=o(1)\|{\bf\Psi}\|_{X_{\eps,h,\theta}}.
\end{split}
\end{equation}
Here \eqref{eq: assumptions htheta} gives
$\eps^{-\tau}\|h\|_{C^{2,\tau}}(1+|\log\eps|)\lesssim \eps^{1-\tau}(1+|\log\eps|)^3=o(1)$.
Using \eqref{eq: compact multiplier Gram positivity}, \eqref{eq: compact weighted test operator bound} and \eqref{eq: compact multiplier error bound} to invert $G_{\eps,h}$ in \eqref{eq: actual multiplier identity} proves \eqref{eq: coefficient local estimate proof}.

\subsubsection*{Proof of \eqref{eq: local Schauder estimate proof}}
Apply Schauder estimates at scale $r_0=\eps$ on $\Omega_{0,h}$ and scale
$r_j$ on each $\Omega_{j,h}$, using slightly larger overlapping regions.
After rescaling an annulus to unit size, the nonnegative zeroth-order
coefficients in the longitudinal connection, transverse, and longitudinal
Higgs equations are, respectively,
\begin{equation*}
 0,\qquad \mu^2\asymp r_j^2/\eps^2,\qquad \mu^2\asymp2\lambda_\eps r_j^2/\eps^2.
\end{equation*}
For $(-\sum_i\partial_i^2+\mu^2)v=f$ on a unit ball, the required estimate is
\begin{equation}\label{eq: parameter Schauder estimate}
 \|v\|_{C^{2,\tau}(B_{1/2})}+\mu^2\|v\|_{C^{0,\tau}(B_{1/2})} \lesssim \|f\|_{C^{0,\tau}(B_1)}+\|v\|_{L^\infty(B_1)},
\end{equation}
Applying \eqref{eq: parameter Schauder estimate} with frozen coefficients on sufficiently small rescaled balls and absorbing their variations by the standard Schauder perturbation argument yields the local estimate. Multiplying the source terms by $r_j^2$, with an additional factor $\eps$ for the connection component, recovers the weights in \eqref{eq: sol norm}. Taking maxima over the regions $\Omega_{j,h}$ and using \eqref{eq: source norm} and \eqref{eq: residual injection estimate proof} proves \eqref{eq: local Schauder estimate proof}.
\end{proof}

\begin{proposition}\label[proposition]{prop: compact supremum estimate}
Under the hypotheses of \Cref{prop: global constrained coercivity}, there exist $C,\eps_0>0$ such that the following holds for $0<\eps<\eps_0$. Let ${\bf F}\in Y_{\eps,h,\theta}$ and suppose that
\begin{equation*}
 ({\bf\Psi},\mathfrak b)\in\mathcal E_{\eps,h,\theta} \times C^{0,\tau}_\eps(\Gamma;\R^4),\qquad {\bf\Psi}=(\Psi,B),
\end{equation*}
satisfy
\begin{equation*}
 {\bf L}_{\eps,h,\theta}{\bf\Psi} ={\bf F}+\mathfrak B_{\eps,h,\theta}^*\mathfrak b.
\end{equation*}
Then
\begin{equation}\label{eq: compact supremum estimate}
 \|\Psi\|_{L^\infty(M)}+\eps\|B\|_{L^\infty(M)} \lesssim \|{\bf F}\|_{Y_{\eps,h,\theta}}.
\end{equation}
The constants are independent of $\eps,h,\theta$ in the stated ranges.
\end{proposition}
\begin{proof}
We first prove two estimates outside the core. For $r_h\geq 2M_0\eps$, the transverse estimate is
\begin{equation}\label{eq: compact transverse tail proof}
 |\Psi^\perp(x)|+\eps|B^\perp(x)|\lesssim (1+r_h(x)/\eps)^{-2}\|{\bf\Psi}\|_{X_{\eps,h,\theta}}.
\end{equation}
For every sufficiently large fixed $A$ and $r_h(x)>4A\eps$, the longitudinal estimate is
\begin{equation}\label{eq: compact neutral tail proof}
|\Psi^\parallel(x)|+\eps|B^\parallel(x)| \lesssim \|{\bf F}\|_{Y_{\eps,h,\theta}}+A^{-2}\|\mathfrak b\|_\infty+\left(\frac{A\eps}{r_h(x)}+A\eps+e^{-cA}\right)\|{\bf\Psi}\|_{X_{\eps,h,\theta}}.
\end{equation}
Together with \eqref{eq: local estimates supremum remainder}, \eqref{eq: compact transverse tail proof} and \eqref{eq: compact neutral tail proof}
force any normalised counterexample to have a nonzero limit near the core.
We then use \Cref{lem: product Liouville} to exclude that limit.

\smallskip

Equation \eqref{eq: compact transverse tail proof} follows directly from the
second term in \eqref{eq: sol norm}.
Write $\sigma_h=\Phi_{h,\theta}/|\Phi_{h,\theta}|$ outside the core. Projection
onto $\R\sigma_h$ in \eqref{eq: model linearised operator} gives scalar and Hodge equations; the derivatives of
$\sigma_h$ contribute to the remainders $e_0,e_1$ below. With
${\bf F}=(F_\Phi,F_A)$, these equations are
\begin{equation}\label{eq: compact projected exterior equations}
\begin{split}
 (\Delta_0+V_\eps)\langle\Psi,\sigma_h\rangle &=\langle F_\Phi+(\mathfrak B_{\eps,h,\theta}^*\mathfrak b)_\Phi,\sigma_h\rangle+e_0,\\
 \Delta_1\langle B,\sigma_h\rangle &=\langle F_A+(\mathfrak B_{\eps,h,\theta}^*\mathfrak b)_A,\sigma_h\rangle+e_1,
\end{split}
\end{equation}
where $V_\eps=\lambda_\eps\eps^{-2}(3|\Phi_{h,\theta}|^2-1)$ satisfies
\begin{equation}\label{eq: compact projected exterior errors}
 V_\eps\geq c\lambda_\eps/\eps^2,\qquad |e_0|+\eps|e_1|\lesssim \eps^{-2}e^{-cr_h/\eps}\|{\bf\Psi}\|_{X_{\eps,h,\theta}}.
\end{equation}
Every coupling remaining after neutral projection contains a transverse background coefficient. Equation \eqref{eq: compact outer background bounds} absorbs the powers of $\eps^{-1/2}$ produced by the gluing cutoffs \eqref{eq: s_eps}.

Since $b_1(M)=0$ by \Cref{def: admissible amb}, the Hodge Green kernel is bounded by $C(1+d_M(x,z)^{2-n})$. Extend $V_\eps$ positively through the core. Since $V_\eps\geq1$ for small $\eps$, comparison with $\Delta_0+1$ gives the same uniform bound for the scalar Green kernel. Integration in the tangential variables yields
\begin{equation}\label{eq: compact Green shell bound}
 \sup_x\int_{\Omega_{j,h}}\bigl(1+d_M(x,z)^{2-n}\bigr)\,d\vol_g(z)\lesssim r_j^2.
\end{equation}
Summing \eqref{eq: compact Green shell bound} with \eqref{eq: source norm} gives a bound of $C\|{\bf F}\|_{Y_{\eps,h,\theta}}$. For the multiplier, \eqref{eq: explicit residual injection}, \eqref{eq: monopole translations}, \eqref{eq: monopole phase mode} and \eqref{eq: approximation profile bounds} give
\begin{equation*}
 |(\mathfrak B_{\eps,h,\theta}^*\mathfrak b)_\Phi|+\eps|(\mathfrak B_{\eps,h,\theta}^*\mathfrak b)_A| \lesssim \eps^2r_h^{-4}\|\mathfrak b\|_\infty\qquad(r_h\geq A\eps).
\end{equation*}
Summing its Green potential over shells with $r_j\geq A\eps$ gives
\begin{equation}\label{eq: compact multiplier Green bound}
 \sum_{r_j\geq A\eps}r_j^2\eps^2r_j^{-4}\|\mathfrak b\|_\infty\lesssim A^{-2}\|\mathfrak b\|_\infty.
\end{equation}

Put $\chi_A=1-\eta(r_h/(A\eps))$. For a scalar or one-form $v$, the
commutator is
\begin{equation}\label{eq: compact cutoff commutator}
 [\Delta_k,\chi_A]v=(\Delta_0\chi_A)v-2\nabla_{\nabla\chi_A}v,
 \qquad k=0,1.
\end{equation}
The expression in \eqref{eq: compact cutoff commutator} is supported where $A\eps\leq r_h\leq2A\eps$ and has size at most $C(A\eps)^{-2}\|{\bf\Psi}\|_{X_{\eps,h,\theta}}$, with the connection component multiplied by $\eps$. The supporting annulus has normal volume $O(A^3\eps^3)$. Combining \eqref{eq: compact projected exterior equations}, \eqref{eq: compact projected exterior errors}, \eqref{eq: compact Green shell bound}, \eqref{eq: compact multiplier Green bound} and \eqref{eq: compact cutoff commutator} therefore proves \eqref{eq: compact neutral tail proof}.

\smallskip

Suppose that no uniform estimate holds for all sufficiently small $\eps$. Then there are $\eps_k\downarrow0$, parameters $(h_k,\theta_k)$ satisfying \eqref{eq: assumptions htheta}, and solutions $({\bf\Psi}_k,\mathfrak b_k)$ with sources ${\bf F}_k$ satisfying \eqref{eq: constrained linear equation}. By homogeneity, we may normalise them so that, writing ${\bf\Psi}_k=(\Psi_k,B_k)$,
\begin{equation*}
 \|{\bf\Psi}_k\|_{X_{\eps_k,h_k,\theta_k}}=1,
 \qquad \|{\bf F}_k\|_{Y_{\eps_k,h_k,\theta_k}}\to 0.
\end{equation*}
Equation \eqref{eq: local estimates supremum remainder} bounds $\mathfrak b_k$ uniformly in $C^{0,\tau}_{\eps_k}$ and gives
\begin{equation*}
 \liminf_{k\to\infty}
 \bigl(\|\Psi_k\|_\infty+\eps_k\|B_k\|_\infty\bigr)>0.
\end{equation*}
Choose first $A$ large and then $r_{h_k}/\eps_k$ large in \eqref{eq: compact neutral tail proof}. Together with \eqref{eq: compact transverse tail proof}, this gives points $x_k\in M$ and constants $c,C>0$, independent of $k$, such that
\begin{equation}\label{eq: compact nonzero core points}
 r_{h_k}(x_k)\lesssim \eps_k,\qquad
 |\Psi_k(x_k)|+\eps_k|B_k(x_k)|\geq c.
\end{equation}
Let $q_k=\pi(\mathscr T_{h_k}^{-1}(x_k))\in\Gamma$. Choose orthonormal coordinates on $T_{q_k}\Gamma$ and define
\begin{equation*}
\begin{split}
 p_k(x)&\coloneqq\exp^\Gamma_{q_k}(\eps_k x),\\
 \mathscr P_k(x,y)&\coloneqq\mathscr F_{\eps_k,h_k}(p_k(x),y),\\
 U_k(x,y)&\coloneqq U_{\eps_k}(p_k(x),y).
\end{split}
\end{equation*}
Here $U_{\eps_k}$ is defined by \Cref{def: rescaling convention} with $\theta=\theta_k$. This blow-up also dilates the tangential variable and uses the full coordinate pullback of the one-form:
\begin{equation}\label{eq: compact blowup rescaling}
 \widehat\Psi_k=\Ad_{U_k}\mathscr P_k^*\Psi_k,\qquad
 \widehat B_k=\Ad_{U_k}\mathscr P_k^*B_k,\qquad
 \eps_k^{-2}\mathscr P_k^*g.
\end{equation}
Thus both normal and tangential one-form coefficients acquire the factor $\eps_k$. By \eqref{eq: lambdaeps}, after passing to a subsequence,
\begin{equation*}
 \lambda_{\eps_k}=\frac{2\eps_k^3}{\delta_{\eps_k}}
 \to \lambda_\infty\in[0,\lambda_*].
\end{equation*}
The limit need not be zero, since \eqref{eq: deltaeps} permits $\delta_\eps=C_*\eps^3$.
By \eqref{eq: compact nonzero core points} and \eqref{eq: compact blowup rescaling}, local compactness gives a nonzero bounded product limit $(\Psi_\infty,B_\infty)$ for ${\bf L}_{\lambda_\infty}$, with bounded residual coefficients obtained from $\mathfrak b_k\circ p_k$. The profiles and modes converge locally at this limiting coupling. The fibre constraints pass to the limit because the omitted weighted moments of bounded fields are $O(T^{-1})$ on $|y|>T$. In the limit of \eqref{eq: compact neutral tail proof}, choose $A=|y|^{1/3}$ for large $|y|$ to obtain
\begin{equation}\label{eq: compact limit neutral decay}
 |\Psi_\infty^\parallel(x,y)|+|B_\infty^\parallel(x,y)|
 \lesssim (1+|y|)^{-2/3}.
\end{equation}
The transverse components satisfy \eqref{eq: compact transverse tail proof}, and hence decay as $O((1+|y|)^{-2})$ in the limit. Together with \eqref{eq: compact limit neutral decay}, \Cref{lem: product Liouville}, also applied to the tangential one-form components, forces $(\Psi_\infty,B_\infty)$ to vanish, a contradiction.
\end{proof}

\begin{corollary}\label[corollary]{cor: local estimates supremum remainder}
Under the hypotheses of \Cref{prop: compact supremum estimate}, every such solution satisfies
\begin{equation}\label{eq: compact full a priori estimate}
 \|{\bf\Psi}\|_{X_{\eps,h,\theta}} +\|\mathfrak b\|_{C^{0,\tau}_\eps(\Gamma;\R^4)} +\|\mathfrak B_{\eps,h,\theta}^*\mathfrak b\|_{Y_{\eps,h,\theta}} \lesssim \|{\bf F}\|_{Y_{\eps,h,\theta}}.
\end{equation}
\end{corollary}
\begin{proof}
Combine \eqref{eq: local estimates supremum remainder} with \eqref{eq: compact supremum estimate}.
\end{proof}

\begin{proof}[Proof of \Cref{prop: uniform inverse}  on compact $M$]
We establish three facts: the fibre constraint has a bounded right inverse,
the constrained variational problem has a unique weak solution, and this
solution has the regularity required in \Cref{prop: uniform inverse}.
The estimate \eqref{eq: compact full a priori estimate} then gives uniform
bounds; coercivity and surjectivity of the constraint give uniqueness of
both unknowns.

\medskip

For fixed $\eps$, the fibre map extends boundedly from
$\mathcal H_{\eps,h,\theta}$ to $H^1(\Gamma;\R^4)$ and admits
\begin{equation}\label{eq: fibre constraint right inverse}
 \mathscr J_{\eps,h,\theta}:H^1(\Gamma;\R^4)\to
 \mathcal H_{\eps,h,\theta},\qquad
 \mathfrak B_{\eps,h,\theta}\mathscr J_{\eps,h,\theta}=\mathrm{Id}.
\end{equation} Indeed, multiply the four transported modes by $\eta_R(z/\eps)$ with fixed large $R$ and invert their fibre moment matrix. This matrix is uniformly positive by the argument for \eqref{eq: normal moment Gram matrix} in the proof of \Cref{lem: normal weighted coercivity}. Differentiation in $q$ proves the asserted $H^1$ bounds, thus $\mathfrak B_{\eps,h,\theta}$ has a bounded right inverse.

\smallskip

\Cref{prop: global constrained coercivity} and Lax--Milgram give a unique constrained weak solution of
\begin{equation}\label{eq: compact variational equation}
 {\bf Q}_{(\Phi_{h,\theta},A_{h,\theta})}({\bf\Psi},{\bf V}) =\langle{\bf F},{\bf V}\rangle_{L^2_\eps(M)} \qquad\mathfrak B_{\eps,h,\theta}{\bf V}=0.
\end{equation}
The functional on the right is bounded for fixed $\eps$. Put $\mathcal R={\bf L}_{\eps,h,\theta}{\bf\Psi}-{\bf F}$ and define the distribution $\mathfrak b(a)=\mathcal R(\mathscr J_{\eps,h,\theta}a)$ on $H^1(\Gamma;\R^4)$. By \eqref{eq: fibre constraint right inverse}, ${\bf V}-\mathscr J_{\eps,h,\theta}\mathfrak B_{\eps,h,\theta}{\bf V}$ satisfies the constraint, so \eqref{eq: compact variational equation} gives
\begin{equation*}
 \mathcal R({\bf V})=\mathfrak b(\mathfrak B_{\eps,h,\theta}{\bf V}).
\end{equation*}
This is precisely \eqref{eq: constrained linear equation} with the adjoint convention \eqref{eq: dual residual injection}; surjectivity of the fibre map makes $\mathfrak b$ unique.
\Cref{prop: global constrained coercivity} and \Cref{lem: local estimates supremum remainder} give the required bounds.
\end{proof}

Lastly, we show that the invertibility theory of $\widehat{\bf L}_{h,\theta}$ can be recovered from the one of ${\bf L}_{h,\theta}$ via a small perturbation.
\begin{corollary}\label[corollary]{cor: compact restricted inverse}
Under the hypotheses of \Cref{prop: global constrained coercivity}, we have
\begin{equation}\label{eq: compact gauge projection bound}
 \left\|\eps^{-1}{\bf G}_{h,\theta}\Pi_h{\bf G}_{h,\theta}^*{\bf\Psi}\right\|_{Y_{\eps,h,\theta}} \lesssim \eps^2(1+|\log\eps|)\|{\bf\Psi}\|_{X_{\eps,h,\theta}}
\end{equation}
for every ${\bf\Psi}\in X_{\eps,h,\theta}$. Consequently, for all sufficiently small $\eps$ and every ${\bf F}\in Y_{\eps,h,\theta}$, there is a unique pair
\begin{equation*}
 ({\bf\Psi},\mathfrak b)\in\mathcal E_{\eps,h,\theta} \times C^{0,\tau}_\eps(\Gamma;\R^4)
\end{equation*}
satisfying
\begin{equation*}
 \widehat{\bf L}_{h,\theta}{\bf\Psi} ={\bf F}+\mathfrak B_{\eps,h,\theta}^*\mathfrak b,
\end{equation*}
and the uniform estimate \eqref{eq: uniform inverse estimate} holds for this pair.
\end{corollary}
\begin{proof}
It suffices to prove \eqref{eq: compact gauge projection bound}: composition
with the inverse in \Cref{prop: uniform inverse}, satisfying \eqref{eq: compact full a priori estimate}, then has norm
$O(\eps^2(1+|\log\eps|))<1$, so its Neumann series gives the restricted inverse.

\medskip

Since $[\Phi_{h,\theta},\xi_{\eps,h}]=0$ and $\chi=1$, \eqref{eq: phase projection} gives
\begin{equation}\label{eq: compact gauge rank one}
 \eps^{-1}{\bf G}_{h,\theta}\Pi_h{\bf G}_{h,\theta}^*{\bf\Psi} =\frac{\int_M\langle B,d_{A_{h,\theta}}\xi_{\eps,h}\rangle\,d\vol_g}{\int_M|\xi_{\eps,h}|^2\,d\vol_g}\,{\bf G}_{h,\theta}\xi_{\eps,h}.
\end{equation}
The denominator is bounded below. Equations \eqref{eq: approximation profile bounds}, \eqref{eq: q_eps} and \eqref{eq: shifted phase approximation} give
\begin{equation}\label{eq: compact phase generator bounds}
 |d_{A_{h,\theta}}\xi_{\eps,h}|\lesssim \frac{\eps}{(\eps+r_h)^2},\qquad \int_M|d_{A_{h,\theta}}\xi_{\eps,h}|\,d\vol_g\lesssim \eps,
\end{equation}
with the corresponding scaled H\"older bounds and exponentially small transverse tails. Hence
\begin{equation}\label{eq: compact gauge numerator}
 \left|\int_M\langle B,d_{A_{h,\theta}}\xi_{\eps,h}\rangle\,d\vol_g\right|\lesssim \eps\|B\|_\infty\lesssim \|{\bf\Psi}\|_{X_{\eps,h,\theta}}.
\end{equation}
It remains to establish \eqref{eq: compact phase source bound}. Since ${\bf G}_{h,\theta}\xi_{\eps,h}=(0,d_{A_{h,\theta}}\xi_{\eps,h})$, only the connection terms of \eqref{eq: source norm} contribute. On each annulus, multiply \eqref{eq: compact phase generator bounds} in scaled H\"older norm by its weight $\eps r_j^2$; for $1\leq j\leq J_\eps+1$, this gives
\begin{equation*}
 \eps r_j^2
 \|(d_{A_{h,\theta}}\xi_{\eps,h})^\parallel\|_{\tau,r_j,\Omega_{j,h}}
 \lesssim \eps^2.
\end{equation*}
The core term and the maximum of the transverse terms are also
$O(\eps^2)$. Summing the $J_\eps+1=O(1+|\log\eps|)$ longitudinal terms gives
\begin{equation}\label{eq: compact phase source bound}
 \|{\bf G}_{h,\theta}\xi_{\eps,h}\|_{Y_{\eps,h,\theta}}\lesssim \eps^2(1+|\log\eps|),
\end{equation}
which, together with \eqref{eq: compact gauge rank one} and \eqref{eq: compact gauge numerator}, proves \eqref{eq: compact gauge projection bound}.
\end{proof}

\Cref{cor: compact restricted inverse} finishes the proof of \Cref{prop: uniform inverse} in the compact case.

\subsection{The noncompact case}
Compact-case existence follows from global coercivity and Lax--Milgram. Here the longitudinal equations at infinity involve $\Delta_0$ and $\Delta_1$, with no assumed $L^2$ spectral gap. Since coercivity might not be uniform along the ends we use a different strategy, combining the weighted inverses of $\Delta_0$ and $\Delta_1$ from \Cref{def: admissible amb noncomp} with a product inverse near $\Gamma$, to obtain solutions with the required decay. The proof of \Cref{prop: uniform inverse} in this case proceeds as follows.
\begin{itemize}
 \item \Cref{lem: noncompact local remainder} combines weighted exterior estimates with the local estimates to control the full solution norm up to a supremum on a fixed compact set.
 \item \Cref{prop: noncompact a priori estimate} removes this compact supremum remainder by a blow-up argument and the product Liouville theorem, giving the uniform a priori estimate.
 \item \Cref{lem: noncompact product constructions} constructs the constrained product inverse, its normal tail bounds, and the smooth longitudinal extensions needed to continue exterior solutions through the core.
 \item \Cref{prop: noncompact ordinary inverse} glues the exterior and interior solutions, restores the fibre constraints, and corrects the small remaining error by a Neumann series to obtain the ordinary inverse.
 \item \Cref{cor: noncompact restricted inverse} treats the gauge projection as a small perturbation, obtaining the restricted inverse and completing the proof of \Cref{prop: uniform inverse}.
\end{itemize}

\smallskip

Choose fixed compact neighbourhoods $K_0\Subset K_1$ containing the tube and all parameter cutoffs.
\begin{lemma}\label{lem: noncompact local remainder}
Every solution $({\bf\Psi},\mathfrak b)$ of \eqref{eq: constrained linear equation}, with ${\bf\Psi}=(\Psi,B)\in\mathcal E_{\eps,h,\theta}$ and $\mathfrak b\in C^{0,\tau}_\eps(\Gamma;\R^4)$, satisfies
\begin{equation}\label{eq: noncompact local remainder}
\|{\bf\Psi}\|_{X_{\eps,h,\theta}}+\|\mathfrak b\|_{C^{0,\tau}_\eps}+\|\mathfrak B_{\eps,h,\theta}^*\mathfrak b\|_{Y_{\eps,h,\theta}}\lesssim  \|{\bf F}\|_{Y_{\eps,h,\theta}}+\|{\bf \Psi}\|_{L^\infty(K_1)}.
\end{equation}
\end{lemma}
\begin{proof}
For the exterior correction $b$, write
\begin{equation}\label{eq: exterior connection}
 A_D\coloneqq A_\sigma+b\sigma\qquad\text{on }M\setminus\Gamma.
\end{equation}
Here $b$ is the real one-form of \Cref{lem: noncompact exterior correction}.
By \eqref{eq: noncompact shifted approximation} and \eqref{eq: noncompact exterior cancellation}, outside $K_0$,
\begin{equation*}
 (\Phi_{h,\theta},A_{h,\theta})=((1-\eps H)\sigma,A_D),\qquad d_{A_D}\sigma=0.
\end{equation*}
Write ${\bf F}=(F_\Phi,F_A)$ and choose a fixed cutoff $\chi$ vanishing near $K_0$ and equal to one outside $K_1$. By \eqref{eq: explicit residual injection} and \eqref{eq: model linearised operator}, the longitudinal coefficients satisfy
\begin{equation}\label{eq: noncompact longitudinal exterior equations}
 \Delta_0\langle\Psi,\sigma\rangle=\langle F_\Phi,\sigma\rangle,
 \qquad
 \Delta_1\langle B,\sigma\rangle=\langle F_A,\sigma\rangle
 \quad\hbox{on }M\setminus K_0.
\end{equation}
Apply the weighted inverses of \Cref{def: admissible amb noncomp} after multiplication by $\chi$. This bounds the longitudinal $C^{2,\tau}_\nu$ norms by the source norm and the $C^{1,\tau}$ norms on the cutoff region $K_1\setminus K_0$. Equation \eqref{eq: noncompact longitudinal exterior equations} also controls the Laplacian terms in \eqref{eq: noncompact exterior solution norm}.

For the transverse components, \eqref{eq: model linearised operator} gives the potential $4\eps^{-2}(1-\eps H)^2\geq c\eps^{-2}$. Multiply the scalar equation by $R^{\nu+2}\chi$ and the one-form equation by $\eps R^{\nu+2}\chi$. Since $d_{A_D}\Phi_{h,\theta}=-\eps\,dH\,\sigma$, the couplings are bounded in this scaling. Smooth $R$ to an equivalent weight with bounded logarithmic derivatives; for a scalar or one-form $v$,
\begin{equation}\label{eq: noncompact conjugated Laplacian}
\begin{split}
 R^{\nu+2}\nabla^*\nabla(R^{-\nu-2}v)
 ={}&\nabla^*\nabla v+2(\nu+2)\nabla_{\nabla\log R}v\\
 &-\bigl((\nu+2)\Delta_0\log R+(\nu+2)^2|d\log R|^2\bigr)v,
\end{split}
\end{equation}
where $\nabla$ is induced by $A_D$ and $g$, and $\Delta_0=d^*d$. The potential absorbs the bounded curvature, coupling and weight terms from \Cref{lem: noncompact exterior correction}, \Cref{def: admissible amb noncomp} and \eqref{eq: noncompact conjugated Laplacian}. The norm maximum principle applied to $R^{\nu+2}\chi\sqrt{|\Psi^\perp|^2+\eps^2|B^\perp|^2}$, followed by \eqref{eq: parameter Schauder estimate}, gives
\begin{equation*}
\begin{split}
 &\|\Psi^\perp\|_{C^{2,\tau}_{\nu+2}}+\eps\|B^\perp\|_{C^{2,\tau}_{\nu+2}}
 +\eps^{-2}\bigl(\|\Psi^\perp\|_{C^{0,\tau}_{\nu+2}}+\eps\|B^\perp\|_{C^{0,\tau}_{\nu+2}}\bigr)\\
 &\qquad\lesssim\|{\bf F}\|_{Y_{\eps,h,\theta}}
 +\|\Psi\|_{C^{1,\tau}(K_1\setminus K_0)}
 +\eps\|B\|_{C^{1,\tau}(K_1\setminus K_0)},
\end{split}
\end{equation*}
with the left-hand norms taken over $M\setminus K_1$. Combine with the longitudinal bounds and \eqref{eq: local estimates supremum remainder}; interpolation absorbs the $C^{1,\tau}$ terms supported in $K_1\setminus K_0$ and leaves the compact supremum in \eqref{eq: noncompact local remainder}.
\end{proof}

\begin{proposition}[The noncompact a priori estimate]\label[proposition]{prop: noncompact a priori estimate}
Every solution as in \Cref{lem: noncompact local remainder} satisfies \eqref{eq: uniform inverse estimate}.
\end{proposition}
\begin{proof}
We adapt the proof of \Cref{prop: compact supremum estimate}, using the weighted inverses in \Cref{def: admissible amb noncomp}. For sources in a fixed compact set, the kernels of the inverses of $\Delta_0$ and $\Delta_1$ satisfy
\begin{equation}\label{eq: noncompact Green kernel bound}
 |G_p(x,z)|\lesssim 1+d_M(x,z)^{2-n},\qquad p=0,1,
\end{equation}
on compact sets, and are bounded by $CR(x)^{-\nu}$ outside a larger compact set. Integration in the tangential variables therefore gives
\begin{equation}\label{eq: noncompact Green shell bound}
 \sup_x\int_{\Omega_{j,h}}|G_p(x,z)|\,d\vol_g(z)\lesssim r_j^2,
 \qquad 0\leq j\leq J_\eps.
\end{equation}
A source of size $t^{-2}$ supported in a tube of radius $t$ contributes at most $Ct$ outside a fixed neighbourhood.

Apply \eqref{eq: noncompact Green kernel bound} and \eqref{eq: noncompact Green shell bound} to sources supported in $K_1$; the remaining source is controlled by \Cref{def: admissible amb noncomp}. The cutoff argument proving \eqref{eq: compact neutral tail proof} therefore applies on $K_1$, with zero scalar potential. If the estimate failed, choose $\eps_k\downarrow0$ and normalised solutions ${\bf\Psi}_k$ with parameters $(h_k,\theta_k)$ and sources ${\bf F}_k$ such that
\begin{equation*}
 \|{\bf\Psi}_k\|_{X_{\eps_k,h_k,\theta_k}}=1,
 \qquad \|{\bf F}_k\|_{Y_{\eps_k,h_k,\theta_k}}\longrightarrow0.
\end{equation*}
Equation \eqref{eq: noncompact local remainder} forces a nonzero supremum on $K_1$, and \eqref{eq: compact neutral tail proof} places it in a fixed rescaled core. The rescaling \eqref{eq: compact blowup rescaling} and moment convergence in the proof of \Cref{prop: compact supremum estimate}, with \eqref{eq: compact limit neutral decay}, give a nonzero solution of the homogeneous product equation for ${\bf L}_0$, with vanishing moments and normal decay $O((1+|y|)^{-2/3})$. This contradicts \Cref{lem: product Liouville}. Hence \eqref{eq: uniform inverse estimate} holds for every solution, as required.
\end{proof}

We next give two constructions on $\R^{n-3}\times\R^3$. Write $(q,z)$ for the physical coordinates and $(x,y)$ for their rescaled counterparts. For pairs ${\bf\Psi}=(\Psi,B)$, ${\bf F}=(F_\Phi,F_A)$ and a coefficient $\mathfrak b\colon\R^{n-3}\to\R^4$, set
\begin{equation}\label{eq: product pair rescaling}
\begin{aligned}
 D_\eps(x,y)&=(\eps x,\eps y),
 &A_{0,\eps}&=(D_\eps^{-1})^*A_0,\\
 {\bf V}&=(D_\eps^*\Psi,D_\eps^*B),
 &{\bf f}&=(\eps^2D_\eps^*F_\Phi,\eps^2D_\eps^*F_A),\\
 c_i(x)&=\mathfrak b_i(\eps x),
 &\mathcal A_r&=\R^{n-3}\times\{r<|z|<4r\}.
\end{aligned}
\end{equation}
Here $A_0$ is pulled back from the normal factor. For any pair ${\bf V}=(V_\Phi,V_A)$, write
\begin{equation*}
 V_A=V_a^N\,dy^a+\sum_{j=1}^{n-3}V_j^T\,dx^j,
 \qquad {\bf V}_N=(V_\Phi,V_a^N\,dy^a),
\end{equation*}
and use the same notation for ${\bf f}$. Put ${\bf E}_i={\bf E}_{i,\delta/(4\eps)}$ in \eqref{eq: normal weighted tests}, evaluated at $\lambda=0$. The equation and constraints to be solved are
\begin{equation}\label{eq: constrained product system}
\begin{cases}
 (-\Delta_x+{\bf L}_0){\bf V}_N
 ={\bf f}_N+\displaystyle\sum_{i=1}^4c_i(x){\bf E}_i(y),\\[2mm]
 (-\Delta_x+\nabla_{A_0}^*\nabla_{A_0}-\ad_{\Phi_0}^2)V_j^T=f_j^T,
 &j=1,\ldots,n-3,\\[2mm]
 \displaystyle\int_{\R^3}\langle{\bf V}_N(x,y),{\bf E}_i(y)\rangle\,dy=0,
 &i=1,\ldots,4.
\end{cases}
\end{equation}

\begin{lemma}[Product inverse and longitudinal extensions]\label{lem: noncompact product constructions}
For every smooth compactly supported pair ${\bf F}$ on $\R^{n-3}\times\R^3$, there exist ${\bf\Psi}=(\Psi,B)$ and $\mathfrak b\colon\R^{n-3}\to\R^4$ whose rescalings \eqref{eq: product pair rescaling} solve \eqref{eq: constrained product system}, with
\begin{equation}\label{eq: product inverse bound}
 \|{\bf\Psi}\|_{X_{\eps,0,0}}+\|\mathfrak b\|_{C^{0,\tau}_\eps(\R^{n-3})}
 \lesssim \|{\bf F}\|_{Y_{\eps,0,0}}.
\end{equation}
If $\supp{\bf F}\subset\{|z|\leq2s\}$ and $s\geq\eps$, the same solution satisfies, for $r>4s$,
\begin{equation}\label{eq: product parametrix tail}
 \sum_{k=0}^2r^k\left(
 \|\nabla_{A_{0,\eps}}^k\Psi\|_{\tau,r,\mathcal A_r}
 +\eps\|\nabla_{A_{0,\eps}}^kB\|_{\tau,r,\mathcal A_r}\right)
 \lesssim (s/r)^{2/3}\|{\bf F}\|_{Y_{\eps,0,0}}.
\end{equation}
Moreover, for every $a\in\R$ and constant real one-form $\vartheta$ on $\R^n$, there is a smooth pair ${\bf U}_{a,\vartheta}(y)$, independent of $x$, solving \eqref{eq: constrained product system} with ${\bf f}=0$ and $c_i=0$ and satisfying
\begin{equation}\label{eq: product longitudinal extensions}
\begin{split}
 \sup_{y\in\R^3}|{\bf U}_{a,\vartheta}(y)|&\lesssim (|a|+|\vartheta|),\\
 \sum_{k=0}^2|y|^k\left|\nabla_{A_0}^k
 \bigl({\bf U}_{a,\vartheta}(y)-(a\sigma(y),\vartheta\sigma(y))\bigr)\right|
 &\lesssim \frac{(|a|+|\vartheta|)}{|y|},\qquad |y|\geq1.
\end{split}
\end{equation}
\end{lemma}
\begin{proof}
Use the rescaling \eqref{eq: product pair rescaling}. For the unit-scale operator in \eqref{eq: constrained product system}, \Cref{lem: normal weighted coercivity} gives
\begin{equation*}
\int_{\R^{n-3}}\left({\bf Q}_0({\bf V}_N)+\|\nabla_x{\bf V}_N\|_{L^2(\R^3)}^2\right)\,dx\geq c\int_{\R^{n-3}}\left(\|{\bf V}_N\|_{\mathcal H_0}^2+\|\nabla_x{\bf V}_N\|_{L^2(\R^3)}^2\right)\,dx.
\end{equation*}
Tangential one-form components satisfy \eqref{eq: normal Hardy bound}. For smooth compactly supported data, the source functional is bounded because
\begin{equation*}
 |\langle{\bf f},{\bf V}\rangle_{L^2}| \leq\||y|{\bf f}\|_{L^2}\,\||y|^{-1}{\bf V}\|_{L^2}.
\end{equation*}
Lax--Milgram therefore gives a weak solution directly on the whole product. The right inverse from \eqref{eq: normal moment Gram matrix} recovers its coefficient. Tangential difference quotients, \eqref{eq: actual multiplier identity}, and normal elliptic regularity give smoothness. Tangential Sobolev estimates and \eqref{eq: normal Hardy bound} first give finite solution and coefficient norms, with constants allowed to depend on the support of the data. Equation \eqref{eq: coefficient local estimate proof}, the Euclidean version of \eqref{eq: compact Green shell bound}, and the compactness argument of \Cref{prop: compact supremum estimate} then give bounds independent of that support. Rescaling gives the product inverse for the compactly supported sources used below.

\medskip

Next, we show \eqref{eq: product parametrix tail}. Write $z\in\R^3$ for the physical normal variable and suppose the product source is supported in $|z|\leq2s$, where $s\geq\eps$. As in \eqref{eq: compact Green shell bound}, integration of the Euclidean Green kernel in the tangential variables leaves the three-dimensional bound $C|z-z'|^{-1}$. Hence a normal shell of radius $r_j\leq2s$ contributes, at $r>4s$, at most $C(r_j/r)$ times its contribution to the source norm. The multiplier satisfies
\begin{equation}\label{eq: product multiplier tail}
 \int_{\R^3}\frac{\eps^2}{|z-z'|(\eps+|z'|)^4}\,dz'\leq C\frac{\eps}{|z|}\qquad(|z|\geq4\eps).
\end{equation}
The inner cutoff commutators and exponentially decaying charged couplings satisfy \eqref{eq: product multiplier tail}, multiplied by the solution norm. Equations \eqref{eq: product inverse bound} and \eqref{eq: product multiplier tail}, with $\eps\leq s$, therefore give an $O(s/r)$ supremum bound. Equation \eqref{eq: local Schauder estimate proof} controls its derivatives and H\"older seminorms; the multiplier contributes the smaller factor $(\eps/r)^{2-\tau}$. Thus \eqref{eq: product parametrix tail} follows. Equation \eqref{eq: product inverse bound} also bounds the coefficient in $C^{0,\tau}_\eps$.

\medskip

To prove \eqref{eq: product longitudinal extensions}, we replace $(a\sigma,\sigma\vartheta)$, which is generally singular at $y=0$, by a smooth solution of the homogeneous linearised equations with the same limit at infinity. The prescribed scalar coefficient $a$ will later be obtained by freezing the longitudinal scalar component of the exterior solution at a point of $\Gamma$.
Let 
\begin{equation*}
	{\bf V}^m\coloneqq\frac{d}{dm}\bigg\vert_{m=1}(m\varphi_0(m|y|)\sigma,\alpha_0(m|y|)A_\sigma).
\end{equation*}
Differentiating \eqref{eq:monopole-profiles} at $\lambda=0$, using \eqref{eq: noncompact mass variation}, gives $\widetilde{\bf L}_0{\bf V}^{m}=0$, and \eqref{eq:monopole-ansatz} gives ${\bf G}_0^*{\bf V}^{m}=0$. Hence \eqref{eq: gauge decom Q} gives ${\bf L}_0{\bf V}^{m}=0$, while \eqref{eq: BPS} gives ${\bf V}^{m}\to(\sigma,0)$. Thus $a{\bf V}^{m}$ provides the scalar limit. For tangential one-form coefficients, \eqref{eq:monopole-profiles} at $\lambda=0$ and \eqref{eq: qlambda} have the same ODE and boundary conditions, so $q_0=\varphi_0$. Consequently, $(0,\Phi_0\,dx^j)$ solves the tangential equation in \eqref{eq: constrained product system} and tends to $(0,\sigma\,dx^j)$.

For the normal one-form coefficients, construct smooth gauge parameters $\gamma_\alpha$, $\alpha=1,2,3$, satisfying
\begin{equation*}
 (\Delta_{A_0}-\ad_{\Phi_0}^2)\gamma_\alpha=0,
 \qquad \gamma_\alpha=y^\alpha\sigma+O(|y|^{-1}).
\end{equation*}
Cutting off $y^\alpha\sigma$ near zero gives a smooth field whose image under this operator decays exponentially outside a ball. Equation \eqref{eq: normal Hardy bound} and Lax--Milgram provide a correction in the homogeneous energy space for the form $\int_{\R^3}(|\nabla_{A_0}\gamma|^2+|[\Phi_0,\gamma]|^2)\,dy$. The argument leading to \eqref{eq: normal longitudinal tails} gives the $O(|y|^{-1})$ longitudinal correction, while the comparison in \eqref{eq: normal transverse comparison} gives exponential decay of its transverse component. Elliptic estimates yield the differentiated bounds. Gauge invariance from \Cref{sec: second variation} then gives
\begin{equation*}
\begin{split}
 {\bf L}_0({\bf G}_0\gamma_\alpha)
 &={\bf G}_0(\Delta_{A_0}-\ad_{\Phi_0}^2)\gamma_\alpha=0,\\
 {\bf G}_0\gamma_\alpha
 &=([\Phi_0,\gamma_\alpha],d_{A_0}\gamma_\alpha)
 \longrightarrow(0,\sigma\,dy^\alpha).
\end{split}
\end{equation*}
Combine these solutions with the coefficients prescribed by $a$ and $\vartheta$, then subtract normal kernel modes to impose the four zero moments against ${\bf E}_i$ from \eqref{eq: normal weighted tests}. The coefficients of this subtraction are determined by the matrix
\begin{equation*}
 \left(\int_{\R^3}\langle{\bf Z}_j,{\bf E}_i\rangle\,dy\right)_{i,j=1}^4,
\end{equation*}
which is the matrix in \eqref{eq: normal moment Gram matrix}, uniformly positive by the proof of \Cref{lem: normal weighted coercivity}. By \eqref{eq: linearised kernel}, \eqref{eq: monopole translations}, \eqref{eq: monopole phase mode} and \eqref{eq: approximation profile bounds}, the resulting pair ${\bf U}_{a,\vartheta}$ satisfies \eqref{eq: constrained product system} with ${\bf f}=0$ and $c_i=0$, together with \eqref{eq: product longitudinal extensions}.
\end{proof}

\begin{proposition}[The noncompact inverse]\label[proposition]{prop: noncompact ordinary inverse}
For every ${\bf F}\in Y_{\eps,h,\theta}$, there is a unique pair $({\bf\Psi},\mathfrak b)\in\mathcal E_{\eps,h,\theta}\times C^{0,\tau}_\eps(\Gamma;\R^4)$ solving \eqref{eq: constrained linear equation}, and this pair satisfies \eqref{eq: uniform inverse estimate}.
\end{proposition}
\begin{proof}
We first take smooth compactly supported data and smooth parameters. We construct a uniformly bounded map
\begin{equation*}
 P_\eps\colon Y_{\eps,h,\theta}\longrightarrow
 \mathcal E_{\eps,h,\theta}\times C^{0,\tau}_\eps(\Gamma;\R^4),
 \qquad P_\eps{\bf F}=({\bf\Psi},\mathfrak b),
\end{equation*}
such that
\begin{equation}\label{eq: noncompact approximate inverse}
\begin{split}
 {\bf L}_{\eps,h,\theta}{\bf\Psi}
 -\mathfrak B_{\eps,h,\theta}^*\mathfrak b
 &={\bf F}+\mathcal R_\eps{\bf F},\\
 \|\mathcal R_\eps\|_{Y_{\eps,h,\theta}\to Y_{\eps,h,\theta}}
 &\lesssim (1+|\log\eps|)^N
 (\eps^{1/8}+\|h\|_{C^{2,\tau}})=o(1),
\end{split}
\end{equation}
for a fixed $N$. We construct $P_\eps$ in three steps: solve for the exterior source and extend through the core; solve for the interior source using the product inverse; restore the fibre constraints. We then remove $\mathcal R_\eps$ by a Neumann series. To record the transports explicitly, write
\begin{equation*}
 \mathscr S_\eps{\bf V}\coloneqq
 (\widehat V_\Phi,\widehat V_N+\widehat b),
\end{equation*}
using all the coefficients of \Cref{def: rescaling convention}; the normal and tangential covectors are identified with the corresponding constant covectors of the product model. Throughout the proof, cutoff errors are computed from
\begin{equation}\label{eq: noncompact cutoff commutator}
 [{\bf L}_{\eps,h,\theta},\chi]{\bf V}
 =(\Delta_0\chi){\bf V}-2\nabla_{A_{h,\theta},\nabla\chi}{\bf V},
\end{equation}
where $[{\bf L}_{\eps,h,\theta},\chi]{\bf V}={\bf L}_{\eps,h,\theta}(\chi{\bf V})-\chi{\bf L}_{\eps,h,\theta}{\bf V}$. On a cutoff annulus of radius $r_j$, the factor $r_j^2$ in \eqref{eq: source norm} compensates for the two derivatives of $\chi$; the remaining terms are controlled by the zeroth and first derivative terms of \eqref{eq: sol norm}.

\subsubsection*{The exterior source}
Decompose the source as
\begin{equation*}
 {\bf F}^{\inn}=\eta\left(\frac{r_h}{\sqrt\eps}\right){\bf F}, \qquad {\bf F}^{\out}=\left(1-\eta\left(\frac{r_h}{\sqrt\eps}\right)\right){\bf F}.
\end{equation*}
Both parts have source norm bounded by $C\|{\bf F}\|_{Y_{\eps,h,\theta}}$. Write ${\bf F}^{\out}=(F_\Phi^{\out},F_A^{\out})$. Write $\sigma_h=\Phi_{h,\theta}/|\Phi_{h,\theta}|$ on $M\setminus\Gamma_h$. For the longitudinal part, use the inverses in \Cref{def: admissible amb noncomp} to solve
\begin{equation*}
 \Delta_0\psi^{\out}=\langle F_\Phi^{\out},\sigma_h\rangle,
 \qquad \Delta_1 b^{\out}=\langle F_A^{\out},\sigma_h\rangle.
\end{equation*}
The right-hand sides vanish near $\Gamma_h$ and are extended by zero there. Equations \eqref{eq: noncompact Green kernel bound}, \eqref{eq: noncompact Green shell bound} and \eqref{eq: source norm} bound $|\psi^{\out}|+\eps|b^{\out}|$ by $C\|{\bf F}\|_{Y_{\eps,h,\theta}}$. Since these coefficients are harmonic for $r_h<\eps^{1/2}$, interior estimates give, near $\Gamma_h$,
\begin{equation}\label{eq: noncompact exterior derivative bounds}
 |\nabla^k\psi^{\out}|+\eps|\nabla^k b^{\out}|
 \lesssim \eps^{-k/2}\|{\bf F}\|_{Y_{\eps,h,\theta}},
 \qquad k=0,1,2,
\end{equation}
with the corresponding scaled H\"older bounds.

The pair $(\psi^{\out}\sigma_h,b^{\out}\sigma_h)$ need not extend smoothly through the core. Freeze its real coefficients on $\Gamma_h$:
\begin{equation*}
 a(q)\coloneqq\psi^{\out}(F_h(q,0)),\qquad
 \vartheta(q)\coloneqq\eps b^{\out}|_{F_h(q,0)}.
\end{equation*}
Here $\vartheta(q)$ is expressed in the transported orthonormal frame and identified with a constant one-form on the product. Define the smooth extension ${\bf U}^{\circ}$ by
\begin{equation*}
 (\mathscr S_\eps{\bf U}^{\circ})(q,y)
 \coloneqq{\bf U}_{a(q),\vartheta(q)}(y),
\end{equation*}
where the right-hand side is supplied by \Cref{lem: noncompact product constructions}. We join it to $(\psi^{\out}\sigma_h,b^{\out}\sigma_h)$ using $\eta\left(\frac{r_h}{\eps^{3/4}}\right)$; the cutoff error is
\begin{equation*}
 \left[{\bf L}_{\eps,h,\theta},\eta\left(\frac{r_h}{\eps^{3/4}}\right)\right]
 \left({\bf U}^{\circ}-(\psi^{\out}\sigma_h,b^{\out}\sigma_h)\right).
\end{equation*}

On the matching annulus, freezing the exterior coefficients and replacing their constant values by ${\bf U}_{a,\vartheta}$ give, respectively, the factors
\begin{equation*}
 \frac{\eps^{3/4}}{\eps^{1/2}}=\eps^{1/4},
 \qquad \frac{\eps}{\eps^{3/4}}=\eps^{1/4}.
\end{equation*}
The first follows from \eqref{eq: noncompact exterior derivative bounds}, and the second from \eqref{eq: product longitudinal extensions}; both multiply $C\|{\bf F}\|_{Y_{\eps,h,\theta}}$ and hold with the required scaled derivatives. Equation \eqref{eq: noncompact cutoff commutator} therefore gives the same bound for the matching error in the source norm. Equation \eqref{eq: constrained product system} cancels the leading normal terms in the core; tangential second derivatives of the frozen coefficients contribute the smaller factor $\eps^{3/2}\eps^{-1}=\eps^{1/2}$. The metric and background errors satisfy \eqref{eq: noncompact parametrix coefficient errors} below.

For the transverse part of ${\bf F}^{\out}$, take the Dirichlet inverse on $r_h>\eps^{1/2}/2$, and denote its solution by ${\bf V}^{\out,\perp}$. The positive zeroth-order coefficient, bounded below by $c\eps^{-2}$, gives existence on bounded domains and then by exhaustion, with the weighted bounds from the proof of \Cref{lem: noncompact local remainder} and \eqref{eq: parameter Schauder estimate}. The source vanishes for $r_h<\eps^{1/2}$, so the solution and its scaled derivatives are exponentially small in a region separated from the source by a fixed multiple of $\eps^{1/2}$. Choose a smooth cutoff $\chi_{\out}$ with
\begin{equation*}
\begin{gathered}
 \chi_{\out}=0\quad\text{for }r_h\leq\frac{\sqrt\eps}{2},\qquad
 \chi_{\out}=1\quad\text{for }r_h\geq\frac{3\sqrt\eps}{4},\\
 |\nabla^k\chi_{\out}|\lesssim\eps^{-k/2}\quad(k=1,2).
\end{gathered}
\end{equation*}
Extending $\chi_{\out}{\bf V}^{\out,\perp}$ by zero gives the global outer solution
\begin{equation}\label{eq: noncompact global outer solution}
\begin{split}
 {\bf\Psi}^{\out}
 ={}&\eta\left(\frac{r_h}{\eps^{3/4}}\right){\bf U}^{\circ}
 +\left(1-\eta\left(\frac{r_h}{\eps^{3/4}}\right)\right)
 (\psi^{\out}\sigma_h,b^{\out}\sigma_h)\\
 &+\chi_{\out}{\bf V}^{\out,\perp}.
\end{split}
\end{equation}
The transverse cutoff error is $[{\bf L}_{\eps,h,\theta},\chi_{\out}]{\bf V}^{\out,\perp}$, which is exponentially small by \eqref{eq: noncompact cutoff commutator}.

\subsubsection*{The interior source and the matching error}
Choose a partition of unity $(\chi_\ell)$ on tangential patches of radius $\eps^{1/8}$, with bounded overlap, and larger cutoffs $\widetilde\chi_\ell$. Extend them constantly in the normal variable and set
\begin{equation*}
 \sum_\ell\chi_\ell=1,\qquad
 \widetilde\chi_\ell=1\text{ on }\supp\chi_\ell,\qquad
 {\bf F}_\ell\coloneqq\chi_\ell{\bf F}^{\inn},\qquad
 \sum_\ell{\bf F}_\ell={\bf F}^{\inn}.
\end{equation*}
Their derivatives of order $k$ are bounded by $C\eps^{-k/8}$. At the centre $q_\ell$, put
\begin{equation*}
 p_\ell(x)\coloneqq\exp_{q_\ell}^\Gamma(\eps x),\qquad
 (\mathscr S_\ell{\bf V})(x,y)\coloneqq
 (\mathscr S_\eps{\bf V})(p_\ell(x),y).
\end{equation*}
Freeze the tangential metric at $q_\ell$. Apply \Cref{lem: noncompact product constructions} with source $\eps^2\mathscr S_\ell{\bf F}_\ell$, and denote the resulting product solution and coefficient by $({\bf U}_\ell,c_\ell)$. Their physical counterparts and the assembled interior field are
\begin{equation}\label{eq: noncompact interior assembly}
\begin{aligned}
 {\bf V}_\ell&\coloneqq\mathscr S_\ell^{-1}{\bf U}_\ell,
 &\mathfrak b_\ell(p_\ell(x))&\coloneqq c_\ell(x),\\
 \eta_{\inn}&\coloneqq\eta\left(\frac{r_h}{\eps^{1/4}}\right),
 &{\bf\Psi}^{\inn}&\coloneqq\eta_{\inn}\sum_\ell\widetilde\chi_\ell{\bf V}_\ell.
\end{aligned}
\end{equation}
Each summand is extended by zero off its patch. Since $\eta_{\inn}\widetilde\chi_\ell=1$ on $\supp{\bf F}_\ell$ for small $\eps$, the source terms in \eqref{eq: noncompact interior assembly} sum to ${\bf F}^{\inn}$. The remaining errors come from the cutoffs and the difference between the frozen product operator and the actual operator.

The source is supported in $r_h\leq2\eps^{1/2}$, whereas the radial cutoff varies at $r_h\simeq\eps^{1/4}$. Apply \eqref{eq: product parametrix tail} with $s=\eps^{1/2}$ and $r\simeq\eps^{1/4}$, then use \eqref{eq: noncompact cutoff commutator}. The radial cutoff error in the source norm is bounded by $C(\eps^{1/2}/\eps^{1/4})^{2/3}\|{\bf F}\|_{Y_{\eps,h,\theta}}\lesssim \eps^{1/6}\|{\bf F}\|_{Y_{\eps,h,\theta}}$.

For tangential cutoffs, the first term in \eqref{eq: noncompact cutoff commutator} contributes $r_j^2\eps^{-1/4}$ on $\Omega_{j,h}$, and the second contributes $r_j\eps^{-1/8}$. Thus their total source norm is bounded by $\|{\bf F}\|_{Y_{\eps,h,\theta}}$ times
\begin{equation}\label{eq: noncompact tangential cutoff error}
 C\sum_{r_j\leq C\eps^{1/4}}\left(\frac{r_j}{\eps^{1/8}}+\frac{r_j^2}{\eps^{1/4}}\right)\lesssim \eps^{1/8}.
\end{equation}
On each patch, the tangential metric differs from its value at the centre by $O(\eps^{1/8})$. Multiplication by $r_j^2$ bounds its contribution on each shell by $C\eps^{1/8}$ times the solution norm; summing introduces a factor of $C(1+|\log\eps|)$. For both the interior solution and the exterior core extension, the remaining Fermi metric, background and graph errors are estimated as in \eqref{eq: coefficient local estimate proof}. Their source norms are bounded by
\begin{equation}\label{eq: noncompact parametrix coefficient errors}
 C\bigl(\eps^{1/4}+\eps+\|h\|_{C^{2,\tau}}\bigr)(1+|\log\eps|)^N\|{\bf F}\|_{Y_{\eps,h,\theta}}
\end{equation}
for a fixed $N$. Only two derivatives of the parameters enter.

\subsubsection*{The fibre constraints and conclusion}
Using \eqref{eq: noncompact global outer solution} and \eqref{eq: noncompact interior assembly}, define the field before imposing the constraints and its residual coefficient by
\begin{equation}\label{eq: noncompact preliminary assembly}
 {\bf\Psi}^{\circ}\coloneqq{\bf\Psi}^{\out}+{\bf\Psi}^{\inn},\qquad
 \mathfrak b\coloneqq\sum_\ell\widetilde\chi_\ell\mathfrak b_\ell.
\end{equation}
The product solutions have zero model moments, but cutting them off and transporting them need not preserve the actual moments. Record their defect as the vector $\mathfrak{m}=(\mathfrak{m}_1,\ldots,\mathfrak{m}_4)^T$, with
\begin{equation}\label{eq: noncompact moment defect}
 \mathfrak{m}_j(q)\coloneqq\mathfrak b_j[{\bf\Psi}^{\circ}](q)
 =\int_{\R^3}\langle\widehat{{\bf\Psi}^{\circ}}(q,y),{\bf E}_j(y)\rangle\,dy,
 \qquad j=1,\ldots,4,
\end{equation}
where ${\bf E}_j={\bf E}_{j,\delta/(4\eps)}$ and the equality follows from \eqref{eq: fibre constraints}. Indeed, the prefactor $\eps^{-2}$ in \eqref{eq: fibre constraints}, the factor $\eps^{-1}$ in the weighted pairing and the normal volume factor $\eps^3$ cancel. Use the right inverse \eqref{eq: fibre constraint right inverse} to set
\begin{equation}\label{eq: noncompact moment correction}
 {\bf\Psi}\coloneqq{\bf\Psi}^{\circ}
 -\mathscr J_{\eps,h,\theta}(\mathfrak{m}),\qquad
 \mathfrak B_{\eps,h,\theta}{\bf\Psi}=0.
\end{equation}
To estimate this subtraction, compare \eqref{eq: noncompact moment defect} with the zero product moments in \eqref{eq: constrained product system}. The outer and inner radial cutoffs remove tails starting, respectively, at
\begin{equation*}
 |y|\simeq\frac{\eps^{3/4}}{\eps}=\eps^{-1/4},\qquad
 |y|\simeq\frac{\eps^{1/4}}{\eps}=\eps^{-3/4}.
\end{equation*}
These tails are small by \eqref{eq: normal moment tail}, \eqref{eq: product parametrix tail} and the mode bounds \eqref{eq: monopole translations}, \eqref{eq: monopole phase mode}. The remaining differences are the coefficient errors in \eqref{eq: noncompact parametrix coefficient errors}. The same estimates apply after tangential differentiation: the frozen outer coefficients satisfy \eqref{eq: noncompact exterior derivative bounds}, so their second derivatives contribute $\eps^2\eps^{-1}=\eps$ in the rescaled operator; the interior cutoffs satisfy \eqref{eq: noncompact tangential cutoff error}.

The explicit construction of $\mathscr J$ following \eqref{eq: fibre constraint right inverse} uses four cut-off modes and the inverse of their uniformly positive moment matrix. Thus $\mathscr J(\mathfrak{m})$ has rescaled coefficients of size $\mathfrak{m}$, and applying the operator uses at most two scaled tangential derivatives of $\mathfrak{m}$:
\begin{equation}\label{eq: noncompact moment correction bound}
\begin{split}
 &\|\mathscr J_{\eps,h,\theta}(\mathfrak{m})\|_{X_{\eps,h,\theta}}
 +\|{\bf L}_{\eps,h,\theta}\mathscr J_{\eps,h,\theta}(\mathfrak{m})\|_{Y_{\eps,h,\theta}}\lesssim\sum_{k=0}^2\eps^k\|\nabla_\Gamma^k \mathfrak{m}\|_{C^{0,\tau}_\eps}.
\end{split}
\end{equation}
The preceding moment estimates bound the right-hand side by the same small factors as the cutoff and coefficient errors.

Equations \eqref{eq: noncompact preliminary assembly}, \eqref{eq: noncompact moment correction}, \eqref{eq: product inverse bound} and \eqref{eq: product longitudinal extensions} now give a uniformly bounded map $P_\eps{\bf F}=({\bf\Psi},\mathfrak b)$. Its error has the exact decomposition
\begin{equation}\label{eq: noncompact residual decomposition}
\begin{split}
 \mathcal R_\eps{\bf F}
 &=({\bf L}_{\eps,h,\theta}{\bf\Psi}^{\out}-{\bf F}^{\out})\\
 &+\sum_\ell\left({\bf L}_{\eps,h,\theta}
 (\eta_{\inn}\widetilde\chi_\ell{\bf V}_\ell)-{\bf F}_\ell
 -\mathfrak B_{\eps,h,\theta}^*(\widetilde\chi_\ell\mathfrak b_\ell)\right)\\
 &-{\bf L}_{\eps,h,\theta}\mathscr J_{\eps,h,\theta}(\mathfrak{m}).
\end{split}
\end{equation}
The three lines are the outer matching error, the localised product errors (including the cutoff of the multiplier tails), and the moment-correction error. The outer matching gives $\eps^{1/4}$ by \eqref{eq: noncompact exterior derivative bounds} and \eqref{eq: product longitudinal extensions}; the inner radial cutoff gives $\eps^{1/6}$ by \eqref{eq: product parametrix tail}; tangential cutoffs give $\eps^{1/8}$ by \eqref{eq: noncompact tangential cutoff error}. Adding \eqref{eq: noncompact parametrix coefficient errors} and \eqref{eq: noncompact moment correction bound} in \eqref{eq: noncompact residual decomposition} yields
\begin{equation*}
\begin{split}
 \|\mathcal R_\eps\|_{Y_{\eps,h,\theta}\to Y_{\eps,h,\theta}} &\lesssim (1+|\log\eps|)^N\left(\eps^{1/4}+\eps^{1/6}+\eps^{1/8}+\eps+\|h\|_{C^{2,\tau}}\right)\\
 &\lesssim (1+|\log\eps|)^N\left(\eps^{1/8}+\|h\|_{C^{2,\tau}}\right).
\end{split}
\end{equation*}
The inverses in \Cref{def: admissible amb noncomp} and \Cref{lem: noncompact local remainder} control the weighted exterior norm \eqref{eq: noncompact exterior solution norm}. This proves \eqref{eq: noncompact approximate inverse}. Choose $\eps$ small enough that $\|\mathcal R_\eps\|<1/2$ and define
\begin{equation*}
 ({\bf\Psi}_{\bf F},\mathfrak b_{\bf F})
 \coloneqq P_\eps(I+\mathcal R_\eps)^{-1}{\bf F},\qquad
 (I+\mathcal R_\eps)^{-1}=\sum_{k=0}^{\infty}(-\mathcal R_\eps)^k.
\end{equation*}
By \eqref{eq: noncompact approximate inverse}, this pair solves \eqref{eq: constrained linear equation} exactly; by \eqref{eq: noncompact moment correction}, it satisfies the fibre constraints. The a priori estimate of \Cref{prop: noncompact a priori estimate} gives uniqueness and the uniform bound. Smooth approximation on compact subsets and lower semicontinuity give the result for the stated data and parameters, as in the compact proof of \Cref{prop: uniform inverse}.
\end{proof}

\begin{corollary}[The restricted noncompact inverse]\label[corollary]{cor: noncompact restricted inverse}
For every ${\bf\Psi}\in X_{\eps,h,\theta}$,
\begin{equation}\label{eq: noncompact restricted perturbation}
 \|\eps^{-1}{\bf G}_{h,\theta}\Pi_h{\bf G}_{h,\theta}^*{\bf\Psi}\|_{Y_{\eps,h,\theta}}
 \lesssim \eps(1+|\log\eps|)\|{\bf\Psi}\|_{X_{\eps,h,\theta}}.
\end{equation}
Consequently, $\widehat{\bf L}_{h,\theta}$ admits the constrained inverse and uniform estimate asserted in \Cref{prop: uniform inverse}.
\end{corollary}
\begin{proof}
The proof is the same as for \Cref{cor: compact restricted inverse}, except that $d\chi$ now contributes on a fixed compact set. By \eqref{eq: compact phase generator bounds}, \eqref{eq: noncompact shifted approximation} and \eqref{eq: noncompact exterior source norm}, the phase generator satisfies
\begin{equation}\label{eq: noncompact phase generator bounds}
\begin{split}
 \|{\bf G}_{h,\theta}\xi_{\eps,h}\|_{Y_{\eps,h,\theta}}&\lesssim \eps^2(1+|\log\eps|),\\
 \int_{\supp\chi}|d_{A_{h,\theta}}\xi_{\eps,h}|\,d\vol_g&\lesssim \eps.
\end{split}
\end{equation}
The denominator $\int_M\chi|\xi_{\eps,h}|^2\,d\vol_g$ is bounded below. Since $[\Phi_{h,\theta},\xi_{\eps,h}]=0$ and $d\chi$ has fixed compact support,
\begin{equation}\label{eq: noncompact gauge projection numerator bound}
 |\langle{\bf\Psi},{\bf G}_{h,\theta}(\chi\xi_{\eps,h})\rangle_{L^2_\eps(M)}|\lesssim \|{\bf\Psi}\|_{X_{\eps,h,\theta}}.
\end{equation}
Equations \eqref{eq: noncompact phase generator bounds} and \eqref{eq: noncompact gauge projection numerator bound} prove \eqref{eq: noncompact restricted perturbation} by \eqref{eq: phase projection}. The Neumann-series argument of \Cref{cor: compact restricted inverse} now applies to \Cref{prop: noncompact ordinary inverse} and gives the restricted inverse.
\end{proof}
Together, \Cref{prop: noncompact ordinary inverse} and \Cref{cor: noncompact restricted inverse} complete the proof of \Cref{prop: uniform inverse} on noncompact $M$.

\section{The reduced equation}\label{sec: reduced equation proof}
This section establishes \Cref{prop: reduced equation} by identifying the reduced operator and estimating its remainder and parameter dependence. We proceed as follows.
\begin{itemize}
 \item Lemmas~\ref{lem: reduced adjoint tests} and~\ref{lem: reduced coefficient operators} construct the adjoint tests and coefficient operators, with the H\"older bounds needed to recover the residual coefficients.
 \item \Cref{lem: reduced matching} gives the tangential matching estimates needed to project the approximation error.
 \item Lemmas~\ref{lem: reduced source size} and~\ref{lem: reduced phase source size} identify the Jacobi and Laplace terms and bound the translation and phase remainders of the approximation.
 \item \Cref{lem: reduced complementary size} bounds the additional remainder produced by the nonlinear correction.
 \item Lemmas~\ref{lem: reduced source differences} and~\ref{lem: reduced complementary differences} control both remainders under changes of $(h,\theta)$, yielding the Lipschitz bounds for the reduced contraction.
\end{itemize}

\medskip

We take $\lambda=\lambda_\eps$ in the compact case and $\lambda=0$ in the noncompact case. Write $\beta^{\mathrm{tan}}=\beta_jdq^j$ for the tangential trace of $\beta$.

\subsection{The coefficient of the product equation}
Our goal is to express the four coefficients in \eqref{eq: complementary nonlinear system} in terms of the source, so that the vanishing of the translation and mean-zero phase coefficients gives the reduced equation; the remaining constant phase coefficient is removed by \Cref{lem: constant phase residual}. We first recover them for a reference problem over $\Gamma$.

\medskip

At this point a distinction arises from the Allen--Cahn or abelian Higgs constructions \cite{Pacard-Ritore2003,BadranDelPino23,BadranDelPino24}. In those settings, the normal translation modes decay exponentially, allowing the reduced coefficients to be extracted by projection onto the kernel, with exponentially small errors from the normal cutoffs. Here the monopole modes have algebraic tails, and our global correction satisfies the weighted fibre constraints \eqref{eq: fibre constraints}. Orthogonality to the weighted tests does not imply orthogonality to the kernel modes themselves, so direct projection leaves tangential derivatives of unknown moments of the correction. We therefore recover the coefficients through adjoint tests adapted to the weighted constraints and the tangential spectrum, see \eqref{eq: reduced resolvent tests}. 

\medskip

Use the fixed normal trivialisation to identify $N\Gamma$ with $\Gamma\times\R^3$, with variables $(q,y)$. Let $\mathcal P_t$ be the transport of normal pairs induced by normal parallel transport along a curve $\gamma$ in $\Gamma$, rotating the normal variable, normal covectors and Lie algebra components simultaneously. For $\gamma(0)=q$ and $\dot\gamma(0)=X$, define
\begin{equation*}
 (\nabla^\top_X{\bf V})(q,\cdot)
 \coloneqq\left.\frac{d}{dt}\right|_{t=0}
 \mathcal P_t^{-1}{\bf V}(\gamma(t),\cdot).
\end{equation*}
Use ${\bf E}_a={\bf E}_{a,\delta/(4\eps)}$ from \eqref{eq: normal weighted tests}, evaluated at $\lambda$. In particular, for $s=s^a\nu_a\in C^\infty(N\Gamma)$ and $\phi\in C^\infty(\Gamma)$, equivariance gives
\begin{equation}\label{eq: reduced test equivariance}
 \nabla^\top_X\left(s^a{\bf E}_a\right)
 =(\nabla^\perp_Xs)^a{\bf E}_a,
 \qquad \nabla^\top_X(\phi{\bf E}_4)=X(\phi){\bf E}_4.
\end{equation}
Define
\begin{equation}\label{eq: reduced product operators}
 \Delta_\Gamma^\perp=(\nabla^\perp)^*\nabla^\perp,
 \qquad J_\Gamma=-\Delta_\Gamma^\perp+V_\Gamma,
 \qquad {\bf L}^{\mathrm{prod}}_{\eps,\lambda}
 ={\bf L}_\lambda+\eps^2(\nabla^\top)^*\nabla^\top,
\end{equation}
where $V_\Gamma$ is the zeroth-order part of the Jacobi operator. Write $\langle\cdot,\cdot\rangle_2$ for the normal $L^2$ pairing. For a normal source pair ${\bf f}(q,y)$, the reference system is
\begin{equation}\label{eq: reduced reference system}
\begin{cases}
 {\bf L}^{\mathrm{prod}}_{\eps,\lambda}{\bf V}
 =-{\bf f}+\displaystyle\sum_{a=1}^4c_a(q){\bf E}_a(y),\\[2mm]
 \langle{\bf E}_a,{\bf V}(q,\cdot)\rangle_2=0,
 \qquad a=1,\ldots,4.
\end{cases}
\end{equation}
The unknown coefficients form a section $c=(c^{\mathrm{tr}},c_4)$ of $N\Gamma\oplus\R$, with $c^{\mathrm{tr}}=c_a\nu_a$.

To recover $c$, take an eigensection $s$ of $\Delta_\Gamma^\perp$ or an eigenfunction $s$ of $\Delta_\Gamma$, with eigenvalue $\mu$. Regard it as a section of $N\Gamma\oplus\R$ by setting its other components to zero, and put $t=\eps^2\mu$. Rotation equivariance of ${\bf L}_\lambda$ shows that it commutes with normal parallel transport. Thus, for $t>0$, the tests
\begin{equation}\label{eq: reduced resolvent tests}
 {\bf u}_a(t)=({\bf L}_\lambda+t)^{-1}{\bf E}_a,
 \qquad {\bf u}_{s,t}(q,y)=\sum_{a=1}^4s^a(q){\bf u}_a(t,y)
\end{equation}
satisfy ${\bf L}^{\mathrm{prod}}_{\eps,\lambda}{\bf u}_{s,t}=\sum_{a=1}^4s^a{\bf E}_a$ by \eqref{eq: reduced test equivariance} and \eqref{eq: reduced product operators}. This explains the shifted normal operator ${\bf L}_\lambda+t$. Testing \eqref{eq: reduced reference system} against \eqref{eq: reduced resolvent tests}, self-adjointness and the zero fibre moments give
\begin{equation}\label{eq: reduced testing identity}
0=\langle{\bf u}_{s,t}, {\bf L}^{\mathrm{prod}}_{\eps,\lambda}{\bf V}\rangle_{L^2(\Gamma\times\R^3)}=-\langle{\bf u}_{s,t},{\bf f}\rangle_{L^2(\Gamma\times\R^3)} +\int_\Gamma\sum_{a,b=1}^4s^a c_b \langle{\bf u}_a(t),{\bf E}_b\rangle_2\,d\mathrm{vol}_\Gamma.
\end{equation}
Rotational symmetry makes the normal pairing matrix diagonal, with equal first three entries. Normalising these entries therefore recovers the spectral coefficients of $c$. Set $T_t=t({\bf L}_\lambda+t)^{-1}$, whose factor $t$ removes the $t^{-1}$ growth on the kernel, and define
\begin{equation}\label{eq: reduced adjoint symbols}
Q_a(t)=\langle{\bf E}_a,T_t{\bf E}_a\rangle_2,\qquad
{\bf J}_{a,t}=\frac{T_t{\bf E}_a}{Q_a(t)},\qquad
F_a(t)=\frac{\langle{\bf E}_a,{\bf Z}_a\rangle_2}{Q_a(t)}.
\end{equation}
Here $Q_a$ ensures $\langle{\bf E}_a,{\bf J}_{a,t}\rangle_2=1$. Dividing \eqref{eq: reduced testing identity} by the corresponding diagonal entry gives
\begin{equation}\label{eq: reduced spectral recovery}
 \langle c,s\rangle_{L^2(\Gamma)}
 =\int_\Gamma\sum_{a=1}^4s^a(q)
 \langle{\bf J}_{a,\eps^2\mu},{\bf f}(q,\cdot)\rangle_2\,d\mathrm{vol}_\Gamma.
\end{equation}
By \eqref{eq: linearised kernel}, $T_t{\bf Z}_a={\bf Z}_a$, so the function $F_a(t)$ gives the coefficient produced by a source with normal profile ${\bf Z}_a$. Write $F_{\mathrm{tr}}=F_1=F_2=F_3$ and $F_{\mathrm{ph}}=F_4$. For $\mu=0$, testing with $\sum_a s^a{\bf Z}_a$ gives \eqref{eq: reduced spectral recovery} with
\begin{equation}\label{eq: reduced zero spectral tests}
 {\bf J}_{a,0}=\frac{{\bf Z}_a}{\langle{\bf E}_a,{\bf Z}_a\rangle_2},
 \qquad F_a(0)=\frac{\|{\bf Z}_a\|_2^2}{\langle{\bf E}_a,{\bf Z}_a\rangle_2}.
\end{equation}

To recover the whole section $c$, choose a real $L^2$-orthonormal eigenbasis $(s_k)_{k\geq0}$ of $\Delta_\Gamma^\perp\oplus\Delta_\Gamma$, with each $s_k$ belonging to one of the two summands and having eigenvalue $\mu_k$. Then \eqref{eq: reduced spectral recovery} gives
\begin{equation*}
\begin{split}
 c&=\sum_{k\geq0}c_k s_k\quad\text{in }L^2(\Gamma;N\Gamma\oplus\R),\\
 c_k&=\int_\Gamma\sum_{a=1}^4s_k^a(q)
 \langle{\bf J}_{a,\eps^2\mu_k},{\bf f}(q,\cdot)\rangle_2\,d\mathrm{vol}_\Gamma(q).
\end{split}
\end{equation*}
Thus the full coefficient is reconstructed using the tests at every spectral value $t=\eps^2\mu_k$, including $t=0$ when $\mu_k=0$.

 This reconstruction is exact for \eqref{eq: reduced reference system}; applying it to the geometric equation requires including the difference of the operators and coefficient injections in the source.

\smallskip

In what follows, we write $\langle y\rangle\coloneqq (1+|y|^2)^{1/2}$.

\begin{lemma}[Normal adjoint tests]\label{lem: reduced adjoint tests}
For $t>0$, the functions $Q_i(t)$ and $F_i(t)$ are bounded above and below by positive constants, uniformly in $t,\eps,\lambda$. For every fixed integer $j\geq0$, the adjoint tests satisfy
\begin{equation}\label{eq: reduced adjoint estimates}
|(t\partial_t)^j{\bf J}_{i,t}(y)|\lesssim _j\langle y\rangle^{-2},\qquad \|(t\partial_t)^j({\bf J}_{i,t}-{\bf J}_{i,0})\|_2 \lesssim _j\min(t^{1/4},1).
\end{equation}
The logarithmic derivatives $(t\partial_t)^jQ_i^{-1}$ and $(t\partial_t)^jF_i^{\pm1}$ are uniformly bounded as well.
\end{lemma}
\begin{proof}
The nonnegativity established in the proof of \Cref{lem: normal weighted coercivity}, with \eqref{eq: linearised kernel} and \eqref{eq: normal moment Gram matrix}, gives
\begin{equation}\label{eq: reduced adjoint symbol bounds}
0<c\leq\frac{\langle {\bf E}_i,{\bf Z}_i\rangle_2^2}{\|{\bf Z}_i\|_2^2}\leq Q_i(t)\leq\|{\bf E}_i\|_2^2\leq C.
\end{equation}
Equations \eqref{eq: reduced adjoint symbols} and \eqref{eq: reduced adjoint symbol bounds} show that $Q_i$ and $F_i$ are bounded above and below by positive constants independent of $t,\eps$, and $\lambda$ in the stated range.

To prove \eqref{eq: reduced adjoint estimates}, apply \Cref{lem: normal weighted coercivity} as follows. Subtract the kernel projection from ${\bf E}_i$, obtaining $g\perp\ker{\bf L}_\lambda$ with $|g|\lesssim \langle y\rangle^{-2}$. Truncate at radius $R$ and restore orthogonality using four fixed compactly supported tests whose pairings with the modes form an invertible matrix. The resulting $g_R$ satisfies
\begin{equation}\label{eq: reduced truncated forcing bounds}
\|g-g_R\|_2\lesssim R^{-1/2},\qquad \|\langle y\rangle g_R\|_2\lesssim  R^{1/2}.
\end{equation}
Testing the resolvent equation and using \eqref{eq: reduced truncated forcing bounds} and \Cref{lem: normal weighted coercivity} gives
\begin{equation}\label{eq: reduced resolvent L2 bound}
\|T_tg\|_2\lesssim \sqrt{tR}+R^{-1/2}.
\end{equation}
Taking $R=t^{-1/2}$ in \eqref{eq: reduced resolvent L2 bound} proves the second estimate in \eqref{eq: reduced adjoint estimates} for $j=0$.

\smallskip

For the pointwise estimate, set ${\bf u}_t=T_t{\bf E}_i$ and ${\bf k}_i={\bf E}_i-g$, the kernel component of ${\bf E}_i$. Write ${\bf L}_\lambda=\nabla_{A_\lambda}^*\nabla_{A_\lambda}+\mathcal V_\lambda$, where the normal Weitzenb\"ock potential satisfies $\langle\mathcal V_\lambda{\bf u},{\bf u}\rangle\geq-C\langle y\rangle^{-2}|{\bf u}|^2$. Wherever ${\bf u}_t\neq0$,
\begin{equation*}
\Delta_0|{\bf u}_t| \leq\frac{\langle\nabla_{A_\lambda}^*\nabla_{A_\lambda}{\bf u}_t,{\bf u}_t\rangle}{|{\bf u}_t|}\leq t|{\bf E}_i|-(t-C\langle y\rangle^{-2})|{\bf u}_t|.
\end{equation*}
The first inequality follows from \eqref{eq: Kato}. Regularising the norm by $(|{\bf u}_t|^2+\eta^2)^{1/2}$ and letting $\eta\downarrow0$ gives the distributional inequality on all of $\R^3$:
\begin{equation}\label{eq: reduced resolvent norm inequality}
 (\Delta_0+t-C\langle y\rangle^{-2})|{\bf u}_t|
 \leq t|{\bf E}_i|.
\end{equation}
Let $r\coloneqq |y|\geq2$. This fixed threshold separates the core from the region where we use radial decay: on $B_{r/4}(y)$, we have $\langle y'\rangle\simeq r$; the ball $|y|<2$ is treated by the interior estimate below. Since ${\bf u}_t={\bf k}_i+T_tg$, \eqref{eq: monopole translations}, \eqref{eq: monopole phase mode}, \eqref{eq: normal profile bounds} and \eqref{eq: reduced resolvent L2 bound} give $\|{\bf u}_t\|_{L^2(B_{r/4}(y))}\lesssim  r^{-1/2}+t^{1/4}$. Dropping the nonnegative term $t|{\bf u}_t|$ in \eqref{eq: reduced resolvent norm inequality}, the scaled local mean-value estimate therefore yields
\begin{equation}\label{eq: reduced local adjoint bound}
\begin{split}
 |{\bf u}_t(y)|
 &\lesssim \left(r^{-3/2}\|{\bf u}_t\|_{L^2(B_{r/4}(y))}
 +tr^2\|{\bf E}_i\|_{L^\infty(B_{r/4}(y))}\right)\\
 &\lesssim \left(r^{-2}+r^{-3/2}t^{1/4}+t\right)
 \lesssim _A r^{-2},\qquad 2\leq r\leq A/\sqrt t.
\end{split}
\end{equation}
The last inequality uses $r\sqrt t\leq A$ and $|{\bf E}_i|\leq C\langle y\rangle^{-2}$.

On $r\geq\max(2,A/\sqrt t)$, choose $A^2\geq2(C+2)$. Then
\begin{equation}\label{eq: reduced adjoint barrier}
 (\Delta_0+t-Cr^{-2})r^{-2}
 =tr^{-2}-(C+2)r^{-4}\geq\tfrac12tr^{-2}.
\end{equation}
By \eqref{eq: reduced adjoint barrier} and \eqref{eq: reduced local adjoint bound}, a sufficiently large multiple $Br^{-2}$ dominates the source $t|{\bf E}_i|$ and the boundary values.  On the fixed core, bounded $t$ is handled by the local estimate and $\|{\bf u}_t\|_2\leq\|{\bf E}_i\|_2$. For sufficiently large $t$, the global barrier $\langle y\rangle^{-2}$ satisfies $(\Delta_0+t-C\langle y\rangle^{-2})\langle y\rangle^{-2}\geq(t/2)\langle y\rangle^{-2}$, giving the same bound by comparison.

For the spectral derivatives, set ${\bf u}_{m,t}=T_t^m{\bf E}_i$ for integers $m\geq0$. Since $T_t$ is an $L^2$ contraction fixing the kernel,
\begin{equation}\label{eq: reduced iterated resolvent bounds}
 ( {\bf L}_\lambda+t){\bf u}_{m,t}=t{\bf u}_{m-1,t},
 \qquad \|{\bf u}_{m,t}-{\bf k}_i\|_2\leq\|T_tg\|_2,
 \qquad m\geq1.
\end{equation}
Equations \eqref{eq: reduced local adjoint bound}, \eqref{eq: reduced adjoint barrier} and \eqref{eq: reduced iterated resolvent bounds} apply inductively, giving $|{\bf u}_{m,t}(y)|\leq C_m\langle y\rangle^{-2}$. Finally,
\begin{equation}\label{eq: reduced resolvent polynomial identities}
\begin{split}
 (t\partial_t)^jT_t&=p_j(T_t),\\
 p_0(x)&=x,\qquad p_{j+1}(x)=x(1-x)p_j'(x).
\end{split}
\end{equation}
Each $p_j$ is divisible by $x$, and $p_j(1)=0$ for $j\geq1$. Thus \eqref{eq: reduced iterated resolvent bounds} and \eqref{eq: reduced resolvent polynomial identities} give the pointwise estimates, while $\|p_j(T_t)g\|_2\lesssim _j\|T_tg\|_2$ gives the $L^2$ estimates. Using \eqref{eq: reduced adjoint symbols} and \eqref{eq: reduced adjoint symbol bounds}, the product and quotient rules yield \eqref{eq: reduced adjoint estimates} and the asserted bounds for $Q_i^{-1}$ and $F_i^{\pm1}$.
\end{proof}

The reconstruction defines an operator $\mathcal C{\bf f}=(\mathcal C_{\mathrm{tr}}{\bf f},\mathcal C_4{\bf f})$ taking the normal source pair ${\bf f}(q,y)$ to the coefficient section $c(q)$ in the reference system. More explicitly, choose real $L^2$-orthonormal eigenbases $(s_k)$ of $\Delta_\Gamma^\perp$ and $(\varphi_\ell)$ of $\Delta_\Gamma$, with eigenvalues $\mu_k$ and $\kappa_\ell$, respectively. Writing $s_k=s_k^a\nu_a$, \eqref{eq: reduced spectral recovery} gives
\begin{equation}\label{eq: reduced coefficient reconstruction}
\begin{split}
 (\mathcal C_{\mathrm{tr}}{\bf f})(q)
 &=\sum_{k\geq0}s_k(q)\int_\Gamma s_k^a(q')
 \langle{\bf J}_{a,\eps^2\mu_k},{\bf f}(q',\cdot)\rangle_2\,d\mathrm{vol}_\Gamma(q'),\\
 (\mathcal C_4{\bf f})(q)
 &=\sum_{\ell\geq0}\varphi_\ell(q)\int_\Gamma\varphi_\ell(q')
 \langle{\bf J}_{4,\eps^2\kappa_\ell},{\bf f}(q',\cdot)\rangle_2\,d\mathrm{vol}_\Gamma(q').
\end{split}
\end{equation}
Thus $c^{\mathrm{tr}}=\mathcal C_{\mathrm{tr}}{\bf f}$ is a section of $N\Gamma$, and $c_4=\mathcal C_4{\bf f}$ is a real function on $\Gamma$. In the fixed normal frame, define the three scalar operators $\mathcal C_i$, $i=1,2,3$, by $\mathcal C_{\mathrm{tr}}{\bf f}=\sum_{i=1}^3(\mathcal C_i{\bf f})\nu_i$.

The spectral multipliers below act directly on coefficient fields: for $g=\sum_{k\geq0}g_ks_k$ and $\psi=\sum_{\ell\geq0}\psi_\ell\varphi_\ell$, define
\begin{equation}\label{eq: reduced spectral multipliers}
\begin{split}
 F_{\mathrm{tr}}(\eps^2\Delta_\Gamma^\perp)^{\pm1}g
 &=\sum_{k\geq0}F_{\mathrm{tr}}(\eps^2\mu_k)^{\pm1}g_ks_k,\\
 F_{\mathrm{ph}}(\eps^2\Delta_\Gamma)^{\pm1}\psi
 &=\sum_{\ell\geq0}F_{\mathrm{ph}}(\eps^2\kappa_\ell)^{\pm1}\psi_\ell\varphi_\ell.
\end{split}
\end{equation}
Equations \eqref{eq: reduced coefficient reconstruction} and \eqref{eq: reduced spectral multipliers} give
\begin{equation}\label{eq: reduced kernel source coefficients}
 \mathcal C\left(g^a{\bf Z}_a+\psi{\bf Z}_4\right)
 =\left(F_{\mathrm{tr}}(\eps^2\Delta_\Gamma^\perp)g,
 F_{\mathrm{ph}}(\eps^2\Delta_\Gamma)\psi\right),
 \qquad g=g^a\nu_a,
\end{equation}
because $\langle{\bf J}_{a,t},{\bf Z}_b\rangle_2=\delta_{ab}F_a(t)$. For $i=1,\ldots,4$, the comparison operator $\mathcal C_i(0){\bf f}(q)=\langle{\bf J}_{i,0},{\bf f}(q,\cdot)\rangle_2$ instead uses the zero-parameter test at every $q$; it retains all tangential modes. The next lemma proves the H\"older bounds for these operators. Rotation equivariance makes $\mathcal C_{\mathrm{tr}}$ independent of the chosen frame.

\begin{lemma}[Coefficient operators]\label{lem: reduced coefficient operators}
The operators $F_{\mathrm{tr}}(\eps^2\Delta_\Gamma^\perp)^{\pm1}$ and $F_{\mathrm{ph}}(\eps^2\Delta_\Gamma)^{\pm1}$ are uniformly bounded on $C^{0,\tau}$. The product coefficient operators satisfy
\begin{subequations}\label{eq: reduced coefficient estimates}
\begin{align}
\|\mathcal C_i f\|_{C^{0,\tau}} &\lesssim \int_{\R^3}\langle y\rangle^{-2} \|f(\cdot,y)\|_{C^{0,\tau}}\,dy,
\label{eq: reduced coefficient weighted bound}\\
\|\mathcal C_i f\|_{C^{0,\tau}} &\lesssim \|f\|_{C^{0,\tau}(\Gamma;L^2_y)},
\label{eq: reduced coefficient L2 bound}\\
\|(\mathcal C_i-\mathcal C_i(0))f\|_{C^{0,\tau}} &\lesssim \sqrt\eps\,\|f\|_{C^{\tau+1/2}(\Gamma;L^2_y)}.
\label{eq: reduced coefficient comparison bound}
\end{align}
\end{subequations}
\end{lemma}
\begin{proof}
For $P=\Delta_\Gamma^\perp$ or $P=\Delta_\Gamma$, choose a smooth partition $1=\sum_{k\geq0}\chi_k(\rho)$ on $[0,\infty)$, with $\chi_0$ supported in $[0,2]$ and $\chi_k(\rho)=\chi(2^{-k}\rho)$ for $k\geq1$, where $\chi$ is supported in $[1/2,2]$. Define the spectral blocks $\mathsf B_k^P=\chi_k(\sqrt P)$. For $0<\alpha<1$, the H\"older norm is equivalent to
\begin{equation}\label{eq: reduced Holder blocks}
 \|v\|_{C^{0,\alpha}}\simeq
 \sup_{k\geq0}2^{k\alpha}\|\mathsf B_k^P v\|_{L^\infty}.
\end{equation}
Indeed, localisation at spatial scale $2^{-k}$ bounds the $k$th block of a H\"older field by $C2^{-k\alpha}\|v\|_{C^{0,\alpha}}$. Conversely, summing the increments of these blocks below and above the scale $2^k d_\Gamma(q,q')=1$ gives a bound by $Cd_\Gamma(q,q')^\alpha$ times the supremum in \eqref{eq: reduced Holder blocks}; the low-frequency block includes the zero eigenspace.

Integration by parts in the local symbols bounds the $L^1$ norms of the localised spectral kernels by finitely many logarithmic derivatives of their symbols. For $F_{\mathrm{tr}}^{\pm1}$ and $F_{\mathrm{ph}}^{\pm1}$, these are bounded by \Cref{lem: reduced adjoint tests}. For $\mathcal C_i$ in \eqref{eq: reduced coefficient reconstruction}, the pointwise and $L^2_y$ bounds in \eqref{eq: reduced adjoint estimates} give kernel bounds weighted by $\langle y\rangle^{-2}$ and uniform bounds as operators from $L^2_y$, respectively. Integration in $y$, or Cauchy--Schwarz in $y$, followed by \eqref{eq: reduced Holder blocks} with $\alpha=\tau$, proves \eqref{eq: reduced coefficient weighted bound}, \eqref{eq: reduced coefficient L2 bound} and the multiplier bounds.

For \eqref{eq: reduced coefficient comparison bound}, use the second bound in \eqref{eq: reduced adjoint estimates}, which controls the symbol of $\mathcal C_i-\mathcal C_i(0)$ and its required logarithmic derivatives on $\sqrt\mu\simeq2^k$ by
\begin{equation}\label{eq: reduced coefficient block difference}
C\min(\eps^{1/2}2^{k/2},1)\lesssim \sqrt\eps\,2^{k/2}.
\end{equation}
Multiplying \eqref{eq: reduced coefficient block difference} by the output weight $2^{k\tau}$ gives $C\sqrt\eps\,2^{k(\tau+1/2)}$, the input weight in \eqref{eq: reduced Holder blocks} with $\alpha=\tau+1/2<1$. Taking the supremum proves \eqref{eq: reduced coefficient comparison bound}. The bounded low-frequency range contains finitely many eigenvalues and satisfies the same $O(\sqrt\eps)$ bound; the difference vanishes at the zero eigenvalue by \eqref{eq: reduced zero spectral tests}.
\end{proof}
For a source ${\bf F}=(F_\Phi,F_A)$, take its rescaled normal forcing $f$ from \Cref{def: rescaling convention}. Multiply these coefficients by a fixed physical cutoff equal to one near the core and define $\mathfrak b_i^{\,\mathrm{prod}}[{\bf F}]=\mathcal C_i f$. Tangential connection components enter the operator comparison below.

\subsection{Decomposition of the reduced remainder}
Define the two remainders of the known source by
\begin{equation}\label{eq: reduced source decomposition}
\begin{split}
\eps^{-1}F_{\mathrm{tr}}(\eps^2\Delta_\Gamma^\perp)^{-1} \mathfrak b^{\,\mathrm{prod}}_{\mathrm{tr}}[{\bf S}_\eps(\Phi_{h,\theta},A_{h,\theta})] &=J_\Gamma h+r_{\mathrm{tr}}^{\,0},\\
\eps^{-2}F_{\mathrm{ph}}(\eps^2\Delta_\Gamma)^{-1} \mathfrak b^{\,\mathrm{prod}}_4[{\bf S}_\eps(\Phi_{h,\theta},A_{h,\theta})] &=\Delta_\Gamma\theta-d_\Gamma^*\beta^{\mathrm{tan}}+r_{\mathrm{ph}}^{\,0},
\end{split}
\end{equation}
and define the contributions of the correction by
\begin{equation}\label{eq: reduced complementary decomposition}
\begin{split}
r_{\mathrm{tr}}^{\,1}
 &=\eps^{-1}F_{\mathrm{tr}}(\eps^2\Delta_\Gamma^\perp)^{-1}
 \bigl(\mathfrak b_{\eps,h,\theta}^{\mathrm{tr}}
 -\mathfrak b_{\mathrm{tr}}^{\,\mathrm{prod}}[{\bf S}_\eps(\Phi_{h,\theta},A_{h,\theta})]\bigr),\\
r_{\mathrm{ph}}^{\,1}
 &=\eps^{-2}F_{\mathrm{ph}}(\eps^2\Delta_\Gamma)^{-1}
 \bigl(\mathfrak b_{4,\eps,h,\theta}
 -\mathfrak b_4^{\,\mathrm{prod}}[{\bf S}_\eps(\Phi_{h,\theta},A_{h,\theta})]\bigr).
\end{split}
\end{equation}
Define the reduced remainder by
\begin{equation*}
\begin{split}
\mathcal R_\eps^{\mathrm{tr}}(h,\theta) &:=J_\Gamma h-\eps^{-1}F_{\mathrm{tr}}(\eps^2\Delta_\Gamma^\perp)^{-1} \mathfrak b_{\eps,h,\theta}^{\mathrm{tr}},\\
\mathcal R_\eps^{\mathrm{ph}}(h,\theta) &:=\Delta_\Gamma\theta-\eps^{-2}F_{\mathrm{ph}}(\eps^2\Delta_\Gamma)^{-1} \mathfrak b_{4,\eps,h,\theta}.
\end{split}
\end{equation*}
Then $\mathcal R_\eps^{\mathrm{tr}}=-r_{\mathrm{tr}}^{\,0}-r_{\mathrm{tr}}^{\,1}$ and
$\mathcal R_\eps^{\mathrm{ph}}=d_\Gamma^*\beta^{\mathrm{tan}}-r_{\mathrm{ph}}^{\,0}-r_{\mathrm{ph}}^{\,1}$. Moreover,
\begin{equation}\label{eq: reduced exact coefficients}
\begin{split}
\mathfrak b_{\eps,h,\theta}^{\mathrm{tr}} &=\eps F_{\mathrm{tr}}(\eps^2\Delta_\Gamma^\perp) (J_\Gamma h-\mathcal R_\eps^{\mathrm{tr}}),\\
\mathfrak b_{4,\eps,h,\theta} &=\eps^2F_{\mathrm{ph}}(\eps^2\Delta_\Gamma) (\Delta_\Gamma\theta-\mathcal R_\eps^{\mathrm{ph}}).
\end{split}
\end{equation}
Since the multipliers are invertible and $F_{\mathrm{ph}}(\eps^2\Delta_\Gamma)$ preserves constants, \eqref{eq: b=0} is equivalent to $\mathcal J(h,\theta)=\Pi_0\mathcal R_\eps(h,\theta)$. It remains to estimate the two remainders and their parameter differences. We estimate the translation and phase components separately, then conclude the proof of \Cref{prop: reduced equation}. 

\subsection{Projection of the approximation}
All calculations take place after pulling the known approximation back by $\mathscr T_h$ and undoing its gauges. Recall that its graphical translation dependence lies in $g_h=\mathscr T_h^*g$, and its phase dependence is the replacement $A_\Gamma\mapsto A_\Gamma-\xi_\eps d\bar\theta$, in both the compact and noncompact approximations. 

Use the fixed horizontal frame $X_j=\partial_{q^j}-\omega_{jb}^az^b\partial_{z^a}$, where $\nabla^\perp_{\partial_{q^j}}\nu_b=\omega_{jb}^a\nu_a$. In this frame, the normal inverse-metric coefficients, mixed inverse-metric coefficients and density $J_h$ from \Cref{sec: projections linear theory} satisfy
\begin{equation}\label{eq: reduced graph metric bounds}
\begin{split}
 |g_h^{ab}-\delta^{ab}|+|J_h-1|&\lesssim r^2+r\|h\|_{C^{2,\tau}}+\|h\|_{C^{2,\tau}}^2,\\
 |g_h^{ja}|&\lesssim r^2+\|h\|_{C^{2,\tau}}.
\end{split}
\end{equation}
These bounds also hold in tangential H\"older norm, with one tangential derivative. 

\begin{lemma}[Tangential matching estimates]\label{lem: reduced matching}
For every fixed integer $k\geq0$,
\begin{equation}\label{eq: reduced conormal matching}
\sum_{\ell=0}^2r^\ell \|\partial_q^k\nabla^\ell(b_r-\beta)\|_{C^{0,\tau}(\Gamma)} \lesssim _k r(1+|\log r|)^{N_k}.
\end{equation}
The same estimate holds for $H-(2r)^{-1}-H_\Gamma$ in the noncompact case. For $k=0,1$, one may take $N_k=1$.
\end{lemma}
\begin{proof}
By \eqref{eq: exterior regular source}, \eqref{eq: exterior tangential regularity} and \eqref{eq: bound R}, the source of $\Delta_1b_r$ is $O(r^{-1})$, also after tangential differentiation. The scalar source $\Delta_0(H-(2r)^{-1})$ has the same bound by minimality, as in the proof of \eqref{eq: noncompact approximation matching}. Subtracting the smooth traces adds only bounded terms. Differentiate the local Green representation tangentially, transferring derivatives in the integration variable by parts; the resulting kernels retain the bounds used in \Cref{lem: approximation regular matching}. The argument for \eqref{eq: approximation regular matching} gives the bound $r(1+|\log r|)^{N_k}$. Scaled Schauder estimates give the two normal derivatives. Thus the logarithm comes from integration of the Green kernel, rather than from an additional singularity of the original source.
\end{proof}

\begin{lemma}[The translation remainder]\label{lem: reduced source size}
There is a rate $\omega_\eps\to 0$, independent of the parameters in \eqref{eq: assumptions htheta}, such that
\begin{equation}\label{eq: reduced source bounds}
\|r_{\mathrm{tr}}^{\,0}\|_{C^{0,\tau}}\lesssim \eps|\log\eps|^2+\omega_\eps\|h\|_{C^{2,\tau}}.
\end{equation}
\end{lemma}
\begin{proof}
Let ${\bf S}=(S_\Phi,S_A)={\bf S}_\eps(\Phi_{h,\theta},A_{h,\theta})$, and let ${\bf f}_{h,\theta}$ be its rescaled normal source in \Cref{def: rescaling convention}, with a radial cutoff equal to one on $r\leq\delta/2$ and supported in $r<\delta$. Set $\mathcal C_{\mathrm{tr}}(0){\bf f}=(\mathcal C_a(0){\bf f})\nu_a$ and define
\begin{equation}\label{eq: translation coefficients}
 \mathfrak t_\eps^0=\frac{\mathcal C_{\mathrm{tr}}(0){\bf f}_{h,\theta}}{\eps F_{\mathrm{tr}}(0)},
 \qquad
 \mathfrak t_\eps=\eps^{-1}F_{\mathrm{tr}}(\eps^2\Delta_\Gamma^\perp)^{-1}\mathcal C_{\mathrm{tr}}{\bf f}_{h,\theta}.
\end{equation}
We prove the two estimates
\begin{align}
 \|\mathfrak t_\eps^0-J_\Gamma h\|_{C^{0,\tau}}
 &\lesssim \eps|\log\eps|^2+\omega_\eps\|h\|_{C^{2,\tau}},\label{eq: translation zero parameter claim}\\
 \|\mathfrak t_\eps-\mathfrak t_\eps^0\|_{C^{0,\tau}}
 &\lesssim  \eps|\log\eps|^2+\omega_\eps\|h\|_{C^{2,\tau}}.\label{eq: translation spectral comparison claim}
\end{align}
Adding \eqref{eq: translation zero parameter claim} and \eqref{eq: translation spectral comparison claim} proves \eqref{eq: reduced source bounds}, since $r_{\mathrm{tr}}^{\,0}=\mathfrak t_\eps-J_\Gamma h$ by \eqref{eq: reduced source decomposition}. Throughout, $\lambda=\lambda_\eps$ in the compact case and $\lambda=0$ in the noncompact case; $N$ denotes a fixed nonnegative integer accounting for logarithmic factors.

\smallskip

We begin with \eqref{eq: translation zero parameter claim}. For a vector field $V$, define ${\bf D}_V\coloneqq(d_A\Phi(V),\iota_VF_A)$ at $(\Phi,A)=(\Phi_{h,\theta},A_{h,\theta})$. We record the identity
\begin{equation}\label{eq: model translation normalization}
 \mu_\lambda\coloneqq E_\lambda(\Phi_\lambda,A_\lambda)
 =\|{\bf Z}_a\|_2^2,\qquad a=1,2,3.
\end{equation}
Indeed, rotational symmetry and \eqref{eq: monopole translations} give
\begin{equation}\label{eq: translation mode energy}
 3\|{\bf Z}_a\|_2^2
 =\sum_{b=1}^3\|{\bf Z}_b\|_2^2
 =\int_{\R^3}\left(|d_{A_\lambda}\Phi_\lambda|^2+2|F_{A_\lambda}|^2\right)dy.
\end{equation}
For $\mathfrak d_t(y)=y/t$, set $(\Phi_{\lambda,t},A_{\lambda,t})=(\mathfrak d_t^*\Phi_\lambda,\mathfrak d_t^*A_\lambda)$. Stationarity under dilations yields
\begin{equation}\label{eq: model dilation identity}
 0=\left.\frac{d}{dt}\right|_{t=1}E_\lambda(\Phi_{\lambda,t},A_{\lambda,t})
 =\frac12\int_{\R^3}\left(|d_{A_\lambda}\Phi_\lambda|^2-|F_{A_\lambda}|^2+\frac{3\lambda}{2}(1-|\Phi_\lambda|^2)^2\right)dy.
\end{equation}
Combining \eqref{eq: translation mode energy} and \eqref{eq: model dilation identity} with \eqref{eq: Elambda} proves \eqref{eq: model translation normalization}, also at $\lambda=0$.

\medskip

For a normal field $k$ on $\Gamma$, let $V_{h,k}=(\partial_t\mathscr T_{h+tk}|_{t=0})\circ\mathscr T_h^{-1}$. We claim that
\begin{equation}\label{eq: Jacobi from error}
 \frac{\langle{\bf S},\eps{\bf D}_{V_{h,k}}\rangle_{L^2_\eps}}{\eps\mu_\lambda}
 =\int_\Gamma\langle J_\Gamma h+R_\eps,k\rangle\,d\mathrm{vol}_\Gamma,
\end{equation}
where $J_\Gamma$ is the Jacobi operator and
\begin{equation}\label{eq: bound R_eps}
 \|R_\eps\|_{C^{0,\tau}}
 \lesssim \eps(1+|\log\eps|)+\|h\|_{C^{2,\tau}}^2+\|h\|_{C^{2,\tau}}^3.
\end{equation}
First freeze the metric at $z=0$. Define the normal metric by
\begin{equation*}
 g_h^\perp(q)\coloneqq\bigl((g_h^{ab}(q,0))_{a,b=1}^3\bigr)^{-1};
\end{equation*}
this is the metric on the normal coordinate vectors after their orthogonal projection onto $(T\Gamma_h)^\perp$. Let $\mathcal E_\lambda(\widetilde g)$ denote the energy of the fixed radial profile in a constant normal metric $\widetilde g$, and let $j_h\,d\mathrm{vol}_\Gamma$ be the graph volume pulled back by $q\mapsto F_h(q,0)$. Integration in the normal variables gives 
\begin{equation}\label{eq: frozen graph energy}
 \mathcal F_\lambda(h)=\int_\Gamma j_h\mathcal E_\lambda(g_h^\perp)\,d\mathrm{vol}_\Gamma,
 \qquad \mathscr{A}(h)=\operatorname{Vol}(\Gamma_h)=\int_\Gamma j_h\,d\mathrm{vol}_\Gamma.
\end{equation}
By \eqref{eq: reduced graph metric bounds}, $g_h^\perp-I=O(\|h\|_{C^{2,\tau}}^2)$. Rotational symmetry and \eqref{eq: model dilation identity} give $d\mathcal E_\lambda(I)=0$, while \eqref{eq: model translation normalization} gives $\mathcal E_\lambda(I)=\mu_\lambda$. Differentiating \eqref{eq: frozen graph energy} therefore yields
\begin{equation}\label{eq: frozen graph energy variation}
\begin{split}
 D\mathcal F_\lambda(h)[k]-\mu_\lambda D\mathscr{A}(h)[k]
 =\int_\Gamma\Bigl[&\bigl(\mathcal E_\lambda(g_h^\perp)-\mu_\lambda\bigr)Dj_h[k]\\
 &+j_h\,d\mathcal E_\lambda(g_h^\perp)[Dg_h^\perp[k]]\Bigr]d\mathrm{vol}_\Gamma.
\end{split}
\end{equation}
Here $\mathcal E_\lambda(g_h^\perp)-\mu_\lambda=O(\|h\|_{C^{2,\tau}}^4)$ and $Dg_h^\perp[k]=O(\|h\|_{C^{2,\tau}}\|k\|_{C^1})$. After integrating derivatives of $k$ by parts, \eqref{eq: frozen graph energy variation} contributes $O(\|h\|_{C^{2,\tau}}^3)$ in $C^{0,\tau}$ to $R_\eps$.

To see \eqref{eq: Jacobi from error}, differentiate the inverse pullbacks in \eqref{eq: shifted phase approximation} and \eqref{eq: noncompact shifted approximation}. The field variation is $-{\bf D}_{V_{h,k}}$ modulo an infinitesimal gauge transformation; hence gauge invariance and \eqref{eq:euler-lagrange} give
\begin{equation}\label{eq: graph energy variation}
 \langle{\bf S},{\bf D}_{V_{h,k}}\rangle_{L^2_\eps}
 =-\left.\frac{d}{dt}\right|_{t=0}E_\eps(\Phi_{h+tk,\theta},A_{h+tk,\theta}).
\end{equation}
Since $\Gamma$ is minimal, the area variation satisfies
\begin{equation}\label{eq: graph area variation}
\begin{split}
 D\mathscr{A}(0)&=0,\\
 D^2\mathscr{A}(0)[h,k]&=-\int_\Gamma\langle J_\Gamma h,k\rangle\,d\mathrm{vol}_\Gamma,\\
 D\mathscr{A}(h)[k]&=-\int_\Gamma\langle J_\Gamma h+\mathcal Q(h),k\rangle\,d\mathrm{vol}_\Gamma,
 \qquad \|\mathcal Q(h)\|_{C^{0,\tau}}\lesssim\|h\|_{C^{2,\tau}}^2.
\end{split}
\end{equation}
By \eqref{eq: frozen graph energy variation}, the leading energy variation is $\mu_\lambda D\mathscr{A}(h)[k]$. Thus \eqref{eq: graph energy variation} and \eqref{eq: graph area variation}, normalised by \eqref{eq: model translation normalization}, give \eqref{eq: Jacobi from error}; $R_\eps$ consists of $\mathcal Q(h)$ and the errors in this comparison of first variations.

To finish \eqref{eq: bound R_eps}, vary the pulled-back metric before estimating the errors and integrate tangential derivatives of $k$ by parts on $\Gamma$. The cubic contribution was bounded in \eqref{eq: frozen graph energy variation}, and the quadratic contribution is given by \eqref{eq: graph area variation}.

Let
\begin{equation*}
 e_\lambda(y)\coloneqq\frac12\left(|d_{A_\lambda}\Phi_\lambda|^2+|F_{A_\lambda}|^2+\frac\lambda2(1-|\Phi_\lambda|^2)^2\right)
\end{equation*}
be the model energy density in \eqref{eq: Elambda}. The remaining metric variation and the curvature correction from the improved approximation are controlled, respectively, by
\begin{equation}\label{eq: translation energy error integrals}
\begin{split}
 \int_{|z|<\delta}|z|\eps^{-3}e_\lambda(z/\eps)dz
 =\eps\int_{|y|<\delta/\eps}|y|e_\lambda(y)dy&
 \lesssim\eps(1+|\log\eps|),\\
 \eps\int_0^\delta\frac{r^2}{(\eps+r)^3}\,dr
 &\lesssim\eps(1+|\log\eps|).
\end{split}
\end{equation}
The first bound in \eqref{eq: translation energy error integrals} follows from \eqref{eq: approximation profile bounds}; the potential contributes $O(\eps)$ for $\lambda>0$ and is absent at $\lambda=0$. For the second, pair the model curvature $O((\eps+r)^{-2})$ with the added curvature $O((\eps+r)^{-1})$, using the connection estimates in \Cref{sec: approximation errors}. The matching estimates \eqref{eq: reduced conormal matching} and \eqref{eq: noncompact approximation matching} give the same bounds in tangential $C^{0,\tau}$ after the integration by parts, including the matching annulus and fixed cutoffs. Equation \eqref{eq: noncompact mass normal cancellation} gives the same bound for \eqref{eq: noncompact mass correction}. Together these bounds give \eqref{eq: bound R_eps}.

\smallskip

Now we compare with the actual translation tests.
On $r<\sqrt\eps$, put $W_{a,h}=(F_h)_*\partial_{z^a}$, with $F_h$ from \Cref{sec: projections linear theory}, so that $V_{h,k}=k^aW_{a,h}$.
The two tests being compared are the prescribed mode ${\bf Z}_{a,\eps,h,\theta}$ from that section and the displacement pair $\eps{\bf D}_{W_{a,h}}$. Define their connection-component difference by
\begin{equation}\label{eq: translation test difference}
 \mathcal B_a\coloneqq ({\bf Z}_{a,\eps,h,\theta})_A-\eps\iota_{W_{a,h}}F_{A_{h,\theta}}.
\end{equation}
Equations \eqref{eq: approximation horizontal cancellation}, \eqref{eq: approximation inner connection bound}, \eqref{eq: approximation parameter shells} and \eqref{eq: reduced graph metric bounds}, applied to \eqref{eq: translation test difference}, give
\begin{equation}\label{eq: translation test error bounds}
\begin{split}
 \|\mathcal B_a(\cdot,z)\|_{C^{0,\tau}(\Gamma)}&\lesssim \frac{\eps(r^2+\|h\|_{C^{2,\tau}}+\eps)}{(\eps+r)^2},\\
 \|S_A(\cdot,z)\|_{C^{0,\tau}(\Gamma)}&\lesssim \left(\frac1{\eps+r}+\frac{\eps+\|h\|_{C^{2,\tau}}}{(\eps+r)^2}+\frac{\|h\|_{C^{2,\tau}}^2}{(\eps+r)^3}\right),
\end{split}
\end{equation}
where $r=|z|<\sqrt\eps$.
Using \eqref{eq: assumptions htheta} in \eqref{eq: translation test error bounds}, the contribution to $\mathfrak t_\eps^0$ in \eqref{eq: translation coefficients} is bounded by
\begin{equation}\label{eq: translation test pairing bound}
 \left\|\int_{|z|<\sqrt\eps}\langle S_A,\mathcal B_a\rangle\,dz\right\|_{C^{0,\tau}(\Gamma)}
 \lesssim \eps^2(1+|\log\eps|)+\eps(1+|\log\eps|)^4\|h\|_{C^{2,\tau}}.
\end{equation}

The Higgs components of ${\bf Z}_{a,\eps,h,\theta}$ and $\eps{\bf D}_{W_{a,h}}$ agree in the compact case; in the noncompact case their difference is $O(\eps)$ in the $C^{0,\tau}(\Gamma)$ norm, uniformly for $r<\sqrt\eps$, by \eqref{eq: noncompact mass correction}.

The Jacobian volume error is controlled by $|J_h-1|\lesssim r^2+\|h\|_{C^{2,\tau}}^2$, giving
\begin{equation}\label{eq: translation Jacobian bound}
 \left\|\frac1{\eps\mu_\lambda}\int_{|z|<\sqrt\eps}(J_h-1)
 \langle{\bf S},\eps{\bf D}_{W_{a,h}}\rangle_\eps\,dz\right\|_{C^{0,\tau}(\Gamma)}
 \lesssim
 \begin{cases}
  \eps^{1+2\tau'},&M\text{ compact},\\
  \eps^2(1+|\log\eps|)^N,&M\text{ noncompact}.
 \end{cases}
\end{equation}
Here $\tau<\tau'<1/2$ is as in \eqref{eq: deltaeps}, and $N$ is a fixed nonnegative integer. At $h=0$, the Higgs density contribution is $O(\eps\ell_\eps^2)=O(\delta_\eps)$ in the compact case and $O(\eps^2(1+|\log\eps|)^N)$ in the noncompact case. The connection contribution and the terms containing $h$ satisfy the latter bound by \eqref{eq: assumptions htheta}, after increasing $N$. In the compact case, \eqref{eq: deltaeps} and $1+2\tau'<2$ therefore bound their sum by $\eps^{1+2\tau'}$. For $h\ne 0$ satisfying \eqref{eq: assumptions htheta}, the terms including $h$ contribute to lower orders. 

\smallskip

For $r>\sqrt\eps$, project each Lie algebra component of a pair onto $\R\sigma$, denoting this projection by a superscript $\parallel$. The outer contribution parallel to $\sigma$ is the normalised moment
\begin{equation*}
 \mathfrak{m}_a^\parallel(q)=\frac1{\eps\mu_\lambda}
 \int_{\eps^{-1/2}<|y|<\delta/\eps}
 \langle{\bf f}_{h,\theta}^\parallel(q,y),{\bf Z}_a^\parallel(y)\rangle\,dy,
 \qquad a=1,2,3.
\end{equation*}
The matching bound \eqref{eq: reduced conormal matching} gives
\begin{equation}\label{eq: translation outer moment bound}
 \|\mathfrak{m}_a^\parallel\|_{C^{0,\tau}}
 \lesssim \eps(1+|\log\eps|)^2+\eps.
\end{equation}
The perpendicular components of the test on this region are exponentially small in $\eps^{-1/2}$. Combining \eqref{eq: translation test pairing bound}, \eqref{eq: translation Jacobian bound} and \eqref{eq: translation outer moment bound} with \eqref{eq: Jacobi from error} and \eqref{eq: assumptions htheta} proves \eqref{eq: translation zero parameter claim}.

\smallskip

To prove \eqref{eq: translation spectral comparison claim}, set
\begin{equation*}
 \mathcal D_\eps=\eps^{-1}\left[F_{\mathrm{tr}}(\eps^2\Delta_\Gamma^\perp)^{-1}\mathcal C_{\mathrm{tr}}-F_{\mathrm{tr}}(0)^{-1}\mathcal C_{\mathrm{tr}}(0)\right].
\end{equation*}
The claim follows from
\begin{equation}\label{eq: translation spectral remainder bound}
 \|\mathcal D_\eps{\bf f}_{h,\theta}\|_{C^{0,\tau}}
 \lesssim \eps\sqrt{\ell_\eps}+\eps^{5/4}(1+|\log\eps|)^N,
\end{equation}
with the first term omitted in the noncompact case.

We prove \eqref{eq: translation spectral remainder bound} as follows. The normalised tests ${\bf J}_{a,t}/F_{\mathrm{tr}}(t)$ satisfy \eqref{eq: reduced adjoint estimates}, by the product rule and $F_a(t)-F_a(0)=\langle{\bf J}_{a,t}-{\bf J}_{a,0},{\bf Z}_a\rangle_2$. Thus \Cref{lem: reduced coefficient operators} gives
\begin{equation}\label{eq: normalized translation comparison}
 \|\mathcal D_\eps{\bf f}\|_{C^{0,\tau}}
 \lesssim\eps^{-1/2}\|{\bf f}\|_{C^{\tau+1/2}(\Gamma;L^2_y)}.
\end{equation}
After subtracting the full principal profiles, collect in ${\bf f}^{\mathrm{odd}}$ the terms whose integrands in the pairings with the translation tests are odd in $y$, and write
\begin{equation}\label{eq: translation source splitting}
\begin{split}
 {\bf f}_{h,\theta}
 &=\eps(-\Delta_\Gamma^\perp h)^a{\bf Z}_a
   +\eps^2(\Delta_\Gamma\theta-d_\Gamma^*\beta^{\mathrm{tan}}){\bf Z}_4\\
 &\quad+{\bf f}^{\mathrm{odd}}+{\bf f}^{(1)}+{\bf f}^{(2)},
\end{split}
\end{equation}
where
\begin{itemize}
 \item ${\bf f}^{(1)}$ contains the remaining terms with at most first tangential derivatives of $h,\theta$, retaining graph-dependent terms only on $r<2\sqrt\eps$;
 \item ${\bf f}^{(2)}$ contains the remaining second-derivative terms, the principal cutoff tail and the graph-dependent source on $r\geq2\sqrt\eps$.
\end{itemize}
The first three terms in \eqref{eq: translation source splitting} are annihilated by $\mathcal D_\eps$, by \eqref{eq: reduced spectral recovery}, \eqref{eq: reduced kernel source coefficients} and parity. Using \eqref{eq: assumptions htheta} to absorb the terms containing $h$, and increasing $N$, the required estimates reduce to
\begin{equation}\label{eq: translation split source bounds}
\begin{split}
 \|{\bf f}^{(1)}\|_{C^{\tau+1/2}(\Gamma;L^2_y)}
 &\lesssim\eps^{3/2}\sqrt{\ell_\eps}+\eps^{7/4}(1+|\log\eps|)^N,\\
 \|\mathcal D_\eps{\bf f}^{(2)}\|_{C^{0,\tau}}
 &\lesssim\eps^{5/4}(1+|\log\eps|)^N.
\end{split}
\end{equation}
The first estimate in \eqref{eq: translation split source bounds} follows from \eqref{eq: approximation profile bounds}, \eqref{eq: approximation parameter shells} and \eqref{eq: reduced conormal matching}, using the rescaling of \Cref{def: rescaling convention}. Integrating the compact scalar factor $e^{-cr/\ell_\eps}$ normally gives the term containing $\sqrt{\ell_\eps}$. In the noncompact case this term is absent; the linear normal error of \eqref{eq: noncompact mass correction} cancels by \eqref{eq: noncompact mass normal cancellation}, and its remaining terms obey \eqref{eq: noncompact mass derivative bound}, \eqref{eq: noncompact phase profile defect} and \eqref{eq: noncompact approximation matching}.

For the second estimate in \eqref{eq: translation split source bounds}, use the direct $L^2_y$ bound \eqref{eq: reduced coefficient L2 bound}. The second-derivative coefficients have the extra factor $r+\|h\|_{C^{2,\tau}}$.

Equation \eqref{eq: approximation profile bounds} gives $\|(\eps|y|){\bf Z}_i\|_{L^2(|y|<\delta/\eps)}\lesssim\sqrt\eps$; hence only $C^{0,\tau}$ control of those derivatives is needed. The fixed physical cutoff tail has $L^2_y$ norm $O(\sqrt\eps)$, while integration of the outer bounds \eqref{eq: reduced conormal matching} and \eqref{eq: approximation parameter shells} over $r>\sqrt\eps$, followed by \eqref{eq: assumptions htheta}, gives the stated bound. Applying \eqref{eq: normalized translation comparison} to ${\bf f}^{(1)}$ and adding the estimate for ${\bf f}^{(2)}$ proves the claim.
\end{proof}

\begin{lemma}[The phase remainder]\label{lem: reduced phase source size}
After potentially increasing the rate $\omega_\eps\to 0$ in \Cref{lem: reduced source size}, we have
\begin{equation}\label{eq: reduced phase source bounds}
\|r_{\mathrm{ph}}^{\,0}\|_{C^{0,\tau}}\leq\omega_\eps.
\end{equation}
\end{lemma}
\begin{proof}
Let ${\bf f}$ be the rescaled known source in \Cref{def: rescaling convention}.
In the compact case set ${\bf f}^{\circ}={\bf f}$. In the noncompact case, let ${\bf f}^{\circ}$ be the source obtained by removing the cutoff mass correction \eqref{eq: noncompact mass correction} from the approximation before applying the graph and gauge maps, with all other fields unchanged. Set $\chi_\eps(y)=\eta(\sqrt\eps\,|y|)$ and define
\begin{equation}\label{eq: phase source splitting}
\begin{split}
 R_{\mathrm{in}}&\coloneqq\eps^{-2}F_{\mathrm{ph}}(\eps^2\Delta_\Gamma)^{-1}\mathcal C_4\bigl[\chi_\eps({\bf f}^{\circ}-\eps^2(\Delta_\Gamma\theta-d_\Gamma^*\beta^{\mathrm{tan}}){\bf Z}_4)\bigr],\\
 R_{\mathrm{out}}&\coloneqq\eps^{-2}F_{\mathrm{ph}}(\eps^2\Delta_\Gamma)^{-1}\mathcal C_4\bigl[(1-\chi_\eps)({\bf f}^{\circ}-\eps^2(\Delta_\Gamma\theta-d_\Gamma^*\beta^{\mathrm{tan}}){\bf Z}_4)\bigr],\\
 R_{\mathrm m}&\coloneqq\eps^{-2}F_{\mathrm{ph}}(\eps^2\Delta_\Gamma)^{-1}\mathcal C_4({\bf f}-{\bf f}^{\circ}).
\end{split}
\end{equation}
By \eqref{eq: reduced spectral recovery} and \eqref{eq: phase source splitting},
$r_{\mathrm{ph}}^{\,0}=R_{\mathrm{in}}+R_{\mathrm{out}}+R_{\mathrm m}$ exactly. We claim that the following bounds hold. 
\begin{subequations}
\begin{align}
 \|R_{\mathrm{in}}\|_{C^{0,\tau}}
 &\lesssim \eps^{1/2-\tau}(1+|\log\eps|)^N,\label{eq: phase source inner bound}\\
 \|R_{\mathrm{out}}\|_{C^{0,\tau}}
 &\lesssim \eps^{1/2-\tau}(1+|\log\eps|)^N,\label{eq: phase source outer bound}\\
 \|R_{\mathrm m}\|_{C^{0,\tau}}
 &\lesssim \eps^{1-\tau}(1+|\log\eps|)^N.\label{eq: phase source mass bound}
\end{align}
\end{subequations}

First, we note that rotation equivariance and parity give real radial functions $A_t,B_t$, the coefficients of the connection component of ${\bf J}_{4,t}$ parallel and perpendicular to $\sigma$, respectively:
\begin{equation}\label{eq: phase radial test form}
 ({\bf J}_{4,t})_\Phi=0,\qquad
 ({\bf J}_{4,t})_a=A_t(|y|)\frac{y^a}{|y|}\sigma
       +B_t(|y|)\left(T_a-\frac{y^a}{|y|}\sigma\right),\quad a=1,2,3.
\end{equation}
For every fixed integer $j\geq0$, these coefficients satisfy
\begin{equation}\label{eq: phase radial test bounds}
|(t\partial_t)^jA_t(|y|)|\lesssim _jq_\lambda'(|y|),\qquad |(t\partial_t)^jB_t(|y|)|\lesssim _j\frac{(1-\alpha_\lambda(|y|))q_\lambda(|y|)}{|y|}.
\end{equation}
To prove \eqref{eq: phase radial test bounds}, write
$B_a=a(r)\frac{y^a}{r}\sigma+b(r)(T_a-\frac{y^a}{r}\sigma)$.
Direct substitution in \eqref{eq: model linearised operator} shows that
${\bf L}_\lambda(0,B)$ has zero Higgs component and that its connection
coefficients are
\begin{equation}\label{eq: phase radial resolvent operator}
\begin{split}
 (\mathcal L_\lambda(a,b))_1
 &=-a''-\frac2r a'
 +\frac{2\bigl(1+(1-\alpha_\lambda)^2\bigr)}{r^2}a
 -\frac{4(1-\alpha_\lambda+r\alpha_\lambda')}{r^2}b,\\
 (\mathcal L_\lambda(a,b))_2
 &=-b''-\frac2r b'
 +\left(4\varphi_\lambda^2+
 \frac{3(1-\alpha_\lambda)^2-1}{r^2}\right)b
 -\frac{2(1-\alpha_\lambda+r\alpha_\lambda')}{r^2}a.
\end{split}
\end{equation}
The regular boundary conditions are $a'(0)=b'(0)=0$ and
$a-b=O(r^2)$; at infinity we use
$\int_0^\infty(a^2+2b^2)r^2\,dr<\infty$.
The radial operator is self-adjoint for this weighted pairing and
nonnegative, as established in the proof of
\Cref{lem: normal weighted coercivity}.

Its off-diagonal coefficients are nonpositive.
Indeed, the radial minimising profile satisfies $0<1-\alpha_\lambda<1$
by truncation and the profile equation.
After decreasing $\lambda_*$, one has $\alpha_\lambda'>0$:
on fixed compact intervals this follows from convergence to
\eqref{eq: BPS}, including the uniform expansion at zero;
for large $r$, \eqref{eq: approximation profile bounds} gives
$\varphi_\lambda\geq\frac12$, and
\eqref{eq:monopole-profiles} gives
$(1-\alpha_\lambda)''>0$.
Since $1-\alpha_\lambda$ tends to zero, its derivative is negative there.
Testing a radial resolvent equation with the componentwise negative
parts, and using nonnegativity of the quadratic form, therefore shows
that $(\mathcal L_\lambda+t)^{-1}$ preserves componentwise positivity
for $t>0$.

By \eqref{eq: qlambda},
$q_\lambda>0$ and $(r^2q_\lambda')'=2(1-\alpha_\lambda)^2q_\lambda>0$.
The positive phase coefficients
\[
 \left(q_\lambda',\frac{(1-\alpha_\lambda)q_\lambda}{r}\right)
\]
lie in the kernel of \eqref{eq: phase radial resolvent operator},
by \eqref{eq:monopole-profiles} and \eqref{eq: qlambda}.
Since ${\bf E}_4=\eta_{\delta/(4\eps)}(1+r^2)^{-1}{\bf Z}_4$,
positivity and $(\mathcal L_\lambda+t){\bf Z}_4=t{\bf Z}_4$
in radial coefficients give, for every integer $m\geq0$,
\[
 0\leq (T_t^m{\bf E}_4)_{\rm rad}
 \leq\left(q_\lambda',\frac{(1-\alpha_\lambda)q_\lambda}{r}\right)
\]
componentwise, by induction.
Here the subscript denotes the two radial coefficients above.
Finally, \eqref{eq: reduced resolvent polynomial identities} expresses
$(t\partial_t)^jT_t$ as a fixed linear combination of these iterates.
Dividing by $Q_4(t)$ and using its logarithmic derivative bounds in
\Cref{lem: reduced adjoint tests} proves
\eqref{eq: phase radial test bounds}.

\smallskip\noindent We prove \eqref{eq: phase source inner bound}.
For ${\bf f}^{\circ}$, use the horizontal frame and covariant derivatives from \eqref{eq: approximation horizontal cancellation}, with the shell estimates \eqref{eq: approximation inner shells}. By \eqref{eq: approximation phase curvature}, the phase replaces $\beta^{\mathrm{tan}}$ by $\beta^{\mathrm{tan}}-d_\Gamma\theta$. In physical normal coordinates $z=\eps y$, for $|z|<\sqrt\eps$, this gives
\begin{equation}\label{eq: phase principal tangential current}
 -g_\Gamma^{ij}(\nabla_A)_iF_{ja}
 =(\Delta_\Gamma\theta-d_\Gamma^*\beta^{\mathrm{tan}})D_a\xi_\eps
 +|\beta^{\mathrm{tan}}-d_\Gamma\theta|^2[\xi_\eps,D_a\xi_\eps],
\end{equation}
Here $\nabla_A$ uses the product metric. Since $D_a\xi_\eps=\eps^{-1}({\bf Z}_4)_a(z/\eps)$, the connection rescaling in \Cref{def: rescaling convention} gives the principal source
\begin{equation*}
 \eps^2(\Delta_\Gamma\theta-d_\Gamma^*\beta^{\mathrm{tan}}){\bf Z}_4.
\end{equation*}
The second term in \eqref{eq: phase principal tangential current} pairs to zero, since
\begin{equation}\label{eq: phase commutator cancellation}
 \langle[\xi_\eps,D_a\xi_\eps],({\bf J}_{4,t})_a\rangle=0:
\end{equation}
the bracket is perpendicular to both $\sigma$ and $T_a-\frac{y^a}{|y|}\sigma$.

The remaining normal curvature terms also cancel before radial integration. Indeed, $F_{ab}$ is a linear combination of $\epsilon_{abd}T_d$ and $\epsilon_{abd}\frac{y^d}{|y|}\sigma$, whereas $(d_A{\bf J}_{4,t})_{ab}$ is a radial multiple of $y^aT_b-y^bT_a$. For the normal inverse-metric coefficients $g_h^{ab}$ in \eqref{eq: reduced graph metric bounds},
\begin{equation}\label{eq: phase normal contractions}
 g_h^{ac}g_h^{bd}\langle F_{ab},(d_A{\bf J}_{4,t})_{cd}\rangle=0,
 \qquad \langle F_{ba},({\bf J}_{4,t})_a\rangle=0.
\end{equation}
The two contractions in the first identity are multiples of
\begin{equation*}
g_h^{ac}y^c\epsilon_{abd}g_h^{bd}=0,\quad g_h^{ac}y^cg_h^{be}y^e\epsilon_{abd}y^d=0.
\end{equation*}
Both vanish by symmetry. The second identity follows by the same antisymmetric contraction. It removes the term containing $D^2h$ with its mixed metric coefficient frozen at $z=0$; the remaining frozen shear terms have odd integrands on each normal sphere. For each $a$, both $[\Phi,D_a\Phi]$ and $[\xi_\eps,D_a\xi_\eps]$ are scalar multiples of $[\sigma,T_a]$, which is perpendicular to the two Lie algebra directions in $({\bf J}_{4,t})_a$. Hence
\begin{equation*}
 \langle[\Phi,D_a\Phi],({\bf J}_{4,t})_a\rangle=0,\qquad a=1,2,3.
\end{equation*}

After the cancellations in \eqref{eq: phase commutator cancellation} and \eqref{eq: phase normal contractions}, the $h$-dependent residual estimated in \eqref{eq: approximation parameter shells} is controlled by the following radial integrals, using \eqref{eq: assumptions htheta}:
\begin{subequations}\label{eq: phase inner graph integrals}
\begin{align}
 I_1&=\int_0^{2\sqrt\eps}
 \frac{r^2(r\|h\|_{C^{2,\tau}}+\|h\|_{C^{2,\tau}}^2)}{(\eps+r)^4}e^{-cr/\eps}\,dr
 \lesssim \eps(1+|\log\eps|)^4,\label{eq: phase inner transverse integral}\\
 I_2&=\int_0^{2\sqrt\eps}
 \frac{r^2(r^2\|h\|_{C^{2,\tau}}+r\|h\|_{C^{2,\tau}}^2+\|h\|_{C^{2,\tau}}^3)}{(\eps+r)^4}\,dr\lesssim  \eps(1+|\log\eps|)^2.\label{eq: phase inner current integral}
\end{align}
\end{subequations}
The first integral bounds the terms containing the exponentially decreasing component $B_t$ after the cancellation. The second bounds the tangential current formed by contracting normal curvature with the product of a normal metric defect and a mixed metric coefficient. These products have size $C(r^2\|h\|_{C^{2,\tau}}+r\|h\|_{C^{2,\tau}}^2+\|h\|_{C^{2,\tau}}^3)$ in $C^{1,\tau}(\Gamma)$. The factors $(\eps+r)^{-4}$ follow from
$|F_{ab}|\leq C(\eps+r)^{-2}$ and
$|\eps^{-1}{\bf J}_{4,t}(z/\eps)|\leq C\eps(\eps+r)^{-2}$, including the normalisation $\eps^{-1}$ of the physical connection pairing. 

To complete the estimate of $R_{\mathrm{in}}$, we now bound the reference contribution at $h=0$; the general dependence on $h$ is controlled by \eqref{eq: phase inner graph integrals}. Pairing the metric error and the tangential-curvature error with the phase test gives, respectively,
\begin{equation}\label{eq: phase reference integrals}
 C\int_0^{2\sqrt\eps}\frac{r^3}{(\eps+r)^3}\,dr\lesssim \sqrt\eps,
 \qquad C\eps\int_0^{2\sqrt\eps}\frac{r^2}{(\eps+r)^3}\,dr\lesssim \eps (1+|\log\eps|).
\end{equation}
Derivatives of $B_t$ are integrated onto the charged curvature profile; the boundary term vanishes at zero and is exponentially small where $d\chi_\eps\neq0$. On the overlap $\sqrt\eps\leq r\leq2\sqrt\eps$, the calculation leading to \eqref{eq: phase outer current integral} gives the same bound after increasing $N$, since $\|h\|_{C^{2,\tau}}/\sqrt\eps\leq K \sqrt\eps (1+|\log\eps|)^2$. The density and transported-covector errors use \eqref{eq: reduced graph metric bounds} and \eqref{eq: phase normal contractions}. Their additional normalised bound is $C_K(1+|\log\eps|)^N(\eps^{1-\tau/2}+\eps^{1-\tau})$, absorbed in \eqref{eq: phase source inner bound}. Passing from scale-$r$ H\"older bounds to the ordinary norm costs at most $\eps^{-\tau}$. Finally, \Cref{lem: reduced coefficient operators} applies \eqref{eq: phase inner graph integrals} and \eqref{eq: phase reference integrals} uniformly in the tangential spectral parameter, proving \eqref{eq: phase source inner bound}.

\smallskip\noindent We prove \eqref{eq: phase source outer bound}.
By \eqref{eq: phase radial test bounds}, the physical test parallel to $\sigma$ is $\eps^{-1}A_t(r/\eps)\sigma\,dr$, with
\begin{equation}\label{eq: phase outer radial test}
 \eps^{-1}|(t\partial_t)^jA_t(r/\eps)|\lesssim_j\eps r^{-2},\qquad r\geq\sqrt\eps.
\end{equation}
Its exterior derivative against a tangential test function contains no radial derivative of $A_t$, since $dr\wedge dr=0$. Use $g_h,J_h$ from \eqref{eq: reduced graph metric bounds}, and raise curvature indices with $g_h$. Subtracting the principal mixed curvature in \eqref{eq: approximation horizontal cancellation}, with the phase change \eqref{eq: approximation phase curvature}, gives the real one-form
\begin{equation}\label{eq: phase outer current}
(j_r^{\mathrm{err}})_i =(g_\Gamma)_{ij}r^2\int_{\Sp^2}J_h\langle F^{jr},\sigma\rangle\,d\mathrm{vol}_{\Sp^2}+4\pi r^2\eps^{-1}q_\lambda'(r/\eps)(\beta^{\mathrm{tan}}-d_\Gamma\theta)_i.
\end{equation}
Include the radial cutoffs in \eqref{eq: phase source splitting}, extending the geometric term by zero beyond the tube. Angular derivatives integrate to zero on $\Sp^2$. Integration by parts on $\Gamma$ with \eqref{eq: phase outer current} therefore expresses the contribution parallel to $\sigma$, up to the exponentially decaying terms, as
\begin{equation}\label{eq: phase outer projected source}
 \eps^{-2}F_{\mathrm{ph}}(\eps^2\Delta_\Gamma)^{-1}
 \int_{\sqrt\eps}^{\infty}
 A_{\eps^2\Delta_\Gamma}(r/\eps)\bigl(d_\Gamma^*j_r^{\mathrm{err}}\bigr)\,dr.
\end{equation}
Here $A_{\eps^2\Delta_\Gamma}(r/\eps)$ multiplies a Laplacian eigenfunction of eigenvalue $\kappa$ by $A_{\eps^2\kappa}(r/\eps)$, as in \eqref{eq: reduced spectral recovery}.

We separate the transition annulus from the exterior region. On
$\sqrt\eps\leq r\leq2\sqrt\eps$, the longitudinal connection difference
$\mathcal D_\eps$ in \eqref{eq: approximation gluing difference} satisfies,
by \eqref{eq: reduced conormal matching},
\begin{equation*}
 \|\mathcal D_\eps^\parallel\|_{C^{2,\tau}(\Gamma)}
 +r\|\nabla\mathcal D_\eps^\parallel\|_{C^{1,\tau}(\Gamma)}
 \lesssim r(1+|\log r|)^N+\frac{\eps}{r}.
\end{equation*}
Here the norms are uniform on the normal sphere; the transverse difference
and its derivatives are exponentially small. Since
$|\nabla\kappa_\eps|\lesssim r^{-1}$ on this annulus, the curvature terms
$d\kappa_\eps\wedge\mathcal D_\eps$ and
$\kappa_\eps d_{A_D}\mathcal D_\eps$ contribute at most
$C(r^2(1+|\log r|)^N+\eps)$ to the current and its tangential derivative.
The graph contribution is bounded by $C\|h\|_{C^{2,\tau}}(1+|\log r|)^N$.
Consequently, \eqref{eq: assumptions htheta} gives
\begin{equation*}
 \|d_\Gamma^*j_r^{\mathrm{err}}\|_{C^{0,\tau}(\Gamma)}
 \lesssim r^2(1+|\log\eps|)^N,
 \qquad \sqrt\eps\leq r\leq2\sqrt\eps.
\end{equation*}

For $r>2\sqrt\eps$, the curvature after undoing the phase gauge is
\begin{equation*}
 F_A=\sigma\bigl(f_0+db-dq_\eps\wedge d\bar\theta\bigr)
       +O\bigl(\eps^{-N}e^{-cr/\eps}\bigr),
\end{equation*}
with the same bounds for the derivatives used below. Write $g_0=g$ and
$J_0=J_h|_{h=0}$. The coclosedness of $f_0+db$, proved in
\eqref{eq: approximation exterior cancellation} and
\eqref{eq: noncompact exterior cancellation}, implies
\begin{equation*}
 d_\Gamma^*\left[
 (g_\Gamma)_{ij}r^2\int_{\Sp^2}
 J_0(f_0+db)^{jr}_{g_0}\,d\mathrm{vol}_{\Sp^2}\,dq^i
 \right]=0.
\end{equation*}
Indeed, the radial component of the divergence consists of a tangential
and an angular divergence, since $(f_0+db)^{rr}=0$; the latter integrates
to zero on the sphere. Thus only the change from $g_0,J_0$ to $g_h,J_h$
contributes to the geometric part of $d_\Gamma^*j_r^{\mathrm{err}}$.
The metric coefficients have $C^{1,\tau}(\Gamma)$ difference bounded by
$C\|h\|_{C^{2,\tau}}$, while \eqref{eq: curvature identity} and
\eqref{eq: reduced conormal matching} give
\begin{equation*}
 \|f_0+db\|_{C^{1,\tau}(\Gamma)}
 \lesssim r^{-2}(1+|\log r|)^N.
\end{equation*}
The factor $r^2$ in \eqref{eq: phase outer current} therefore bounds this
geometric contribution by
$C\|h\|_{C^{2,\tau}}(1+|\log r|)^N$.

The phase and reference contributions are bounded separately by $C\eps$.
In particular, the fixed cutoff in \eqref{eq: q_eps} gives
\begin{equation*}
 \partial_r q_\eps
 =\chi_\Gamma\eps^{-1}q_\lambda'(r/\eps)
  +(\partial_r\chi_\Gamma)(q_\lambda(r/\eps)-1),
 \qquad |\partial_r q_\eps|\lesssim\eps r^{-2}.
\end{equation*}
Moreover, $dq_\eps\wedge d\bar\theta
=\chi_\Gamma(\partial_rq_\eps)dr\wedge d_\Gamma\theta$;
the term containing $\partial_r\chi_\Gamma$ in $d\bar\theta$ vanishes
because $dr\wedge dr=0$. These formulas also control the fixed collar.
The reference term has size
$r^2\eps^{-1}|q_\lambda'(r/\eps)|\lesssim\eps$ by
\eqref{eq: approximation profile bounds}. We obtain
\begin{equation*}
 \|d_\Gamma^*j_r^{\mathrm{err}}\|_{C^{0,\tau}(\Gamma)}
 \lesssim
 (\eps+\|h\|_{C^{2,\tau}})(1+|\log\eps|)^N
 +\eps^{-N}e^{-cr/\eps},
 \qquad 2\sqrt\eps<r<\delta.
\end{equation*}
Beyond the tube only the subtracted reference term remains, with the
same $C\eps$ bound. Integrating the transition and exterior bounds gives
\begin{equation}\label{eq: phase outer current integral}
\begin{split}
 \int_{\sqrt\eps}^{\infty}r^{-2}
 \|d_\Gamma^*j_r^{\mathrm{err}}\|_{C^{0,\tau}(\Gamma)}\,dr
 &\lesssim(1+|\log\eps|)^N
 \left(\sqrt\eps+
 \frac{\eps+\|h\|_{C^{2,\tau}}}{\sqrt\eps}\right)\\
 &\lesssim\sqrt\eps(1+|\log\eps|)^{N+2}.
\end{split}
\end{equation}
Only $h,Dh,d\theta$ enter the current, so its tangential divergence uses
at most two parameter derivatives. Equations
\eqref{eq: phase outer radial test}, \eqref{eq: phase outer projected source}
and \eqref{eq: phase outer current integral}, together with
\Cref{lem: reduced coefficient operators} and the H\"older conversion
factor $\eps^{-\tau}$, prove \eqref{eq: phase source outer bound}.
Differentiating these same current terms before estimating them gives
\begin{equation*}
 \int_{\sqrt\eps}^{\infty}r^{-2}
 \|D_h(d_\Gamma^*j_r^{\mathrm{err}})[k]\|_{C^{0,\tau}(\Gamma)}\,dr
 \lesssim\eps^{-1/2}(1+|\log\eps|)^N\|k\|_{C^{2,\tau}},
\end{equation*}
since the differentiated coefficients involve only $k,Dk,D^2k$.

\smallskip\noindent We prove \eqref{eq: phase source mass bound}.
Equation \eqref{eq: noncompact mass normal cancellation} gives the linear normal cancellation for \eqref{eq: noncompact mass correction}. Its remaining normal curvature contribution pairs to zero by \eqref{eq: phase normal contractions}, and its scalar component pairs with $({\bf J}_{4,t})_\Phi=0$ by \eqref{eq: phase radial test form}.

The phase generator remains $\xi_\eps$, and the change in its covariant derivative is exactly
\begin{equation*}
 d_{A_\Gamma^0+C_\eps^A}\xi_\eps-d_{A_\Gamma^0}\xi_\eps
 =[C_\eps^A,\xi_\eps].
\end{equation*}
By \eqref{eq: noncompact mass correction}, its contribution to the rescaled principal source is $O(\eps^3(1+|y|)^Ne^{-c|y|})$, hence $O(\eps)$ after division by $\eps^2$. Terms with second tangential derivatives of the mass correction pair to zero by parity. The remaining tangential derivatives and the trace-lift defect satisfy \eqref{eq: noncompact mass derivative bound} and \eqref{eq: noncompact phase profile defect}; the latter contributes to the normal coefficient through the errors in \eqref{eq: reduced graph metric bounds}. Their normalised bound is $C(\eps+\|h\|_{C^{2,\tau}})$.

On the matching annulus, \eqref{eq: noncompact approximation matching} controls the scalar difference $\eps(H-(2r)^{-1}-H_\Gamma)\sigma$, which commutes with the connection parallel to $\sigma$; the connection mass correction is exponentially small. Equations \eqref{eq: noncompact exterior cancellation} and \eqref{eq: noncompact outer shells}, with the same H\"older conversion used for \eqref{eq: phase source outer bound}, followed by \eqref{eq: assumptions htheta}, give
\begin{equation*}
 \|R_{\mathrm m}\|_{C^{0,\tau}}
 \lesssim\eps^{-\tau}(1+|\log\eps|)^N(\eps+\|h\|_{C^{2,\tau}})
 \lesssim\eps^{1-\tau}(1+|\log\eps|)^{N+2},
\end{equation*}
which proves \eqref{eq: phase source mass bound} after increasing $N$.
\end{proof}

\subsection{Contribution of the correction}
\begin{lemma}\label{lem: reduced complementary size}
For the solution of \eqref{eq: complementary nonlinear system}, the remainders in \eqref{eq: reduced complementary decomposition} satisfy
\begin{equation}\label{eq: reduced complementary bounds}
\frac{\|r_{\mathrm{tr}}^{\,1}\|_{C^{0,\tau}}}{\eps|\log\eps|^2}
+\|r_{\mathrm{ph}}^{\,1}\|_{C^{0,\tau}}
\lesssim (1+|\log\eps|)^N
\bigl(a_\eps\eps^{-1-\tau}+a_\eps^2\eps^{-2-\tau}\bigr),
\end{equation}
where $a_\eps$ is the approximation error rate defined in \eqref{eq: error rate for contraction}.
\end{lemma}
\begin{proof}
Write ${\bf\Psi}=(\Psi,B)={\bf\Psi}_{h,\theta}$, $\mathfrak b=\mathfrak b_{\eps,h,\theta}$. Equation \eqref{eq: complementary solution bound} gives
\begin{equation*}
 \|{\bf\Psi}\|_{X_{\eps,h,\theta}}\lesssim a_\eps,
 \qquad \|\mathfrak b\|_{C^{0,\tau}}
 \lesssim \eps^{-\tau}a_\eps.
\end{equation*}
Use $\lambda=\lambda_\eps$ in the compact case and $\lambda=0$ in the noncompact case.

We compare \eqref{eq: complementary nonlinear system} with \eqref{eq: reduced reference system} using \Cref{def: rescaling convention}. Choose the fixed physical cutoff $\zeta$ equal to one for $r\leq\delta/2$ and zero for $r\geq3\delta/4$, and write $\zeta_\eps(y)=\zeta(\eps|y|)$. Define the normal sources
\begin{equation}\label{eq: reduced complementary source terms}
\begin{split}
 {\bf f}_{\mathrm{err}}&=\eps^2\zeta_\eps\widehat{{\bf S}_\eps(\Phi_{h,\theta},A_{h,\theta})},\\
 {\bf f}_{\mathrm{nl}}&=\eps^2\zeta_\eps\widehat{\mathcal N_{\eps,h,\theta}({\bf\Psi})},\\
 {\bf e}_{\mathrm{lin}}
 &=\zeta_\eps\bigl(\eps^2\widehat{{\bf L}_{\eps,h,\theta}{\bf\Psi}}
       -{\bf L}^{\mathrm{prod}}_{\eps,\lambda}\widehat{\bf\Psi}\bigr),\\
 {\bf e}_{\mathrm{cut}}
 &=-[{\bf L}^{\mathrm{prod}}_{\eps,\lambda},\zeta_\eps]\widehat{\bf\Psi},\\
 {\bf e}_{\mathrm{g}}
 &=-\eps\zeta_\eps\widehat{{\bf G}_{h,\theta}\Pi_h{\bf G}_{h,\theta}^*{\bf\Psi}},\\
 {\bf e}_{\mathrm{jac}}
 &=\bigl(J_h(q,\eps y)^{-1}-1\bigr)
       \sum_{i=1}^4\mathfrak b_i(q){\bf E}_i(y).
\end{split}
\end{equation}
The subscripts $\mathrm{err}$ and $\mathrm{nl}$ denote the approximation error and the nonlinear term; $\mathrm{lin}$ denotes the linear operator difference, $\mathrm{cut}$ the cutoff commutator\footnote{For a linear operator $L$ and a scalar function $\zeta$, we write $[L,\zeta]v=L(\zeta v)-\zeta Lv$.} on the outer collar, $\mathrm g$ the gauge term, and $\mathrm{jac}$ the error from the Jacobian factor in \eqref{eq: explicit residual injection}. The linear source uses the full physical correction, including the tangential components of $B$. All cut-off sources are extended by zero to $\Gamma\times\R^3$.

We claim the coefficient identity
\begin{equation}\label{eq: reduced complementary coefficient identity}
 \mathfrak b-\mathfrak b^{\,\mathrm{prod}}
       [{\bf S}_\eps(\Phi_{h,\theta},A_{h,\theta})]
 =\mathcal C\bigl({\bf f}_{\mathrm{nl}}+{\bf e}_{\mathrm{lin}}
       +{\bf e}_{\mathrm{cut}}+{\bf e}_{\mathrm{g}}
       -{\bf e}_{\mathrm{jac}}\bigr)
\end{equation}
and the following estimates in which $\mathcal C$ is the four-component source operator of \Cref{lem: reduced coefficient operators}:
\begin{subequations}\label{eq: reduced complementary constituent bounds}
\begin{align}
 \|\mathcal C{\bf e}_{\mathrm{lin}}\|_{C^{0,\tau}}
 &\lesssim a_\eps\eps^{1-\tau}(1+|\log\eps|)^N,
 \label{eq: reduced correction linear bound}\\
 \|\mathcal C{\bf e}_{\mathrm{cut}}\|_{C^{0,\tau}}
       +\|\mathcal C{\bf e}_{\mathrm{g}}\|_{C^{0,\tau}}
 &\lesssim \eps a_\eps,
 \label{eq: reduced correction collar gauge bound}\\
 \|\mathcal C{\bf f}_{\mathrm{nl}}\|_{C^{0,\tau}}
 &\lesssim \eps^{-\tau}a_\eps^2,
 \label{eq: reduced correction nonlinear bound}\\
 \|\mathcal C{\bf e}_{\mathrm{jac}}\|_{C^{0,\tau}}
 &\lesssim a_\eps\eps^{2-\tau}(1+|\log\eps|)^N.
 \label{eq: reduced correction injection bound}
\end{align}
\end{subequations}
By \eqref{eq: reduced complementary decomposition}, the translation remainder divides this coefficient by $\eps$, whereas the phase remainder divides it by $\eps^2$. The inverse multipliers are uniformly bounded by \Cref{lem: reduced coefficient operators}; thus, factoring in the further $\eps|\log\eps|^2$ normalisation in the translation term, \eqref{eq: reduced complementary coefficient identity} and \eqref{eq: reduced complementary constituent bounds} imply \eqref{eq: reduced complementary bounds}.

\smallskip\noindent
We prove \eqref{eq: reduced complementary coefficient identity}.
The transported vectors $\nu_{a,h}$ are orthonormal, since both transports in their definition preserve the metric. Thus \eqref{eq: explicit residual injection} gives the exact identity
\begin{equation}\label{eq: reduced rescaled injection}
 \eps^2\widehat{\mathfrak B_{\eps,h,\theta}^*\mathfrak b}
 =J_h(q,\eps y)^{-1}\sum_{i=1}^4\mathfrak b_i(q){\bf E}_i(y).
\end{equation}
Since $\zeta_\eps=1$ on the support of every ${\bf E}_i$, the actual constraint \eqref{eq: fibre constraints} gives
\begin{equation}\label{eq: reduced cutoff constraint}
 \langle{\bf E}_i,\zeta_\eps\widehat{\bf\Psi}\rangle_2
 =\langle{\bf E}_i,\widehat{\bf\Psi}\rangle_2=0,
 \qquad i=1,\ldots,4.
\end{equation}
Substituting \eqref{eq: restricted linearised operator} and \eqref{eq: reduced rescaled injection} in \eqref{eq: complementary nonlinear system}, using \eqref{eq: reduced complementary source terms}, yields
\begin{equation}\label{eq: reduced correction product equation}
 {\bf L}^{\mathrm{prod}}_{\eps,\lambda}(\zeta_\eps\widehat{\bf\Psi})
 =-{\bf f}_{\mathrm{err}}-{\bf f}_{\mathrm{nl}}-{\bf e}_{\mathrm{lin}}
       -{\bf e}_{\mathrm{cut}}-{\bf e}_{\mathrm{g}}
       +{\bf e}_{\mathrm{jac}}+\sum_{i=1}^4\mathfrak b_i{\bf E}_i.
\end{equation}
Applying \eqref{eq: reduced spectral recovery} to \eqref{eq: reduced correction product equation} with \eqref{eq: reduced cutoff constraint} proves \eqref{eq: reduced complementary coefficient identity}.

\smallskip

Next, we prove \eqref{eq: reduced correction linear bound}. Use the physical annuli $\Omega_j$ and radii $r_j$ from \Cref{sec: norms}, and their rescaled radii $t_j=r_j/\eps$ as in the proof of \eqref{eq: residual injection estimate proof}. The second-order metric coefficients in \eqref{eq: reduced complementary source terms} are $g_h^{ab}-\delta^{ab}$, $2g_h^{ia}$ and $g_h^{ij}-g_\Gamma^{ij}$, as in \eqref{eq: multiplier operator difference}; the normal connection is retained in \eqref{eq: reduced product operators}. Their scaled norms are bounded by $C(r_j+\|h\|_{C^{2,\tau}})$, while their differentiated coefficients and the remaining first-order connection coefficients are bounded by $C(1+\|h\|_{C^{2,\tau}}/r_j)$. In \eqref{eq: model linearised operator}, the remaining terms are the differences of
\begin{equation*}
\begin{gathered}
 \ad_B^*F_A,\qquad 2\ad_B^*d_A\Phi,\qquad -2\eps^{-2}[d_A\Phi,\Psi],\\
 -\eps^{-2}[\Phi,[\Phi,\Psi]],\qquad -\eps^{-2}[\Phi,[\Phi,B]],\\
 \frac{2\eps}{\delta_\eps}\bigl((|\Phi|^2-1)\Psi+2\langle\Phi,\Psi\rangle\Phi\bigr),
\end{gathered}
\end{equation*}
evaluated at the approximation and at the rescaled model; the last term is absent in the noncompact case. These give
\begin{equation}\label{eq: reduced complementary local linear bound}
\left\|\left(\widehat{{\bf L}_{\eps,h,\theta}{\bf\Psi}} -\eps^{-2}{\bf L}^{\mathrm{prod}}_{\eps,\lambda}\widehat{\bf\Psi}\right)(q,z/\eps)\right\|_{\tau,r_j,\Omega_j}\lesssim a_\eps(1+|\log\eps|)^N\frac{r_j+\|h\|_{C^{2,\tau}}}{r_j^2}.
\end{equation}
Equation \eqref{eq: reduced complementary local linear bound} also holds on annuli meeting the matching region. By \eqref{eq: s_eps}, there,
\begin{equation*}
 r_j\simeq\sqrt\eps,\qquad
 \|\nabla^k\kappa_\eps\|_{\tau,r_j,\Omega_j}
 \lesssim\eps^{-k/2}\simeq r_j^{-k},\qquad k=1,2.
\end{equation*}
Thus the weights $r_j^k$ absorb derivatives of the cutoff, while \eqref{eq: approximation regular matching} and \eqref{eq: reduced conormal matching} bound the differentiated connection. Compact scalar cutoff errors obey \eqref{eq: approximation profile bounds}; the noncompact scalar matching terms obey \eqref{eq: reduced conormal matching}. On the fixed outer collar all scales are bounded below. Thus the regions in \eqref{eq: reduced complementary local linear bound} exhaust the support of $\zeta$, in both cases.

Passing from the scaled tangential seminorm to $C^{0,\tau}(\Gamma)$ costs $Cr_j^{-\tau}$. Thus \eqref{eq: assumptions htheta} and \eqref{eq: reduced complementary local linear bound} give, on $|y|\simeq t_j\geq1$,
\begin{equation}\label{eq: reduced correction rescaled linear bound}
 \|{\bf e}_{\mathrm{lin}}(\cdot,y)\|_{C^{0,\tau}(\Gamma)}
 \lesssim a_\eps\eps^{1-\tau}(1+|\log\eps|)^Nt_j^{-1-\tau}.
\end{equation}
Applying \eqref{eq: reduced coefficient weighted bound} to \eqref{eq: reduced correction rescaled linear bound} uses $\int_1^\infty R^{-1-\tau}\,dR<\infty$; the factor $R^2$ in normal volume cancels the weight $R^{-2}$ of the test. The core has the same bound with $t_j=1$. This proves \eqref{eq: reduced correction linear bound}.

\smallskip\noindent
Next, we show \eqref{eq: reduced correction collar gauge bound}.
The radial cutoff is parallel under $\nabla^\top$, and hence \eqref{eq: reduced product operators} gives
\begin{equation}\label{eq: reduced correction cutoff commutator}
 [{\bf L}^{\mathrm{prod}}_{\eps,\lambda},\zeta_\eps]\widehat{\bf\Psi}
 =(\Delta_y\zeta_\eps)\widehat{\bf\Psi}
       -2(\partial_{y^a}\zeta_\eps)D_a\widehat{\bf\Psi}.
\end{equation}
Here $D_a$ is defined by \eqref{eq: model normal covariant derivative} and acts componentwise on normal pairs. On $\delta/2<r<3\delta/4$, \eqref{eq: sol norm} and \eqref{eq: complementary solution bound} bound the $C^{0,\tau}(\Gamma)$ size of \eqref{eq: reduced correction cutoff commutator} by $Ca_\eps\eps^2$. Its pairing with $\langle y\rangle^{-2}$ over $|y|\simeq\eps^{-1}$ is therefore $O(\eps a_\eps)$.

For the gauge term, \eqref{eq: phase projection} gives
\begin{equation}\label{eq: reduced correction gauge projection}
 \Pi_h{\bf G}_{h,\theta}^*{\bf\Psi}
 =\frac{\eps\int_M\langle B,\chi d_{A_{h,\theta}}\xi_{\eps,h}+\xi_{\eps,h}d\chi\rangle\,d\vol_g}
 {\int_M\chi|\xi_{\eps,h}|^2\,d\vol_g}\,\xi_{\eps,h}.
\end{equation}
Estimates \eqref{eq: compact phase generator bounds} and \eqref{eq: noncompact gauge projection numerator bound}, together with $\eps\|B\|_\infty\lesssim\|{\bf\Psi}\|_{X_{\eps,h,\theta}}\lesssim a_\eps$, bound the scalar factor by $C\eps a_\eps$ in the compact case and $Ca_\eps$ in the noncompact case, respectively. Since ${\bf G}_{h,\theta}\xi_{\eps,h}=(0,d_{A_{h,\theta}}\xi_{\eps,h})$ and
\begin{equation*}
 \|\widehat{{\bf G}_{h,\theta}\xi_{\eps,h}}(\cdot,y)
       \|_{C^{0,\tau}(\Gamma)}\lesssim \langle y\rangle^{-2}
 \quad\hbox{on }\supp\zeta_\eps,
\end{equation*}
Equation \eqref{eq: reduced coefficient weighted bound}, applied to \eqref{eq: reduced correction gauge projection}, yields $\|\mathcal C{\bf e}_{\mathrm g}\|_{C^{0,\tau}}\lesssim \eps a_\eps$. Together with the bound from \eqref{eq: reduced correction cutoff commutator}, this proves \eqref{eq: reduced correction collar gauge bound}.

\smallskip\noindent
We turn our attention to \eqref{eq: reduced correction nonlinear bound}.
Given any ${\bf F}$, \eqref{eq: source norm} gives on the rescaled annulus $|y|\simeq t_j$ the bound
\begin{equation}\label{eq: reduced correction source L2 rescaling}
 \|\eps^2\zeta_\eps\widehat{\bf F}
       \|_{C^{0,\tau}(\Gamma;L^2_y(|y|\simeq t_j))}
 \lesssim \eps^{-\tau}t_j^{-1/2-\tau}
       \|{\bf F}\|_{Y_{\eps,h,\theta}}.
\end{equation}
Indeed, \Cref{def: rescaling convention} and \eqref{eq: source norm} give $t_j^{-2}$, normal volume gives $t_j^{3/2}$, and the tangential seminorm gives $\eps^{-\tau}t_j^{-\tau}$. The geometric series in $t_j$ converges, including the core and fixed collar. Equations \eqref{eq: reduced correction source L2 rescaling} and \eqref{eq: reduced coefficient L2 bound}, with \Cref{lem: quadratic estimate}, therefore imply \eqref{eq: reduced correction nonlinear bound}.

\smallskip\noindent
We prove \eqref{eq: reduced correction injection bound}.
On $\supp\rho_\eps$ the shift is $z\mapsto z+h(q)$. Its coordinate Jacobian has determinant one, and hence $J_h(q,z)=J_0(q,z+h(q))$. By \eqref{eq: reduced graph metric bounds},
\begin{equation}\label{eq: reduced correction Jacobian bound}
 \|J_h(\cdot,z)^{-1}-1\|_{C^{0,\tau}(\Gamma)}
 \lesssim (|z|^2+\|h\|_{C^{2,\tau}}^2).
\end{equation}
By \eqref{eq: normal weighted tests}, \eqref{eq: monopole translations}, \eqref{eq: monopole phase mode} and \eqref{eq: approximation profile bounds}, $|{\bf E}_i(y)|\lesssim \langle y\rangle^{-4}$, so both $\|{\bf E}_i\|_2$ and $\||y|^2{\bf E}_i\|_2$ are bounded uniformly. Equations \eqref{eq: reduced correction Jacobian bound} and \eqref{eq: complementary solution bound} give
\begin{equation*}
 \|{\bf e}_{\mathrm{jac}}\|_{C^{0,\tau}(\Gamma;L^2_y)}
 \lesssim (\eps^2+\|h\|_{C^{2,\tau}}^2)\|\mathfrak b\|_{C^{0,\tau}}
 \lesssim a_\eps\eps^{2-\tau}(1+|\log\eps|)^N.
\end{equation*}
Estimate \eqref{eq: reduced coefficient L2 bound} proves \eqref{eq: reduced correction injection bound}, completing the proof of \eqref{eq: reduced complementary bounds}. In the noncompact case all coefficient sources were cut off in a fixed tube; \eqref{eq: noncompact exterior solution norm} supplies the collar bounds, and $\chi$ has compact support in \eqref{eq: reduced correction gauge projection}. No integration at infinity has been used.

Finally, \eqref{eq: error rate for contraction}, together with \eqref{eq: deltaeps} in the compact case, gives
\begin{equation*}
 a_\eps\eps^{-1-\tau}
 \lesssim \bigl(\eps^{\tau'-\tau}
       +\eps^{1/2-\tau}(1+|\log\eps|)\bigr).
\end{equation*}
In the noncompact case, the first term on the right is absent, directly by \eqref{eq: error rate for contraction}. Every fixed logarithmic multiple tends to zero because $\tau<\tau'<1/2$. Also $a_\eps^2\eps^{-2-\tau} =(a_\eps/\eps)(a_\eps\eps^{-1-\tau})$ and $a_\eps/\eps\to0$. Thus the right-hand side of \eqref{eq: reduced complementary bounds} tends to zero in both cases.
\end{proof}

\subsection{Parameter differences}
\begin{lemma}[Differences of the approximation remainder]\label{lem: reduced source differences}
For any two parameter pairs $(h_1,\theta_1)$ and $(h_2,\theta_2)$ satisfying \eqref{eq: assumptions htheta},
\begin{subequations}\label{eq: reduced source differences}
\begin{equation}\label{eq: reduced translation source difference}
\|r_{\mathrm{tr}}^{\,0}(h_1,\theta_1)-r_{\mathrm{tr}}^{\,0}(h_2,\theta_2)\|_{C^{0,\tau}}\lesssim  \omega_\eps\left(\|h_1-h_2\|_{C^{2,\tau}}+\eps|\log\eps|^2\|\theta_1-\theta_2\|_{C^{2,\tau}}\right),
\end{equation}
\begin{equation}\label{eq: reduced phase source difference}
\|r_{\mathrm{ph}}^{\,0}(h_1,\theta_1)-r_{\mathrm{ph}}^{\,0}(h_2,\theta_2)\|_{C^{0,\tau}}\lesssim  \omega_\eps\left(\frac{\|h_1-h_2\|_{C^{2,\tau}}}{\eps|\log\eps|^2}+\|\theta_1-\theta_2\|_{C^{2,\tau}}\right).
\end{equation}
\end{subequations}
\end{lemma}
\begin{proof}
We show independently the translation and the phase contributions of the two estimates.

\smallskip

We begin with the translation contribution in \eqref{eq: reduced translation source difference}, thus fixing $\theta$. For $h_t=h_2+t(h_1-h_2)$, $0\leq t\leq1$, use
\begin{equation*}
g_{h_1}^{ab}-g_{h_2}^{ab}=\int_0^1\partial_t g_{h_t}^{ab}\,dt,\qquad J_{h_1}-J_{h_2}=\int_0^1\partial_t J_{h_t}\,dt,
\end{equation*}
and similarly for the mixed indices $g_{h_1}^{ja}-g_{h_2}^{ja}$.
Differentiating the smooth Fermi expansions underlying \eqref{eq: reduced graph metric bounds} gives
\begin{equation}\label{eq: graph metric parameter derivatives}
\begin{aligned}
 \|\partial_t g_{h_t}^{ab}\|_{C^{0,\tau}(\Gamma)}
 +\|\partial_t J_{h_t}\|_{C^{0,\tau}(\Gamma)}
 &\lesssim\bigl(r+\|h_t\|_{C^{2,\tau}}\bigr)\|h_1-h_2\|_{C^{2,\tau}},\\
 \|\partial_t g_{h_t}^{ja}\|_{C^{0,\tau}(\Gamma)}
 &\lesssim\|h_1-h_2\|_{C^{2,\tau}}.
\end{aligned}
\end{equation}
Integrating \eqref{eq: graph metric parameter derivatives} and using \eqref{eq: assumptions htheta}, the differences of the normal metric and density defects are bounded in $C^{0,\tau}(\Gamma)$ by $C(r+\eps|\log\eps|^2)\|h_1-h_2\|_{C^{2,\tau}}$, whereas the mixed defects have bound $C\|h_1-h_2\|_{C^{2,\tau}}$. The same estimates hold with one tangential derivative for coefficients depending only on $h,Dh$. Terms containing $D^2h$ are subtracted in $C^{0,\tau}(\Gamma)$: keep the difference of the second derivatives in one term, and the difference of their coefficients in the other.

For the graph contribution to \eqref{eq: reduced translation source difference}, subtract \eqref{eq: Jacobi from error} for the two graphs and compare with the prescribed tests as in the proof of \eqref{eq: translation zero parameter claim}. The resulting radial integrals satisfy
\begin{equation}\label{eq: graph difference radial integrals}
\int_0^{\sqrt\eps}\frac{r^2}{(\eps+r)^3}\,dr\lesssim (1+|\log\eps|),\quad \int_0^{\sqrt\eps}\frac{r^2}{(\eps+r)^4}\,dr\lesssim \eps^{-1},\quad \int_0^{\sqrt\eps}\frac{r^2}{(\eps+r)^5}\,dr\lesssim \eps^{-2}.
\end{equation}
Equations \eqref{eq: assumptions htheta} and \eqref{eq: graph difference radial integrals} give an inner bound of $C\eps(1+|\log\eps|)^N\|h_1-h_2\|_{C^{2,\tau}}$. The scalar density term is no larger, using \eqref{eq: approximation profile bounds} in the compact case and \eqref{eq: reduced conormal matching} in the noncompact case.

Estimate \eqref{eq: reduced coefficient comparison bound} gives a factor $\sqrt\eps$ for terms involving only first derivatives of the parameters. For second-derivative terms, their coefficient is $O(r+\eps|\log\eps|^2)$, and the direct $L^2_y$ estimate \eqref{eq: reduced coefficient L2 bound} gives the same bound, up to logarithms. On $r\geq \sqrt\eps$, the source difference, with its connection component multiplied by $\eps$, has bound $C\eps(1+|\log\eps|)^Nr^{-2}\|h_1-h_2\|_{C^{2,\tau}}$.
Its normalised coefficient is therefore at most
\begin{equation}\label{eq: translation graph outer difference}
C\eps^{1/4}(1+|\log\eps|)^N\|h_1-h_2\|_{C^{2,\tau}}.
\end{equation}

\smallskip

For the translation contribution to \eqref{eq: reduced phase source difference}, apply \eqref{eq: phase normal contractions} to each graph and differentiate along $h_t$ the metric coefficients estimated in \eqref{eq: phase inner transverse integral} and \eqref{eq: phase inner current integral}. Their derivatives are bounded, respectively, by
\begin{equation*}
\begin{gathered}
 C\bigl(r+\|h_t\|_{C^{2,\tau}}\bigr)\|h_1-h_2\|_{C^{2,\tau}},\\
 C\bigl(r^2+r\|h_t\|_{C^{2,\tau}}+\|h_t\|_{C^{2,\tau}}^2\bigr)\|h_1-h_2\|_{C^{2,\tau}}.
\end{gathered}
\end{equation*}
Integrate in $t$ and then against the respective radial weights $r^2e^{-cr/\eps}(\eps+r)^{-4}$ and $r^2(\eps+r)^{-4}$ in \eqref{eq: phase inner transverse integral} and \eqref{eq: phase inner current integral}. Using \eqref{eq: assumptions htheta} gives $C(1+|\log\eps|)^N\|h_1-h_2\|_{C^{2,\tau}}$. The conversion to the prescribed covectors has the same bound because the frozen $D^2h$ term vanishes by \eqref{eq: phase normal contractions}. The current estimate \eqref{eq: phase outer current integral} contributes $C(1+|\log\eps|)^N\eps^{-1/2}\|h_1-h_2\|_{C^{2,\tau}}$. Including the conversion from scaled to ordinary H\"older norms gives the graph variation of the phase remainder:
\begin{equation}\label{eq: phase graph source difference}
 \|r_{\mathrm{ph}}^{\,0}(h_1,\theta)-r_{\mathrm{ph}}^{\,0}(h_2,\theta)\|_{C^{0,\tau}}
 \lesssim\eps^{-1/2-\tau}(1+|\log\eps|)^N\|h_1-h_2\|_{C^{2,\tau}}.
\end{equation}
Multiplication of \eqref{eq: phase graph source difference} by $\eps|\log\eps|^2$ makes this coefficient tend to zero, as required for the $h$-difference in \eqref{eq: reduced phase source difference}.

\smallskip

We now vary $\theta$ at fixed $h$. First we estimate the change in $r_{\mathrm{tr}}^{\,0}$ to complete \eqref{eq: reduced translation source difference}; then we estimate the change in $r_{\mathrm{ph}}^{\,0}$ to complete \eqref{eq: reduced phase source difference}. We retain the dependence on the radius $K$. By \eqref{eq: shifted phase approximation} and \eqref{eq: noncompact shifted approximation}, after undoing the phase gauge, the connection is $A_\Gamma-\xi_\eps d\bar\theta$. Since $[\xi_\eps,\Phi_\Gamma]=0$ and $[\xi_\eps d\bar\theta\wedge\xi_\eps d\bar\theta]=0$, the known residual is a polynomial of degree at most two in $\theta$. Equations \eqref{eq: reduced spectral recovery} and \eqref{eq: reduced source decomposition} preserve this property.

For the translation remainder, we first estimate $D_\theta r_{\mathrm{tr}}^{\,0}(0,\theta)$ and then use the graph estimate to control its change with $h$. Let $(S_\Phi(h,\theta),S_A(h,\theta))$ denote the components of the known residual after undoing both gauges. We differentiate these components with respect to $\theta$ in a direction $\vartheta\in C^{2,\tau}(\Gamma)$. Equations \eqref{eq: approximation horizontal cancellation} and \eqref{eq: approximation profile bounds} give, in the inner horizontal frame,
\begin{equation*}
\begin{gathered}
 [\xi_\eps,D_j\Phi]=0,\qquad D_j\xi_\eps=0,\\
 \|D_a\xi_\eps(\cdot,z)\|_{C^{0,\tau}(\Gamma)}\lesssim\frac{\eps}{(\eps+r)^2},\qquad
 \|D_bD_a\xi_\eps(\cdot,z)\|_{C^{0,\tau}(\Gamma)}\lesssim\frac{\eps}{(\eps+r)^3}.
\end{gathered}
\end{equation*}
Differentiating \eqref{eq: approximation phase curvature}, the normal connection equation contains $D_a\xi_\eps$ times at most two tangential derivatives of $\vartheta$, and brackets of $\xi_\eps$ with $D_a\xi_\eps$ times $(\beta^{\mathrm{tan}}-d\theta)$ and $d\vartheta$. The identity $D_aD_a\xi_\eps=0$ from \eqref{eq: qlambda} cancels the term $(\partial_j\vartheta)D_aD_a\xi_\eps$. The remaining second normal derivatives are multiplied by the $O(r^2)$ metric defects; their first derivatives are $O(r)$. The scalar variation contains only mixed metric coefficients and the exponentially decaying brackets $[\xi_\eps,D_a\Phi]$.

On the matching annulus and in the exterior, the neutral part of $d_A\xi_\eps$ is $dq_\eps\,\sigma$. It commutes with the neutral matching connection, so differentiating in phase introduces no derivative of its matching cutoff. Charged terms, including cutoff derivatives, are exponentially small. At the fixed outer cutoff, $dq_\eps\wedge d\chi_\Gamma=0$. In the noncompact case the normal harmonic defect is $O(\eps^{-1}(1+r/\eps)^Ne^{-cr/\eps})$, 
by \eqref{eq: noncompact phase profile defect}, and obeys the same bound below. Thus
\begin{equation}\label{eq: reference phase derivative source}
\|D_\theta S_\Phi(0,\theta)[\vartheta](\cdot,z)\|_{C^{0,\tau}} +\eps\|D_\theta S_A(0,\theta)[\vartheta](\cdot,z)\|_{C^{0,\tau}}\lesssim \bigl(1+\|\theta\|_{C^{2,\tau}}\bigr) \frac{\eps^2\|\vartheta\|_{C^{2,\tau}}}{(\eps+r)^2}.
\end{equation}
Apply \eqref{eq: reduced coefficient weighted bound}, with the rescaling of \Cref{def: rescaling convention}. The radial integral is
\begin{equation*}
 \eps^2\int_0^\delta\frac{r^2}{(\eps+r)^4}\,dr\lesssim \eps.
\end{equation*}
Consequently,
\begin{equation}\label{eq: reference translation phase derivative}
 \|D_\theta r_{\mathrm{tr}}^{\,0}(0,\theta)[\vartheta]\|_{C^{0,\tau}}
 \lesssim \eps(1+\|\theta\|_{C^{2,\tau}})\|\vartheta\|_{C^{2,\tau}}.
\end{equation}

The implicit constants in \eqref{eq: reference phase derivative source} and \eqref{eq: reference translation phase derivative} are independent of $K$.

Now we include $h$. Apply the bounds from \eqref{eq: graph difference radial integrals} and \eqref{eq: translation graph outer difference} to the quadratic polynomial
$r_{\mathrm{tr}}^{\,0}(h,\theta+t\vartheta)-r_{\mathrm{tr}}^{\,0}(0,\theta+t\vartheta)$,
with $\|\vartheta\|_{C^{2,\tau}}\leq1$. Since its degree is at most two,
\begin{equation}\label{eq: translation phase polynomial identity}
\begin{split}
 D_\theta r_{\mathrm{tr}}^{\,0}(h,\theta)[\vartheta]
 -D_\theta r_{\mathrm{tr}}^{\,0}(0,\theta)[\vartheta]
 &=\tfrac12\bigl(r_{\mathrm{tr}}^{\,0}(h,\theta+\vartheta)-r_{\mathrm{tr}}^{\,0}(0,\theta+\vartheta)\bigr)\\
 &-\tfrac12\bigl(r_{\mathrm{tr}}^{\,0}(h,\theta-\vartheta)-r_{\mathrm{tr}}^{\,0}(0,\theta-\vartheta)\bigr).
\end{split}
\end{equation}
Since
\begin{equation*}
 \|\theta\pm\vartheta\|_{C^{2,\tau}}
 \leq\|\theta\|_{C^{2,\tau}}+\|\vartheta\|_{C^{2,\tau}}
 \leq K+1,
\end{equation*}
the bounds from \eqref{eq: graph difference radial integrals} and \eqref{eq: translation graph outer difference} apply at $\theta+\vartheta$ and $\theta-\vartheta$ and give
\begin{equation}\label{eq: translation phase graph derivative}
 \|D_\theta r_{\mathrm{tr}}^{\,0}(h,\theta)
       -D_\theta r_{\mathrm{tr}}^{\,0}(0,\theta)\|_{C^{2,\tau}\to C^{0,\tau}}
 \lesssim \eps^{1/4}(1+|\log\eps|)^N\|h\|_{C^{2,\tau}}.
\end{equation}
Combining \eqref{eq: translation phase graph derivative} with \eqref{eq: reference translation phase derivative} gives
\begin{equation}\label{eq: small translation phase coefficient}
 \frac{\|D_\theta r_{\mathrm{tr}}^{\,0}(h,\theta)\|_{C^{2,\tau}\to C^{0,\tau}}}
 {\eps|\log\eps|^2}
 \lesssim |\log\eps|^{-2}+\eps^{1/4}(1+|\log\eps|)^N
 \to 0
\end{equation}
for every fixed $K$. Set
\begin{equation}\label{eq: phase parameter segment}
 \theta_s\coloneqq\theta_2+s(\theta_1-\theta_2),\qquad
 0\leq s\leq1,\qquad \|\theta_s\|_{C^{2,\tau}}\leq K.
\end{equation}
Then
\begin{equation*}
 r_{\mathrm{tr}}^{\,0}(h,\theta_1)-r_{\mathrm{tr}}^{\,0}(h,\theta_2)
 =\int_0^1D_\theta r_{\mathrm{tr}}^{\,0}(h,\theta_s)[\theta_1-\theta_2]\,ds.
\end{equation*}
Estimate \eqref{eq: small translation phase coefficient}, with the bounds from \eqref{eq: graph difference radial integrals} and \eqref{eq: translation graph outer difference}, therefore proves \eqref{eq: reduced translation source difference}.

\smallskip

For the phase remainder in \eqref{eq: reduced phase source difference}, apply the quadratic-polynomial identity used in \eqref{eq: translation phase polynomial identity} to $t\mapsto r_{\mathrm{ph}}^{\,0}(h,\theta+t\vartheta)$ and use \eqref{eq: reduced phase source bounds} for $\|\theta\|_{C^{2,\tau}}\leq K+1$. With $\theta_s$ from \eqref{eq: phase parameter segment},
\begin{equation*}
 r_{\mathrm{ph}}^{\,0}(h,\theta_1)-r_{\mathrm{ph}}^{\,0}(h,\theta_2)
 =\int_0^1D_\theta r_{\mathrm{ph}}^{\,0}(h,\theta_s)[\theta_1-\theta_2]\,ds.
\end{equation*}
Equation \eqref{eq: reduced phase source bounds} gives
\begin{equation}\label{eq: phase phase source difference}
 \|r_{\mathrm{ph}}^{\,0}(h,\theta_1)-r_{\mathrm{ph}}^{\,0}(h,\theta_2)\|_{C^{0,\tau}}
 \lesssim \eps^{1/2-\tau}(1+|\log\eps |)^N\|\theta_1-\theta_2\|_{C^{2,\tau}}.
\end{equation}
Adding \eqref{eq: phase graph source difference} and \eqref{eq: phase phase source difference} establishes \eqref{eq: reduced phase source difference}. We may take $\omega_\eps\lesssim  \bigl(|\log\eps|^{-2}+(1+|\log\eps|)^N(\eps^{1/4}+\eps^{1/2-\tau})\bigr)\to0$ after increasing the fixed integer $N$.
\end{proof}

\begin{lemma}[Differences of the correction contribution]\label{lem: reduced complementary differences}
For any two parameter pairs $(h_1,\theta_1)$ and $(h_2,\theta_2)$ satisfying \eqref{eq: assumptions htheta},
\begin{equation}\label{eq: reduced complementary differences}
\begin{split}
\frac{\|r_{\mathrm{tr}}^{\,1}(h_1,\theta_1)-r_{\mathrm{tr}}^{\,1}(h_2,\theta_2)\|_{C^{0,\tau}}}{\eps|\log\eps|^2}&+\|r_{\mathrm{ph}}^{\,1}(h_1,\theta_1)-r_{\mathrm{ph}}^{\,1}(h_2,\theta_2)\|_{C^{0,\tau}}\\
&  \lesssim a_\eps\eps^{-1-\tau}(1+|\log\eps|)^N d_\eps\bigl((h_1,\theta_1),(h_2,\theta_2)\bigr),
\end{split}
\end{equation}
\end{lemma}
\begin{proof}
The corrections are defined on $M$, but their fibre constraints \eqref{eq: fibre constraints} depend on $(h,\theta)$. To subtract their equations, we first identify these constraints with the one at $(0,0)$. Use $\mathscr T_h$ and $u_\theta$ from the construction of \eqref{eq: shifted phase approximation}. For $x\in\Tub_\delta(\Gamma)$, let $P_h(x)\colon T_xM\to T_{\mathscr T_h(x)}M$ be Levi--Civita parallel transport along the short geodesic used to define $\nu_{a,h}$ in \Cref{sec: projections linear theory}. For ${\bf\Psi}=(\Psi,B)$, set
\begin{equation*}
\begin{aligned}
 (\mathcal I_{h,\theta}{\bf\Psi})_\Phi(x)
 &=\Ad_{u_\theta(x)}\Psi(\mathscr T_h(x)),\\
 (\mathcal I_{h,\theta}{\bf\Psi})_A(x)(v)
 &=\Ad_{u_\theta(x)}\bigl[B_{\mathscr T_h(x)}(P_h(x)v)\bigr],
 \qquad v\in T_xM.
\end{aligned}
\end{equation*}
We retain the common gauge $g_\eps$. Since $\mathscr T_h=\mathrm{id}$ and $u_\theta=1$ outside the tube, $\mathcal I_{h,\theta}$ extends by the identity there. Parallel transport and the adjoint action preserve the pairing in \eqref{eq: fibre constraints}; \eqref{eq: lifted monopole modes} and $g_\eps u_\theta=u_\theta g_\eps$ therefore give
\begin{equation}\label{eq: transported fibre constraint}
 \mathcal I_{0,0}=\mathrm{id},\qquad
 \mathfrak B_{\eps,h,\theta}\mathcal I_{h,\theta}^{-1}
 =\mathfrak B_{\eps,0,0}.
\end{equation}
Thus both transported corrections lie in $\mathcal E_{\eps,0,0}$, and their equations can be subtracted on this fixed space.
The transported norms are uniformly equivalent to $X_{\eps,0,0}$ and $Y_{\eps,0,0}$. For $i=1,2$, denote the sources in \eqref{eq: reduced complementary source terms} at $(h_i,\theta_i)$ by ${\bf f}_{\mathrm{nl},i}$ and ${\bf e}_{\mathrm{lin},i},{\bf e}_{\mathrm{cut},i},{\bf e}_{\mathrm g,i},{\bf e}_{\mathrm{jac},i}$.

We first establish the correction estimate
\begin{equation}\label{eq: reduced response difference}
\begin{split}
&\|\mathcal I_{h_1,\theta_1}{\bf\Psi}_{h_1,\theta_1}-\mathcal I_{h_2,\theta_2}{\bf\Psi}_{h_2,\theta_2}\|_{X_{\eps,0,0}}+\|\mathfrak b_{\eps,h_1,\theta_1}-\mathfrak b_{\eps,h_2,\theta_2}\|_{C^{0,\tau}_\eps}\\
&\qquad\lesssim (1+|\log\eps|)^N\bigl(\eps|\log\eps|^2\|h_1-h_2\|_{C^{2,\tau}}+a_\eps\|\theta_1-\theta_2\|_{C^{2,\tau}}\bigr)\\
&\qquad\lesssim a_\eps(1+|\log\eps|)^N d_\eps\bigl((h_1,\theta_1),(h_2,\theta_2)\bigr).
\end{split}
\end{equation}
We then establish the five coefficient estimates
\begin{subequations}\label{eq: correction source difference bounds}
\begin{align}
 \|\mathcal C({\bf e}_{\mathrm{lin},1}-{\bf e}_{\mathrm{lin},2})\|_{C^{0,\tau}}
 &\lesssim a_\eps\eps^{1-\tau}(1+|\log\eps|)^N
 d_\eps\bigl((h_1,\theta_1),(h_2,\theta_2)\bigr),\label{eq: correction linear difference bound}\\
 \|\mathcal C({\bf e}_{\mathrm{cut},1}-{\bf e}_{\mathrm{cut},2})\|_{C^{0,\tau}}
 &\lesssim a_\eps\eps(1+|\log\eps|)^N
 d_\eps\bigl((h_1,\theta_1),(h_2,\theta_2)\bigr),\label{eq: correction cutoff difference bound}\\
 \|\mathcal C({\bf e}_{\mathrm g,1}-{\bf e}_{\mathrm g,2})\|_{C^{0,\tau}}
 &\lesssim a_\eps\eps(1+|\log\eps|)^N
 d_\eps\bigl((h_1,\theta_1),(h_2,\theta_2)\bigr),\label{eq: correction gauge difference bound}\\
 \|\mathcal C({\bf f}_{\mathrm{nl},1}-{\bf f}_{\mathrm{nl},2})\|_{C^{0,\tau}}
 &\lesssim a_\eps^2\eps^{-\tau}(1+|\log\eps|)^N
 d_\eps\bigl((h_1,\theta_1),(h_2,\theta_2)\bigr),\label{eq: correction nonlinear difference bound}\\
 \|\mathcal C({\bf e}_{\mathrm{jac},1}-{\bf e}_{\mathrm{jac},2})\|_{C^{0,\tau}}
 &\lesssim a_\eps\eps^{2-\tau}(1+|\log\eps|)^N
 d_\eps\bigl((h_1,\theta_1),(h_2,\theta_2)\bigr).\label{eq: correction Jacobian difference bound}
\end{align}
\end{subequations}
Since $a_\eps/\eps\to0$ by \eqref{eq: error rate for contraction}, \eqref{eq: correction cutoff difference bound}--\eqref{eq: correction Jacobian difference bound} are absorbed by \eqref{eq: correction linear difference bound}. Subtract \eqref{eq: reduced complementary coefficient identity} and use \eqref{eq: reduced complementary decomposition}: division by $\eps^2|\log\eps|^2$ for the normalised translation remainder and by $\eps^2$ for phase gives \eqref{eq: reduced complementary differences}, using the multiplier bounds in \Cref{lem: reduced coefficient operators}.

\smallskip\noindent
First, we prove \eqref{eq: reduced response difference}. The transported linear operators and residual injections satisfy
\begin{equation}\label{eq: correction transported operator difference}
\begin{split}
 &\|\mathcal I_{h_1,\theta_1}\widehat{\bf L}_{\eps,h_1,\theta_1}\mathcal I_{h_1,\theta_1}^{-1}
 -\mathcal I_{h_2,\theta_2}\widehat{\bf L}_{\eps,h_2,\theta_2}\mathcal I_{h_2,\theta_2}^{-1}\|_{X_{\eps,0,0}\to Y_{\eps,0,0}}\\
 &\quad+\|\mathcal I_{h_1,\theta_1}\mathfrak B_{\eps,h_1,\theta_1}^*
 -\mathcal I_{h_2,\theta_2}\mathfrak B_{\eps,h_2,\theta_2}^*\|_{C^{0,\tau}_\eps\to Y_{\eps,0,0}}\\
 &\qquad\lesssim(1+|\log\eps|)^N
 \bigl(\|h_1-h_2\|_{C^{2,\tau}}+\|\theta_1-\theta_2\|_{C^{2,\tau}}\bigr).
\end{split}
\end{equation}
Use \eqref{eq: graph metric parameter derivatives} in \eqref{eq: model linearised operator}. The transported Higgs field is independent of the parameters by \eqref{eq: shifted phase approximation} and \eqref{eq: noncompact shifted approximation}, so its commutator mass cancels. The density and gauge projection terms are given by \eqref{eq: explicit residual injection} and \eqref{eq: restricted linearised operator}.

The difference of the transported known sources satisfies
\begin{equation}\label{eq: correction known source difference}
\begin{split}
&\|\mathcal I_{h_1,\theta_1}{\bf S}_\eps(\Phi_{h_1,\theta_1},A_{h_1,\theta_1})-\mathcal I_{h_2,\theta_2}{\bf S}_\eps(\Phi_{h_2,\theta_2},A_{h_2,\theta_2})\|_{Y_{\eps,0,0}}\\
&\qquad\lesssim (1+|\log\eps|)^N\bigl(\eps|\log\eps|^2\|h_1-h_2\|_{C^{2,\tau}}+\eps^2|\log\eps|^4\|\theta_1-\theta_2\|_{C^{2,\tau}}\bigr).
\end{split}
\end{equation}
To prove this estimate, differentiate the residual itself in the fixed pullback and phase gauge of \eqref{eq: approximation phase curvature}. Write $(\Phi,A)=(\Phi_\Gamma^{\mathrm c},A_\Gamma^{\mathrm c}-\xi_\eps d\bar\theta)$ in the compact case, and use the corresponding noncompact pair otherwise; the metric is $g_h=\mathscr T_h^*g$. For a graph variation $k$, the coordinate formulas
\begin{equation*}
\begin{aligned}
 d_A^{*,g_h}\omega
 &=-|g_h|^{-1/2}D_\mu\bigl(|g_h|^{1/2}g_h^{\mu\nu}\omega_\nu\bigr),\\
 (d_A^{*,g_h}F)_\alpha
 &=-g_{h,\alpha\beta}|g_h|^{-1/2}
 D_\mu\bigl(|g_h|^{1/2}g_h^{\mu\nu}g_h^{\beta\kappa}F_{\nu\kappa}\bigr)
\end{aligned}
\end{equation*}
show that only the metric and its first derivatives vary. Thus only $k,Dk,D^2k$ occur; the apparent third derivatives from differentiating the pulled-back connection cancel already in its curvature. For a phase variation $\vartheta$, \eqref{eq: approximation phase curvature} gives directly
\begin{equation*}
\begin{aligned}
 D_\theta(d_A\Phi)[\vartheta]&=0,\qquad
 D_\theta F_A[\vartheta]=-d_A\xi_\eps\wedge d\bar\vartheta,\\
 D_\theta E_\Phi[\vartheta]
 &=-\ad_{\xi_\eps d\bar\vartheta}^{*,g_h}d_A\Phi,\\
 D_\theta E_A[\vartheta]
 &=-d_A^{*,g_h}(d_A\xi_\eps\wedge d\bar\vartheta)
   -\ad_{\xi_\eps d\bar\vartheta}^{*,g_h}F_A.
\end{aligned}
\end{equation*}
The leading normal contraction in the last line is
$(D_aD_a\xi_\eps)d_\Gamma\vartheta$; it vanishes by \eqref{eq: approximation normal cancellation}, with the noncompact defect controlled by \eqref{eq: noncompact phase profile defect}. Applying \eqref{eq: graph metric parameter derivatives}, \eqref{eq: approximation profile bounds}, \eqref{eq: approximation horizontal cancellation} and \eqref{eq: FAtan} to these differentiated formulas gives, for $0\leq j\leq J_\eps$,
\begin{equation*}
\begin{aligned}
 &r_j^2\bigl(\|D_hE_{\Phi,h,\theta}[k]\|_{\tau,r_j}
 +\eps\|D_hE_{A,h,\theta}[k]\|_{\tau,r_j}\bigr)
 \\
 &\quad\lesssim(1+|\log\eps|)^N\eps
 \left(1+\frac{\|h\|_{C^{2,\tau}}}{\eps+r_j}\right)\|k\|_{C^{2,\tau}},\\
 &r_j^2\bigl(\|D_\theta E_{\Phi,h,\theta}[\vartheta]\|_{\tau,r_j}
 +\eps\|D_\theta E_{A,h,\theta}[\vartheta]\|_{\tau,r_j}\bigr)
 \\
 &\quad\lesssim(1+|\log\eps|)^N
 \bigl(\eps^2+\eps\|h\|_{C^{2,\tau}}+\|h\|_{C^{2,\tau}}^2\bigr)
 \|\vartheta\|_{C^{2,\tau}}.
\end{aligned}
\end{equation*}
The same bounds hold on the fixed cutoff region; the exterior parameter derivative vanishes. The conversion from coordinate pullback to $\mathcal I_{h,\theta}$ adds only the derivative of the frame map times the known residual, and satisfies these bounds as well. Sum the longitudinal contributions, take the maximum of the transverse ones, and use \eqref{eq: approximation dyadic sums} and \eqref{eq: assumptions htheta}. Integration along the segment between $(h_1,\theta_1)$ and $(h_2,\theta_2)$ proves \eqref{eq: correction known source difference}, after increasing $N$.

For a fixed transported correction ${\bf V}$ with $\|{\bf V}\|_{X_{\eps,0,0}}\lesssim a_\eps$, expansion of \eqref{eq: nonlinear remainder definition} gives
\begin{equation}\label{eq: correction fixed nonlinear difference}
\begin{split}
 &\|\mathcal I_{h_1,\theta_1}\mathcal N_{\eps,h_1,\theta_1}(\mathcal I_{h_1,\theta_1}^{-1}{\bf V})
 -\mathcal I_{h_2,\theta_2}\mathcal N_{\eps,h_2,\theta_2}(\mathcal I_{h_2,\theta_2}^{-1}{\bf V})\|_{Y_{\eps,0,0}}\\
 &\qquad\lesssim a_\eps^2(1+|\log\eps|)^N
 \bigl(\|h_1-h_2\|_{C^{2,\tau}}+\|\theta_1-\theta_2\|_{C^{2,\tau}}\bigr).
\end{split}
\end{equation}
Using \eqref{eq: transported fibre constraint}, subtract \eqref{eq: complementary nonlinear system} on $\mathcal E_{\eps,0,0}$. Apply \Cref{prop: uniform inverse} with \eqref{eq: correction transported operator difference}, \eqref{eq: correction known source difference} and \eqref{eq: correction fixed nonlinear difference}; \Cref{lem: quadratic estimate} absorbs the correction difference with coefficient $Ca_\eps<1/2$. This proves \eqref{eq: reduced response difference}. The same argument applies to the noncompact exterior, where $\mathcal I_{h,\theta}$ is fixed.

\smallskip\noindent
We next establish the linear estimate \eqref{eq: correction linear difference bound}. Subtract the definition of ${\bf e}_{\mathrm{lin}}$ in \eqref{eq: reduced complementary source terms}, separating the variation of the correction from that of the operator. Use \eqref{eq: reduced response difference}, the coefficient derivatives \eqref{eq: graph metric parameter derivatives}, and the shell estimate \eqref{eq: reduced complementary local linear bound}. The weighted coefficient bound \eqref{eq: reduced coefficient weighted bound} gives \eqref{eq: correction linear difference bound}, exactly as for \eqref{eq: reduced correction linear bound}.

\smallskip\noindent
For the cutoff estimate \eqref{eq: correction cutoff difference bound}, the cutoff is independent of the parameters, so use \eqref{eq: reduced response difference} in \eqref{eq: reduced correction cutoff commutator} and \eqref{eq: reduced correction collar gauge bound}. The gauge estimate \eqref{eq: correction gauge difference bound} follows from \eqref{eq: reduced response difference} and \eqref{eq: reduced correction collar gauge bound} after subtracting \eqref{eq: reduced correction gauge projection}. To obtain the nonlinear estimate \eqref{eq: correction nonlinear difference bound}, combine \Cref{lem: quadratic estimate} with \eqref{eq: reduced response difference}, \eqref{eq: reduced correction source L2 rescaling} and \eqref{eq: reduced coefficient L2 bound}.

\smallskip\noindent
Lastly, we prove the Jacobian estimate \eqref{eq: correction Jacobian difference bound}. Integrating \eqref{eq: graph metric parameter derivatives} along $h_t$ gives
\begin{equation}\label{eq: correction density difference}
 \|J_{h_1}(\cdot,z)^{-1}-J_{h_2}(\cdot,z)^{-1}\|_{C^{0,\tau}(\Gamma)}
 \lesssim\bigl(r+\eps|\log\eps|^2\bigr)\|h_1-h_2\|_{C^{2,\tau}}.
\end{equation}
In \eqref{eq: reduced complementary source terms}, subtract the density factors and the coefficients $\mathfrak b$ separately. Use \eqref{eq: correction density difference}, \eqref{eq: reduced response difference} and $|{\bf E}_i(y)|\lesssim\langle y\rangle^{-4}$ in \eqref{eq: reduced coefficient L2 bound}, as for \eqref{eq: reduced correction injection bound}. This proves \eqref{eq: correction Jacobian difference bound} and completes the proof.
\end{proof}

\subsection{Conclusion of the reduction}
\begin{proof}[Proof of \Cref{prop: reduced equation}]
The identities \eqref{eq: reduced exact coefficients} and the multiplier bounds in \Cref{lem: reduced coefficient operators} give the equivalence with \eqref{eq: reduced Jacobi Laplacian system}. Equations \eqref{eq: reduced source bounds}, \eqref{eq: reduced phase source bounds} and \eqref{eq: reduced complementary bounds} give \eqref{eq: reduced size estimates}, while \eqref{eq: reduced source differences} and \eqref{eq: reduced complementary differences} give \eqref{eq: reduced Lipschitz estimates}. The scalar rate $\upsilon_\eps$ in \eqref{eq: reduced Lipschitz estimates} can be chosen so that
\begin{equation*}
\upsilon_\eps\lesssim  \left[|\log\eps|^{-2}+(1+|\log\eps|)^N\bigl(\eps^{1/4}+\eps^{1/2-\tau}+a_\eps\eps^{-1-\tau}\bigr)\right]\longrightarrow0.
\end{equation*}
Both normalised coefficients in the translation estimate therefore tend to zero. The phase forcing $d_\Gamma^*\beta^{\mathrm{tan}}$ is independent of the parameters and cancels in their differences.
\end{proof}

\section{Technical proofs}\label{sec: technical}
Here we prove \Cref{lem: quadratic estimate}, which gives the quadratic bound used in the correction contraction, and \Cref{lem: constant phase residual}, which removes the remaining constant phase coefficient and gauge residual.
\subsection{The quadratic estimate}\label{sec: quadratic estimate proof}
\begin{proof}[Proof of \Cref{lem: quadratic estimate}]
Fix $(h,\theta)$ satisfying \eqref{eq: assumptions htheta}, and let ${\bf\Psi}_i=(\Psi_i,B_i)$, $i=1,2$, and $R\leq1$ be as in the statement. Put $\lambda=\lambda_\eps$ in the compact case and $\lambda=0$ in the noncompact case, and write
\begin{equation*}
 (N_\Phi,N_A)\coloneqq\mathcal N_{\eps,h,\theta}({\bf\Psi}_1)
 -\mathcal N_{\eps,h,\theta}({\bf\Psi}_2).
\end{equation*}
We estimate separately the core, the two annular components, and the noncompact exterior, using the scales $r_j$ and regions $\Omega_{j,h}$ from \Cref{sec: norms}. On the core, we claim that
\begin{equation}\label{eq: quadratic core estimate}
 r_0^2\bigl(\|N_\Phi\|_{\tau,r_0,\Omega_{0,h}}
 +\eps\|N_A\|_{\tau,r_0,\Omega_{0,h}}\bigr)
 \lesssim R\|{\bf\Psi}_1-{\bf\Psi}_2\|_{X_{\eps,h,\theta}}.
\end{equation}
For $1\leq j\leq J_\eps+1$ in the compact case, and $1\leq j\leq J_\eps$ in the noncompact case, we claim
\begin{subequations}\label{eq: quadratic annular estimates}
\begin{equation}\label{eq: quadratic longitudinal estimate}
\begin{split}
 &r_j^2\bigl(\|N_\Phi^\parallel\|_{\tau,r_j,\Omega_{j,h}}
 +\eps\|N_A^\parallel\|_{\tau,r_j,\Omega_{j,h}}\bigr)\\
 &\qquad\lesssim R\|{\bf\Psi}_1-{\bf\Psi}_2\|_{X_{\eps,h,\theta}}
 \left(\frac{\eps^2}{r_j^2}
 +\frac{\lambda r_j^2/\eps^2}{(1+\lambda r_j^2/\eps^2)^2}\right).
\end{split}
\end{equation}
\begin{equation}\label{eq: quadratic transverse estimate}
 r_j^2\bigl(\|N_\Phi^\perp\|_{\tau,r_j,\Omega_{j,h}}
 +\eps\|N_A^\perp\|_{\tau,r_j,\Omega_{j,h}}\bigr)
 \lesssim R\|{\bf\Psi}_1-{\bf\Psi}_2\|_{X_{\eps,h,\theta}}.
\end{equation}
\end{subequations}
Finally, in the noncompact case,
\begin{equation}\label{eq: quadratic exterior estimate}
 \|(N_\Phi,N_A)\|_{Y^\nu_{\eps,h,\theta}(\Omega_{J_\eps+1,h})}
 \lesssim R\|{\bf\Psi}_1-{\bf\Psi}_2\|_{X_{\eps,h,\theta}}.
\end{equation}
The left-hand sides are precisely the core, longitudinal, transverse and exterior contributions to the source norm in \Cref{sec: norms}. Sum \eqref{eq: quadratic longitudinal estimate} over $j$, and take the maximum in \eqref{eq: quadratic transverse estimate}. The longitudinal sum is bounded because
\begin{equation*}
 \sum_j\frac{\eps^2}{r_j^2}
 +\sum_j\frac{\lambda r_j^2/\eps^2}{(1+\lambda r_j^2/\eps^2)^2}\lesssim 1.
\end{equation*}
The first sum is geometric. For the second, split at $\lambda r_j^2/\eps^2=1$; the summands grow geometrically before this scale and decay geometrically afterwards. Together with \eqref{eq: quadratic core estimate} and, in the noncompact case, \eqref{eq: quadratic exterior estimate}, this proves the lemma once the four estimates are established.

\smallskip

We first prove \eqref{eq: quadratic core estimate}. Since $r_0=\eps$, \eqref{eq: sol norm} gives, for any ${\bf\Psi}=(\Psi,B)$ of norm at most $R$,
\begin{equation*}
 \sum_{k=0}^2\eps^k\bigl(\|\nabla_{A_{h,\theta}}^k\Psi\|_{\tau,\eps,\Omega_{0,h}}
 +\eps\|\nabla_{A_{h,\theta}}^kB\|_{\tau,\eps,\Omega_{0,h}}\bigr)\lesssim  R.
\end{equation*}
The quadratic remainder consists of products of a correction with its first derivative, double commutators, the Higgs potential, and the gauge term in its definition. Since $[\Phi_{h,\theta},\xi_{\eps,h}]=0$, integration by parts gives
\begin{equation}\label{eq: quadratic projection coefficient}
 \Pi_h{\bf G}_{h,\theta}^*{\bf\Psi}
 =\frac{\eps\int_M\langle B,d_{A_{h,\theta}}(\chi\xi_{\eps,h})\rangle\,d\vol_g}
 {\int_M\chi|\xi_{\eps,h}|^2\,d\vol_g}\,\xi_{\eps,h}.
\end{equation}
The denominator is uniformly positive. The profile bounds and the fixed support of $d\chi$ bound the quotient by $C\|{\bf\Psi}\|_{X_{\eps,h,\theta}}$. Using the definition of the gauge function,
\begin{equation*}
\begin{gathered}
 \gamma_{\bf\Psi}=\eps d_{A_{h,\theta}}^*B-\eps^{-1}[\Phi_{h,\theta},\Psi]
 -\Pi_h{\bf G}_{h,\theta}^*{\bf\Psi},\\
 \|\gamma_{\bf\Psi}\|_{\tau,\eps,\Omega_{0,h}}\lesssim R/\eps.
\end{gathered}
\end{equation*}
Consequently, every quadratic term contributes at most $CR^2$ after multiplication by the source weights $\eps^2$ and $\eps^3$. Cubic terms satisfy the same bound because $R\leq1$. For differences, use bilinearity, for instance
\begin{equation*}
 [B_1,\Psi_1]-[B_2,\Psi_2]
 =[B_1-B_2,\Psi_1]+[B_2,\Psi_1-\Psi_2],
\end{equation*}
and expand cubic terms in the same way. One factor is bounded by $\|{\bf\Psi}_1-{\bf\Psi}_2\|_{X_{\eps,h,\theta}}$ and the others by $R$, proving \eqref{eq: quadratic core estimate}.

\smallskip

Next, we prove \eqref{eq: quadratic longitudinal estimate} and \eqref{eq: quadratic transverse estimate}. All norms with scale $r_j$ are taken on $\Omega_{j,h}$. The solution norm and interpolation give
\begin{subequations}\label{eq: quadratic annular field bounds}
\begin{align}
 \|\Psi\|_{\tau,r_j}+\eps\|B\|_{\tau,r_j}&\lesssim R,
 \label{eq: quadratic annular zeroth bound}\\
 \|\Psi^\perp\|_{\tau,r_j}+\eps\|B^\perp\|_{\tau,r_j}&\lesssim R\eps^2r_j^{-2},
 \label{eq: quadratic annular transverse zeroth bound}\\
 \|\nabla_{A_{h,\theta}}\Psi\|_{\tau,r_j}
 +\eps\|\nabla_{A_{h,\theta}}B\|_{\tau,r_j}&\lesssim Rr_j^{-1},
 \label{eq: quadratic annular derivative bound}\\
 \|\nabla_{A_{h,\theta}}\Psi^\perp\|_{\tau,r_j}
 +\eps\|\nabla_{A_{h,\theta}}B^\perp\|_{\tau,r_j}&\lesssim R\eps r_j^{-2}.
 \label{eq: quadratic annular transverse derivative bound}
\end{align}
\end{subequations}
To obtain \eqref{eq: quadratic annular transverse derivative bound}, interpolate \eqref{eq: quadratic annular transverse zeroth bound} with the second-derivative bound in \eqref{eq: sol norm}. The splitting is the one defined in \Cref{sec: norms}, with $\sigma_h=\Phi_{h,\theta}/|\Phi_{h,\theta}|$. Its derivatives satisfy
\begin{equation}\label{eq: quadratic splitting derivative bounds}
 \|\nabla_{A_{h,\theta}}\sigma_h\|_{\tau,r_j}\lesssim \eps r_j^{-2},\qquad
 \|\nabla_{A_{h,\theta}}^2\sigma_h\|_{\tau,r_j}\lesssim \eps r_j^{-3}.
\end{equation}
The final term in \eqref{eq: sol norm} also gives
\begin{equation}\label{eq: quadratic longitudinal Higgs bound}
 \|\Psi^\parallel\|_{\tau,r_j}\lesssim\frac{R}{1+\lambda r_j^2/\eps^2}.
\end{equation}
Together with \eqref{eq: quadratic projection coefficient}, these estimates imply
\begin{equation*}
 \|\gamma_{\bf\Psi}^\parallel\|_{\tau,r_j}\lesssim R/r_j,\qquad
 \|\gamma_{\bf\Psi}^\perp\|_{\tau,r_j}\lesssim R\eps/r_j^2.
\end{equation*}
Commutators with the projected term in \eqref{eq: quadratic projection coefficient} are transverse because $\xi_{\eps,h}$ is longitudinal.
The identities
\begin{equation*}
 [\R\sigma_h,\R\sigma_h]=0,\qquad
 [\R\sigma_h,\sigma_h^\perp]\subset\sigma_h^\perp
\end{equation*}
show that the longitudinal part of a commutator uses the transverse parts of both inputs. In a local orthonormal frame $(e_i)_{i=1}^n$, for example,
\begin{equation*}
 r_j^2\left\|\sum_{i=1}^n
 [B^\perp(e_i),(\nabla_{A_{h,\theta},e_i}\Psi)^\perp]\right\|_{\tau,r_j}
 \lesssim r_j^2\left(\frac{R\eps}{r_j^2}\right)^2
 =\frac{R^2\eps^2}{r_j^2}.
\end{equation*}
Here the derivative of $\Psi^\parallel$ contributes at most $CR\eps r_j^{-2}$ to its transverse part by \eqref{eq: quadratic splitting derivative bounds}. The other longitudinal commutators have the same bound after including the connection weight $\eps$. For the transverse part of a commutator, we use \eqref{eq: quadratic annular zeroth bound} and \eqref{eq: quadratic annular derivative bound} for one input, and \eqref{eq: quadratic annular transverse zeroth bound} or \eqref{eq: quadratic annular transverse derivative bound} for the other. After multiplication by $r_j^2$ and the connection weight $\eps$, its bound is $CR^2$.

The potential remainder is
\begin{equation*}
 \frac{\lambda}{\eps^2}\left(
 2\langle\Phi_{h,\theta},\Psi\rangle\Psi
 +|\Psi|^2\Phi_{h,\theta}+|\Psi|^2\Psi\right).
\end{equation*}
Its purely longitudinal terms satisfy \eqref{eq: quadratic longitudinal Higgs bound} twice, giving the second summand on the right of \eqref{eq: quadratic longitudinal estimate}. The longitudinal potential terms containing transverse corrections satisfy the first summand; the transverse potential terms have bound $CR^2$ after multiplication by $r_j^2$, using \eqref{eq: quadratic annular transverse zeroth bound} and the fixed upper bound for $\lambda$. Cubic terms are controlled because $R\leq1$. Expanding differences as in the core proves \eqref{eq: quadratic longitudinal estimate} and \eqref{eq: quadratic transverse estimate}.

\smallskip

Lastly, we show \eqref{eq: quadratic exterior estimate}. On $\Omega_{J_\eps+1,h}$ we have $d_{A_D}\sigma=0$, with $A_D$ as in \eqref{eq: exterior connection}. All weighted norms below are taken on this region. The exterior solution norm and interpolation give
\begin{equation*}
\begin{split}
 \|\Psi^\parallel\|_{C^{1,\tau}_\nu}
 +\eps\|B^\parallel\|_{C^{1,\tau}_\nu}&\lesssim R,\\
 \|\Psi^\perp\|_{C^{0,\tau}_{\nu+2}}
 +\eps\|B^\perp\|_{C^{0,\tau}_{\nu+2}}&\lesssim R\eps^2,\\
 \|\nabla_{A_D}\Psi^\perp\|_{C^{0,\tau}_{\nu+2}}
 +\eps\|\nabla_{A_D}B^\perp\|_{C^{0,\tau}_{\nu+2}}&\lesssim R\eps.
\end{split}
\end{equation*}
There is no potential, and every nonlinear commutator contains a transverse correction. Products of corrections therefore have decay weight at least $2\nu+2\geq\nu+2$, and their powers of $\eps$ give precisely the exterior source weights $1$ and $\eps$. Terms containing the projection coefficient in \eqref{eq: quadratic projection coefficient} contain one transverse correction and satisfy the same estimate. Expanding the differences proves \eqref{eq: quadratic exterior estimate}.
\end{proof}

\subsection{Gauge compatibility}\label{sec: technical gauge}

The two errors in \eqref{eq: remaining phase residual} are eliminated by two gauge tests. Gauge invariance makes the residual orthogonal to both
\begin{equation*}
	{\bf G}_{(\Phi_\eps,A_\eps)}\gamma_{\bf\Psi}\quad\text{and}\quad {\bf G}_{(\Phi_\eps,A_\eps)}\xi_{\eps,h}.
\end{equation*} 
The first test controls the gauge error by $|c_\eps|$; the second detects the constant phase residual and, using the small mixed pairing, forces $c_\eps=0$. Returning to the first test then removes the gauge error and recovers the original equation.

\begin{proof}[Proof of \Cref{lem: constant phase residual}]
Write $(\Phi_\eps,A_\eps)=(\Phi_{h,\theta}+\Psi,A_{h,\theta}+B)$. We first treat the noncompact case. The unprojected gauge function
\begin{equation*}
\beta\coloneqq{\bf G}_{h,\theta}^*{\bf\Psi}
\end{equation*}
satisfies, by the definition of $\gamma_{\bf\Psi}$ and \eqref{eq: quadratic projection coefficient},
\begin{equation}\label{eq: gauge phase normalization proof}
\begin{split}
\gamma_{\bf\Psi}&=\beta-\Pi_h\beta,\\
\|\Pi_h\beta\|_{L^\infty}&\lesssim\|{\bf\Psi}\|_{X_{\eps,h,\theta}},\\
\int_M\chi\langle\gamma_{\bf\Psi},\xi_{\eps,h}\rangle\,d\vol_g&=0.
\end{split}
\end{equation}
The required claims are the gauge identities
\begin{subequations}\label{eq: gauge test identities proof}
\begin{equation}\label{eq: gauge test integrability proof}
 {\bf G}_{(\Phi_\eps,A_\eps)}\gamma_{\bf\Psi},\,
 {\bf G}_{(\Phi_\eps,A_\eps)}\xi_{\eps,h}\in L^2_\eps,
\end{equation}
\begin{equation}\label{eq: gauge test vanishing proof}
\begin{aligned}
 \langle{\bf S}_\eps(\Phi_\eps,A_\eps),{\bf G}_{(\Phi_\eps,A_\eps)}\gamma_{\bf\Psi}\rangle_{L^2_\eps}&=0,\\
 \langle{\bf S}_\eps(\Phi_\eps,A_\eps),{\bf G}_{(\Phi_\eps,A_\eps)}\xi_{\eps,h}\rangle_{L^2_\eps}&=0,
\end{aligned}
\end{equation}
\end{subequations}
the estimate
\begin{equation}\label{eq: gauge phase pairing proof}
|\langle{\bf G}_{(\Phi_\eps,A_\eps)}v, {\bf G}_{(\Phi_\eps,A_\eps)}\xi_{\eps,h}\rangle_{L^2_\eps}|\lesssim\eps\bigl(\sqrt\eps+\|{\bf\Psi}\|_{X_{\eps,h,\theta}}\bigr) \|{\bf G}_{(\Phi_\eps,A_\eps)}v\|_{L^2_\eps},
\end{equation}
for gauge functions $v$ satisfying $\|{\bf G}_{(\Phi_\eps,A_\eps)}v\|_{L^2_\eps}<\infty$ and
$\int_M\chi\langle v,\xi_{\eps,h}\rangle\,d\vol_g=0$, and the residual bounds
\begin{subequations}\label{eq: gauge residual pairing proof}
\begin{equation}\label{eq: gauge residual injection norm proof}
 \|\mathfrak B_{\eps,h,\theta}^*(0,0,0,1)\|_{L^2_\eps}\lesssim\eps^{-1},
\end{equation}
\begin{equation}\label{eq: gauge residual positive pairing proof}
 \langle\mathfrak B_{\eps,h,\theta}^*(0,0,0,1),
 {\bf G}_{(\Phi_\eps,A_\eps)}\xi_{\eps,h}\rangle_{L^2_\eps}\geq c_0>0.
\end{equation}
\end{subequations}
Indeed, testing the residual equation in \Cref{lem: constant phase residual}
against the two fields in \eqref{eq: gauge test identities proof} gives
\begin{equation*}
\begin{split}
\|{\bf G}_{(\Phi_\eps,A_\eps)}\gamma_{\bf\Psi}\|_{L^2_\eps}&\lesssim|c_\eps|,\\
c_0|c_\eps|&\leq\eps^{-1}
 |\langle{\bf G}_{(\Phi_\eps,A_\eps)}\gamma_{\bf\Psi},
 {\bf G}_{(\Phi_\eps,A_\eps)}\xi_{\eps,h}\rangle_{L^2_\eps}|\\
&\lesssim\bigl(\sqrt\eps+\|{\bf\Psi}\|_{X_{\eps,h,\theta}}\bigr)|c_\eps|.
\end{split}
\end{equation*}
Choose $m_0$ and then $\eps$ small enough to absorb the last term. Then
$c_\eps=0$ and ${\bf G}_{(\Phi_\eps,A_\eps)}\gamma_{\bf\Psi}=0$, which gives
\eqref{eq: final eq}. We prove these claims below.

\smallskip

We start by proving \eqref{eq: gauge test integrability proof} and \eqref{eq: gauge test vanishing proof}.
Choose a sufficiently large smooth connected compact domain $K$ containing $\supp\chi$ and all shifted gluing regions, with $K\setminus\Gamma$ connected. Outside $K$, the background is $((1-\eps H)\sigma,A_D)$, with $A_D$ as in \eqref{eq: exterior connection} and $\Delta_gH=0$ by \eqref{eq: noncompact scalar potential}. The exterior correction in \Cref{lem: noncompact exterior correction} gives
\begin{equation*}
\begin{gathered}
 d_{A_D}\sigma=0,\qquad F_{A_D}=\sigma f_1,\qquad df_1=d_g^*f_1=0,\\
 \sum_{j=0}^2|\nabla^j f_1|\lesssim1\qquad\text{on }M\setminus K.
\end{gathered}
\end{equation*}
There $\xi_{\eps,h}=\sigma$ and $\mathfrak B_{\eps,h,\theta}^*(0,0,0,1)=0$. Write ${\bf\Psi}^\perp=(\Psi^\perp,B^\perp)$ and $|{\bf\Psi}^\perp|_\eps^2=|\Psi^\perp|^2+\eps^2|B^\perp|^2$, using the induced norms for derivatives and coefficients below. The term $\eps^{-1}{\bf G}_{(\Phi_\eps,A_\eps)}{\bf G}_{h,\theta}^*{\bf\Psi}$ cancels the mixed second derivatives, and every perpendicular nonlinear term contains ${\bf\Psi}^\perp$ or $\nabla_{A_D}{\bf\Psi}^\perp$. The projection contributes the term
\begin{equation*}
 -\eps^{-1}([\Psi,\Pi_h\beta],[B,\Pi_h\beta]).
\end{equation*}
In a local orthonormal frame $e_1,\ldots,e_n$, let $P^1,\ldots,P^n,Q$ denote the resulting first- and zeroth-order coefficient endomorphisms. Then
\begin{equation}\label{eq: gauge exterior equation proof}
\begin{split}
\nabla_{A_D}^*\nabla_{A_D}{\bf\Psi}^\perp+4\eps^{-2}(1-\eps H)^2{\bf\Psi}^\perp&=\sum_{j=1}^n P^j\nabla_{A_D,e_j}{\bf\Psi}^\perp+Q{\bf\Psi}^\perp,\\
|P|&\lesssim1+\|{\bf\Psi}\|_{X_{\eps,h,\theta}}/\eps,\\
|Q|&\lesssim1+\eps^{-1}+\|{\bf\Psi}\|_{X_{\eps,h,\theta}}\eps^{-2}.
\end{split}
\end{equation}
Here $|P|^2=\sum_j|P^j|^2$, with the operator norms induced by $|\cdot|_\eps$.
Taking the inner product of \eqref{eq: gauge exterior equation proof} with ${\bf\Psi}^\perp$, use
\begin{equation*}
 2|P||{\bf\Psi}^\perp|_\eps|\nabla_{A_D}{\bf\Psi}^\perp|_\eps\leq |\nabla_{A_D}{\bf\Psi}^\perp|_\eps^2+|P|^2|{\bf\Psi}^\perp|_\eps^2.
\end{equation*}
We absorb $|P|^2+|Q|$ into the positive mass term by choosing $m_0$ and then $\eps$ small. Thus
\begin{equation*}
(\Delta_0+c\eps^{-2})|{\bf\Psi}^\perp|_\eps^2\leq0.
\end{equation*}
Choose a smooth nonnegative function $\rho$ that vanishes near $\partial K$, is comparable to $\dist(x,K)$ outside a larger compact set, and has bounded first and second derivatives. For sufficiently small fixed $\kappa>0$, the function
$C\eps^4\|{\bf\Psi}\|_{X_{\eps,h,\theta}}^2e^{-2\kappa\rho/\eps}$
is a supersolution of $\Delta_0+c\eps^{-2}$ and dominates $|{\bf\Psi}^\perp|_\eps^2$ on the boundary, by the exterior solution norm. Comparison, followed by interior estimates at scale $\eps$, yields, after enlarging $K$ and decreasing $\kappa$,
\begin{equation}\label{eq: gauge exponential decay proof}
|{\bf\Psi}^\perp|_\eps+\eps|\nabla_{A_D}{\bf\Psi}^\perp|_\eps\lesssim\eps^2\|{\bf\Psi}\|_{X_{\eps,h,\theta}}\,e^{-\kappa\dist(x,K)/\eps}.
\end{equation}
The same estimates control the second derivatives. Bounded geometry gives at most exponential volume growth, so ${\bf G}_{(\Phi_\eps,A_\eps)}\xi_{\eps,h}\in L^2_\eps$ and ${\bf G}_{(\Phi_\eps,A_\eps)}^*{\bf G}_{(\Phi_\eps,A_\eps)}\xi_{\eps,h}\in L^1$.

The exterior decay norms imply that $\beta$ tends to zero. The local gauge identity ${\bf G}_{(\Phi_\eps,A_\eps)}^*{\bf S}_\eps(\Phi_\eps,A_\eps)=0$ gives
\begin{equation*}
\begin{split}
{\bf G}_{(\Phi_\eps,A_\eps)}^*{\bf G}_{(\Phi_\eps,A_\eps)}\beta&=\eps c_\eps{\bf G}_{(\Phi_\eps,A_\eps)}^*\mathfrak B_{\eps,h,\theta}^*(0,0,0,1)+{\bf G}_{(\Phi_\eps,A_\eps)}^*{\bf G}_{(\Phi_\eps,A_\eps)}(\Pi_h\beta)\in L^1.
\end{split}
\end{equation*}
For $t>0$, the truncation $(1-t/|\beta|)_+\beta$ is compactly supported because $\beta$ tends to zero. Testing the equation against it gives
\begin{equation*}
\begin{split}
&\int_M\left(1-\frac{t}{|\beta|}\right)_+\left(\eps|d_{A_\eps}\beta|^2+\eps^{-1}|[\Phi_\eps,\beta]|^2\right)\,d\vol_g\leq\|\beta\|_{L^\infty}\|{\bf G}_{(\Phi_\eps,A_\eps)}^*{\bf G}_{(\Phi_\eps,A_\eps)}\beta\|_{L^1}.
\end{split}
\end{equation*}
Monotone convergence proves finite gauge energy, and dominated convergence gives convergence of the truncations in that energy. Thus ${\bf G}_{(\Phi_\eps,A_\eps)}\gamma_{\bf\Psi}\in L^2_\eps$, proving \eqref{eq: gauge test integrability proof}. Gauge invariance extends to the test ${\bf G}_{(\Phi_\eps,A_\eps)}\beta$. For ${\bf G}_{(\Phi_\eps,A_\eps)}\xi_{\eps,h}$, the first-variation boundary current
\begin{equation*}
\eps^{-1}\langle d_{A_\eps}\Phi_\eps,[\Phi_\eps,\xi_{\eps,h}]\rangle +\eps\langle F_{A_\eps},d_{A_\eps}\xi_{\eps,h}\rangle
\end{equation*}
has vanishing cutoff integrals by \eqref{eq: gauge exponential decay proof} and the volume-growth bound. Subtracting the gauge identity for $\Pi_h\beta$, which is a constant multiple of $\xi_{\eps,h}$, from the one for $\beta$ proves \eqref{eq: gauge test vanishing proof}.

\smallskip

Next, we prove \eqref{eq: gauge phase pairing proof}.
We use the following two bounds; the second requires $\int_M\chi\langle v,\xi_{\eps,h}\rangle\,d\vol_g=0$:
\begin{subequations}\label{eq: gauge weighted inequalities proof}
\begin{align}
\eps^{-1}\int_K(1+r_h/\eps)^{-4}|v|^2\,d\vol_g&\lesssim\|{\bf G}_{h,\theta}v\|_{L^2_\eps}^2,\label{eq: gauge core weighted bound proof}\\
\eps\|v\|_{L^2(K)}^2&\lesssim\|{\bf G}_{h,\theta}v\|_{L^2_\eps}^2.\label{eq: gauge normalized L2 bound proof}
\end{align}
\end{subequations}
To prove \eqref{eq: gauge core weighted bound proof}, pull back $v$ by $\mathscr F_{\eps,h}$ and use the transported gauge from \Cref{sec: projections linear theory}. The model is irreducible in the following sense: the only gauge function $w$ on $B_2\subset\R^3$ satisfying
\begin{equation*}
 d_{A_\lambda}w=0,\qquad [\Phi_\lambda,w]=0
\end{equation*}
is $w=0$. Indeed, \eqref{eq:monopole-ansatz} and \eqref{eq: BPS} give $D_a\Phi_0(0)=\frac23T_a$, so differentiating $[\Phi_0,w]=0$ yields $[T_a,w(0)]=0$ for $a=1,2,3$. Hence $w(0)=0$, and $d_{A_0}w=0$ implies $w=0$. Compactness and the trace inequality therefore give, for every $w\in H^1(B_2;\su(2))$,
\begin{equation*}
 \|w\|_{L^2(B_2)}^2+\|w\|_{L^2(\partial B_2)}^2
 \lesssim\int_{B_2}\bigl(|d_{A_\lambda}w|^2+|[\Phi_\lambda,w]|^2\bigr)\,dy.
\end{equation*}
Local convergence of the profiles to \eqref{eq: BPS} makes this estimate uniform for $0\leq\lambda\leq\lambda_*$, after decreasing $\lambda_*$ if necessary. The rescaled approximation is a small perturbation of this model, including the term \eqref{eq: noncompact mass correction}, so the same bound applies. Scale back and apply the one-dimensional Hardy inequality along normal rays. For \eqref{eq: gauge normalized L2 bound proof}, use the stated normalisation. If the bound fails uniformly for small $\eps$, choose $\eps_k\downarrow0$, admissible $(h_k,\theta_k)$ and gauge functions $v_k$ with $\int_M\chi\langle v_k,\xi_{\eps_k,h_k}\rangle\,d\vol_g=0$ such that
\begin{equation*}
 \|v_k\|_{L^2(K)}=1,\qquad
 \eps_k^{-1}\|{\bf G}_{h_k,\theta_k}v_k\|_{L^2_{\eps_k}}^2\longrightarrow0.
\end{equation*}
Consequently,
\begin{equation}\label{eq: gauge compactness limits proof}
 \|d_{A_{h_k,\theta_k}}v_k\|_{L^2(K)}\longrightarrow0,\qquad
 \eps_k^{-1}\|[\Phi_{h_k,\theta_k},v_k]\|_{L^2(K)}\longrightarrow0.
\end{equation}
The inequality $|d|v_k||\leq|d_{A_{h_k,\theta_k}}v_k|$ and Rellich compactness give strong $L^2(K)$ convergence of $|v_k|$ along a subsequence, excluding concentration at $\Gamma$. Off $\Gamma$, \eqref{eq: gauge compactness limits proof} forces $v_k$ to converge to $c\sigma$ in the exterior gauge, with one constant $c$ since $K\setminus\Gamma$ is connected. The normalisation gives $c\int_M\chi\,d\vol_g=0$, contradicting $\|v_k\|_{L^2(K)}=1$ and proving \eqref{eq: gauge normalized L2 bound proof}. On the exterior, propagation through overlapping fixed-radius balls gives, for small $\eps$,
\begin{equation*}
\int_{M\setminus K}e^{-\kappa\dist(x,K)/\eps}|v|^2\,d\vol_g \lesssim\|v\|_{L^2(K)}^2+\|d|v|\|_{L^2(M\setminus K)}^2.
\end{equation*}
Indeed, both the propagation constants and the number of balls grow at most exponentially with distance. Since
\begin{equation*}
({\bf G}_{(\Phi_\eps,A_\eps)}-{\bf G}_{h,\theta})v=([\Psi,v],[B,v]),
\end{equation*}
the solution bounds and \eqref{eq: gauge exponential decay proof} give, for $v$ with $\int_M\chi\langle v,\xi_{\eps,h}\rangle\,d\vol_g=0$,
\begin{align}
\|({\bf G}_{(\Phi_\eps,A_\eps)}-{\bf G}_{h,\theta})v\|_{L^2_\eps}&\lesssim\|{\bf\Psi}\|_{X_{\eps,h,\theta}}\|{\bf G}_{h,\theta}v\|_{L^2_\eps},\notag\\
\|({\bf G}_{(\Phi_\eps,A_\eps)}-{\bf G}_{h,\theta})\xi_{\eps,h}\|_{L^2_\eps}&\lesssim\eps\|{\bf\Psi}\|_{X_{\eps,h,\theta}}.\label{eq: gauge phase perturbation bound proof}
\end{align}

Normal harmonicity of the phase profile and minimality bound the terms in $\Delta_{A_{h,\theta}}\xi_{\eps,h}$ that do not contain the mass correction by
\begin{equation*}
 C\left(\frac{\eps}{\eps+r_h}+\frac{\eps\|h\|_{C^{2,\tau}}}{(\eps+r_h)^2}+\frac{\eps\|h\|_{C^{2,\tau}}^2}{(\eps+r_h)^3}\right).
\end{equation*}
This follows from the Fermi defects $O(r^2+r\|h\|_{C^{2,\tau}}+\|h\|_{C^{2,\tau}}^2)$ and $O(r+\|h\|_{C^{2,\tau}})$; the fixed collar contributes $O(\eps)$. The connection component of \eqref{eq: noncompact mass correction} adds a term bounded by $C\eps^{-1}(1+r_h/\eps)^{-4}$. For the actual approximation, write $F_1$ for the terms in ${\bf G}_{h,\theta}^*{\bf G}_{h,\theta}\xi_{\eps,h}$ containing the connection mass correction and $F_0$ for the remaining terms. Then
\begin{equation*}
\begin{split}
{\bf G}_{h,\theta}^*{\bf G}_{h,\theta}\xi_{\eps,h}&=F_0+F_1,\\
\|F_0\|_{L^2}&\lesssim\eps^2,\qquad |F_1|\lesssim(1+r_h/\eps)^{-4},\\
\|{\bf G}_{h,\theta}\xi_{\eps,h}\|_{L^2_\eps}&\lesssim\eps,
\end{split}
\end{equation*}
Both terms are supported in $K$, and $F_1=0$ in the compact case. Estimate \eqref{eq: gauge core weighted bound proof} and the normal volume element $r_h^2\,dr_h$ give
\begin{equation*}
\begin{split}
|\langle F_1,v\rangle_{L^2}|&\lesssim\left(\int_K(1+r_h/\eps)^{-4}\,d\vol_g\right)^{1/2}\left(\int_K(1+r_h/\eps)^{-4}|v|^2\,d\vol_g\right)^{1/2}\\
&\lesssim\eps^{3/2}\eps^{1/2}\|{\bf G}_{h,\theta}v\|_{L^2_\eps}=\eps^2\|{\bf G}_{h,\theta}v\|_{L^2_\eps}.
\end{split}
\end{equation*}
Integration by parts, followed by \eqref{eq: gauge weighted inequalities proof} and \eqref{eq: gauge phase perturbation bound proof}, therefore gives, for $v$ with $\int_M\chi\langle v,\xi_{\eps,h}\rangle\,d\vol_g=0$,
\begin{equation*}
\begin{split}
&|\langle{\bf G}_{(\Phi_\eps,A_\eps)}v,{\bf G}_{(\Phi_\eps,A_\eps)}\xi_{\eps,h}\rangle_{L^2_\eps}|\\
&\qquad\lesssim\eps^2\|v\|_{L^2(K)}+(\eps^2+\eps\|{\bf\Psi}\|_{X_{\eps,h,\theta}})\|{\bf G}_{h,\theta}v\|_{L^2_\eps}\\
&\qquad\lesssim\eps\bigl(\sqrt\eps+\|{\bf\Psi}\|_{X_{\eps,h,\theta}}\bigr)\|{\bf G}_{(\Phi_\eps,A_\eps)}v\|_{L^2_\eps}.
\end{split}
\end{equation*}
This proves \eqref{eq: gauge phase pairing proof}.

\smallskip

Lastly, we show  \eqref{eq: gauge residual injection norm proof} and \eqref{eq: gauge residual positive pairing proof}.
Put $\lambda=\lambda_\eps$ in the compact case and $\lambda=0$ in the noncompact case. Pulling back the pairing in \eqref{eq: dual residual injection} by $\mathscr F_{\eps,h}$, with the transported modes of \Cref{sec: projections linear theory}, cancels the volume Jacobian and gives
\begin{equation*}
\begin{split}
&\|\mathfrak B_{\eps,h,\theta}^*(0,0,0,1)\|_{L^2_\eps}\lesssim\eps^{-1},\\
&\langle\mathfrak B_{\eps,h,\theta}^*(0,0,0,1),{\bf G}_{(\Phi_\eps,A_\eps)}\xi_{\eps,h}\rangle_{L^2_\eps}\\
&\qquad=\operatorname{Vol}(\Gamma)\int_{\R^3}\frac{|{\bf W}_\lambda(y)|^2}{1+|y|^2}\,dy+O(\|{\bf\Psi}\|_{X_{\eps,h,\theta}})+o(1)\\
&\qquad\geq c_0>0.
\end{split}
\end{equation*}
The $O(\|{\bf\Psi}\|_{X_{\eps,h,\theta}})$ error follows by pairing \eqref{eq: gauge phase perturbation bound proof} with the residual injection, using \eqref{eq: gauge residual injection norm proof}. The integral has a uniform positive lower bound for $0\leq\lambda\leq\lambda_*$, by continuity of the phase mode and its nonzero BPS limit, after fixing $\lambda_*$ sufficiently small. Decreasing $m_0$ and then $\eps$ therefore proves \eqref{eq: gauge residual positive pairing proof}.

For compact $M$, take $\chi\equiv1$ and $K=M$. The gauge tests are immediate, and the same local estimates and absorption apply. The phase profile is $q_{\lambda_\eps}$, whose pairing has the uniform positive lower bound just established for $0\leq\lambda\leq\lambda_*$; the gauge-invariant Higgs potential does not alter the identities. This completes the proof.

\end{proof}

\appendix
\section{Several connected components}\label{app: disconnected components}
Let $\Gamma=\bigsqcup_{a=1}^m\Gamma_a$, with $m<\infty$. We outline the changes needed for both admissible ambient classes. Choose disjoint fixed tubes and carry out the approximation and normal linear theory in each, using the same global exterior correction. The additional step is to remove one constant phase residual on each component.

\smallskip\noindent
\emph{Componentwise reduction.}
Retain the constant phases as parameters:
\begin{equation*}
 \theta|_{\Gamma_a}=\vartheta|_{\Gamma_a}+\alpha_a,
 \qquad \int_{\Gamma_a}\vartheta\,d\vol_{\Gamma_a}=0,
 \qquad \alpha=(\alpha_1,\ldots,\alpha_m).
\end{equation*}
Choose cutoffs $\eta_a$ equal to one on the corresponding tubes, with disjoint supports. Place $d\eta_a$ in surrounding collars where $\xi_{\eps,h}=\sigma$, $d_{A_{h,\theta}}\sigma=0$ and $[\Phi_{h,\theta},\sigma]=0$. In \eqref{eq: shifted phase approximation} and \eqref{eq: noncompact shifted approximation}, extend the phase by $\bar\vartheta+\sum_a\alpha_a\eta_a$, where $\bar\vartheta$ is the original extension of its mean-zero part. Replace \eqref{eq: mean zero projection} by
\begin{equation*}
 \left.\Pi_0(g,\varphi)\right|_{\Gamma_a}
 =\left(g|_{\Gamma_a},\varphi|_{\Gamma_a}-\fint_{\Gamma_a}\varphi\,d\vol_{\Gamma_a}\right).
\end{equation*}
Keep the single gauge projection $\Pi_h$ and normalisation \eqref{eq: gauge phase normalization proof}. The complement of finitely many codimension-three submanifolds is connected, so the gauge compactness argument still removes only one global constant.

For each fixed $\alpha$, solve the correction equation and the componentwise reduced equation. The Jacobi and mean-zero Laplace inverses act separately on the components; the other arguments use finitely many disjoint tubes. This gives a corrected pair ${\bf U}_\alpha=(\Phi_{\eps,\alpha},A_{\eps,\alpha})$ with residual
\begin{equation}\label{eq: components residual}
\begin{split}
 {\bf S}_\eps({\bf U}_\alpha)
 &=\sum_{a=1}^m c_a(\alpha){\bf R}_a
 -\eps^{-1}{\bf G}_{{\bf U}_\alpha}\gamma_\alpha,\\
 {\bf R}_a&\coloneqq\mathfrak B_{\eps,h,\theta}^*(0,0,0,\mathbf1_{\Gamma_a}),
 \qquad \gamma_\alpha\coloneqq(I-\Pi_h){\bf G}_{h,\theta}^*{\bf\Psi}(\alpha).
\end{split}
\end{equation}
All parameters here depend on $\alpha$; ${\bf G}_{{\bf U}_\alpha}$ is the gauge operator at the corrected pair.

\smallskip\noindent
\emph{Choosing the constant phases.}
Regard $\alpha$ as a point of $(\R/2\pi\Z)^m$. Increasing $\alpha_a$ by $2\pi$ changes the family by the gauge $\exp(2\pi\xi_{\eps,h})$ near $\Gamma_a$, extended by the identity beyond its collar. This extension is smooth because $\xi_{\eps,h}=\sigma$ and $\exp(2\pi\sigma)=1$ there. The subtraction in the phase ansatz cancels the derivative of the cutoff. Gauge covariance and uniqueness of the contractions therefore make $E_\eps({\bf U}_\alpha)$ periodic. Differentiating the contractions gives a $C^1$ family.

Minimise this energy on the phase torus. To interpret its derivatives, compare $\partial_{\alpha_a}{\bf U}_\alpha$ with $\eta_a{\bf G}_{{\bf U}_\alpha}\xi_{\eps,h}$. After undoing the phase gauge, the remaining parameter dependence lies on the collars and acts through exponentially decaying perpendicular components. The comparison follows by differentiating the two fixed-point equations and using \eqref{eq: gauge exponential decay proof}.

Pair these variations with \eqref{eq: components residual}. The calculation of \eqref{eq: gauge residual positive pairing proof} gives a matrix with positive diagonal leading terms, one from each component. The gauge contribution is handled by the argument for \eqref{eq: gauge phase pairing proof}, retaining the global normalisation. The equations $\partial_{\alpha_a}E_\eps({\bf U}_\alpha)=0$ then remove all the constants. Testing the residual against ${\bf G}_{{\bf U}_\alpha}\gamma_\alpha$ finishes the gauge argument and gives
\begin{equation*}
 c_1=\cdots=c_m=0,\qquad {\bf S}_\eps({\bf U}_\alpha)=0.
\end{equation*}

\smallskip\noindent
\emph{The noncompact case.}
Use the rescaled energy $E_\eps(\Phi,A)=\eps E(\eps^{-1}\Phi,A)$ and keep all phase cutoffs in a fixed compact set. To justify minimisation, apply the weighted scalar and one-form inverses to the exterior fields and their parameter derivatives. Hodge integration by parts on annuli, together with \eqref{eq: noncompact weighted volume}, gives finite energy and allows differentiation under exhaustion. The perpendicular components are treated by \eqref{eq: gauge exponential decay proof}.

The remaining exterior gauge pairing must also be controlled proportionally to $(c_1,\ldots,c_m)$. For this, apply gauge invariance to \eqref{eq: components residual}, solve the resulting gauge equation by scalar Green comparison, and use the global normalisation to determine its constant part. The same weighted-volume assumption makes the annular pairings summable. This supplies the noncompact version of the energy argument, with no additional assumption on the ends. The admissible parameter range and uniform rescaled energy bounds remain unchanged.

\medskip

\noindent{\bf Acknowledgements.}   The author was supported by the European Research Council
under Grant Agreement No.~ 101165368 (MAGNETIC).

\medskip

\noindent{\bf AI usage statement.} ChatGPT-6 Astra was used throughout the preparation of this paper to help organise the arguments and simplify several proofs. Its assistance was particularly valuable in two specific aspects of the analysis:
\begin{itemize}
\item It suggested the dyadic structure of the norms \eqref{eq: source norm} and \eqref{eq: sol norm} in \Cref{sec: norms}, in particular the summation of the longitudinal source contributions over scales in \eqref{eq: source norm}. This choice provides the control needed for Poisson equations with borderline source behaviour. The author claims no credit for this idea.
\item It proposed the strategy adopted in \Cref{sec: reduced equation proof}, combining a spectral decomposition along $\Gamma$ with the adjoint tests in \eqref{eq: reduced adjoint symbols}, constructed from the resolvents
$$
({\bf L}_\lambda+t)^{-1},\qquad t=\eps^2\mu>0,
$$
where $\mu$ is a tangential eigenvalue. These tests recover the coefficients of the reduced equation mode by mode through \eqref{eq: reduced spectral recovery}, substantially simplifying the analysis. The proofs of \Cref{lem: reduced adjoint tests} and \Cref{lem: reduced coefficient operators} were mainly produced by ChatGPT, with minor modifications by the author. Although a more conventional approach based on computations in local coordinates is available, the author adopted the suggested strategy for its clarity and economy.
\end{itemize}
The author takes full responsibility for the mathematical content and final presentation of the paper.

\bibliography{Bib}
\bibliographystyle{siam}

\end{document}